%% file: main.tex
\documentclass{amsart}
\pdfoutput=1
\usepackage[T1]{fontenc}
\usepackage[margin=1in]{geometry}
\usepackage{amssymb, latexsym, amsmath, verbatim, amsthm, amscd}
\usepackage{stmaryrd}
\usepackage{xfrac}
\usepackage{enumitem}
\usepackage{mathtools}
\usepackage{xcolor}
\usepackage{tabularx}
\usepackage{subcaption} 
\usepackage{booktabs}
\usepackage{pinlabel}
\usepackage{graphicx}
\usepackage{tikz}
\usetikzlibrary{calc,positioning}
\usepackage{tikz-3dplot}
\usetikzlibrary{3d}
\usetikzlibrary{decorations.pathreplacing}
\usepackage{microtype}
\usepackage{hyperref}
\allowdisplaybreaks

\title{Hyperbolic models and stability recognition for Coxeter groups}
 
\author{Christopher H. Cashen} 
\address{
  TU Wien\\ Institute of Discrete Mathematics and Geometry\\ Wiedener
  Hauptstrasse~8-10\\ 1040 Vienna \\ Austria\\ \href{https://orcid.org/0000-0002-6340-469X}{\includegraphics[scale=.75]{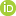}
  0000-0002-6340-469X}}
\email{christopher.cashen@tuwien.ac.at}
\author{Michelle Chu}
\address{University of Minnesota \\ School of Mathematics \\ 206 Church St. SE \\
Minneapolis, MN 55455}
\email{mchu@umn.edu}
\author{Jing Tao} 
\address{University of Oklahoma \\
    Department of Mathematics\\
    601 Elm Avenue Room 423\\Norman, OK 73019}
\email{jing@ou.edu}

\date{\today}
\thanks{This research was supported in part by the Austrian Science
Fund (FWF) \href{https://doi.org/10.55776/PAT7799924}{10.55776/PAT7799924}, NSF DMS-2304920, and NSF DMS-2441034.}

\hypersetup{
    pdftitle={Hyperbolic models and stability recognition for Coxeter groups},    
    pdfauthor={Christopher H. Cashen, Michelle Chu, Jing Tao},     
    pdfkeywords={Coxeter group, Davis complex, Morse quasigeodesic, Morse
  recognition, acylindrically hyperbolic group}, 
    colorlinks=true,       
    linkcolor=black,          
    citecolor=black,        
    filecolor=black,      
    urlcolor=black           
}

\input{preamble}

\begin{document}
\begin{abstract}
\input{abstract}

\end{abstract}

\maketitle

\tableofcontents



\section{Introduction}\label{sec:intro}
\input{mcgfirstintro}

\section{Preliminaries}\label{sec:prelim}
    \input{preliminaries}

\section{Characterizing the Morse property for 
  quasigeodesics}\label{sec:char_morse}
\input{characterization_leadin}

    \subsection{Bridges and divergence}\label{sec:bridge}
    \fullref{sec:cat_zero_bridges} recalls the CAT(0) bridge
    machinery of Bestvina, Kleiner, and Sageev \cite{BesKleSag08}, and
    specializes those results to walls in Davis complexes.
    \fullref{sec:wall_morse} states and proves the main result of this
    subsection, \fullref{thm:wall_morse}, which characterizes the
    Morse property for quasigeodesic paths in terms of dense
    transverse chains. 
    \subsubsection{CAT(0) bridges}\label{sec:cat_zero_bridges}
        \input{bridge}

    \subsubsection{Wall characterization of the Morse property}\label{sec:wall_morse}
        \input{wall_morse}

\subsection{Grids and wide parabolics}\label{sec:grid}
\fullref{sec:grid_tech} describes grids and uses them to produce wide
parabolics.
\fullref{sec:flats} states and proves the main result of this subsection,
\fullref{thm:wide_parallel}, which relates a quasigeodesic's interactions with wide parabolic subcomplexes and
dense transverse chains. 
\subsubsection{Grids}\label{sec:grid_tech}

\input{grid}

\subsubsection{Avoiding wide parabolic subcomplexes}\label{sec:flats}
\input{flats}

\subsection{The characterization theorems}\label{sec:proof}
    \input{proof_main_theorem}

    \section{Hyperbolic models}\label{sec:hyperbolic}
    \input{hyperbolic_intro}

    \subsection{The fine chain metric}\label{sec:fine_chain_metric}
    \input{hyperbolic}

    \subsection{Recognizing Morse quasigeodesics and parabolics}\label{sec:morse_recognition}
    \input{morse_recognition}

    \subsection{The coned-off space}\label{sec:coneoff}\mbox{}
    \input{coneoff}

    \section{Extending the MCG analogy}\label{sec:extend}
    We do not know if the Davis complex of a finite rank Coxeter system is
a hierarchically hyperbolic space.
However, there are similarities.
In \fullref{sec:HHSlike} we develop Coxeter group
counterparts to some of the machinery from the mapping class groups/hierarchically hyperbolic
space (MCG/HHS) settings.
In the following two subsections we use these tools to prove Coxeter
versions of two MCG/HHS theorems:
\begin{itemize}
\item $W\act (\coneoff,\dcone)$ is an acylindrical action, \fullref{sec:acylindrical}.
  \item The stable subgroups of $W$ can be characterized as those that
    are undistorted and purely loxodromic with respect to $W\act
    (\coneoff,\dcone)$, \fullref{sec:stable}.
\end{itemize}

The MCG/HHS tools (see \cite{BehHagSis17,BehHagSis19}) for which we obtain analogues include: one-sided distance
formula, large link condition,
bounded geodesic image, descent by induction on a notion of
complexity, which for us is rank of a parabolic subgroup.
We do not have analogues for hierarchy paths and the two-sided
distance formula.
For the two target applications we do not need the full power
of those tools; where they would normally occur in an MCG/HHS proof we
can replace them with arguments specific to walls and parabolic subgroups in Coxeter groups. 
    \subsection{Parabolic projections and hierarchical features}\label{sec:HHSlike}

\input{projection}

    \subsection{Acylindricity}\label{sec:acylindrical}
    \input{acylindricity}

    \subsection{Dynamical characterization of stability}\label{sec:stable}
    \input{stable}

    \section*{AI use}
    The authors conceived the arguments, developed the proofs, and
    wrote the entirety of this article. Over the course of the project
    several versions of ChatGPT5 were used as interactive sounding boards to explore possible proof strategies and test arguments.



\bibliographystyle{hypersshort}
\bibliography{hyperbolicCoxeter}

\end{document}

%% file: preamble.tex
\theoremstyle{plain}
\newtheorem{theorem}{Theorem}[section]

\newtheorem{lemma}{Lemma}[section]
\newtheorem{proposition}{Proposition}[section]
\newtheorem{corollary}{Corollary}[section]
\newtheorem*{summarytheorem}{Theorem}

\theoremstyle{remark}
\newtheorem{claim}{Claim}[theorem]

\newtheorem*{remark}{Remark}

\theoremstyle{definition}
\newtheorem{definition}{Definition}[section]
\newtheorem{example}{Example}[section]

\def\makeautorefname#1#2{\expandafter\def\csname#1autorefname\endcsname{#2}}
\let\fullref\autoref

\makeautorefname{theorem}{Theorem} 
\makeautorefname{lemma}{Lemma} 
\makeautorefname{proposition}{Proposition} 
\makeautorefname{corollary}{Corollary} 
\makeautorefname{definition}{Definition}
\makeautorefname{example}{Example}
\makeautorefname{section}{Section}
\makeautorefname{subsection}{Section}
\makeautorefname{subsubsection}{Section}
\makeautorefname{claim}{Claim}
\makeautorefname{question}{Question}
\makeautorefname{conjecture}{Conjecture}

\makeatletter 
\let\c@lemma=\c@theorem 
\makeatother
\makeatletter 
\let\c@proposition=\c@theorem 
\makeatother
\makeatletter 
\let\c@corollary=\c@theorem 
\makeatother
\makeatletter 
\let\c@definition=\c@theorem 
\makeatother
\makeatletter 
\let\c@example=\c@theorem 
\makeatother
\makeatletter 
\let\c@conjecture=\c@theorem 
\makeatother
\makeatletter
\@addtoreset{claim}{proposition}
\makeatother
\makeatletter
\@addtoreset{claim}{lemma}
\makeatother
\mathtoolsset{centercolon} 

\newenvironment{claimproof}[1][\unskip]{\vspace{1ex}\noindent{\it
Proof of Claim #1:}\hspace{0.5em}}{\hfill$\lozenge$\vspace{1ex}}

\newcommand{\bdry}{\partial} 
\newcommand{\interior}[1]{\mathring{#1}} 
\newcommand{\act}{\curvearrowright} 
\newcommand{\from}{\colon\thinspace} 
\newcommand{\onto}{\twoheadrightarrow} 
\newcommand{\into}{\hookrightarrow} 

\renewcommand{\setminus}{\smallsetminus}

\usepackage{bbm}
\newcommand{\one}{\mathbbm{1}}

\newcommand{\refl}[1]{\mathfrak{r}_{#1}}

\DeclareFontFamily{U}{mathx}{}
\DeclareFontShape{U}{mathx}{m}{n}{<-> mathx10}{}
\DeclareSymbolFont{mathx}{U}{mathx}{m}{n}
\DeclareMathAccent{\widehat}{0}{mathx}{"70}
\DeclareMathAccent{\widecheck}{0}{mathx}{"71}

\newcommand{\Davis}{\Sigma}
\newcommand{\NR}{X_{\mathrm{N\! R}}}
\DeclareMathOperator{\stab}{Stab} 
\DeclareMathOperator{\fix}{Fix} 
\DeclareMathOperator{\pc}{Pc} 
\DeclareMathOperator{\rank}{rank}
\newcommand{\transverse}{\pitchfork}
\DeclareMathOperator{\bridge}{Bridge} 
\DeclareMathOperator{\diam}{diam} 
\DeclareMathOperator{\Isom}{Isom} 
\newcommand{\nbhd}{\mathcal{N}}
\newcommand{\dihedral}{\mathcal{D}} 

\newcommand{\coneoff}{\widehat{\Davis}}
\newcommand{\dcone}{\hat{d}}

\newcommand{\dcomb}{d_1}
\newcommand{\dcat}{d_2}

\newcommand{\dwall}{d_{w}}
\newcommand{\dchain}{d_{c}}
\newcommand{\dfine}{d_{\mkern-2.75mu f}}

\newcommand{\cleq}{\preceq}
\newcommand{\cgeq}{\succeq}
\newcommand{\ceq}{\asymp}
\newcommand{\cequiv}{\stackrel{+}{\ceq}}
\newcommand{\centralizer}{Z}

\newcommand{\normalizer}{N}

\newcommand{\maxsphere}{\Delta}
\newcommand{\radius}{\mathfrak{R}}
\newcommand{\caprace}{\ell}
\newcommand{\length}{\mathfrak{L}}
\newcommand{\density}{\rho_0}
\newcommand{\Epsilon}{E}
\newcommand{\wallgrowth}{\omega}
\newcommand{\gate}{\mathfrak{p}}
\newcommand{\walls}{\mathcal{M}}
\newcommand{\shortcut}[1]{\widehat{#1}}
\newcommand{\greedy}[1]{#1^*}
\newcommand{\realgreedy}[1]{\greedy{\overline{#1}}}
\DeclareMathOperator{\FP}{FP}
\DeclareMathOperator{\collapse}{Collapse}
\DeclareMathOperator{\factorsum}{FactorSum}
\newcommand{\fibers}{\stab^{-1}}

%% file: abstract.tex
Given a finite rank Coxeter system, we construct a hyperbolic space by
coning off the wide parabolic subcomplexes of its Davis complex.
The resulting space is ``stability recognizing'', in the sense that stable
subspaces of the Davis complex are exactly the quasigeodesically
connected subspaces whose images in the coned-off space are
quasiisometrically embedded.
Equivalently, the coned-off space is ``Morse recognizing'': a
quasigeodesic in the Davis complex is Morse if and only if it is
quasigeodesic in the coned-off space.

Similar results are true for the mapping class group of a hyperbolic
surface acting on its curve graph. We strengthen this analogy with
additional results.
First, the action of the Coxeter group on the coned-off space is
acylindrical and  gives the largest acylindrically hyperbolic
structure for the Coxeter group.
Second, the stable subgroups of a Coxeter group are characterized as
the finitely generated subgroups that are undistorted and purely
loxodromic with respect to the action on the coned-off space.

The coned-off space, being constructed directly from the Coxeter structure,
facilitates concrete characterizations of Morseness and stability
for two important Coxeter-native families of objects:
combinatorial geodesics and special subgroups.

%% file: mcgfirstintro.tex
This paper is about the coarse geometry of finitely generated Coxeter
groups, as revealed by their wall structure and parabolic subgroup
structure. 
The CAT(0) geometry of a finite rank Coxeter system $(W,S)$ is described
by its Davis complex $\Davis$, which admits a piecewise Euclidean,
CAT(0) metric $\dcat$.
Our goal is to understand the large scale hyperbolic aspects of the
geometry, which include the \emph{Morse property} and \emph{stability}
for subspaces.
The Morse property requires that a quasigeodesic with both endpoints in
the subspace stays close to the subspace. A subspace is stable if it
contains uniformly Morse quasigeodesics between every pair of points. 
In hyperbolic spaces these properties are equivalent to each other and to quasiconvexity.
A geodesic metric space is hyperbolic
 if and only if its geodesics are uniformly Morse, so we
 think of identifying the Morse geodesics as identifying the
 ``hyperbolic directions'' in the space.  

The main thesis of this paper is that hyperbolic aspects of Coxeter groups can be described and studied in close analogy with those of the mapping class group (MCG) of a hyperbolic surface.

We recall several features of MCGs that motivate our Coxeter group constructions. 
MCG has a geometric action
on its marking complex and a proper, but not cocompact, action on its
Teichm\"uller space. 
Work of Minsky shows Teichm\"uller geodesics are Morse if
and only if they avoid the thin parts of Teichm\"uller space \cite{Min96},
which have approximate product structures \cite{Min96b}, making them the obvious
sources of non-hyperbolicity.
Masur and Minsky \cite{MasMin99, MasMin00} showed that the curve graph for the surface is a hyperbolic graph that is
quasiisometric to the space obtained by coning off the thin parts of
Teichm\"uller space and to the space obtained from the marking complex
by coning off subcomplexes corresponding to cosets of curve
stabilizers, one for
each orbit of essential simple closed curve in the surface.
The elements of MCG that act loxodromically on the curve graph are
precisely the pseudo-Anosov mapping classes.
Bowditch \cite{Bow08} proved that the action of MCG on
its curve graph is acylindrical.
Abbott, Balasubramanya, and Osin \cite{AbbBalOsi19} showed that this action
gives the largest acylindrically hyperbolic structure on MCG.
Kent and Leininger \cite{KenLei08} proved a result that, in an
equivalent formulation established by Durham and Taylor \cite{DurTay15}, says that the curve graph is a
stable subgroup recognizing space for MCG, in the
sense that a subgroup of MCG is stable if and only
if the orbit map into the curve graph is a quasiisometric embedding.
Bestvina, Bromberg, Kent, and Leininger \cite{BBKL20}  gave  an
element-wise, dynamical characterization of stability: an undistorted subgroup of
MCG is stable if and only if its action on the curve graph is purely
loxodromic, that is, all infinite order elements act loxodromically.

\smallskip

In the Coxeter group setting, the \emph{wide parabolic subgroups} are the obvious
sources of non-hyperbolicity. 
Here, ``wide''  means the canonical direct product splitting
either contains multiple infinite factors or contains a factor that
is of higher rank affine type.
Wide parabolic subgroups are not hyperbolic, and Moussong's characterization of hyperbolicity in Coxeter groups \cite{Mou88,Mol25} can be rephrased:
$W$ is hyperbolic if and only if it has no wide parabolic subgroups.

We construct a geodesic hyperbolic metric space $(\coneoff,\dcone)$,
the \emph{coned-off space}, by coning off subcomplexes of the Davis
complex $\Davis$ corresponding to cosets of wide parabolic subgroups.

\begin{summarytheorem}[{\fullref{cor:coneoff_is_morse_recognizing} and \fullref{prop:recognizing_is_recognizing}}]
  There is a hyperbolic, geodesic metric space
  $(\coneoff,\dcone)$ obtained by coning off the wide parabolic
  subcomplexes of $\Davis$.
  Moreover, $(\coneoff,\dcone)$ is ``stability recognizing'', in the
  sense that a quasigeodesically connected subspace of $\Davis$ is stable if and
  only if its natural map into $\coneoff$ is a quasiisometric
  embedding. 
  Equivalently, $(\coneoff,\dcone)$ is ``Morse recognizing'', in the sense
  that a quasigeodesic $\gamma\from I\to (\Davis,\dcat)$ is Morse if
  and only if its natural composition into $\coneoff$ is also a quasigeodesic.
\end{summarytheorem}

An important ingredient in the ``Morse recognizing'' property is a complete
characterization of the Morse property for a restricted class of
quasigeodesic paths, \fullref{thm:morse_characterization},
which is of independent interest, but whose statement is long and technical, so we do not
detail it here.  

Niblo and Reeves constructed a CAT(0) cube complex
$\NR$ whose wall structure is isomorphic to that of $\Davis$ \cite{NibRee03}.
The action of $W$ on $\NR$ is proper, but is cocompact
if and only if $W$ has no irreducible higher rank affine parabolic subgroups \cite{CapMuh05}.
Still, our proof that $(\coneoff,\dcone)$ is hyperbolic proceeds through the wall structure. We first
construct a non-geodesic, quasiruled hyperbolic metric $\dfine$ in terms of the wall system,
so it applies equally well to $\NR$ and to $\Davis$, and then show that $(\coneoff,\dcone)$ and $(\Davis,\dfine)$
are quasiisometric, \fullref{thm:coneoff_qi_to_dfine}.
We observe that the metric $\dfine$ collapses the difference between $\Davis$ and $\NR$,
\fullref{prop:Davis_vertices_coarsely_dense_in_NR_dfine}, just as passage to the curve graph collapses the difference
between the marking complex and Teichm\"uller space.

The cone-off construction introduces infinite valence vertices with
infinite $W$--stabilizers.
A weakening of the notion of a proper group action appropriate to such
settings is acylindricity. 

\begin{summarytheorem}[{\fullref{thm:acylindricity}}]
    The action $W\act (\coneoff,\dcone)$ is acylindrical. 
\end{summarytheorem}
Moreover, this acylindrical action is \emph{universal}, \fullref{cor:universal}, meaning that
$\coneoff$ sees all \emph{generalized loxodromics} in the sense of
Osin \cite{Osi16}.
In fact, we prove a stronger statement: 
\begin{summarytheorem}[{\fullref{thm:largest}}]
  $\coneoff$ defines the largest acylindrically hyperbolic structure
  for $W$, in the sense of Abbott-Balasubramanya-Osin \cite{AbbBalOsi19}.
\end{summarytheorem}

Finally, as in the MCG case, stability detection in terms of a
uniform subgroup-level condition can be replaced by a dynamical
element-wise condition:
\begin{summarytheorem}[{\fullref{thm:chara_stable}}]
  A finitely generated subgroup of $W$ is stable if and only if it  is
  undistorted and purely loxodromic for the action $W\act\coneoff$.
\end{summarytheorem}

\subsection{Benefits of the Coxeter-native construction}
The construction of the hyperbolic \emph{fine chain metric} $\dfine$ is a mild
variant of a construction of Genevois \cite{Gen19}.
Genevois actually defines a parameterized family of hyperbolic metrics
on a CAT(0) cube complex, such that quasigeodesics with a common Morse
gauge are eventually recognized in a hyperbolic space with large enough
parameter. 
Notably, our result says that for a given Coxeter group there is a
single parameter at which we can already recognize all Morse quasigeodesics, so we only need one
specific hyperbolic metric, not a parameterized family or a limit of such a family.
There are several recent results 
\cite{PetSprZal24,PetZal24,Zbi24} extending Genevois's idea  by
constructing, in great generality, 
hyperbolic spaces that detect strongly contracting/rank-one/Morse
behavior.
These results provide an abstract hyperbolic Morse recognizing space
for $\Davis$. 
Our approach is complementary.
We transfer hyperbolicity of $(\Davis,\dfine)$ to the coned-off space, which is Coxeter-native, in the sense of being
defined directly in terms of the Coxeter structure. 
This is important for Theorems~\ref{thm:acylindricity}, 
\ref{thm:largest}, and \ref{thm:chara_stable}, whose proofs
take place in $(\coneoff,\dcone)$, but it has additional direct
Coxeter group payoffs---it enables us to give concrete characterizations of the Morse and
stable properties for combinatorial geodesics and
special subgroups. 

 \begin{summarytheorem}[{\fullref{main_for_combinatorial}}]
    There are effective quantitative equivalences between the
    following items, uniformly over all combinatorial geodesics $\gamma$ in $\Davis$.
  \begin{enumerate}[label=(\arabic*)]
  \item There exists $\mu$ such that $\gamma$ is $\mu$--Morse.\label{introcomb:morse}
    \item There is an upper bound on the lengths of subsegments
      $\gamma'$ of $\gamma$ for which there exists a wide parabolic
      subcomplex $\Upsilon$ with
      $\gamma'\subset\Upsilon$.\label{introcomb:parabolic_wide}
      \item There is an upper bound on the lengths of subsegments
        $\gamma'$ of $\gamma$ for which there exists 
        $T\subset S$ such that $W_T$ is a wide special
        subgroup and all edges of $\gamma'$ have label
        in $T$.\label{introcomb:special_wide}
  \end{enumerate}
  \end{summarytheorem}
Cordes and Levcovitz proved the special case of \fullref{main_for_combinatorial}
for Coxeter groups that are ``affine-free'', meaning they have no
irreducible higher rank affine parabolic subgroups \cite[Theorem~D]{CorLev25}.
They conjectured the result should be true
without the affine-free hypothesis.
\fullref{main_for_combinatorial} confirms that conjecture.

\begin{summarytheorem}[{\fullref{thm:characterize_morse_stable_for_parabolics}}]
  For $T\subset S$, let $\Gamma_T$ be the subdiagram of the Coxeter
  diagram of the finite rank Coxeter system $(W,S)$ spanned by the vertices $T$. 
  Call $\Gamma_T$ ``spherical'' if $W_T$ is spherical.
  
$W_T$ is stable if and only if for every wide
special subgroup $W_U$ the diagram $\Gamma_{T\cap U}$ is spherical. 
  
$W_T$ is Morse if and only if for every wide
special subgroup $W_U$, one of the following is true:
  \begin{enumerate}
  \item The intersection of every connected component of $\Gamma_U$ with
    $\Gamma_T$ is spherical. 
    \item Every nonspherical connected component of $\Gamma_U$ is
      contained in $\Gamma_T$.
  \end{enumerate}
\end{summarytheorem}
\fullref{thm:characterize_morse_stable_for_parabolics} generalizes the corresponding results from the right-angled case, compare \cite[Theorem~1.11]{Tra19}, \cite[Proposition~4.9]{Gen19},
\cite[Theorem 7.5]{RusSprTra23}.

\subsection{Why is this strategy successful for Coxeter groups?}

There are two essential properties of Coxeter groups that enable the
comparison to MCGs:
\begin{itemize}
\item Parabolic subcomplexes are gated.
  \item There is a local-to-global result relating the wall structure
    to the parabolic subgroup structure.
\end{itemize}

The first property means that the combinatorial closest point projection
in the 1-skeleton of $\Davis$ is a gate map when the target is a
parabolic subcomplex. This gives strong versions of projection and
parallelism between parabolic subcomplexes that are very similar to the
situation for convex subsets of CAT(0) cube complexes/median graphs. This is a special
feature of Coxeter groups that arises from the combinatorial group
theory of reflection groups, not from CAT(0) polyhedral geometry.
Indeed, even the CAT(0) geodesics of $\Davis$ are not gated  in this
way.
This gate property is what makes our version of subsurface projection
viable.

The second property is a result of Caprace \cite{Cap06}, recalled in
\fullref{caprace}, that roughly says that a sufficiently complicated local
pattern of wall intersections can only arise if the reflections in
those walls are contained in a wide parabolic subgroup. 
This is a deep result proved using the Tits symmetric bilinear form
on the canonical root system for $(W,S)$, so it is very much a
Coxeter-specific result. 
It is another case in which the effective behavior of the Coxeter
group is like a cube complex, even when $\Davis$ is not cubical;
specifically, in a CAT(0) cube complex large ``grids'' composed of
transverse chains of walls identify flat/product-like behavior \cite{Hag14,Gen19}.
In Coxeter groups two chains of walls can be transverse but not
intersect at right-angles, so we do not immediately get a product-type
parabolic subcomplex, but Caprace's result says that the only thing
that can go wrong is that we might instead get an affine-type parabolic
subcomplex. 
Both types of subcomplexes are wide, so our cone-off construction is
collapsing exactly the regions where the wall structure gets
complicated, leaving us to work with simpler wall combinatorics to
understand large distances in $\coneoff$.

\smallskip 

Beyond mapping class groups, stable subgroup recognition
theorems are known for 
groups hyperbolic relative to peripherals with linear divergence
\cite{AouDurTay17}, right-angled Artin groups (RAAGs) \cite{KobManTay17},
and some hierarchically hyperbolic groups (HHGs) \cite{AbbBehDur21}.
HHGs are a common generalization of toral relatively hyperbolic
groups, MCGs, RAAGs, and right-angled Coxeter groups (RACGs) \cite{BehHagSis17,BehHagSis19}.
Our coned-off space coarsely agrees with the construction of a
universal acylindrical hyperbolic space for RACGs by Abbott, Behrstock, and
Durham \cite{AbbBehDur21}.
However, not all Coxeter groups are HHGs \cite{PetSpr23,AzuHag}\footnote{Specifically, it follows from
  \cite[Corollary~7.11]{AzuHag} that if $(W,S)$ is a Coxeter system
  such that there is a $T\subset S$ with $T^\perp$ spherical and
  $T$ is affine of type $\tilde{A}_2$, $\tilde{A}_n$ for $n\geq 4$,
  $\widetilde{E}_6$, $\widetilde{E}_7$, $\widetilde{E}_8$, $\widetilde{F}_4$, or
  $\widetilde{G}_2$, then $W$ is not a hierarchically hyperbolic
  group.
It turns out, applying work of Hagen \cite{Hag14crystal}, that the
irreducible higher rank affine Coxeter groups of types
$\tilde{A}_3$, $\tilde{B}_n$, $\tilde{C}_n$, and
$\tilde{D}_n$ actually are cocompactly cubulated, even though their
actions on their Niblo-Reeves cube complexes are not
cocompact.\label{fn:azuhag}}, so simply
being an HHG is not the mechanism for our results.
The novelty of the approach in this paper is to handle the case of general
Coxeter groups, including those with affine parabolics.

It is not known if general Coxeter groups act geometrically on a
hierarchically hyperbolic space (HHS), and Coxeter groups are
conjectured to be HHGs if and only if they are cocompactly cubulated
\cite{HHSProblemList}. 
It seems that any attempt to build an HHS structure for arbitrary
Coxeter groups would require some choices in how to decompose affine
parabolics, but the existence of non-HHG Coxeter group examples shows that such choices cannot always be made in a $W$-invariant way.
We take a different approach. While we do develop a partial MCG/HHS-like 
toolbox for Coxeter
groups, we do not have the full HHS suite. For example, we do not
have hierarchy paths or a two-sided distance formula.
Moreover, there are some differences in our approach from the typical MCG/HHS
constructions.
An example is that in the MCG setting, if $X\subsetneq Y\subsetneq Z$ are (nonannular)
nested subsurfaces, they all have curve graphs, defined in the same
way, and progress supported
entirely in $X$ is bounded in both curve graphs $C(Y)$ and $C(Z)$. 
In our Coxeter group setting, the coned-off graph for a
proper parabolic subgroup is obtained by coning off all proper
parabolic cosets, rather than only the wide ones, as in the case of the parent.
The construction must be different because the two coned-off graphs play
different roles in the proofs; coning off wide parabolics for $W$
yields hyperbolicity, but for nested parabolics $Q\lneq P\lneq W$,
progress supported in $Q$ can still be large in $\coneoff$ when $Q$
and $P$ are not wide. Coning off all proper subparabolics of
$P$ permits an inductive descent on rank detecting whether
progress in the coned-off graph of $P$ is genuinely a $P$ phenomenon
or is actually supported in a deeper parabolic.
When $P$ is irreducible affine it has no proper nonspherical parabolic
subgroups, so its subsidiary coned-off graph has the same
quasiisometry type as $P$.
In particular, the higher rank affine parabolics are not hyperbolic,
so their coned-off graphs cannot be
part of an HHS structure.

\subsection{Connections to other results in the literature}
Caprace and Fujiwara \cite[Proposition~4.5]{CapFuj10}, building on
work of Caprace and Haglund \cite{CapHag09}, characterized
\emph{rank-one} elements of $W$.
Their result is essentially the case of
\fullref{thm:morse_characterization} in which a group element acts loxodromically on $\Davis$
with a Morse axis.
It is not clear how to prove our \fullref{thm:morse_characterization} directly from their
result.
A rank-one geodesic in a CAT(0) space is one that does not bound a half-plane \cite{BalBri95}. 
For a periodic geodesic in a proper CAT(0) space this is equivalent to
being strongly contracting \cite[Theorem~5.4]{BesFuj09}, but this fact relies on periodicity;
arbitrary geodesics may encounter an irregular collection of non-hyperbolic regions,
as opposed to a single half-plane. 

Recent work of Ciobanu and Genevois \cite[Theorem~5.1]{CioGen} showed that the
Caprace-Fujiwara rank-one elements actually have strongly contracting
axes in $(\Davis^{(1)},\dcomb)$. Even more recently, work of
Dru\c{t}u, Spriano, and Zbinden \cite[Corollary~1.13,
Example~G.3]{DruSprZbi},  showed that Morseness and strong contraction
are equivalent for quasigeodesics in $(\Davis^{(1)},\dcomb)$.
We extensively use the corresponding equivalence in $(\Davis,\dcat)$,
which follows from the CAT(0) geometry. However, the two conclusions
are separate; strong contraction need not be preserved by
quasiisometries \cite{ArzCasGrua}. This is also not an HHG fact; Rafi
and Verberne constructed an example of a pseudo-Anosov element whose
quasiaxis in a Cayley graph of its mapping class group is not strongly
contracting \cite{RafVer21}.

Behrstock and Charney \cite{BehCha12} classified rank-one elements in
RAAGs.
Using similar techniques, 
Charney and Sultan \cite{ChaSul15} gave a characterization of Morse
geodesic rays in uniformly locally finite CAT(0) cube complexes that is very similar
to \ref{main:morse}$\iff$\ref{main:bounded_width_chain} of our \fullref{thm:morse_characterization}, except that
they require control on a chain of walls boundedly spaced along the ray.
In the non-right-angled Coxeter group setting, this is too much to ask, so instead we weaken
that condition to a lower bound on chain density. 
\fullref{sec:wall_morse} is largely an adaptation of Charney and
Sultan's construction to Davis complexes in a way that relates the
Morse property to wide parabolics. 
The divergence argument in  \fullref{sec:wall_morse} is essentially
the same as that of Cordes and Levcovitz \cite{CorLev25}, but we need to expose additional characterizations of Morseness in terms of the
wall structure to interface with later sections.
\fullref{sec:bridge} does not really use the Coxeter structure in an essential
way; it is a geometric argument about discrete wall systems in a
CAT(0) space.

Hagen's contact graph for CAT(0) cube complexes \cite{Hag14} and 
Kim and Koberda's extension graph for RAAGs \cite{KimKob13,KimKob14}
were curve graph counterparts in MCG-analogous
structure theories for their respective groups, prior to the
unification of MCGs, RAAGs, and other cubical groups as HHGs. 

Outside the HHS paradigm, the free factor complex and free splitting
complex provide curve graph surrogates for the outer automorphism group
of a free group in MCG-inspired approaches to the study of its
large-scale geometry \cite{HatVog98, HanMos13, BesFei14,HilHor17},
and, recently, Kudlinska and Petyt \cite{KudPet25} constructed a Morse
recognizing hyperbolic space for free-by-cyclic groups by coning off maximal
product subgroups.

\subsection{Plan of the paper}
\fullref{sec:prelim} is a preliminary section collecting existing
results on the structure of Coxeter groups, coarse geometry,
non-geodesic spaces, and variations of the Morse property.

\fullref{sec:char_morse} works towards the characterization,
\fullref{thm:morse_characterization}, of the
Morse property for a restricted class of quasigeodesic paths;
specifically, paths for which the inverse image of their intersection
with any wall is a connected set in the parameter space.
Both combinatorial and CAT(0) geodesics have this property. 
\fullref{sec:bridge} recalls the construction of bridges in CAT(0)
spaces and uses them to characterize Morse quasigeodesics
in terms of their interaction with the wall structure.
In \fullref{sec:grid} we use grids, as in \cite{CapMar13}, to make the
translation between conditions on sets of walls and the existence of
wide parabolics.
A main contribution of this section, \fullref{prop:induction_argument}, is an inductive argument showing
that when a quasigeodesic has a sufficiently long interaction
with a grid, then it has a long subsegment contained in a fixed radius
neighborhood of some wide parabolic subcomplex, possibly different
than the one associated with the original grid.

The proof of \fullref{thm:morse_characterization} in
\fullref{sec:proof} is largely a formal combination of the results of the two preceding subsections.
The reinterpretation of this result in terms of edge labels in
\fullref{main_for_combinatorial} requires additional work using
Coxeter-theoretic facts. 

In \fullref{sec:hyperbolic} we construct the fine chain metric and
work out its properties. Then in \fullref{sec:coneoff} we prove
$\Davis$ with the fine chain metric is quasiisometric to the coned-off
space. Roughly, the proof compares, quantitatively, the ways in
which combinatorial geodesics fail to make progress in the
metrics $\dfine$ and $\dcone$ during intervals that they spend in a
wide parabolic subcomplex. 

\fullref{sec:HHSlike} builds out the dictionary of
\fullref{tab:coxeter-mcg-dictionary-precise}.
These tools are used in \fullref{sec:acylindrical} and
\fullref{sec:stable} to prove acylindricity and the characterization
of stability in terms of being purely loxodromic, respectively.
These proofs are remarkably similar in structure to their MCG/HHS
predecessors, \cite[Theorem~14.3]{BehHagSis17} and \cite[Main
Theorem]{BBKL20}, but both require significant Coxeter input to
replace parts of the proofs where the relevant MCG/HHS machinery is missing.

\begin{table}[htbp!]
    \centering
    \renewcommand{\arraystretch}{1.3}
    \begin{tabular}{p{0.44\textwidth} | p{0.44\textwidth}}
        \hline
        \textbf{Mapping class group $\text{MCG}(S)$} & \textbf{Coxeter system $(W,S)$} \\
        \hline

        Teichm\"uller space $\mathcal{T}(S)$ & Niblo-Reeves cube complex $\NR$\\[2pt]

         Marking graph $\mathcal{M}(S)$
        & Davis complex $\Davis=\Davis(W,S)$
        \\[2pt]
        
         Curve graph $\mathcal{C}(S)$
        & Universal coned-off space $\coneoff$
      \\[2pt]

       Proper subsurface $Z$, its curve graph $\mathcal{C}(Z)$, $d_{\mathcal{M}(Z)}$, and $d_{\mathcal{C}(Z)}$
        & Pre-wide  parabolic $P$, its coned-off parabolic graph
          $C(P)$, $d_P$, and $\delta_P$
        \\[2pt]

        Reducible element preserving $Z$
        & Element of the normalizer $\normalizer_W(P)$ of $P$
        \\[2pt]
        
        Complexity of $Z$
        & Rank of $P$
        \\[2pt]
        
        Purely pseudo-Anosov
        & Purely $\coneoff$-loxodromic
        \\[2pt]
        
        Convex cocompact
        & Stable
        \\
        
        \hline
    \end{tabular}
    \caption{MCG/Coxeter dictionary}
    \label{tab:coxeter-mcg-dictionary-precise}
\end{table}

%% file: preliminaries.tex
Let $(W,S)$ be a finite rank Coxeter system, meaning $W$ is a Coxeter
group with a finite set $S$ of fundamental generators.
A Coxeter system $(W,S)$ determines a \emph{Coxeter diagram} that
encodes a group presentation for $W$.
The vertices of the Coxeter diagram are in bijection with $S$, whose
elements are the generators of $W$, each of order 2.
When $s$ and $t$ are distinct elements of $S$ and $st$ has order $m\in\{2,3,\dots\}\cup\{\infty\}$ in $W$, then the Coxeter diagram
has no edge between $s$ and $t$ when $m=2$ and has an edge labelled
$m$ between $s$ and $t$ otherwise. (By convention, the label 3 is
usually not written.)

A subset $T\subset S$ determines a \emph{special subgroup} $W_T$ of $W$
generated by $T$ that is isomorphic to the Coxeter group defined by
the sub-Coxeter system involving only the generators in $T$.
The special subgroups and their conjugates are called \emph{parabolic
  subgroups}.
If $P$ is a parabolic subgroup conjugate to a special subgroup $W_T$
then $|T|$ is the \emph{rank} of $P$, and $P$ is called
\emph{reducible} if there exists a nontrivial decomposition $T=T_0\sqcup T_1$
such that there are no edges in the Coxeter diagram between vertices
of $T_0$ and vertices of $T_1$. This implies $W_T=W_{T_0}\times
W_{T_1}$.
The parabolic is \emph{irreducible} if it is not reducible.

Conceivably, a parabolic subgroup $P$ could be conjugate to two different special subgroups
$W_T$ and $W_U$, but when this happens there is a
conjugation in $W$ taking the set $T$ bijectively to the set $U$
\cite[Proposition~4.5.10]{davisbook}.
Since the edge labels of the Coxeter diagram are determined by the
orders of products of generators, this further implies that the
subdiagram spanned by vertices of $T$ is isomorphic as a Coxeter
diagram to the subdiagram spanned by vertices of $U$. 
Thus, the rank and reducibility of $P$ are well defined, independent
of which special subgroup is taken as a representative of its
conjugacy class.

The intersection of two parabolic subgroups is a parabolic subgroup of
$(W,S)$, and also of the induced Coxeter systems of the two parabolic subgroups
\cite[Lemma~5.3.6]{davisbook}.
In particular, if $P\lneq Q$ are parabolic subgroups then
$\rank(P)<\rank(Q)$.

A Coxeter group or one of its parabolic subgroups is called \emph{spherical} when it is a finite group.

We use the same adjectives for $T\subset S$ as for the corresponding
special subgroup $W_T$; ie, ``$T$ is spherical'' means $W_T$ is
spherical.

Associated to a Coxeter system, there is a canonical 
decomposition of $W$ as a product of its irreducible special
subgroups, $W=\prod_i W_{T_i}$, where the $T_i$ are the
vertex sets of the connected components of the Coxeter diagram.
For our purposes it will often be convenient to combine the spherical
factors, so $W=W_{T_1}\times\cdots\times W_{T_n}\times W_U$, where the
$T_i$ are the nonspherical irreducible subsets of $S$ and  $U$ is the
union of the spherical irreducible subsets of $S$. Since parabolic
subgroups inherit a Coxeter system from $(W,S)$, we also get canonical
product decompositions of a parabolic subgroup into irreducible
parabolics/nonspherical irreducible parabolics plus one spherical factor.

The  \emph{Davis complex} $\Davis=\Davis(W,S)$ of the Coxeter system
is a cell complex whose cells correspond to cosets of spherical
special subgroups.
Its $n$--skeleton is denoted $\Davis^{(n)}$. 
It admits a piecewise Euclidean, locally finite, CAT(0) polyhedral
structure,  constructed as in
\cite[Section~7.3 and Section~12.1]{davisbook}.
We normalize so that
edges have length 1 (by taking $\mathbf{d}\equiv 1/2$ in
the construction of  \cite[Section~12.1]{davisbook}).
With this normalization, the geometry of $\Davis$ is completely
determined by the Coxeter system $(W,S)$.

For background on Coxeter groups, see
\cite{davisbook,Bou02,Hum90}.
For background on CAT(0) metric spaces, see \cite{BriHae99,davisbook}.

\medskip

For the remainder of the paper we regard $(W,S)$ and $\Davis$ as
fixed, so terms referred to as ``constants'', e.g. in
\fullref{constants}, are constants with respect to this fixed
Coxeter system and choice of metric on $\Davis$.

\subsection{Coxeter groups and walls}
\subsubsection{Background on Coxeter groups}\label{sec:coxeter}
\input{Coxeter_background}

\subsubsection{Parabolic centralizers and normalizers}
\input{parabolicnormalizer}

\subsubsection{Caprace toolbox}\label{caprace}
We state some results of Caprace that we will use.
\begin{theorem}[{\cite[Theorem~8]{Cap06}}]\label{Caprace_8}
  There is a constant $N$ such that the following holds.
  Suppose  $\mathcal{K}=\{K_i\}_{0\leq
    i\leq n}$ is a chain and $M$ and $M'$ are walls that are each
  transverse to $\mathcal{K}$ and such that 
  $\emptyset\neq M\cap M'\subset K_0$.
  If $n>N$ then $W(\mathcal{K}\cup\{M,M'\})$ is a Euclidean triangle
  group that is contained in a higher rank irreducible affine
  parabolic.
\end{theorem}

\begin{lemma}[{\cite[Lemma~11]{Cap06}}]\label{Caprace_11}
  Let $\mathcal{M}$ be a set of walls with
  $W(\mathcal{M})\cong\dihedral_\infty$.
  If a wall $M$ is transverse to at least 8 walls of $\mathcal{M}$
  then it is transverse to all of them and either $\refl{M}$
  centralizes $W(\mathcal{M})$ or $W(\mathcal{M}\cup \{M\})$ is a
  Euclidean triangle group.
\end{lemma}

\begin{proposition}[{\cite[Proposition~16]{Cap06}}]\label{Caprace_16}
  If $R$ is a higher rank irreducible affine reflection subgroup then
  so is its parabolic closure.
\end{proposition}

\begin{lemma}[{\cite[Lemma~17]{Cap06}}]\label{Caprace_17}
  Let $\mathcal{K}$ be a chain containing at least 2 walls.
  Let $M$ be a wall. If one of the following is true then
  $\refl{M}\in\pc(\mathcal{K})$.
  \begin{itemize}
  \item $M$ separates two walls of $\mathcal{K}$.
  \item $W(\mathcal{K}\cup M)\cong\dihedral_\infty$
    \item $W(\mathcal{K}\cup M)$ is a Euclidean triangle subgroup.
  \end{itemize}
\end{lemma}

\subsection{Metric geometry}
\subsubsection{Coarse geometry}
Let $(X,d)$ be a metric space.
For $x\in X$ and  $Y\subset X$, $d(x,Y):=\inf_{y\in Y}d(x,y)$.
For $R\in\mathbb{R}$, the \emph{$R$--neighborhood} of $Y$ is
$\nbhd_R(Y):=\{x\in X\mid d(x,Y)<R\}$, and the \emph{closed
  $R$--neighborhood} is  $\bar\nbhd_R(Y):=\{x\in
X\mid d(x,Y)\leq R\}$.

Subsets $Y$ and $Z$ of $X$ are \emph{coarsely equivalent} if they have
finite Hausdorff distance, i.e. $\inf\{R\mid Y\subset\bar\nbhd_R(Z)\text{
  and }Z\subset\bar\nbhd_R(Y)\}$ is finite.

For functions $\phi$ and $\psi$ from  $X^n$ to $\mathbb{R}$, write $\phi\cleq\psi$ if
there exist $\lambda$ and $\epsilon$ such that for all
$\underline{x}\in X^n$ we have $\phi(\underline{x})\leq
\lambda\psi(\underline{x})+\epsilon$.
Write $\phi\stackrel{+}{\cleq}\psi$ if $\phi\cleq\psi$ with $\lambda=1$.
Write $\phi\ceq\psi$ if $\phi\cleq\psi$ and $\psi\cleq \phi$, and write
$\phi\cequiv\psi$ if $\phi\stackrel{+}{\cleq}\psi$ and
$\psi\stackrel{+}{\cleq}\phi$.

A \emph{quasiisometric embedding} is a map $\phi\from X \to Y$ such
that for $\phi^*(d_Y)(x,x'):=d_Y(\phi(x),\phi(x'))$ we have $d_X\ceq
\phi^*(d_Y)$ as maps on $X^2$.
It is a \emph{quasiisometry} if it is a quasiisometric embedding and
$\phi(X)$ is coarsely equivalent to $Y$.
It is a \emph{biLipschitz embedding} if it is a quasiisometric
embedding with 0 additive error.

A \emph{path} is a continuous map from an interval of $\mathbb{R}$.
A \emph{geodesic} is an isometric embedding of an interval of
$\mathbb{R}$, which is a segment, ray, etc according to the
corresponding property of the domain.
A \emph{quasigeodesic} is a quasiisometric embedding of
an interval of $\mathbb{R}$. A $(\lambda,\epsilon)$--quasigeodesic is
a quasigeodesic such that the multiplicative and additive errors of
the quasiisometric embedding are $\lambda$ and $\epsilon$,
respectively.

A quasigeodesic need not be a path, but there is a standard connect-the-dots trick,
\cite[III.H.1.11]{BriHae99}, that replaces a quasigeodesic in a
geodesic metric space by a piecewise geodesic quasigeodesic path at bounded Hausdorff distance from the
original, where the Hausdorff distance and the new quasigeodesic
constants depend only on the original quasigeodesic constants.

 \subsubsection{Distances}
 \input{distance}

\subsubsection{Constants}\label{constants}
We fix some constants of the Coxeter system
$(W,S)$ and the choice of $\Davis$.

Let $\maxsphere$ be the maximal $\dcomb$--diameter of a spherical
special subgroup of $(W,S)$.

Let $\wallgrowth(r)$ be the maximum number of walls that intersect any
closed $\dcat$--ball of radius $r$ in $\Davis$. This is well-defined, by local
finiteness of the wall structure and cocompactness of $W\act \Davis$.

Let $\caprace:=\max\{N+2,8\}$, where $N$ is the constant of
\fullref{Caprace_8} (and 8 is the constant of \fullref{Caprace_11}).

Take $\radius$ large enough that for all $x\in\Davis$ and $\Upsilon$ a parabolic
subcomplex, \fullref{lem:distance_to_standard} implies:
\[\dcat(x,\Upsilon)\geq \radius\implies \dchain(x,\Upsilon)\geq 4\caprace-1\]

Next we define functions $\density$ and $\length$ that take as input
quasigeodesic parameters and carry information about chain density
along a quasigeodesic. Suppose that \fullref{all_metrics_are_equivalent} gives $\Lambda$ and $\Epsilon$
such that for all $x,y\in\Davis$:
\[\dcat(x,y)/\Lambda-\Epsilon\leq\dchain(x,y)\leq \Lambda\dcat(x,y)+\Epsilon\]

Given $\lambda\geq 1$ and $\epsilon\geq 0$, define $\density(\lambda,\epsilon):=\frac{1}{2\lambda\Lambda}$ and 
$\length(\lambda,\epsilon):=2\lambda(\Lambda\Epsilon+\epsilon)$.

Observe that if $\gamma\from I \to\Davis$ is a
$(\lambda,\epsilon)$--quasigeodesic and $[x,y]\subset I$ then:
\[|y-x|\geq\length(\lambda,\epsilon)\implies \frac{\dchain(\gamma(x),\gamma(y))}{|y-x|}\geq\density(\lambda,\epsilon)\]

So if $|y-x|$ is large enough then a chain realizing
$\dchain(\gamma(x),\gamma(y))$ has length at least
$|y-x|\density(\lambda,\epsilon)$.
In particular, there do exist chains separating $\gamma(x)$ and
$\gamma(y)$ whose \emph{density}, in terms of number of walls per unit
length in the parameter space, is at least $\density(\lambda,\epsilon)$.

\begin{remark}
  All of these constants can be effectively bounded.
  The point is that there are only finitely many conjugacy classes of
  spherical subgroups in $W$, and they can be enumerated in terms of
  the known types of irreducible spherical Coxeter diagrams appearing as
  full subgraphs in the Coxeter diagram for $(W,S)$.
  Caprace's constant from \fullref{Caprace_8} and a bound on
  $\dim(\NR)$ \cite[Lemma~3]{NibRee03} are expressed in terms of the values of
  the Tits form on subsystems of the canonical root system
  corresponding to spherical subgroups.
  Since the
  Tits form has an explicit closed form, these quantities are
  computable.
  Likewise, the cells of the Davis complex are determined by the
  orthogonal representations of the spherical subgroups, which can be
  explicitly computed, so there are only finitely many isometry types
  of cells in $\Davis$ and they are isometric to convex Euclidean
  polytopes that are described by linear algebra. 
  Effective bounds on the distance relations in \fullref{all_metrics_are_equivalent}
and \fullref{lem:distance_to_standard} can be extracted from the
geometry of cells of $\Davis$ and the bound on $\dim(\NR)$.
\end{remark}

\subsection{Morse and hyperbolic geometry}
\subsubsection{Variations of the Morse property}\label{sec:morse_variation}
Let $Z$ be a subset of a metric space $X$.
We define several properties that describe quantitatively how $Z$ sits
inside $X$ in a manner similar to a quasiconvex subspace of a
hyperbolic space.
These properties and equivalences between them were first stated for
the case of $Z$ being a geodesic or quasigeodesic
\cite{BesFuj09,Alg11, Sul14, Sis18}, see also the survey \cite{Cas19Hab}.

\begin{definition}\label{def:morse_property}
A subset $Z\subset X$ is \emph{Morse} or has the
\emph{Morse property} if it is quasigeodesically quasiconvex, in the
sense that there is a real valued function
$\mu(\lambda,\epsilon):=\sup_\gamma\sup_{x\in\gamma}d(x,Z)$, where the
first supremum is taken over $(\lambda,\epsilon)$--quasigeodesic
segments $\gamma$ with both endpoints on $Z$.
A function bounding $\mu$ from above is known as a \emph{Morse gauge}
for $Z$; the given $\mu$ is the \emph{optimal Morse gauge}.
\end{definition}
The content of the definition is whether or not $\mu$ defines a real
valued function; if $Z$ is not Morse then there is some input pair
$(\lambda,\epsilon)$ on which $\mu$ takes value $\infty$.

The next definitions make sense using a \emph{set
  valued} ``closest points projection'':
\[\pi_Z\from X\to 2^Z : x\mapsto\{z\in Z\mid
  d(x,z)=d(x,Z)\}\]
In fact, in more general situations an ``almost closest points
projection'' is good enough; the key point is that point images must be
nonempty.
In the settings that we apply this concept, we have proper $X$ and
closed $Z$, so point images of $\pi_Z$ are
nonempty. A priori they could have arbitrarily large diameters.
If $Z$ is convex and $X$ is CAT(0) then point images of $\pi_Z$ are single points,
and, in fact, $\pi_Z\from X\to Z$ is 1--Lipschitz.
\begin{definition}\label{def:strong_contraction}
  A closed subset $Z$ of a proper geodesic metric space $X$ is \emph{$\chi$--strongly
  contracting} if $\chi\geq\sup_{x,y}\diam(\pi_Z(x)\cup\pi_Z(y))$,
where the supremum is taken over points $x$ and $y$ such that
$d(x,y)\leq d(x,Z)$  and the \emph{diameter} of $Y\subset X$ is $\diam(Y):=\sup_{y,y'\in Y}d(y,y')$.
\end{definition}

\begin{definition}[Divergence]\label{def:divergence}
  For a closed subset $Z$ of a proper geodesic metric space $X$ and parameters
  $0<M\leq 1$ and $N>2$, define the \emph{optimal divergence gauge}
  $\delta(r;M,N):=\inf_{x,y}\bar d_{Mr}(x,y)$, where the infimum is taken
  over all pairs $x$ and $y$ at distance $r$ from $Z$ with $d(x,y)\geq
  Nr$.
The modified distance $\bar d_{Mr}(x,y)$ denotes the infimal length of a path
  in $X$ connecting $x$ to $y$ and staying outside the $Mr$--neighborhood of
  $Z$, or $\infty$ if no such path exists.
  Any lower bound on $\delta(r;M,N)$ is a \emph{divergence gauge} for
  $Z$ with respect to parameters $M$ and $N$. 
\end{definition}

This notion of divergence appears in \cite{Cas19Hab}.
It is the same, up to a suitable equivalence
relation on functions, to other familiar versions of divergence, and
the equivalence class of the function does not depend on the choices
of parameters $M$ and $N$ \cite[Proposition~8.4]{Cas19Hab}.
The equivalence relation distinguishes functions of different
polynomial degrees, so phrases such as ``$Z$ has at least quadratic
divergence'' make sense, and are established by showing there is some
valid choice of $M$ and $N$ such that $Z$ has a divergence gauge with
respect to $M$ and $N$ that is bounded
below by a quadratic function of $r$.

By convention, $\sup\emptyset=-\infty$ and $\inf\emptyset=\infty$, so if for some $r$ there are no suitable pairs $(x,y)$ in
\fullref{def:divergence} or no suitable paths between them, then $\delta(r;M,N)=\infty$.
  This is the expected behavior: every closed bounded set certainly has some Morse/contraction/divergence gauge that can be defined in terms of its
  diameter.
For instance, for a geodesic segment of length $L>0$ in the Euclidean plane, the
 boundary of its $r$--neighborhood has diameter $L+2r$ and length
 $2\pi r+2L$, so any two points on it can be connected by a subsegment
 of the boundary of length at most $\pi r+L$, but for $r>L$ there is
 no pair of points on the boundary at distance at least $3r$ from each other. 
 The optimal divergence gauge $\delta(r;1,3)$ is between  $3r$ and $\pi
 r+L$ while $r\leq L$, but
 then is $\infty$ for $r>L$.
 The segment admits a quadratic divergence gauge $q_L(r):=3r^2/L$.
The behavior of  $\delta(r;1,3)$ is Euclidean-like
on scales at most $L$, while on large scales the bounded
segment begins to look ``point-like'', hence Morse in a trivial way.
The bound $q_L$ degenerates as $L\to\infty$. In contrast, divergence
gauges for geodesic segments in the hyperbolic plane can be bounded
uniformly, independent of the segment diameters.

  \begin{definition}[BGIP]\label{def:BGIP}
    A closed subset $Z$ of a proper geodesic metric space has the
    \emph{Bounded Geodesic Image Property} if there exists $C$ such
    that for every geodesic $\gamma$, if
    $\gamma\cap\nbhd_C(Z)=\emptyset$ then $\diam(\pi_Z(\gamma))\leq C$.
  \end{definition}
  \begin{definition}\label{def:strong_constrict}
    A closed subset $Z$ of a proper geodesic metric space is \emph{strongly
      constricting} if there exists $C$ such that for every geodesic
    segment $\gamma$ with endpoints $x$ and $y$, if
    $\diam(\pi_Z(x)\cup\pi_Z(y))>C$ then $\gamma$ passes within distance
    $C$ of both $\pi_Z(x)$ and $\pi_Z(y)$. 
  \end{definition}

  \begin{proposition}[{\cite{Cas20}, \cite[Proposition~3.1]{Tra19},
      \cite[Proposition~8.15, Corollary~3.6]{Cas19Hab}}]\label{morse_divergent_contracting_in_cat_zero}
    When $X$ is proper and CAT(0), Morse, strongly contracting, and at least
    quadratic divergence are equivalent. Moreover, there are effective
    quantitative bounds between the relevant gauges.
  \end{proposition}
  
  \begin{proposition}[{\cite[Proposition~2.9]{ArzCasTao15}}]\label{equivalence_of_strong_contraction_conditions}
    Strong contraction, strong constriction, and the Bounded Geodesic
    Image Property are equivalent, with effective bounds between the
    relevant gauges.
  \end{proposition}

  We close this subsection with some contraction robustness lemmas. 
  \begin{lemma}\label{lem:qgeod_Hdf}
    Let $X$ be a proper geodesic metric space. 
    Given $\chi$, $\lambda$, and $\epsilon$, there exists $D$ such
    that for every $\chi$--strongly contracting
    $(\lambda,\epsilon)$--quasigeodesic path $\gamma\from I\to X$ and for
    every $s,t\in I$ the Hausdorff distance between $\gamma([s,t])$
    and any geodesic between $\gamma(s)$ and $\gamma(t)$ is at most $D$.
  \end{lemma}
  \begin{proof}
    Pick a geodesic $\alpha\from[0,L]\to X$ between $\gamma(s)$ and $\gamma(t)$.
    By \cite[Lemma~2.14]{ArzCasTao15}, there is a $D'$ depending on
    $\chi$, $\lambda$, and $\epsilon$ such that
    $\gamma([s,t])\subset\bar\nbhd_{D'}(\alpha)$.
    
Conversely, suppose $Z$ is a connected with
$\{\alpha(0),\alpha(L)\}\subset Z\subset\bar\nbhd_{D'}(\alpha)$.
Define the Gromov product coordinate of $x\in X$ along $\alpha$ by:
\[f(x):=\left(d(\alpha(0),\alpha(L))+d(\alpha(0),x)-d(\alpha(L),x)\right)/2\]
  The function $f\from X\to[0,L]$ is
     1--Lipschitz and for $p\in [0,L]$ satisfies $f\circ\alpha(p)=p$.
    It follows that $f(Z)=[0,L]$.
    In particular, for every $p\in [0,L]$ there exists $z\in Z$ such
    that $f\circ\alpha(p)=f(z)$.
    By hypothesis, there exists
    $q\in[0,L]$ with $d(\alpha(q),z)\leq D'$. 
    We bound the difference between the two ``projections'' of $z$:
    \[|p-q|=|f(\alpha(p))-f(\alpha(q))|=|f(z)-f(\alpha(q))|\leq D'\]
    Then the triangle inequality gives $d(\alpha(p),z)\leq
    d(\alpha(p),\alpha(q))+d(\alpha(q),z)\leq 2D'$. Since $p\in[0,L]$
    was arbitrary, 
    $\alpha\subset\bar\nbhd_{2D'}(Z)$.
    Applying this argument with $D:=2D'$ and $Z:=\gamma([s,t])$
    completes the lemma.
  \end{proof}
  
  \begin{lemma}\label{lem:robust}
    Given $\chi$, $\lambda$, and $\epsilon$, there exists $\chi'$,
    $\lambda'$, $\epsilon'$, and
    $D$  such
    that if $\gamma\from
    I\to X$ is a $\chi$--strongly contracting,
    $(\lambda,\epsilon)$--quasigeodesic in a proper geodesic metric
    space $X$, then:
    \begin{itemize}
      \item There is a continuous, $\chi'$--strongly contracting
        $(\lambda',\epsilon')$--quasigeodesic at bounded Hausdorff
        distance from $\gamma$. Moreover, if $\gamma$ is compact or a
        closed ray we can choose the continuous version to have the same endpoints/basepoint.
    \item Every subsegment of $\gamma$ is  $\chi'$--strongly
      contracting.
      \item Every geodesic segment with both endpoints on $\gamma$ is
        $\chi'$--strongly contracting and is $D$--Hausdorff equivalent
        to any subsegment of $\gamma$ with the same endpoints.
        \item Every geodesic ray based on $\gamma$ and asymptotic to
          $\gamma$ is $\chi'$--strongly contracting and is
          $D$--Hausdorff equivalent to any subray of $\gamma$ with
          which it is asymptotic and shares a basepoint. 
          \item Every biinfinite geodesic asymptotic to $\gamma$ is
            $\chi'$--strongly contracting and $D$--Hausdorff
            equivalent to $\gamma$. 
    \end{itemize}
  \end{lemma}
  \begin{proof}
    The first item follows by using the connect-the-dots trick to get a continuous quasigeodesic that is
    coarsely equivalent to $\gamma$, with Hausdorff distance bounded
    in terms of $\lambda$ and $\epsilon$.
    By \cite[Lemma~6.3]{ArzCasGrub}, strong contraction is preserved
    by coarse equivalence, so the new quasigeodesic is strongly contracting, with
    contraction constant determined by that of $\gamma$ and the
    Hausdorff distance between them.

    Take $D_0$ to be the constant provided by \fullref{lem:qgeod_Hdf}.
    Using properness to take limits of geodesic segments in the ray
    and biinfinite cases, there is a geodesic $\alpha$ of the same
    type as $\gamma$ and at Hausdorff distance at most $D_0$ from
    $\gamma$.
    By \cite[Lemma~6.3]{ArzCasGrub}, $\alpha$ is strongly contracting,
    with contraction constant determined by $\chi$ and $D_0$.
    By \cite[Lemma~3.6]{CasMac19}, subsegments of $\alpha$ are
    uniformly strongly contracting, with contraction constant
    depending only on that of $\alpha$.

    Fix $s<t$ in $I$ and suppose $\beta$ is a geodesic from
    $\gamma(s)$ to $\gamma(t)$. Choose $p$ and $q$ in $\alpha$ within
    distance $D_0$ of $\gamma(s)$ and $\gamma(t)$, respectively.
    The subsegment of $\alpha$ from $p$ to $q$ is strongly
    contracting, as are the geodesic segments of length at most $D_0$
    joining $\gamma(s)$ to $p$ and joining $q$ to $\gamma(t)$.
    The triangle lemma \cite[Lemma~3.8]{CasMac19} says that if two
    sides of a geodesic (almost) triangle are strongly contracting,
    then so is the third side, with contraction bound determined by
    those of the first two sides. 
    One application of the triangle lemma gives that a
    geodesic from $\gamma(s)$ to $q$ is strongly contracting, and a
    second application gives that $\beta$ is strongly contracting,
    with contraction constant determined by that of $\alpha$ and the
    quantity $D_0$.
    \fullref{lem:qgeod_Hdf} gives that $\beta$ and $\gamma([s,t])$ are
    at Hausdorff distance at most $D_0$, so another application of
    \cite[Lemma~6.3]{ArzCasGrub} transfers strong contraction of
    $\beta$ to strong contraction of $\gamma([s,t])$.
    This proves the second and third items.

    Uniform strong contraction of compact subsegments of quasigeodesic
    transfers to strong contraction, with the same constant, of the
    whole quasigeodesic. This is an easy point-by-point argument using
    a long enough compact subsegment to control each projection.
    Thus, the previous argument also implies that any subray
    $\gamma'$ of $\gamma$
    is uniformly strongly contracting. 
Let $\beta$ be a
geodesic ray with the same basepoint and asymptotic to $\gamma'$.
Then strong contraction of $\gamma'$ forces $\beta$ and $\gamma'$ to
be at uniformly bounded Hausdorff distance,
\cite[Lemma~4.4]{CasMac19}, and then
    \cite[Lemma~6.3]{ArzCasGrub} transfers strong contraction of
    $\gamma'$ to $\beta$.

    Finally, suppose that $\gamma$ is biinfinite and
    $\beta$ is some biinfinite geodesic asymptotic to $\gamma$.
    Properness gives geodesic rays $\alpha^+$ and $\alpha^{-}$ based
    at $\gamma(0)$ and asymptotic to $\gamma([0,\infty))$ and
    $\gamma([0,-\infty))$, respectively.
    By the preceding paragraph, $\alpha^+$ and $\alpha^{-}$ are
    uniformly strongly contracting and uniformly Hausdorff equivalent
    to the corresponding subray of $\gamma$.
    The rays $\alpha^+$ and $\alpha^{-}$ together with $\beta$ form an
    ideal geodesic triangle with two strongly contracting sides, so
    the truncation argument in the proof of the triangle lemma
    gives a uniform Hausdorff distance bound between $\beta$ and
    $\alpha^+\cup\alpha^{-}$.
    Hence, $\beta$ is uniformly Hausdorff equivalent to $\gamma$, so
    \cite[Lemma~6.3]{ArzCasGrub} makes $\beta$ uniformly strongly
    contracting.

    Take $\chi'$ and $D$ to be the maxima of the strong contraction
    and Hausdorff distance bounds from the various cases of this argument. 
  \end{proof}

\subsubsection{Working in non-geodesic spaces}
To work in the setting of hyperbolic spaces that are not necessarily geodesic, 
Blach\`ere, Ha\"issinsky, and Mathieu \cite{BlaHaiMat11} introduced
the following notion:
\begin{definition}
  A map $\gamma\from I\to X$ from a subinterval of $\mathbb{R}$ to a
  metric space $X$ is a \emph{(parameterized) $(\lambda,\epsilon,\tau)$--quasiruler} 
  if it is a $(\lambda,\epsilon)$--quasigeodesic and the triangle
  inequality is almost degenerate, in the sense that for all $r<s<t$ in
  $I$:
  \[d(\gamma(r),\gamma(s))+d(\gamma(s),\gamma(t))-\tau\leq
    d(\gamma(r),\gamma(t))\leq
    d(\gamma(r),\gamma(s))+d(\gamma(s),\gamma(t))\]

  The space $X$ is \emph{quasiruled} if there exist
  $\lambda$, $\epsilon$, and $\tau$ such that every pair of points in
  $X$ is the pair of endpoints of a $(\lambda,\epsilon,\tau)$--quasiruler.
\end{definition}

Abbott and Incerti-Medici \cite{AbbInc22} make the following
refinement:
\begin{definition}
   A map $\gamma\from I\to X$ from a subinterval of $\mathbb{R}$ to a
  metric space $X$ is an \emph{unparameterized $\tau$--quasiruler} if
  it satisfies two conditions:

  \begin{description}
\item[bounded jumps]  $\forall t\in I$, $\limsup_{s\in I, \, |t-s|\to
    0}d(\gamma(s),\gamma(t))\leq \tau$
  \item[almost degeneracy] $\forall r<s<t\in I$, 
  $d(\gamma(r),\gamma(s))+d(\gamma(s),\gamma(t))-\tau\leq
    d(\gamma(r),\gamma(t))$
  \end{description}
\end{definition}

There is another definition in the literature that is related:
\begin{definition}
  A $(1,\epsilon)$--quasigeodesic is also known as an
\emph{$\epsilon$--rough geodesic}, and a space for which there exists
$\epsilon$ such that every pair of points is connected by an
$\epsilon$--rough geodesic is a \emph{rough geodesic space}.
\end{definition}

Abbott and Incerti-Medici prove \cite[Lemma~A.4]{AbbInc22} that given $\tau$ there exist
$\lambda$ and $\epsilon$ such that every unparameterized
$\tau$--quasiruler satisfying an additional metric properness condition can be
reparameterized to make it a $(\lambda,\epsilon,\tau)$--quasiruler. 
They also prove \cite[Lemma~A.6]{AbbInc22} that quasiruled spaces are
roughly geodesic, although the rough
geodesics need not be the reparameterizations of the quasirulers. 
In \fullref{sec:hyperbolic} we introduce a non-geodesic metric on $\Davis$ and $\NR$
such that combinatorial geodesics in $\NR$ can be reparameterized to
be rough geodesic quasirulers, see \fullref{cor:combinatorial_geodesics_rough_and_quasiruler}.
Getting the same reparameterization to be both roughly geodesic and a
uniform quasiruler, in this particular setting, is the reason for
proving \fullref{cor:combinatorial_geodesics_rough_and_quasiruler}
instead of quoting Abbott and Incerti-Medici.

\subsection{Parabolic subcomplexes}
\subsubsection{Wide groups}\label{sec:wide}
A group is \emph{constricted} if all of its asymptotic cones have cut
points or \emph{unconstricted} otherwise \cite{DruSap05}.
A group is \emph{wide} if none of its asymptotic cones have cut
points \cite{DruMozSap10}, which is a strictly stronger condition than
being unconstricted.
For finitely generated groups, the existence of a biinfinite Morse
geodesic implies that the group is constricted.

For Coxeter groups there is a nice wide/constricted dichotomy visible
from the Coxeter diagram: 
\begin{proposition}\label{wide_characterization}
  Let $W=P_1\times\cdots \times P_n\times K$ be the canonical product
  decomposition of $(W,S)$ where the $P_i$ are special subgroups corresponding to
nonspherical connected components of the Coxeter diagram, and $K$ is the
maximal spherical special subgroup.
Suppose $W\neq K$.
Then $W$ is wide if and only if either $n\geq 2$ or $n=1$ and $P_1$ is higher
rank affine.  
\end{proposition}
\begin{proof}
  If $n\geq 2$ or $n=1$ and $P_1$ is higher rank affine then $W$ is
  virtually a product of infinite groups. Asymptotic cones of products
  are products of asymptotic cones, so any asymptotic cone of a
  product of two infinte groups has no cut points, so the group is wide.
  Otherwise, Caprace and Fujiwara \cite[Proposition~4.5]{CapFuj10}
  say that $W$ contains a rank-one element (equivalently, a Morse element), so it is constricted. 
\end{proof}
\begin{definition}[{cf \cite[Definition~3.1]{CorLev25} and the
definition $\mathbb{T}_0$ of \cite[Section~A.1]{BehHagSis17cox}}]\label{def:wide}
  A \emph{wide parabolic subgroup} $P$ of $(W,S)$ is a nonspherical parabolic
  that is wide as a Coxeter group, in the sense of  
\fullref{wide_characterization}, which occurs exactly if
  either $P=B\times C$, where $B$ and $C$ are nonspherical
  parabolics of $(W,S)$, or $P=A\times K$, where $A$ is an irreducible
  higher rank affine parabolic of $(W,S)$ and $K$ is a spherical
  parabolic of $(W,S)$.

  A \emph{wide parabolic subcomplex} is a parabolic subcomplex whose
  stabilizer is a wide parabolic subgroup.
\end{definition}

\subsubsection{Gate projections to parabolic subcomplexes}\label{sec:gate}
\begin{lemma}\label{def:gate_projection}
  Given a parabolic subcomplex $\Upsilon:=w\Davis_T$ and a vertex $x\in
  \Davis^{(0)}=W$ there exists a unique $\dcomb$--closest vertex
  $\gate_\Upsilon(x)$ of
  $\Upsilon$ to $x$.
  Furthermore, for every vertex $y\in \Upsilon^{(0)}$:
    \[\dcomb(x,y)=\dcomb(x,\gate_\Upsilon(x))+\dcomb(\gate_\Upsilon(x),y)\]

The map $\gate_\Upsilon\from
  \Davis^{(0)}\to \Upsilon^{(0)}$ is the \emph{gate map} or \emph{gate projection to
    $\Upsilon$}.
  It is 1--Lipschitz.
\end{lemma}
\begin{proof}
  By \cite[Lemma~4.3.1]{davisbook}, there is a unique shortest element
  $v$ of $W_\emptyset x^{-1}wW_T$, and there exists $z\in W_T$ such
  that $x^{-1}y=vz$ with $|x^{-1}y|_S=|v|_S+|z|_S$.
  Take $\gate_\Upsilon(x):=xv$.
  Then
  $\dcomb(x,y)=|x^{-1}y|_S=|v|_S+|z|_S=\dcomb(x,\gate_\Upsilon(x))+\dcomb(\gate_\Upsilon(x),y)$.
  Uniqueness follows immediately.

  The fact that $\gate_\Upsilon$ is 1--Lipschitz is a standard, easy
  computation. 
\end{proof}

\begin{corollary}\label{cor:composition_of_nested_gate_maps}
  If $U\subset T\subset S$ then $\gate_{\Davis_U}=\gate_{\Davis_U}\circ\gate_{\Davis_T}$.
\end{corollary}
\begin{proof}
  \begin{align*}
    \dcomb(x,&\gate_{\Davis_U}(x))\\
             &=\dcomb(x,\gate_{\Davis_T}(x))+\dcomb(\gate_{\Davis_T}(x),\gate_{\Davis_U}(x))\\
    &=
      \dcomb(x,\gate_{\Davis_T}(x))+\dcomb(\gate_{\Davis_T}(x),\gate_{\Davis_U}\circ\gate_{\Davis_T}(x))+\dcomb(\gate_{\Davis_U}\circ\gate_{\Davis_T}(x),
      \gate_{\Davis_U}(x))\\
    &\geq \dcomb(x,\gate_{\Davis_U}\circ\gate_{\Davis_T}(x)) + \dcomb(\gate_{\Davis_U}\circ\gate_{\Davis_T}(x),
      \gate_{\Davis_U}(x))\\
    &\geq \dcomb(x,\gate_{\Davis_U}(x)) + \dcomb(\gate_{\Davis_U}\circ\gate_{\Davis_T}(x),
      \gate_{\Davis_U}(x))\\
  \end{align*}
  Thus,  $\dcomb(\gate_{\Davis_U}\circ\gate_{\Davis_T}(x),
      \gate_{\Davis_U}(x))=0$.
\end{proof}

\begin{corollary}\label{cor:wall_gate}
  A wall $M$ separates $x$ from $\gate_\Upsilon(x)$ if and only if $M$
  separates $x$ from $\Upsilon$.
\end{corollary}
\begin{proof}
  A wall that separates $x$ from all of $\Upsilon$ certainly separates
  $x$ from $\gate_\Upsilon(x)\in\Upsilon$.
For the converse, suppose $M$ is a wall
that separates $x$ from $y:=\gate_\Upsilon(x)$ but not from $\Upsilon$,
so there is a vertex  $z\in \Upsilon$ such that $x$ and $z$ are on one
side of $M$, and $y$ is on the other.
Then the concatenation of the $\dcomb$--geodesic from $x$ to $y$ with the
$\dcomb$--geodesic from $y$ to $z$ crosses $M$ twice, so it is not a
combinatorial geodesic, so $\dcomb(x,z)<\dcomb(x,y)+\dcomb(y,z)$.
This contradicts the gate property $\dcomb(x,z)=\dcomb(x,y)+\dcomb(y,z)$.
\end{proof}

The following interaction between gate maps and the group action
follows directly from the definition.
\begin{corollary}\label{cor:W_action_on_projection}
  For $w\in W$, $\Upsilon$ a parabolic subcomplex, and $x\in
  \Davis^{(0)}$:
  \[\gate_{w\Upsilon}(x)=w\gate_\Upsilon(w^{-1}x)\]
\end{corollary}

\begin{lemma}\label{lem:perp_dont_change_gate}
  For any $U\subset S$ and $g\in W_{U^\perp}$, for all $x\in \Davis^{(0)}$:
  \[\gate_{\Davis_U}(gx)=\gate_{\Davis_U}(x)\]
\end{lemma}
\begin{proof}
  Let $T:=U\sqcup U^\perp$.
  Write $\gate_{\Davis_T}(x)=ab$ for $a\in W_U$ and
  $b\in W_{U^\perp}$.
Using that $g\Davis_T=\Davis_T$ and
\fullref{cor:W_action_on_projection},
$\gate_{\Davis_T}(gx)=\gate_{g\Davis_T}(gx)=g\gate_{\Davis_T}(g^{-1}gx)=g\gate_{\Davis_T}(x)$.
This plus \fullref{cor:composition_of_nested_gate_maps} gives:
\[\gate_{\Davis_U}(g\gate_{\Davis_T}(x))=\gate_{\Davis_U}\circ\gate_{\Davis_T}(gx)=\gate_{\Davis_U}(gx)\]
Thus, $\gate_{\Davis_U}(gx)=\gate_{\Davis_U}(gab)=\gate_{\Davis_U}(agb)$.
The latter is $agb$ times the shortest element of
$(agb)^{-1}W_U=b^{-1}g^{-1}W_U$. Since $b^{-1}g^{-1}\in W_{U^\perp}$,
there is no cancellation between $b^{-1}g^{-1}$ and words in $W_U$, so
the shortest element of $b^{-1}g^{-1}W_{U}$ is
$b^{-1}g^{-1}$, which gives $\gate_{\Davis_U}(agb)=agbb^{-1}g^{-1}=a$.
The same argument gives $\gate_{\Davis_U}(ab)=a$.
Thus:
\[\gate_{\Davis_U}(x)=\gate_{\Davis_U}\circ\gate_{\Davis_T}(x)=\gate_{\Davis_U}(ab)=a=\gate_{\Davis_U}(gab)=\gate_{\Davis_U}(gx)\qedhere\]
\end{proof}

\begin{lemma}\label{lem:gate_projection_image_is_residue}
 For $i\in\{0,1\}$, let $\Upsilon_i:=w_i\Davis_{T_i}$ be parabolic
 subcomplexes.
 Let $u$ be the unique shortest element of
 $W_{T_0}w_0^{-1}w_1W_{T_1}$.
 Take $v_0\in W_{T_0}$ and $v_1\in W_{T_1}$ such that
  $w_0^{-1}w_1=v_0uv_1^{-1}$.
 Let $U_0:=\{s\in T_0\mid u^{-1}su\in
 T_1\}$ and  $U_1:=\{s\in T_1\mid usu^{-1}\in T_0\}$.
 Then $\gate_{\Upsilon_0}(\Upsilon_1)=w_0v_0W_{U_0}$ is the vertex set of
 a parabolic subcomplex $w_0v_0\Davis_{U_0}$.
 Furthermore, right-multiplication by $u$ defines an
 isomorphism $w_0v_0\Davis_{U_0}\to w_1v_1\Davis_{U_1}$. 
\end{lemma}
\begin{proof}
  The element $u$ exists by \cite[Lemma~4.3.1]{davisbook}, and
  $\dcomb(\Upsilon_0,\Upsilon_1)=|u|_S$.
  For every vertex $w_0v_0x\in w_0v_0W_{U_0}\subset\Upsilon_0^{(0)}$ we have
  $\dcomb(w_0v_0x,w_0v_0xu)=|u|_S$ and $w_0v_0xu=w_0v_0uu^{-1}xu\in
  w_0v_0uW_{U_1}=w_1v_1W_{U_1}\subset\Upsilon_1^{(0)}$.
Thus $w_0v_0W_{U_0}\subset\gate_{\Upsilon_0}(\Upsilon_1)$.  

  Conversely, suppose $z\in W_{T_0}$ such that
  $w_0v_0z\in\gate_{\Upsilon_0}(\Upsilon_1)$.
  Then there is some $y\in W_{T_1}$ such that
  $w_0v_0z=\gate_{\Upsilon_0}(w_1v_1y)$.
  By construction, $\gate_{\Upsilon_0}(w_1v_1y)$ is $w_1v_1yx$, where
  $x$ is the unique shortest element of $(w_1v_1y)^{-1}w_0v_0W_{T_0}=y^{-1}u^{-1}W_{T_0}$.
  Let $a:=x^{-1}y^{-1}u^{-1}\in W_{T_0}$.
  We have $w_0v_0z=w_1v_1\cdot yx=w_0v_0u\cdot u^{-1}a^{-1}$, so
  $z=a^{-1}$ and $x=y^{-1}u^{-1}z$.
  Since $u^{-1}$ is the shortest element of $W_{T_1}u^{-1}W_{T_0}\ni
  x$, there are $b\in W_{T_0}$ and $c\in W_{T_1}$ such that
  $x=cu^{-1}b$ and $|x|_S=|c|_S+|u^{-1}|_S+|b|_S$.
  But since $x$ is shortest in $y^{-1}u^{-1}W_{T_0}\subset
  W_{T_1}u^{-1}W_{T_0}$, the element $b$ is trivial.
  Thus, $x=cu^{-1}$, which gives $W_{T_0}\ni z=uycu^{-1}\in uW_{T_1}u^{-1}$.
By the remark following \cite[Lemma~5.3.6]{davisbook},
$uW_{T_1}u^{-1}\cap W_{T_0}=W_{U_0}$, so $z\in W_{U_0}$.
Thus, $w_0v_0W_{U_0}\supset\gate_{\Upsilon_0}(\Upsilon_1)$.

An edge in the full subcomplex spanned by
$\gate_{\Upsilon_0}(\Upsilon_1)$ is of the form $w_0v_0p-w_0v_0pq$ for
some $p\in W_{U_0}$ and $q\in U_0$. 
Right-multiplication by $u$ sends these vertices to
$w_0v_0pu=w_0v_0uu^{-1}pu=w_1v_1u^{-1}pu$ and
$w_0v_0pqu=w_1v_1u^{-1}pu\cdot u^{-1}qu$, respectively. 
These are two elements of $w_1v_1W_{U_1}$ that differ by $u^{-1}qu\in
U_1$, so they are adjacent vertices of $w_1v_1W_{U_1}$.
Thus, right-multiplication by $u$ gives a bijection from
$w_0v_0W_{U_0}$ to $w_1v_1W_{U_1}$ that preserves adjacency, so
extends to an isomorphism of the full subcomplexes spanned by those
vertex sets. 
\end{proof}
\begin{corollary}\label{cor:cross_projections_have_same_stabilizer}
 $\stab(\gate_{\Upsilon_0}(\Upsilon_1))=\stab(\gate_{\Upsilon_1}(\Upsilon_0))$
\end{corollary}
The next result is a partial converse of
\fullref{cor:cross_projections_have_same_stabilizer}.
\begin{lemma}\label{lem:same_stabilizer_implies_parallel}
  Let $P$ be a nonspherical irreducible parabolic subgroup of $(W,S)$.
  Let $T\subset S$ be the unique subset such that $P$ is conjugate to
  $W_T$.
  Suppose $\Upsilon_0=w_0\Davis_{T_0}$ and $\Upsilon_1=w_1 \Davis_{T_1}$ are parabolic
  subcomplexes with $P=\stab(\Upsilon_0)=\stab(\Upsilon_1)$.
  Then $T=T_0=T_1$ and the unique shortest element $u$ of
  $W_Tw_0^{-1}w_1W_T$ belongs to $W_{T^\perp}$ and satisfies:
  \begin{itemize}
  \item $\gate_{\Upsilon_1}(\Upsilon_0)=\Upsilon_0^{(0)}u=\Upsilon_1^{(0)}$.
    \item Left-multiplication by $w_0uw_0^{-1}\in P^\perp$ and
      right-multiplication by $u\in W_{T}^\perp$ induce the same isomorphism $\Upsilon_0\to\Upsilon_1$.
  \end{itemize}
\end{lemma}
\begin{proof}
The set $T$ exists by \fullref{irreducible_parabolics_have_type}, and
\fullref{lem:nonspherical_normalizer}~\eqref{item:irreducible_rigidity}
further implies $T=T_0=T_1$.
\fullref{lem:nonspherical_normalizer}~\eqref{item:normalizer_splits}
gives $w_0^{-1}w_1\in\normalizer_W(W_T)=W_T\times W_{T^\perp}$, so we
can write $w_0^{-1}w_1=vu$ where $v\in W_T$ and $u\in W_{T^\perp}$ is
the shortest element of $W_Tw_0^{-1}w_1W_T$.
The projection formula in
\fullref{lem:gate_projection_image_is_residue} gives $U_0=U_1=T$, and
that right-multiplication by $u$ is the gate map $\gate_{\Upsilon_1}(\Upsilon_0)=\Upsilon_0^{(0)}u=\Upsilon_1^{(0)}$.
Finally, for any $x\in W_T$ we have $w_0uw_0^{-1}\cdot
w_0x=w_0ux=w_0x\cdot u$, so on $\Upsilon_0^{(0)}$ left-multiplication
by $w_0uw_0^{-1}$, right-multiplication by $u\in W_{T^\perp}=W_T^\perp$, and
$\gate_{\Upsilon_1}$ all define the same map.
\end{proof}

\begin{lemma}\label{cpp_and_gate_agree}
  If $x\in \Davis^{(0)}$ and $\Upsilon$ is a parabolic subcomplex then
  $\gate_\Upsilon(x)$ is the unique $\dcat$--closest vertex of
  $\Upsilon$ to $x$.
  Furthermore, there exists a closed chamber containing both $\pi_\Upsilon(x)$ and $\gate_\Upsilon(x)$.
\end{lemma}
\begin{proof}
  Let $y\in\Upsilon^{(0)}\setminus\{\gate_\Upsilon(x)\}$.
  Let $z$ be the first vertex after $y$ on a
  $\dcomb$--geodesic from $y$ to $\gate_\Upsilon(x)$.
  Since $\Upsilon$ is $\dcomb$--convex, $z\in\Upsilon$.
  Let $M$ be the wall dual to the edge between $z$ and $y$.
  Then $M$ cuts $\Upsilon$, so by \fullref{cor:wall_gate} it does not
  separate $x$ and $\gate_\Upsilon(x)$.
  Thus, vertices $x$, $\gate_\Upsilon(x)$, and $z$ are in one
  open halfspace $M^{-}$ of $M$ and $y$ is in the other.
  Let $p$ be the point at which the unique $\dcat$--geodesic from $x$ to $y$ crosses $M$.
 The reflection through $M$ is an isometry that fixes $p$ and
 exchanges $y$ and $z$, so:
 \[\dcat(x,y)=\dcat(x,p)+\dcat(p,y)=\dcat(x,p)+\dcat(p,z)\geq\dcat(x,z)\]
In fact, $\dcat(x,p)+\dcat(p,z)>\dcat(x,z)$, because $x$ and $z$ are
contained in the $\dcat$--convex open halfspace $M^{-}$ and $p\in M$
is not, so $p$ does not lie on the unique $\dcat$--geodesic from $x$
to $z$.
Iterating along the $\dcomb$--geodesic, conclude
$\dcat(x, \gate_\Upsilon(x))<\dcat(x,y)$.

The proof of the further statement is similar: suppose $M$ is a wall
with $\gate_\Upsilon(x)\in M^{-}$ and $\pi_\Upsilon(x)\in M^+$.
Then $M$ cuts $\Upsilon$, so $M$ does not separate $x$ and
$\gate_\Upsilon(x)$, so $x\in M^{-}$. 
Using the reflect and straighten trick as in the previous argument,
this produces a point $\refl{M}(\pi_\Upsilon(x))$ of $\Upsilon$ that is strictly
closer to $x$ than $\pi_\Upsilon(x)$ is.
That is a contradiction, so no wall separates $\gate_\Upsilon(x)$ and $\pi_\Upsilon(x)$.
\end{proof}

%% file: Coxeter_background.tex
A \emph{reflection} in $(W,S)$ is a conjugate of an element of $S$.
A \emph{wall} in $\Davis$ is the fixed set of a reflection.
The unique reflection fixing wall $M$ and exchanging the two
components of $\Davis\setminus M$ is denoted $\refl{M}$.
Unlike in the right-angled case, general Coxeter groups can have
conjugate generators, so the reflection $\refl{M}$ is not necessarily
conjugate to a unique element of $S$.

If $M$ and $M'$ are distinct walls and $x\in M\cap M'$, then in the
cell $\sigma$ of $\Davis$ containing $x$ in its relative interior the walls
restrict to distinct Euclidean hyperplanes in the Euclidean polytope
$\sigma$, so their intersection is transverse. Hence, distinct walls
in $\Davis$ are either disjoint or intersect transversely, and in the
latter case we write $M\transverse M'$.

\begin{lemma}\label{lem:point_stabilizers_spherical}
 Point stabilizers in $\Davis$ are spherical parabolic subgroups. 
\end{lemma}
\begin{proof}
 This is observed immediately following \cite[Definition~5.2.6]{davisbook}.
\end{proof}
\begin{lemma}[{\cite[Corollary~D.2.9]{davisbook}}]\label{tits_finite}
  Every finite subgroup of $W$ is contained in a spherical parabolic subgroup.
\end{lemma}
\begin{lemma}\label{lem:dihedral}
  The group $G:=\langle\refl{M},\refl{M'}\rangle$ generated by
  reflections in distinct walls $M$ and $M'$ is dihedral.
  It is finite if and only if $M\transverse M'$.
\end{lemma}
\begin{proof}
  A group generated by a pair of distinct involutions is dihedral.
  Distinct walls $M$ and $M'$ are transverse if and only if they intersect.
If  $x\in M\cap M'$ then $G\leq\stab(x)$, which is
  a spherical parabolic subgroup, by
  \fullref{lem:point_stabilizers_spherical}. 
  Thus, $G$ is a finite dihedral group.
  Conversely, if $G$ is finite then by \fullref{tits_finite} it is contained in $wW_Tw^{-1}$ for some $w\in W$ and spherical
  $T\subset S$.
  That parabolic subgroup stabilizes the cell $\sigma$ with vertices $wW_T$. By
  construction, $\sigma$ is a convex Euclidean polytope and 
  $wW_Tw^{-1}$ acts on it by Euclidean isometries preserving $wW_T$,
  so fixing the Euclidean barycenter $x$ of the vertices. 
Thus, $x\in\fix(\refl{M})\cap\fix(\refl{M'})=M\cap M'$.
\end{proof}
A \emph{reflection subgroup} of $(W,S)$ is a subgroup of $W$ that is generated by reflections.
\begin{theorem}[{Dyer \cite{MR1076077}/Deodhar\cite{MR1023969}}]\label{thm:reflection_subgroup}
Every reflection subgroup $W'$ of $(W,S)$ is a Coxeter group, and admits a
canonical Coxeter system $(W',S')$ for which the reflections are
exactly the reflections of $(W,S)$ contained in $W'$.
\end{theorem}

Every wall $M$ is convex and cuts $\Davis$ into two components.
We say $M$ \emph{separates} $A$ and $B$ if $A$ and $B$ are contained
in different components of $\Davis\setminus M$.
The components of $\Davis\setminus M$ are the \emph{open halfspaces}
corresponding to $M$. A \emph{closed halfspace} of $M$ is the union of
$M$ with an open halfspace of $M$. 

By construction, walls miss the vertices of $\Davis$ and intersect
higher dimensional cells in a Euclidean codimension--1 hyperplane.
This gives a natural piecewise Euclidean structure on a wall $M$ from its intersections with
the cells of $\Davis$.
Similarly, there are no tangential intersections between walls $M$
and subcomplexes $\Upsilon$ of $\Davis$: either they are disjoint or
$M$ cuts an edge of $\Upsilon$, in which case we say \emph{$M$ cuts
  $\Upsilon$}. In particular, the following are equivalent:
$M\cap\Upsilon\neq\emptyset$, $M$ cuts $\Upsilon$, $M$ cuts an edge of
$\Upsilon$, $\Upsilon$ contains vertices from both
components of $\Davis\setminus M$. 

The connected components of $\Davis$ minus the walls are the
\emph{open chambers}.
Each open chamber contains a unique vertex of $\Davis$, and every
vertex is contained in a chamber.
By choosing a base vertex  $\one\in\Davis$ we get a bijection between
$W$ and vertices of $\Davis$ by referring to the vertex
$w.\one$ as $w$.
The chamber containing $\one$ is the \emph{fundamental chamber}, whose
closure is a fundamental domain for $W\act\Davis$.
This choice of basepoint identifies the Cayley graph of $W$ with
respect to $S$ with the 1--skeleton of $\Davis$.

It will sometimes be convenient to replace $\Davis$ by a subdivision
in which walls of $\Davis$ are actual subcomplexes.
The barycentric subdivision has this property, but that subdivision is
finer than necessary. Instead we define an intermediate complex:
\begin{definition}\label{def:Hasse}
  Let $\Davis'$ be the polyhedral subdivision of $\Davis$ obtained by
  subdividing each cell of $\Davis$ along its intersections with the
  walls of $\Davis$.
  Equivalently, the cells of $\Davis'$ are the intersections of cells
  of $\Davis$ with closed chambers.
  Since finitely many walls meet each cell, $\Davis'$ is still a
  locally finite complex.
It is the coarsest subdivision of $\Davis$ in which walls
  and their intersections are subcomplexes.
  Its 1--skeleton is the Hasse diagram of the face
poset of $\Davis$, i.e.\ one vertex for every cell of $\Davis$ and an
  edge connecting a cell-vertex with its codimension-1 face-vertices.
\end{definition}

With respect to our choice of $\one$, the vertices of $\Davis$
corresponding to elements of a special subgroup $W_T$ span a subcomplex, denoted $\Davis_T$, isomorphic to the
Davis complex of the Coxeter system $(W_T,T)$.
Furthermore, $\Davis_T^{(1)}$ is $\dcomb$--convex in $\Davis^{(1)}$,
\cite[Lemma~4.3.3]{davisbook}, and $\Davis_T$ is $\dcat$--convex in
$\Davis$, \cite[Proposition~A.2]{Swi16}.
Similarly, a coset of a special subgroup $wW_T$ corresponds to a
\emph{parabolic subcomplex} $w\Davis_T$, which is the full subcomplex
of $\Davis$ spanned by vertices in the coset $wW_T$.
These subcomplexes are closely related to the \emph{residues} used in
the buildings literature. The parabolic subcomplex $w\Davis_T$
corresponds to the residue that is the union of chambers containing
the vertices of $w\Davis_T$.
The parabolic subcomplex and the residue have the same stabilizer,
which is the parabolic subgroup $wW_Tw^{-1}$.
However, a parabolic subgroup may be the stabilizer of multiple parallel
 parabolic subcomplexes; we return to this point in \fullref{normalizer_parabolics_have_well_defined_subcomplex}.
 
\begin{lemma}\label{lem:nested_residues_have_nested_stabilizers}
  If $a\Davis_U\subset b\Davis_T$ then $\stab(a\Davis_U)\subset\stab(b\Davis_T)$.
\end{lemma}
\begin{proof}
  Since edges of $a\Davis_U$ are labelled by elements of $U$ and edges
  of $b\Davis_T$ are labelled by elements of $T$, we have $U\subset T$.
 Passing to vertex sets, $aW_U\subset bW_T$, so $a=bc$ for some $c\in
 W_T$, so $cW_Uc^{-1}\subset W_T$.
 Thus:
 \[\stab(a\Davis_U)=aW_Ua^{-1}=bcW_Uc^{-1}b^{-1}\leq bW_Tb^{-1}=\stab(b\Davis_T)\qedhere\]
\end{proof}
\begin{corollary}
  If $\Upsilon$ is a parabolic subcomplex, a wall $M$ cuts $\Upsilon$
  if and only if $\refl{M}\in\stab(\Upsilon)$.
\end{corollary}
\begin{proof}
If $M$ cuts $\Upsilon$ then it cuts an edge of $\Upsilon$, which is a parabolic subcomplex contained in $\Upsilon$ and stabilized by
  $\refl{M}$. \fullref{lem:nested_residues_have_nested_stabilizers}
  gives $\refl{M}\in\stab(\Upsilon)$.
If $M$ does not cut $\Upsilon$ then it separates $\Upsilon$ and
$\refl{M}(\Upsilon)$, so $\refl{M}\notin\stab(\Upsilon)$. 
\end{proof}

An \emph{affine} Coxeter group is one that is virtually a Euclidean
reflection group.
Coxeter classified the irreducible spherical and irreducible affine
Coxeter groups \cite{Cox34}, see also \cite[Table~6.1 and Appendix~C]{davisbook}.
In particular, examination of \cite[Table~6.1]{davisbook} shows that
every proper subdiagram of an irreducible Euclidean diagram is a
disjoint union of irreducible spherical diagrams, giving the well-known
fact that an irreducible affine parabolic subgroup has no proper nonspherical subparabolics.

\begin{definition}
  A Coxeter system is called \emph{irreducible higher rank affine} if
  it is irreducible and  affine of
rank at least 3, in the sense that the fundamental generating set has
cardinality at least 3. This corresponds to the geometric rank, the
dimension of the Euclidean space on which it acts, being at least 2.
\end{definition}

\begin{definition}
  If $\mathcal{M}$ is a set of walls in $\Davis$ and $M_1$,\dots,$M_n$
  are additional walls, let $W(\mathcal{M},M_1,\dots,M_n):=\langle
\refl{M}\mid M\in\mathcal{M}\cup\bigcup_{i=1}^n\{M_i\}\rangle$ be the subgroup generated by
reflections in those walls.

  For any subset $U$ of $W$ there is a unique smallest parabolic
  subgroup $\pc(U)$
  containing it, known as its \emph{parabolic closure}. In particular,
  with notation as above, $\pc(\mathcal{M},M_1,\dots,M_n):=\pc(W(\mathcal{M},M_1,\dots,M_n))$.
\end{definition}
By \fullref{thm:reflection_subgroup}, $W(\mathcal{M},M_1,\dots,M_n)$
is a Coxeter group. $\pc(\mathcal{M},M_1,\dots,M_n)$ is also a Coxeter
group, since it is a parabolic subgroup of $(W,S)$.

\begin{definition}\label{def:chain} 
  A \emph{chain} $\mathcal{K}=\{K_i\}_{i\in I}$ in $\Davis$,
  indexed by $\emptyset\neq I\subset\mathbb{Z}$,  is a
  collection of disjoint walls $K_i$, such that $i<j<k$ in $I$ implies
  $K_j$ separates $K_i$ and $K_k$.
  The number of walls in the chain is denoted $|\mathcal{K}|$.
\end{definition}
\fullref{def:chain} is a slight abuse of terminology. The complementary halfspaces of the walls $K\in\mathcal{K}$ can be
labelled $K^+$ and $K^-$
in such a way that for $i<j$ we have $K_i^-\subset K_j^-$ and
$K_i^+\supset K_j^+$.
Then the set of halfspaces $K_i^-$, ordered consistently with the
indices, gives a chain in the partially ordered set of halfspaces with
respect to inclusion.
Thus, it is really the nested sequence of halfspaces that is the
``chain''.
Since the ordered index set is part of \fullref{def:chain},
$\mathcal{K}$ contains the same information as the nested sequence of
halfspaces $K_i^{-}$, so the terminology is justified.

We say a wall $M$ or path $\gamma$ is \emph{transverse to a chain}
$\mathcal{K}$ if it is
transverse to every wall in the chain, and write $M\transverse
\mathcal{K}$ or $\gamma\transverse\mathcal{K}$, respectively.
We say \emph{two chains $\mathcal{V}$ and $\mathcal{H}$ are transverse}, and write $\mathcal{V}\transverse\mathcal{H}$, if
$V\transverse H$ for all $V\in\mathcal{V}$ and all $H\in\mathcal{H}$.

The \emph{extremes} of a chain are the at most two walls of the chain
that do not separate any of the other walls of the chain, i.e., the
first and last walls, if they exist. 

If $\mathcal{K}:=\{K_i\}_{i\in I}$ is a chain and $\emptyset\neq J\subset I$
then $\{K_j\}_{j\in J}$ is a \emph{subchain} of $\mathcal{K}$. It is a
\emph{convex subchain} if $j_0,j_1\in J$ and $i\in I$ with $j_0<i<j_1$
implies $i\in J$.

\begin{lemma}\label{lem:chains_irreducible}
  If $\mathcal{K}:=\{K_i\}_{i\in I}$ is a chain then $W(\mathcal{K})$
  and $\pc(\mathcal{K})$ are
    irreducible and there exist $i_0,i_1\in I$ such that
      $\pc(\mathcal{K})=\pc(K_{i_0},K_{i_1})$. If $|I|>1$ then
      $W(\mathcal{K})$ is infinite.
\end{lemma}
\begin{proof}
  If $\mathcal{K}$ has only one wall then
  $W(\mathcal{K})=\pc(\mathcal{K})\cong\mathbb{Z}/2\mathbb{Z}$ is
  irreducible.
  Otherwise, it contains disjoint walls.
   Reflections in disjoint walls generate an infinite dihedral group,
  by \fullref{lem:dihedral},
  so $W(\mathcal{K})$ is infinite once $\mathcal{K}$ contains distinct
  walls, and no two reflections through distinct walls of
  $\mathcal{K}$ commute.
Irreducibility of $W(\mathcal{K})$ and $\pc(\mathcal{K})$ then follows from \cite[Lemma~2.2~(ii)]{CapMar13}.
 
  For the claim that the parabolic closure is determined by reflections
  in at most two of the walls, \cite[Lemma~17]{Cap06} (see
  \fullref{Caprace_17}) implies that if
  $i<j<k$ are distinct indices in $I$ then
  $\refl{K_j}\in\pc(\refl{K_i},\refl{K_k})$, so for any finite
  subinterval $[i_0,i_1]\subset I$ we have $\pc(\{K_i\}_{i_0\leq i\leq
    i_1})=\pc(\{K_{i_0},K_{i_1}\})$.
  Take an exhaustion of $I$ by finite subintervals. This gives a
  nested sequence of parabolic closures of the corresponding convex
  subchains. A strict inclusion of parabolic subgroups requires a
  strict increase in rank, but the ranks of parabolic subgroups are
  bounded above by $|S|$, so the sequence of parabolic closures
  stabilizes at some finite stage. 
\end{proof}

%% file: parabolicnormalizer.tex
There has been extensive work on describing normalizers and
centralizers of parabolic subgroups of $(W,S)$
\cite{Bri96, MR576184, AleMicNer13,
  BriHow99, Bor98,All12,All13,Nui11}. The case of spherical parabolics
is complicated. The nonspherical irreducible case is much simpler, as we now briefly recall.

\begin{lemma}\label{lem:reflection_splitting}
 A reflection subgroup of a direct product of special subgroups splits
as a direct product of reflection subgroups of the factors. 
\end{lemma}
\begin{proof}
  Suppose $(W,S)$ is a Coxeter system and $r$ is a reflection in
  $W_{T_0}\times W_{T_1}\leq W$ for some $T_0,T_1\subset S$.
  Then $r=wsw^{-1}$ for
some $w\in W_{T_0}\times W_{T_1}$ and $s\in T_0\sqcup T_1$ \cite[Lemma~4.2.3]{davisbook}.
Suppose $w=uv$ for $u\in W_{T_0}$ and $v\in W_{T_1}$. 
If $s\in T_0$ then $r=uvsv^{-1}u^{-1}=usu^{-1}\in W_{T_0}$. 
If $s\in T_1$ then $r=uvsv^{-1}u^{-1}=vsv^{-1}\in W_{T_1}$.
Thus, each reflection is in one factor or the other, and the factors
commute, by hypothesis. 
\end{proof}

\begin{proposition}\label{lem:nonspherical_normalizer}
  Let $(W,S)$ be a Coxeter system. Let $T\subset S$ be irreducible and
  nonspherical.
  Let $T^\perp:=\{s\in S\setminus T\mid \forall t\in T,\,st=ts\}$.
  Then:
  \begin{enumerate}
    \item If $wTw^{-1}\subset S$ for some $w\in W$ then $wTw^{-1}=T$.\label{item:irreducible_rigidity}
    \item  The normalizer of $W_T$ in $W$ is $\normalizer_W(W_T)=W_T\times
      W_{T^\perp}$.\label{item:normalizer_splits}
      \item The centralizer of $W_T$ in $W$ is $\centralizer_W(W_T)=W_{T^\perp}$.
  \item Every reflection in $\normalizer_W(W_T)$ is either in
  $W_T$ or in $W_{T^\perp}$.\label{item:reflections_in_normalizer}
  \end{enumerate} 
\end{proposition}
\begin{proof}
  These results are well known, see \cite[Lemma~3.3]{CapMin13}, 
 \cite[Lemma~22 and the preceding line]{Cap06}, \cite[Lemma~2.2]{Cap09},
 \cite[Theorem~5.7]{CorLev25}.
 They are usually attributed to Deodhar \cite[Proposition~5.5]{Deo82}, with further explanation in \cite[Section~3.1]{Kra09}.
  Deodhar's result, and the text prior to it, generalizes some root system facts used by Howlett in
  the description of parabolic normalizers for finite Coxeter groups
  \cite{MR576184}, which would later be used by Brink and Howlett
  \cite{BriHow99} to describe parabolic normalizers for infinite
  Coxeter groups.

  The proposition is the irreducible nonspherical special
  case of the parabolic normalizer theorem obtained from Howlett's
  argument, together with Deodhar's refinements of the elementary
  conjugating elements. The final assertion follows
  from \fullref{lem:reflection_splitting}.
\end{proof}
\begin{lemma}[{\cite[Proposition~4.5.10]{davisbook}}]\label{special_conjugates}
  If $T,U\subset S$ and $wW_Tw^{-1}=W_U$ then there is an element
  $u\in W_U$ such that for $v:=u^{-1}w$ we have $vTv^{-1}=U$.
\end{lemma}
\begin{corollary}\label{irreducible_parabolics_have_type}
If $P$ is an irreducible nonspherical parabolic subgroup of $(W,S)$,
then there is a unique $T\subset S$ such that $P$ is
conjugate to $W_T$.
\end{corollary}
\begin{proof}
  If $P$ is conjugate to $W_U$ and $W_T$ then
  \fullref{special_conjugates} says $vTv^{-1}=U$.
  Then since $P$ is irreducible and nonspherical,
  \fullref{lem:nonspherical_normalizer}~\eqref{item:irreducible_rigidity}
  says $T=U$.
\end{proof}

The following corollary of \fullref{lem:nonspherical_normalizer}
establishes the definition of the \emph{orthogonal factor} parabolic
$P^\perp$ of an irreducible nonspherical parabolic $P$.

\begin{corollary}\label{cor:P_or_refP}
  If $P$ is an irreducible nonspherical parabolic subgroup of $(W,S)$, then
  $\normalizer_W(P)$ is parabolic and there is a canonical parabolic
  subgroup $P^\perp$ such that $\normalizer_W(P)=P\times
  P^\perp$.
  
Furthermore, if $R\subset\normalizer_W(P)$ is a set of reflections
then  $\langle R\rangle=\langle R\cap P\rangle\times \langle R\cap P^\perp\rangle$.
\end{corollary}
\begin{proof}
  By definition, there exist an irreducible, nonspherical $T\subset S$ and $w\in W$
  such that $wPw^{-1}=W_T$.
  \fullref{lem:nonspherical_normalizer} says
  $w\normalizer_W(P)w^{-1}=\normalizer_W(W_T)=W_T\times W_{T^\perp}$
  and $\centralizer_W(W_T)=W_{T^\perp}$. 
  Take $P^\perp:=\centralizer_W(P)=w^{-1}W_{T^\perp}w$.
  Then $\normalizer_W(P)=P\times P^\perp$. 

The direct product splitting of $\langle R\rangle$ follows from the claim of
\fullref{lem:nonspherical_normalizer} that reflections in
$\normalizer_W(W_T)$ are contained in either $W_T$ or $W_{T^\perp}$.
\end{proof}

\begin{remark}
  The construction of \fullref{cor:P_or_refP} is symmetric if
  $T^\perp$ is also irreducible and nonspherical, but that is not
  always the case.
  In general,
  $W_T\times W_{T^\perp}$ can be a proper subgroup of
  $\normalizer_W(W_{T^\perp})$. In the most extreme case, $T^\perp$
  could be empty, with $W_{T^\perp}$ trivial and  $\normalizer_W(W_{T^\perp})=W$.
\end{remark}

\begin{corollary}\label{normalizer_parabolics_have_well_defined_subcomplex}
  If $P$ is an irreducible nonspherical parabolic subgroup then there is a
  unique parabolic subcomplex whose stabilizer is $\normalizer_W(P)$.
\end{corollary}
\begin{proof}
  Suppose $P=xW_Tx^{-1}$ for some $x\in W$ and $T\subset S$.
   \fullref{lem:nonspherical_normalizer}~\eqref{item:normalizer_splits}
   and \fullref{cor:P_or_refP} imply that 
  $\normalizer_W(P)=x(W_T\times W_{T^\perp})x^{-1}=xW_{T\sqcup
    T^\perp}x^{-1}$.
  The parabolic subcomplex $x\Davis_{T\sqcup T^\perp}$ has stabilizer
  $\normalizer_W(P)$.
  Suppose $w\Davis_U$ for some $w\in W$ and $U\subset S$ is another
  parabolic subcomplex with stabilizer equal to $\normalizer_W(P)$.
  Its stabilizer is $wW_Uw^{-1}$, so $W_U=w^{-1}xW_{T\sqcup T^\perp}x^{-1}w$.
  By \fullref{special_conjugates}, there exists
  $u\in W_U$ such that for $v:=u^{-1}w^{-1}x$ we have $v(T\sqcup T^\perp)v^{-1}=U$.
  Let $U'$ be the irreducible component of $U$ such that $vTv^{-1}=U'$.
  Since $T$ is irreducible and nonspherical,
  \fullref{lem:nonspherical_normalizer}~\eqref{item:irreducible_rigidity}
  implies $T=U'$, so $v\in\normalizer_W(W_T)=W_{T\sqcup T^\perp}$, which implies
  $W_{T\sqcup T^\perp}v^{-1}=W_{T\sqcup T^\perp}$.
  Then: 
  \[wW_U=xv^{-1}u^{-1}W_U=xv^{-1}W_U=xv^{-1}(vW_{T\sqcup T^\perp}v^{-1})=xW_{T\sqcup T^\perp}\]
  Thus, $x \Davis_{T\sqcup T^\perp}$ and $w\Davis_U$ are the same
  parabolic subcomplex.
\end{proof}

%% file: distance.tex
We have several different measures of separation in $\Davis$:
\begin{itemize}
\item $\dcomb(x,y)$ is the length of the shortest edge path between
  vertices $x$ and $y$.
\item $\dcat(x,y)$ is the CAT(0) distance between $x$ and $y$.
  \item $\dwall(x,y)=|\walls(x,y)|$ is the cardinality of the set
    $\walls(x,y)$ of walls separating points $x$
    and $y$ (which is exactly $\dcomb(x,y)$, when $x,y\in\Davis^{(0)}$).
    \item $\dchain(x,y)$ is the length of the longest chain of walls
      separating $x$ and $y$.
    \end{itemize}
    Of these, only $\dcat$ is a metric on $\Davis$.
    The combinatorial metric $\dcomb$ extends to  a geodesic metric on the 1--skeleton
    of $\Davis$.
    The wall and chain measures $\dwall$ and $\dchain$ are metrics on
    the 0--skeleton.
    There are ways to extend them to metrics on all of $\Davis$, but we will not use this fact.
    For us it will be enough to take coarse extensions, in the sense
    that if $x$ and $y$ are arbitrary points of $\Davis$, then for
    $d\in\{\dcomb,\dwall,\dchain\}$ we let $d(x,y)$ be the minimum of
    $d(x',y')$ over vertices $x'$ and $y'$ such that $x'$ is contained
    in a common closed chamber with $x$ and $y'$ is
    contained in a common closed chamber with $y$.
 
Niblo and Reeves \cite{NibRee03}  constructed a finite dimensional,
locally finite CAT(0) cube complex $\NR$ corresponding to the Coxeter
system $(W,S)$. Its wall structure is isomorphic to that of $\Davis$,
and after choosing basepoints the orbit map gives a $W$--equivariant
isometric embedding $(\Davis^{(0)},\dwall)\to(\NR^{(0)},\dwall)$.
By the Helly property for walls in a CAT(0) cube complex, 
      a collection of pairwise intersecting walls in $\NR$ has a
      common point of intersection, so the size of such a set is
      bounded by $\dim(\NR)$. Since the wall structures of $\Davis$
      and $\NR$ are isomorphic, $\dim(\NR)$ is also the maximum
      cardinality of a set of pairwise intersecting walls in $\Davis$,
      even though in $\Davis$ such a set need not have a common
      intersection point. 

\begin{lemma}\label{all_metrics_are_equivalent}
  $\dcat\ceq\dcomb\ceq\dwall\ceq\dchain$ as maps $\Davis^2\to\mathbb{R}$.
    \end{lemma}
    \begin{proof}
      There are only finitely many isometry types of cells of
      $\Davis$, so every point is within uniformly bounded distance of
      a vertex.
      Moreover, any closest vertex to a point $x$ is not separated
      from $x$ by any wall.
      Thus, up to increasing the additive error afterwards, it
      suffices to prove the lemma on $\Davis^{(0)}$. 

      $W$ acts geometrically on $\Davis$, freely and transitively on
      $\Davis^{(0)}$, and $\Davis^{(1)}$ is precisely the Cayley graph
      of $W$ with respect to $S$, so $\dcomb$ and $\dcat$  on
      $\Davis^{(0)}$ are
      quasiisometric to the word metric on $W$ with respect to $S$.
      
On $\Davis^{(0)}$ $\dcomb=\dwall$; in one direction every edge is cut
by a unique wall, so the number of walls separating vertices $x$ and
$y$ is at most the length of a combinatorial geodesic between them. In
the other direction, the Exchange Condition for
    reflection systems \cite[Section~3.2]{davisbook} implies that a
    combinatorial geodesic is cut at most once by any given wall. 

    Since the wall structure of $\Davis$ is isomorphic to that of
    $\NR$, the relation between $\dwall$ and $\dchain$ can be computed
    there. 
      It is immediate from the definitions that $\dchain\leq\dwall$.
      Conversely, for any $w_1$ and $w_2$ in $W$, consider the set of
      walls separating $w_1$ and $w_2$.
      If $H$ is such a wall, let $H^-$ be the halfspace with $w_1\in
      H^-$ and $\bdry H^-=H$, and let $H^+$ be its complementary
      halfspace.
      Consider the set of halfspaces $H^-$ partially ordered by
      inclusion.
      Since they all contain $w_1$, two such halfspaces $H_1^-$ and
      $H_2^-$ are incomparable if and only if $H_1$ and $H_2$ cross.
      Therefore, antichains in the partially ordered set
      correspond to collections of pairwise crossing walls.
      Dilworth's Theorem says that in a partially ordered set the
      maximum size of an antichain is equal to the minimum size of a
      partition of the set into chains.
      At least one of these has at least average size, so there is
      some chain separating $w_1$ from $w_2$ containing at least
      $\dwall(w_1,w_2)/\dim(\NR)$--many walls.
      Thus, $\dchain\leq\dwall\leq\dim(\NR)\dchain$.
    \end{proof}
\fullref{all_metrics_are_equivalent}  enables the following little
trick, which we will use frequently:
\begin{lemma}[4--gon trick]\label{quad_trick}
  Let $x_0$, $x_1$, $y_0$, and $y_1$ be four points. 
  Let $\mathcal{V}:=\{V_i\}_{i\in I}$ be a chain separating $x_0$ and $y_0$, oriented from $x_0$
      to $y_0$, so that for all $i\in I$ we have $x_0\in V_i^{-}$ and
      $y_0\in V_i^{+}$.
  Let $I_x:=\{i\in I\mid x_1\in
  \bar V_i^+\}$, let $I_y:=\{i\in I\mid y_1\in
  \bar V_i^{-}\}$, let $I':= I\setminus(I_x\cup I_y)$.
  Let $\mathcal{V}_x$, $\mathcal{V}_y$, and $\mathcal{V}'$ be the
  subchains of $\mathcal{V}$ with walls index sets $I_x$,
  $I_y$, and $I'$, respectively.
  Then $I_x\subset I$ is an initial subinterval, $I_y\subset I$ is a terminal
  subinterval, and $\mathcal{V}'$, if nonempty, is a convex
  subchain of $\mathcal{V}$ separating $x_1$ from $y_1$.
  See \fullref{fig:fourgontrick}.
  Furthermore:

\[|\mathcal{V}|-|\mathcal{V}'|\leq\dchain(x_0,x_1)+\dchain(y_0,y_1)+2\ceq
  \dcat(x_0,x_1)+\dcat(y_0,y_1)\]
In particular, when $|\mathcal{V}|$ is much larger than
$\dcat(x_0,x_1)$ and $\dcat(y_0,y_1)$, 
then $\mathcal{V}'$ is ``most of'' $\mathcal{V}$.

More generally, if $\mathcal{M}$ is any collection of walls separating
$x_0$ and $y_0$ and $\mathcal{M}''$ is the subset of $\mathcal{M}$ separating $x_1$ and
$y_1$ then:
\[|\mathcal{M}|-|\mathcal{M}''|\stackrel{+}{\cleq} \dwall(x_0,x_1)+\dwall(y_0,y_1)\]
\end{lemma}
 \begin{figure}[h]
        \centering
        \includegraphics{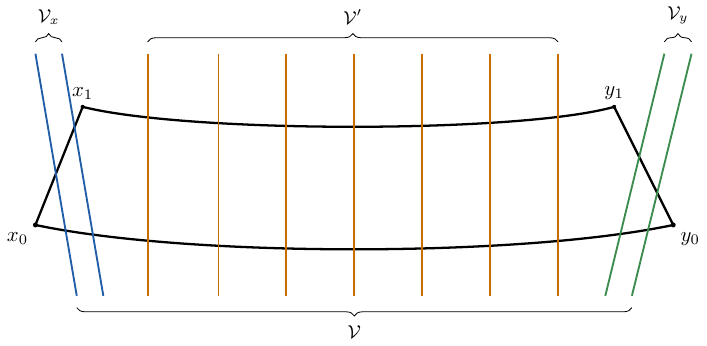}
        \caption{4-gon trick: If $\dcat(x_0,x_1)$ and $\dcat(y_0,y_1)$ are small then
          most of the walls separating $x_0$ and $y_0$ also separate
          $x_1$ and $y_1$.}
        \label{fig:fourgontrick}
      \end{figure}
    \begin{proof}
      If $i<j$ are indices in $I$ with $j\in I_x$ then $x_1\in\bar
      V_j^+\subset\bar V_i^+$, so $i\in I_x$. Thus, $I_x$ is an
      initial subinterval of $I$.
      A similar argument shows $I_y$ is a terminal subinterval of $I$.
      If $I'\neq\emptyset$, then it is the middle subinterval of $I$
      separating $I_x$ and $I_y$.
      For $i\in I'$, the condition $i\notin I_x$ implies $x_1\in
      V_i^{-}$, and the condition $i\notin I_y$ implies $y_1\in
      V_i^+$, so $V_i$ separates $x_1$ and $y_1$.

      Define $I'_x:=\{i\in I\mid x_1\in V_i^+\}$.
      Since $\mathcal{V}$ is a chain, its walls are disjoint, so $x_1$
      belongs to at most one of its walls, so $I'_x$ is all except
      possibly the last index of $I_x$.
      Then $\{V_i\}_{i\in I'_x}$ is a chain separating
      $x_0$ and $x_1$, giving $\dchain(x_0,x_1)\geq |I'_x|\geq |I_x|-1$.
      A similar estimate holds for $y_0$ and $y_1$, so:
      \begin{align*}
        |\mathcal{V}|-|\mathcal{V}'|&=|I|-|I'|\\
                                    &\leq |I_x|+|I_y|\\
        &\leq \dchain(x_0,x_1)+\dchain(y_0,y_1)+2
      \end{align*}
The final coarse equivalence comes from
\fullref{all_metrics_are_equivalent}.

For the more general $\dwall$ version, the walls of $\Davis$ are
countable, so we can number them and make all the same index
definitions as in the previous case, just without any geometric
significance to the indexing.
The corresponding $\mathcal{M}'$ is a subset of $\mathcal{M}''$, so it
suffices to coarsely bound $|\mathcal{M}|-|\mathcal{M}'|$.
For that,  a similar computation to the previous case applies,
because the maximal number of pairwise intersecting walls is
$\dim\NR$, so $|I_x|-|I'_x|$ and $|I_y|-|I'_y|$ are still uniformly bounded.
    \end{proof}    

    We can extend \fullref{all_metrics_are_equivalent} to a
    point-parabolic subcomplex distance comparison:
    \begin{lemma}\label{lem:distance_to_standard}
      The following functions on $W=\Davis^{(0)}$ are coarsely
      equivalent, with constants 
      independent of $T\subset S$:
      \[w\mapsto
        \begin{cases}
          \dcomb(w,\Davis_T)\\
          \dcat(w,\Davis_T)\\
          \dwall(w,\Davis_T)\\
          \dchain(w,\Davis_T)
        \end{cases}
\]
\end{lemma}
\begin{proof}
  Let $w\in W=\Davis^{(0)}$.
  By \cite[Lemma~4.3.3]{davisbook} there is a unique decomposition 
   $w=vu$ where $v\in W_T$ and $u$ is the unique minimal length
   representative of $W_Tw$.
   Thus, $v$ is the closest vertex of $W_T$ to $w$.
  Take a $\dcomb$--geodesic $\gamma$ from $w$ to $v$.
   A wall $M$ separating
   $w$ from $v$ cuts the interior of $\gamma$.
   Split $\gamma$ as a concatenation $\gamma'+\gamma''$ of nontrivial subpaths
   meeting at $M\cap \gamma$.
   Then $\gamma'+\refl{M}(\gamma'')$ tightens to an edge path from $w$
   to $\refl{M}(v)$ of length at most  $|\gamma|-1$, with $v$ and $\refl{M}(v)$ on
   opposite sides of $M$.
   If $M$ separates $w$ from $v$ but does not separate $w$ from
   $\Davis_T$ then it cuts $\Davis_T$.
   However, the reflections whose walls cut $\Davis_T$ are precisely
   the reflections in $W_T$, which stabilize $\Davis_T$.
   This gives a contradiction, because then  $\refl{M}(v)$ is a vertex
   of $\Davis_T$ that is closer to $w$ than $v$.
Thus, the walls separating $w$ from $\Davis_T$ are precisely the
walls separating $w$ from $v$.
Now apply \fullref{all_metrics_are_equivalent} for distances between
$w$ and $v$.
This accounts for the wall-based distances in the statement.
The CAT(0) distance follows because every point in $\Davis_T$ is
within uniformly bounded distance of a vertex, so $\dcat(w,\Davis_T)\cequiv\dcat(w,\Davis_T^{(0)})$.
\end{proof}

We will not use the following result, but it must be
mentioned that in addition to point-point and
point-parabolic subcomplex coarse equivalences $\dcomb\ceq\dcat\ceq\dchain$,
there are also point-wall and wall-wall versions.

\begin{theorem}[{Strong Parallel Walls Theorem, \cite[Theorem~C]{Cap06} and
      \cite[Theorem~2.3]{CapPrz11}}]\label{SPWT}
Let $\Davis=\Davis(W,S)$ be the Davis complex of a
finite rank Coxeter system.
Let $M$ and $M'$ be disjoint walls in $\Davis$, and let $x$ be a
    point in $\Davis$.
    \begin{enumerate}
    \item $\dchain(M,x)\ceq\dcomb(M,x)\ceq\dcat(M,x)$\label{SPWT:wall_vertex}
      \item $\dchain(M,M')\ceq\dcomb(M,M')\ceq\dcat(M,M')$\label{SPWT:wall_wall}
    \end{enumerate}
    The constants of the equivalences are independent
      of $M$, $M'$, and $x$. 
    \end{theorem}
    We remark that the references only state that these distances are
    related by some function, but it is straightforward to check that
    the proof actually produces a linear function. 

In contrast to \fullref{lem:distance_to_standard} and \fullref{SPWT},
distances to arbitrary CAT(0)-convex sets are not approximated
by wall distances.
Here are two explicit examples:
    \begin{example}
      Consider the standard square tiling of
      $\mathbb{R}^2$.
      
      First, let $L$ be a line that is not vertical or horizontal.
      It is CAT(0)--convex, but it crosses every vertical and
      horizontal line, so the intersection of halfspaces containing it
      is an empty intersection, hence is the whole plane.
      There are points of the plane arbitrarily far from $L$, but they
      are not separated from $L$ by any wall. 

      Second, for each $n\in\mathbb{N}$, consider $v:=(n,n)$ and the closed ball
      $B$ of radius $n$ about $(0,0)$.
      The intersection of halfspaces containing $B$ is the
      square bounded by vertical and horizontal lines at distance
      $n+1/2$ from $(0,0)$. This set contains $v$, so no wall
      separates $v$ and $B$, but $\dcat(v,B)=(\sqrt{2}-1)n$.
    \end{example}

%% file: characterization_leadin.tex
\fullref{sec:bridge} reviews the concept of a ``bridge'' between convex
subsets of a CAT(0) space, and then characterizes the
    Morse property for quasigeodesic paths in terms of dense
    transverse chains, \fullref{thm:wall_morse}.
\fullref{sec:grid} describes ``grids'' and uses them to produce wide
parabolics.
\fullref{thm:wide_parallel} in that subsection relates a quasigeodesic's interactions with wide parabolic subcomplexes and
dense transverse chains.
\fullref{thm:morse_characterization} in \fullref{sec:proof} gives the
main result of this section characterizing the Morse property for
certain quasigeodesics. 

The guiding intuition is that a quasigeodesic should be Morse if it
avoids the wide parabolic subcomplexes.
For a precise characterization, we need to make ``avoids'' quantitative
and \emph{uniform}.
If $\gamma$ is a quasigeodesic path with a long subsegment $\gamma'$ contained
in the $R$--neighborhood of a wide parabolic subcomplex $\Upsilon$, then
\fullref{quad_trick} implies that for any long chain $\mathcal{V}$
separating the endpoints of $\gamma'$ there is a convex subchain $\mathcal{V}'$ of $\mathcal{V}$
that also cuts $\Upsilon$, with $|\mathcal{V}|-|\mathcal{V}'|$ bounded
in terms of $R$. 
One might hope in the other direction that if a long chain
$\mathcal{V}$ is transverse to both $\gamma$ and a wide parabolic
subcomplex $\Upsilon$, then, maybe after throwing away some walls from
the ends of $\mathcal{V}$, the subsegment $\gamma'$ between the
extremes of the remaining chain would be contained in a uniform neighborhood of $\Upsilon$.
That is not true, see \fullref{ex:wrong_parabolic}.
The main novelty of this section is an argument, whose heart is
\fullref{prop:induction_argument}, that we can make a trade: there is a convex
subchain $\mathcal{V}'$ of $\mathcal{V}$, consisting of a definite
fraction of its walls, such that $\gamma$ is closer in the
wall geometry to the wide parabolic subcomplex 
$\Upsilon'$ stabilized by $\normalizer_W(\pc(\mathcal{V}'))$ than it
was to $\Upsilon$.
An induction on rank shows that after uniformly finitely many
such trades we get a subsegment of $\gamma$ contained in a uniform
neighborhood of the resulting wide parabolic subcomplex.

\begin{example}\label{ex:wrong_parabolic}
  Take the Coxeter diagram with vertices $S:=\{a,b,c,x,y,z\}$ and edges
  labelled $\infty$ for $ab$, $ac$, $bc$, $xy$, $xz$, $yz$, and $cz$
  (all other generator pairs commute).
  Let $U:=\{a,b\}$ and $T:=\{a,b,c\}$.
  Then $U^\perp=\{x,y,z\}$ and $T^\perp=\{x,y\}$, and both $W_T\times W_{T^\perp}$ and $W_U\times
  W_{U^\perp}$ are wide.

  Let $\Upsilon:=\Davis_{T\sqcup T^\perp}$ and $\Upsilon':=\Davis_{U\sqcup
    U^\perp}$.
  Let $\gamma_0\subset \Davis_U\subset \Davis_T$ be a long geodesic segment, let
  $\mathcal{V}$ be a chain of $U$--walls separating its endpoints, and
  consider $\gamma_n:=(zx)^n\gamma_0$.
  The product structure on $\Upsilon'$ shows the $\gamma_n$ are all
  geodesic segments contained in $\Upsilon'$ that are all cut by $\mathcal{V}$. 
However, they get arbitrarily far from $\Upsilon$ as
$n\to\infty$: the retraction $W\onto W_{\{x,z\}}$ killing
$\{a,b,c,y\}$ is $1$--Lipschitz and sends $\gamma_n^{(0)}$ to $(zx)^n$ while sending $\Upsilon^{(0)}$ onto
$\{1,x\}$. 

Thus, $\Upsilon$ is simply the wrong wide parabolic subcomplex,
despite being cut by the same chain $\mathcal{V}$ that cuts
$\gamma_n$. 
We trade shrinking $T$ to $U$ for expanding $T^\perp$ to $U^\perp$, thereby
 exchanging $\Upsilon$ for a different wide parabolic subcomplex
$\Upsilon'$ closer to, and in this example actually containing, $\gamma_n$.
\end{example}

%% file: bridge.tex
\begin{lemma}[{CAT(0) bridge lemma (cf \cite[Lemma~2.3]{BesKleSag08})}]\label{lem:bridge}
  Let $X$ be a complete CAT(0) space.
  Let $A$ and $B$ be nonempty closed convex subsets of $X$ such that
  $\dcat(A,B)$ is realized. 
  Define $A_0:=\{a\in A\mid
  \dcat(a,B)=\dcat(A,B)\}$, $B_0:=\{b\in B\mid
  \dcat(b,A)=\dcat(A,B)\}$, and
   \[\bridge(A,B):=\bigcup_{a\in A_0}[a,\pi_B(a)]\]
  \begin{enumerate}
  \item $A_0$ and $B_0$ are closed and convex.
  \item $\pi_B|_{A_0}\from A_0\to B_0$ and $\pi_A|_{B_0}\from B_0\to
    A_0$ are inverse isometries. 
  \item $\bridge(A,B)$ is closed and convex and is isometric to a product $A_0\times
    [0,\dcat(A,B)]$.
    \item\label{item:dependent_constant}
    Suppose there exists $\epsilon>0$ such that:
    \begin{equation}
      \label{flare}
      a\in A,\,\dcat(a,A_0)=1  \implies \dcat(a,B)\geq \dcat(A,B)+\epsilon\tag{$\Yleft$}
    \end{equation}
 
    Then for every $a\in A$ with $\dcat(a,A_0)\geq 1$ we have: 
    \[\dcat(a,B)  \geq\ \dcat(A,B)+\epsilon \dcat(a,A_0)\]
    \item\label{item:subcomplexes_flare} If $X$ is a locally finite complex with a cocompact isometry
      group then for every $D$ there exists $\epsilon$ such that if
      $A$ and $B$ are convex subcomplexes with $\dcat(A,B)\leq D$ then
      $\epsilon$ is a \emph{flaring bound} for $A$ and $B$ satisfying \eqref{flare}. 
   \end{enumerate}
 \end{lemma}
 As an example, consider two lines in $\mathbb{E}^2$ intersecting at
 angle $0<\theta<\pi/2$.
 Their bridge is the single point of intersection. Projection of one
 line to the other is surjective, so projections between convex sets need not be coarsely equivalent
 to the bridge (unlike in the cubical case). The optimal flaring bound satisfying
 \eqref{flare} is $\epsilon=\sin(\theta)$, so it depends on the pair of convex
 sets.

\fullref{lem:bridge} is just \cite[Lemma~2.3]{BesKleSag08} stated
slightly differently, with the additional hypothesis that $X$ is
complete, and with the final item of \cite[Lemma~2.3]{BesKleSag08}
split into two separate statements.

\begin{corollary}\label{cor:walls_are_tame}
  Walls in $\Davis$ flare: walls are convex, distances between walls are realized, and 
for all $D$ there exists $\epsilon$ such
that if $M$ and $M'$ are walls of $\Davis$ with $\dcat(M,M')\leq D$ then
$\epsilon$ is a flaring bound for $M$ and $M'$ satisfying
\eqref{flare}.
 \end{corollary}
 \begin{proof}
   There is a subdivision $\Davis'$ of $\Davis$, as in \fullref{def:Hasse}, in which the walls of $\Davis$
   are subcomplexes.
Walls are convex in $\Davis$, so they are also convex in $\Davis'$,
since the underlying metric space is the same.
Since $\Davis'$ is locally finite and $W$
   acts cocompactly, distances between convex subcomplexes are
   realized.
   Thus, walls satisfy the hypotheses of \fullref{lem:bridge} and the
   additional hypotheses of 
   \fullref{lem:bridge}~\eqref{item:subcomplexes_flare}.
 \end{proof}

\begin{lemma}\label{lem:length_wall_crossers_improved}
For all $D$ there exists
  $k>0$  such that for all walls $M$, $M'$ of $\Davis$ with
  $\dcat(M,M')\leq D$, all paths $\gamma$ from $M$ to $M'$, and all $R$ we have:
  \[\gamma\not\subset\nbhd_R(\bridge(M,M'))\implies|\gamma|\geq kR-2\]
\end{lemma}
\begin{proof}
  Let $D_0:=\dcat(M,M')\leq D$.
   By \fullref{cor:walls_are_tame}, given $D$  there exists a flaring bound
   $\epsilon$ for $M$ and $M'$.
   Define $k:=\frac{\epsilon}{1+\epsilon}$.
  
  Let $M_0:=\{m\in M\mid \dcat(m,M')=\dcat(M,M')\}$ and $M'_0:=\{m\in
  M'\mid \dcat(m,M)=\dcat(M,M')\}$.
  Suppose there exists $x\in\gamma$ at least $R$--far from the bridge.
  Let $m:=\pi_M(x)$ and $m':=\pi_{M'}(x)$.
  \begin{equation}
    \label{eq:6}
    |\gamma|\geq \dcat(x,m)+\dcat(x,m')\geq \dcat(m,m')\geq\dcat(m,M')
  \end{equation}

  If $\dcat(x,m)\geq R$ or $\dcat(x,m')\geq R$ then $|\gamma|\geq R>kR-2$.
  
  If $\dcat(m,M_0)< 1$ and $\dcat(m',M'_0)< 1$ then $|\gamma|> 2(R-1)>kR-2$.

  Up to swapping the roles of $M$ and $M'$, it remains to consider the case that
  $\dcat(m,M_0)\geq 1$ and $\dcat(x,m)<R$.
In this case \fullref{lem:bridge} gives:
\[  \dcat(m,M')\geq D_0+\epsilon\dcat(m,M_0)\geq
    D_0+\epsilon(\dcat(x,M_0)-\dcat(x,m))\]
     
  Since $M_0\subset\bridge(M,M')$, we have $\dcat(x,M_0)\geq\dcat(x,\bridge(M,M'))\geq R$, so:
   \begin{equation}
     \label{eq:7}\dcat(m,M')\geq D_0+\epsilon(R-\dcat(x,m))
       \end{equation}

  Combining equations \eqref{eq:6} and \eqref{eq:7}:
  \begin{align*}
      |\gamma|&\geq\max\{D_0+\epsilon(R-\dcat(x,m)),\dcat(x,m)\}\\
              &\geq \frac{\epsilon}{1+\epsilon}R+\frac{D_0}{1+\epsilon}\\
    &>kR-2\qedhere
  \end{align*}
\end{proof}

\begin{corollary}\label{cor:qgeo_wall_crossers}
  For all $D$ there exists
  $k>0$, as in \fullref{lem:length_wall_crossers_improved}, such that for all walls $M$, $M'$ of $\Davis$ with
  $\dcat(M,M')\leq D$, all $(\lambda,\epsilon)$--quasigeodesics
  $\gamma\from [0,L]\to\Davis$ from $M$ to $M'$, and all $R$ we have:
  \[\gamma\not\subset\nbhd_R(\bridge(M,M'))\implies L\geq  (k/\lambda)R-(2/\lambda)(\epsilon+1)\]
\end{corollary}
  \begin{proof}
    Suppose $t\in [0,L]$ is such that
    $\gamma(t)\notin\nbhd_R(\bridge(M,M'))$.
    By \fullref{lem:length_wall_crossers_improved}, the length of the
    concatenation of the geodesic from $\gamma(0)$ to $\gamma(t)$ with
    the geodesic from $\gamma(t)$ to $\gamma(L)$ has length at least
    $kR-2$, so $\lambda t+\epsilon +
    \lambda(L-t)+\epsilon\geq kR-2$, which implies
    $L\geq (k/\lambda)R-(2/\lambda)(\epsilon+1)$.
  \end{proof}

\begin{lemma}\label{lem:bounded_bridge}
  For all $C$ and $D$ there exists $B$ such that for all walls $M$,
  $M'$ of $\Davis$ such that 
  $\dcat(M,M')\leq D$ and such that the longest chain transverse to
  both $M$ and $M'$ contains at most $C$--many walls, the diameter of 
$\bridge(M,M')$ is at most $B$.
\end{lemma}
\begin{proof}
  Let $D_0:=\dcat(M,M')$ and $M_0:=\{m\in M\mid \dcat(m,M')=D_0\}$.
  Since $\bridge(M,M')$ is isometric to a product $M_0\times[0,D_0]$
  and $D_0\leq D$, it suffices to bound the diameter of $M_0$.
  For any $x,y\in M_0$, consider
  the geodesic 4--gon whose sides are $[x,y]$, $[y,\pi_{M'}(y)]$,
  $[\pi_{M'}(y),\pi_{M'}(x)]$, and $[\pi_{M'}(x),x]$.
Let $\mathcal{K}$ be a chain realizing
$\dchain(x,y)$ and $\mathcal{K}'$ the subchain of $\mathcal{K}$
consisting of walls separating $\pi_{M'}(x)$ and $\pi_{M'}(y)$. 
By hypothesis
  $|\mathcal{K}'|\leq C$, so
  \fullref{all_metrics_are_equivalent} and \fullref{quad_trick} give:
  \[    \dcat(x,y)\ceq \dchain(x,y) 
              \cequiv \dchain(x,y) -C
              \leq |\mathcal{K}|-|\mathcal{K}'|\cleq 2D_0\qedhere\]
\end{proof}
\begin{remark}
The converse of \fullref{lem:bounded_bridge} is not true: having
a small bridge does not limit the number of common crossers. This is
different than in the right-angled case.
For example, the Davis complex of the 3,3,3 triangle group is the
regular hexagonal tiling of the plane. The walls are three families of
parallel lines.
Any line $M$ from the first family crosses any line $M'$ from the
second family, so their bridge is the single point $M\cap M'$, but the
entire infinite third family consists of lines that cross both $M$ and $M'$.
\end{remark}

%% file: wall_morse.tex
\begin{theorem}\label{thm:wall_morse}
  For all $\lambda$ and $\epsilon$ the following are equivalent for
  all $(\lambda,\epsilon)$--quasigeodesic paths $\gamma\from
  I\to\Davis$. Moreover, there are explicit bounds between the constraints of each
item that are independent of $\gamma$. 
  \begin{enumerate}[label=(\alph*)]
  \item There exist $K>2$, $r_0$, and a quadratic polynomial $q(r)$
    with positive leading coefficient such that $\gamma$ has
    divergence gauge $\delta(r;1,K)$ satisfying $\delta(r;1,K)\geq
    q(r)$ for all $r\geq r_0$.\label{wall_morse:divergence}
      \item There exists $\mu$ such that $\gamma$ is $\mu$--Morse.\label{wall_morse:morse}
    \item There exists $\chi$ such that $\gamma$ is $\chi$--strongly contracting. \label{wall_morse:strong_contraction}
    \item There exist $C$ and $L_0$ such that for every compact
      subinterval $I'\subset I$ of length at least $L_0 $ there exists a chain
      $\mathcal{V}:=\{V_i\}_{1\leq i\leq n}$ separating the endpoints
      of $\gamma(I')$  such that for every pair $V_i$, $V_{i+1}$ of consecutive walls of
      $\mathcal{V}$ there are at most $C$--many walls transverse to
      $\{V_i,V_{i+1}\}$  and such that for all $0\leq i\leq n$ we
      have $|t_{i+1}-t_i|\leq 7L_0$, where $I'=[t_0,t_{n+1}]$ and 
      $t_i:=\min\{t\in I'\mid \gamma(t)\in V_i\}$. \label{wall_morse:bounded_gap_bounded_walls}
     \item There exist $C$ and $0<\rho$ such that there
       exists $L_1$ such that for every compact subinterval $I'\subset I$
       of length at least $L_1$ there exists a chain
       $\mathcal{V}$ separating the endpoints of $\gamma(I')$ with
       $|\mathcal{V}|/|I'|\geq \rho$ such
      that for every pair $V$, $V'$ of consecutive walls of
      $\mathcal{V}$ the longest chain $\mathcal{H}$ transverse to
      $\{V,V'\}$ has $|\mathcal{H}|\leq C$. \label{wall_morse:bounded_gap_bounded_chains}
  \end{enumerate}
\end{theorem}
The proof of the theorem will follow from some auxiliary results.
\fullref{cor:subsegment_containing_projection} and
\fullref{lem:regular_small_projection} use strong contraction to
produce walls that are "almost orthogonal to" $\gamma$. 
\fullref{prop:contracting_implies_separated_walls} assembles these
walls into a chain with good properties. 
\fullref{separated_chain_implies_divergent} derives quadratic
divergence from such a chain.

\begin{lemma}\label{cor:subsegment_containing_projection}
  For all $\chi$ there exists $\chi'$ such that for all $\lambda$ and $\epsilon$ there exists $\chi''$ such that for every
  $\chi$--strongly contracting $(\lambda,\epsilon)$-quasigeodesic path
   $\gamma\from I\to\Davis$, and for every wall $M$
  such that $\diam(\pi_\gamma(M))>2\chi'$, there is a subinterval
  $J_M\subset I$ such that:
  \[\gamma^{-1}(\pi_\gamma(M))\subset J_M\text{\quad and\quad}\gamma(J_M)\subset\bar\nbhd_{\chi''}(M)\]
\end{lemma}
\begin{proof}
  By \fullref{equivalence_of_strong_contraction_conditions}, strong
  contraction is quantitatively equivalent to strong constriction, so there exists $\chi'$
depending only on $\chi$ such that if $[x,y]\subset \Davis$ is a geodesic,  $\diam(\pi_\gamma(x)\cup\pi_\gamma(y))>\chi'$ implies
$[x,y]$ passes within distance $\chi'$ of $\pi_\gamma(x)$ and $\pi_\gamma(y)$.
Suppose $M$ is a wall such that $\diam(\pi_\gamma(M))>2\chi'$.
For every $x_0\in\pi_\gamma(M)$ there exists $x_1\in\pi_\gamma(M)$ with $\dcat(x_0,x_1)>\chi'$, else $\diam(\pi_\gamma(M))$ would be at most
$2\chi'$.
Choose $m_0,m_1\in M$ 
 such that $x_i\in\pi_\gamma(m_i)$.
Since $M$ is convex, it contains the geodesic $[m_0,m_1]$, which, by
the result above, passes within distance $\chi'$ of $\pi_\gamma(m_0)$
and $\pi_\gamma(m_1)$.
Since $\diam(\pi_\gamma(m_0))\leq\chi$, the geodesic $[m_0,m_1]\subset
M$
passes within distance $\chi+\chi'$ of $x_0$.
Since $x_0$ was an arbitrary point of $\pi_\gamma(M)$, this gives
$\pi_\gamma(M)\subset \bar\nbhd_{\chi+\chi'}(M)$.

Take $J_M$ to be the smallest closed
subinterval of $I$ containing $\gamma^{-1}(\pi_\gamma(M))$.
Suppose $I''$ is a maximal open subinterval of $J_M$ with 
$\gamma(I'')\subset\bar\nbhd_{\chi+\chi'}^c(M)$.
Convexity of $\dcat(\cdot,M)$ implies the geodesic between the
endpoints of $\gamma(\bar I'')$ is contained in
$\bar\nbhd_{\chi+\chi'}(M)$.
\fullref{lem:qgeod_Hdf} provides $D$ depending on $\chi$, $\lambda$,
and $\epsilon$ such that $\gamma(\bar I'')$
and the geodesic between its endpoints have Hausdorff distance at
most $D$.
Take $\chi'':=\chi+\chi'+D$.
\end{proof}

\begin{lemma}\label{lem:regular_small_projection}
For all $\chi$, $\lambda$, and $\epsilon$ there exist $L_0$ such that
for every $\chi$--strongly contracting
$(\lambda,\epsilon)$--quasigeodesic path segment $\gamma\from I\to\Davis$,  for 
every subinterval $I'\subset I$ of length $L_0$ there
exists a wall $M$ that separates the endpoints of $\gamma(I')$,
separates the endpoints of $\gamma(I)$, and admits a subinterval
$J_M\subset I$ such that $\gamma^{-1}(\pi_\gamma(M))\subset J_M$ and
$|J_M|\leq L_0$.
\end{lemma}
\begin{proof}
  Take $\chi'$ and $\chi''$ of
  \fullref{cor:subsegment_containing_projection}.
  Let $\wallgrowth$ be the wall growth function of
  \fullref{constants}.
  \fullref{lem:qgeod_Hdf} provides $D$ depending on $\chi$, $\lambda$,
and $\epsilon$ such that $\gamma(I)$
and the geodesic $\alpha$ between its endpoints have Hausdorff distance at
most $D$.
From the construction of $\chi''$ in the proof of
\fullref{cor:subsegment_containing_projection}, we have $\chi''\geq D$.
By \fullref{all_metrics_are_equivalent}, $\dwall\ceq\dcat$, and 
$\gamma$ gives coarse equivalence between
$|\cdot|$ on its domain and $\dcat$ on its image, so we can choose $L_0\geq\lambda(2\chi'+\epsilon)$ large
enough that $|x-y|\geq L_0$ implies 
  $\dwall(\gamma(x),\gamma(y))>2\wallgrowth(\chi'')$.

Suppose $[x,y]=I'\subset I=[a,b]$, with $|I'|=L_0$.
Consider 
$M\in\walls(\gamma(x),\gamma(y))\setminus\walls(\gamma(a),\gamma(b))$,
the walls that separate endpoints of $\gamma(I')$ but fail to separate
endpoints of $\gamma(I)$.
Since $M$ fails to separate the endpoints of $\gamma(I)$, which are
the endpoints of $\alpha$, by convexity of halfspaces there is a
closed halfspace $\bar M^+$ of $M$ containing $\alpha$.
Since $M$ separates $\gamma(x)$ and $\gamma(y)$, we may assume
$\gamma(x)\in M^{-}$ and $\gamma(y)\in M^+$.
Then $\dcat(\gamma(x),M)=\dcat(\gamma(x),\bar M^+)\leq\dcat(\gamma(x),\alpha)\leq D$.
Thus, every wall that separates the endpoints of
$\gamma(I')$ but fails to separate the endpoints of $\gamma(I)$ passes
within distance $D\leq\chi''$ of one of the endpoints of $\gamma(I')$. 

For each wall $M\in\walls(\gamma(x),\gamma(y))$, let $J_M$ be the
smallest closed interval of $I$ such that
$\gamma^{-1}(\pi_\gamma(M))\subset J_M$.
Notice that since $\gamma$ is continuous, for every
$M\in\walls(\gamma(x),\gamma(y))$ there exists $t\in(x,y)$ such that
$\gamma(t)\in M$, and $J_M$ contains this time $t$, so every $J_M$ has
nonempty intersection with $[x,y]$.
Let $\mathcal{M}^+\subset \walls(\gamma(x),\gamma(y))$ be the walls
for which $|J_M|>L_0$.
For each $M\in\mathcal{M}^+$, the interval $J_M$ contains a point of
$[x,y]$ and is strictly longer than $[x,y]$, so it contains at least
one of the endpoints, $x$ or $y$.
Since $J_M$ is smallest, each endpoint of $J_M$ is either in, or is an
accumulation point of, $\gamma^{-1}(\pi_\gamma(M))$, so  $\diam(\pi_\gamma(M))$ is at least the
$\dcat$--distance between the endpoints of $\gamma(J_M)$, which is at
least $|J_M|/\lambda-\epsilon>L_0/\lambda-\epsilon\geq 2\chi'$.
\fullref{cor:subsegment_containing_projection} says
$\gamma(J_M)\subset\bar\nbhd_{\chi''}(M)$.
Since $J_M$ contains at least one of $x$ or $y$, the wall $M$ passes
within distance $\chi''$ of either $\gamma(x)$ or $\gamma(y)$.

A wall of $\walls(\gamma(x),\gamma(y))$ can fail to satisfy the lemma
either by failing to separate the endpoints of $\gamma(I)$, or by
belonging to $\mathcal{M}^+$.
We have shown in both cases that such a wall passes within distance $\chi''$ of an endpoint
of $\gamma(I')$.
There are only at most $2\wallgrowth(\chi'')$--many walls passing
within distance $\chi''$ of one of the two endpoints of $\gamma(I')$,
but the choice of $L_0$ implies that
$|\walls(\gamma(x),\gamma(y))|=\dwall(\gamma(x),\gamma(y))>2\wallgrowth(\chi'')$,
so there are walls of $\walls(\gamma(x),\gamma(y))$ that separate the
endpoints of $\gamma(I)$ and have $|J_M|\leq L_0$. 
\end{proof}

\begin{proposition}\label{prop:contracting_implies_separated_walls}
 For all $\chi$, $\lambda$, and $\epsilon$ there exist $C$ and $L_0$
 such that for every  $\chi$--strongly contracting
 $(\lambda,\epsilon)$--quasigeodesic path $\gamma\from
 I\to \Davis$, with $I$ a compact interval of length at least $L_0$,  there exists a chain
  $\mathcal{V}$ separating the endpoints of $\gamma(I)$
  such that for every pair $V$, $V'$ of consecutive walls of $\mathcal{V}$ the
  number of walls transverse to $\{V,V'\}$ is at most $C$, and such
  that $|t_{i+1}-t_i|\leq 7L_0$, where the $t_i$ are defined as follows.
  If $\mathcal{V}=\{V_i\}_{1\leq i\leq n}$, then let $t_0:=\min I$,
  $t_{n+1}:=\max I$, and for each $1\leq i\leq
  n$, let $t_i:=\min\gamma^{-1}(\gamma(I)\cap V_i)$.   
\end{proposition}
\begin{proof}
Given $\chi$, $\lambda$, and $\epsilon$ take $\chi'$, $\chi''$, $L_0$,
and $\wallgrowth$ as in
\fullref{lem:regular_small_projection}, and take $C:=\wallgrowth(\chi'')$.
By shifting the parameterization of $\gamma$, if necessary, we may
assume $(0,L_0)\subset I=[a,b]$.
Let $s_i:=iL_0$ for $i\in\mathbb{Z}$.
If $(s_i,s_{i+1})\subset I$ then by
\fullref{lem:regular_small_projection} there exists a wall
$V_i\in\walls(\gamma(s_i),\gamma(s_{i+1}))\cap\walls(\gamma(a),\gamma(b))$
and an interval $J_i\subset I$ of length
at most $L_0$ such that $\gamma^{-1}(\pi_\gamma(V_i))\subset J_i$.
In particular, $J_i$ contains every time $t$ such that $\gamma(t)\in
V_i$, since at such times $\pi_\gamma(\gamma(t))=\gamma(t)$.
Since $J_i$ has length at most $L_0$ and $J_i\cap
(s_i,s_{i+1})\neq\emptyset$, we have $J_i\subset((i-1)L_0,(i+2)L_0)\cap I$.
In particular, if $j\geq i+3$ and  $(s_i,s_{j+1})\subset I$ then $J_i$ is disjoint from
$J_{j}\subset ((j-1)L_0,(j+2)L_0)$, which implies that $V_i$ and $V_j$
are disjoint, since a point in $V_i\cap V_j$ would project via
$\pi_\gamma$ to a point in $\pi_\gamma(V_i)\cap\pi_\gamma(V_j)$, but
$\gamma^{-1}(\pi_\gamma(V_i)\cap\pi_\gamma(V_j))\subset J_i\cap J_j=\emptyset$.

Suppose for $i+1<j<k-1$ in $\mathbb{Z}$ that $(s_i,s_{k+1})\subset I$ and $M$ is a wall transverse to both $V_i$
and $V_k$.
For $z\in M\cap V_i$, we have
$\pi_\gamma(z)\subset\pi_\gamma(M)\cap\pi_\gamma(V_i)$, so
$\gamma^{-1}(\pi_\gamma(M))$ contains a point in $J_i$.
Similarly, $\gamma^{-1}(\pi_\gamma(M))$ contains a point in $J_k$.
Now, $J_i\subset((i-1)L_0,(i+2)L_0)$ and
$J_k\subset((k-1)L_0,(k+2)L_0)$
so $[jL_0,(j+1)L_0]$ separates $J_i$ and $J_k$ in $I$.
It follows that $\diam(\pi_\gamma(M))> L_0/\lambda-\epsilon>2\chi'$,
so \fullref{cor:subsegment_containing_projection} implies
$\gamma([jL_0,(j+1)L_0])\subset\bar\nbhd_{\chi''}(M)$.
Thus, every point of $\gamma([jL_0,(j+1)L_0])$ is $\chi''$--close to
every wall that is transverse to both $V_i$ and $V_k$.
The constant $C$ is the maximum number of walls intersecting any
$\chi''$--ball, so there are at most $C$--many walls transverse to both
$V_i$ and $V_k$ when $k\geq i+4$ and $(s_i,s_{k+1})\subset I$.

Take $I':=\{i\in 4\mathbb{Z}\mid
(iL_0,(i+1)L_0)\subset  I\}$.
Then $\mathcal{V}:=\{{V}_i\}_{i\in I'}$ contains at least one wall, $V_0$, since $(0,L_0)\subset I$.
Since wall indices in $\mathcal{V}$ differ by at least 4, the argument
above gives that they are disjoint and there are at most $C$ walls transverse to any
distinct pair of them.
Furthermore, since all the walls were chosen to be in
$\walls(\gamma(a),\gamma(b))$, there are no facing triples in $\mathcal{V}$, so it is a
chain. 

Finally, we bound the differences in first crossing times between
consecutive walls of $\mathcal{V}$.
There is some $k$ such that the $i$-th wall of $\mathcal{V}$ is
$V_{4(i+k)}$, so $t_i\in J_{4(i+k)}\subset
({(4(i+k)-1)L_0},{(4(i+k)+2)L_0})$.
Then $|t_{i+1}-t_i|\leq
(4(i+1+k)+2)L_0-(4(i+k)-1)L_0=7L_0$.
The first wall is in $J_{4(1+k)}\subset ({(4(1+k)-1)L_0},{(4(1+k)+2)L_0})$.
The fact that it is the first wall means 
$(4kL_0,(4k+1)L_0)\not\subset I$, so the initial endpoint of $I$
is greater than $4kL_0$, so $|t_1-t_0|\leq (4k+6)L_0-4kL_0=6L_0\leq
7L_0$.
The symmetric argument at the terminal endpoint of $I$ gives $|t_{n+1}-t_n|\leq 6L_0\leq 7L_0$.
  \end{proof}

  \begin{proposition}\label{separated_chain_implies_divergent}
    Given $\lambda$, $\epsilon$, $C$, 
    $0<\rho$, and $L$, suppose $\gamma\from I\to \Davis$ is a compact
    $(\lambda,\epsilon)$--quasigeodesic path with the property that
    for every compact
    subinterval $I'\subset I$ of length at least $L$ there exists a chain
  $\mathcal{V}$ that separates the endpoints of $\gamma(I')$ with
  $|\mathcal{V}|/|I'|\geq\rho$ such that for every pair of
  consecutive walls $V$, $V'$ of $\mathcal{V}$, every chain $\mathcal{H}\transverse\{V,V'\}$ has $|\mathcal{H}|\leq C$.
  Then there are $K=K(\lambda,\rho)$, $r_0=r_0(\lambda,\epsilon,\rho,C,L)$,
  and a quadratic polynomial $q(r)$ with positive leading coefficient whose coefficients depend on
  $\lambda$, $\epsilon$, $\rho$, and $C$, such that $\gamma$ has
  divergence gauge $\delta(r;1,K)$ satisfying $\delta(r;1,K)\geq q(r)$ for $r\geq r_0$.
\end{proposition}
\begin{proof}
  Suppose $\lambda_0$ and $\epsilon_0$ are constants such that
  \fullref{all_metrics_are_equivalent} implies that for all
  $x,y\in\Davis$:
  \[\dcat(x,y)/\lambda_0-\epsilon_0\leq \dchain(x,y)\leq \lambda_0\dcat(x,y)+\epsilon_0\]
  Take $K:=2+2\lambda(1+2\lambda_0)/\rho$.
  Consider $r\geq r_0$ where $r_0$ is a constant that is a maximum of
  three quantities appearing in the course of the proof.

  Suppose there exist points $x$ and $y$ such that
  $\dcat(x,\gamma)=\dcat(y,\gamma)=r$ and $\dcat(x,y)\geq Kr$, so that
  $\dcat(\pi_\gamma(x),\pi_\gamma(y))\geq (K-2)r$.
    Suppose there exists a path $\alpha$ from $x$ to $y$ in
    $\nbhd_r^c(\gamma)$.
  
  Pick $x'\in\pi_\gamma(x)$ and $y'\in\pi_\gamma(y)$ and let $I'$ be
  the smallest closed interval whose endpoints map to $x'$ and $y'$.
  This implies $|I'|\geq ((K-2)r-\epsilon)/\lambda$, which is at least
  $L$ once $r_0\geq (\lambda L+\epsilon)/(K-2)$, so there exists a chain
  $\mathcal{V}:=\{V_i\}_{1\leq i\leq n}$ as in the statement. 

  For each $i$, let $t_i$ be the first time in $I'$ such that $\gamma(t_i)\in \bar
  V_i^+$. Since $\bar V_{i+1}^+\subset V_i^+$, the first crossing
  times $t_i<t_{i+1}$ are ordered compatibly with the index.

  We have $n-1$ intervals $[t_i,t_{i+1}]$ with
  $\sum_{i=1}^{n-1}|t_{i+1}-t_i|\leq |I'|$ and
  $n=|\mathcal{V}|\geq\rho|I'|$, so fewer than $|I'|/(2/\rho)$
  many of the intervals
  have length strictly greater than $2/\rho$.

By \fullref{quad_trick}  the convex subchain $\mathcal{V}'$ of $\mathcal{V}$
  consisting of the walls that separate $x$ and $y$ satisfies: 
\[|\mathcal{V}|-|\mathcal{V}'|\leq
  \dchain(x,x')+1+\dchain(y,y')+1\leq 
2(\lambda_0\dcat(x,x')+\epsilon_0+1)=
2(\lambda_0r+\epsilon_0+1)\]

See \fullref{fig:bridge_quad_div}. 

The number of intervals $[t_{i},t_{i+1}]$ such that both $V_i$ and
$V_{i+1}$ are in $\mathcal{V}'$ is $|\mathcal{V}'|-1$, and:
\[|\mathcal{V}'|-1\geq
  |\mathcal{V}|-2(\lambda_0r+\epsilon_0+1)-1\geq\rho|I'|-2\lambda_0r-(2\epsilon_0+3)\]
Fewer than $\rho|I'|/2$ of those are long, so the number of intervals
with length at most $2/\rho$ such that $V_i$ and
$V_{i+1}$ are in $\mathcal{V}'$ is at least:
\begin{equation}
  \label{eq:22}
  \begin{split}
      \rho|I'|-2\lambda_0r-(2\epsilon_0+3)-\rho|I'|/2&=\rho|I'|/2
      -2\lambda_0r-(2\epsilon_0+3)\\
      &\geq\frac{\rho}{2}\frac{(K-2)r-\epsilon}{\lambda}-2\lambda_0r-(2\epsilon_0+3)\\
      &=r-\left(\frac{\rho\epsilon}{2\lambda}+2\epsilon_0+3 \right)
  \end{split}
\end{equation}
Assume $r_0\geq \frac{\rho\epsilon}{2\lambda}+2\epsilon_0+4$, so that
\eqref{eq:22} gives at least one good interval.

\begin{figure}[h]
  \centering
  \includegraphics{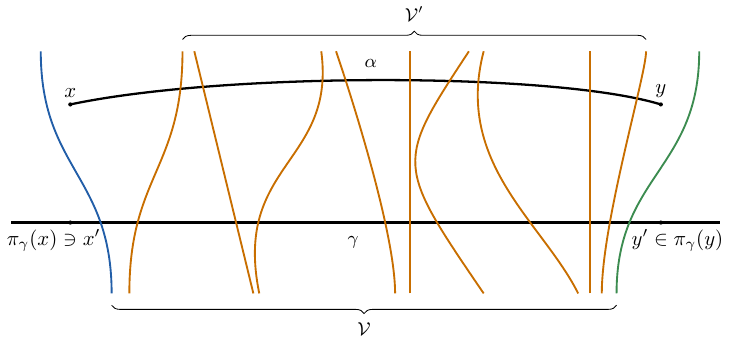}
  \caption{Density control on the chain separating $x',y'\in \gamma$ forces
    linearly many of the interwall subsegments of $\gamma$ to be short
  and the corresponding interwall subsegments of $\alpha$ to be long.}
  \label{fig:bridge_quad_div}
\end{figure}

Let $V_i$ and $V_{i+1}$ be a pair of consecutive walls of
$\mathcal{V}'$ such that $0<t_{i+1}-t_i\leq 2/\rho$.
This implies $\dcat(V_i,V_{i+1})\leq 2\lambda/\rho+\epsilon$.
\fullref{cor:qgeo_wall_crossers} says the length of a quasigeodesic
with fixed constants connecting $V_i$ and $V_{i+1}$ not contained in
$\nbhd_R(\bridge(V_i,V_{i+1}))$ has parameter interval length growing
linearly in $R$, with constants depending on the quasigeodesic
constants and $\dcat(V_i,V_{i+1})$ but not
  on $V_i$ and $V_{i+1}$.
  Since $t_{i+1}-t_i\leq 2/\rho$, the corresponding segment of
  $\gamma$ is contained in a uniform
  neighborhood of $\bridge(V_i,V_{i+1})$.
  \fullref{lem:bounded_bridge} says $\diam(\bridge(V_i,V_{i+1}))$ is bounded
  in terms of $2\lambda/\rho+\epsilon$ and $C$.
  Thus, $\bridge(V_i,V_{i+1})$ and $\gamma([t_i,t_{i+1}])$ are coarsely equivalent,
  with constants depending on $\rho$, $\lambda$, $\epsilon$, and $C$.

On the other hand, since $\alpha$ is required to stay far from
$\gamma$, this implies $\alpha$ stays outside $\nbhd_R(\bridge(V_i,V_{i+1}))$
for some $R\cequiv r$, so
\fullref{lem:length_wall_crossers_improved}  says the length of the subsegment of
$\alpha$ from its first encounter with $V_i$ to its first encounter
with $V_{i+1}$ is bounded below by a linear function of
$r$, with positive slope, whose constants depend on $\dcat(V_i,V_{i+1})$, which is bounded
above by $2\lambda/\rho+\epsilon$.
Assume $r_0$ is large enough that this linear function takes positive
values for all $r\geq r_0$.
The bound from \eqref{eq:22} says the number of such subsegments of
$\alpha$ is linear in $r$ with slope 1.
Those subsegments of $\alpha$ have disjoint interiors, so the total length of $\alpha$ is bounded
below by a quadratic polynomial $q(r)$ with positive leading
coefficient whose coefficients  depend on $\rho$, $\lambda$, $\epsilon$, and $C$.

We have shown that $\delta(r;1,K)\geq q(r)$ when $r\geq r_0$ as
desired.
Notice that if for some $r$ there is no such pair of points $(x,y)$ or
there is no such path $\alpha$, then $\delta(r;1,K)=\infty$, so it is still true that
$\delta(r;1,K)\geq q(r)$.
\end{proof}

\begin{proof}[{Proof of \fullref{thm:wall_morse}}]
The equivalence of \ref{wall_morse:divergence}, \ref{wall_morse:morse}, and
\ref{wall_morse:strong_contraction} is \fullref{morse_divergent_contracting_in_cat_zero}.

Suppose \ref{wall_morse:strong_contraction} holds.
By \fullref{lem:robust}, there exists $\chi'$ depending only on
$\chi$, $\lambda$, and $\epsilon$, such that every compact subsegment
of $\gamma$ is $\chi'$--strongly contracting.
\fullref{prop:contracting_implies_separated_walls} then implies
\ref{wall_morse:bounded_gap_bounded_walls}.

Assuming  \ref{wall_morse:bounded_gap_bounded_walls}, set
$\rho:=1/14L_0$ and $L_1:=14L_0$.
Let $I'$ be a compact subinterval of $I$ of length at
least $L_1$, take $\mathcal{V}:=\{V_i\}_{1\leq i\leq n}$ to be the
chain supplied by \ref{wall_morse:bounded_gap_bounded_walls}, and
define $t_0$ and $t_{n+1}$ to be the endpoints of $I'$.
Since the subsegments given by
\ref{wall_morse:bounded_gap_bounded_walls} have length at most $7L_0$:
\[|\mathcal{V}|/|I'|=n/|I'|\geq \frac{\frac{|I'|}{7L_0}-1}{|I'|}\geq\frac{1}{7L_0}-\frac{1}{14L_0}=\frac{1}{14L_0}=\rho\]
This gives \ref{wall_morse:bounded_gap_bounded_chains}.

Finally, the proof
\ref{wall_morse:bounded_gap_bounded_chains}$\implies$\ref{wall_morse:divergence}
amounts to showing that uniform quadratic divergence on all compact
subintervals passes to the whole quasigeodesic.
Suppose \ref{wall_morse:bounded_gap_bounded_chains} holds
with constant $C$, $\rho$, and $L_1$. Take $K$, $r_0$, and $q$ as
provided by \fullref{separated_chain_implies_divergent} with respect
to $\lambda$, $\epsilon$, $C$, $\rho$, and $L_1$.
Given $r\geq r_0$ and points $x$ and $y$ at $\dcat$--distance $r$ from
$\gamma$ with $\dcat(x,y)\geq Kr$, take $x'\in\pi_\gamma(x)$ and
$y'\in\pi_\gamma(y)$, and consider a compact subinterval $J\subset
I$ containing a point from $\gamma^{-1}(x')$ and a point from $\gamma^{-1}(y')$.
Then $\dcat(x,\gamma(J))=\dcat(y,\gamma(J))=r$.
\fullref{separated_chain_implies_divergent} says that every path
from $x$ to $y$ that stays outside $\nbhd_r(\gamma(J))$ has length bounded below by $q(r)$, for the quadratic
function $q$.
The paths from $x$ to $y$ that stay outside $\nbhd_r(\gamma(I))$ are a
subset of those that stay outside $\nbhd_r(\gamma(J))$, so the same
lower length bound applies. 
This shows that $q$ is a quadratic lower bound on the divergence of
$\gamma(I)$.
\end{proof}

%% file: grid.tex
\begin{definition}\label{def:grid}
  A \emph{grid} is a pair $(\mathcal{V},\mathcal{H})$
  (vertical/horizontal) of transverse chains that each contain at
  least two walls.
\end{definition}

\begin{definition}
  Let $\caprace$ be as in \fullref{constants}.
If $\mathcal{V}$ is a chain, define its \emph{interior}
$\interior{\mathcal{V}}$ to be all the walls $V$ of $\mathcal{V}$ such
that there are at least $\caprace-1$ walls of $\mathcal{V}$ that come
before $V$ in the chain and at least $\caprace-1$ walls that come after $V$.
\end{definition}

The goal of this subsection is to prove\footnote{The proof of the Grid
Lemma presented in \cite{CapMar13} has a gap. In our language, that
proof supposes that there is a grid $(\mathcal{V},\mathcal{H})$ with
$\mathcal{V}=\{V_i\}_{0\leq i\leq k}$, and supposes there is \emph{any}
$0\leq i_0\leq k$ such that there is an
$H_j\in\interior{\mathcal{H}}$ such that $\refl{V_{i_0}}$ and $\refl{H_j}$
do not commute. It then applies 
\cite[Lemma~2.9]{CapMar13}, which is \fullref{Caprace_8}, to
$\{V_i\}_{0\leq i\leq i_0}$. However, \fullref{Caprace_8} requires the
chain to be sufficiently long, which is not true if $i_0$ is small.
\fullref{prop:grid} below argues that given the noncommuting pair of walls
there are four possible configurations to which one might try to apply
\fullref{Caprace_8}, and at least one of them satisfies that theorem's
hypotheses. This fills the gap.}
the following proposition, which is essentially the Grid Lemma attributed to Caprace and
Przytycki by Caprace and Marquis \cite[Lemma~2.8]{CapMar13}.

\begin{proposition}\label{lem:grid_wide}
  Let $(\mathcal{V},\mathcal{H})$ be a grid in $\Davis$ such that
  $|\mathcal{V}|\geq 2\caprace-1$ and $|\mathcal{H}|\geq 2\caprace$.
  Then $\pc(\mathcal{V})\times\pc(\mathcal{V})^\perp $ is a wide
  parabolic and either $\pc(\mathcal{V})=\pc(\mathcal{H})$ is
  irreducible higher rank affine or $\pc(\interior{\mathcal{H}})\leq\pc(\mathcal{V})^\perp$.
\end{proposition}

\begin{lemma}[{cf \cite[Lemma~2.3]{Cap09}\footnote{\fullref{prop:affine} is a corrected version of \cite[Lemma~2.3]{Cap09} stating what is
actually proven there.
The original omits the hypothesis that $\langle R_2\rangle$
is infinite in the final sentence, but still concludes $\pc(R_1)=\pc(R_2)$ is
higher rank affine.
This is an error, as demonstrated by \fullref{caprace_counterexample}.
The specific error in the proof of \cite[Lemma~2.3]{Cap09} is the
application of \cite[Lemma~2.2]{Cap09} in the case $R_2\leq \pc(R_1)$
without explicitly requiring $\langle R_2\rangle$ to be nonspherical.}}]\label{prop:affine}
Let $R_1$ and $R_2$ be finite sets
of reflections in $(W,S)$ such that $\langle R_1\rangle$ and $\langle R_2\rangle$
are irreducible, and such that $r_1r_2=r_2r_1$ for all $r_1\in R_1$
and $r_2\in R_2$.
Suppose $\langle R_1\rangle$ is infinite.
Then one of the following is true:
\begin{itemize}
\item $\pc(R_1\cup R_2)=\pc(R_1)\times\pc(R_2)$
  \item $R_2\leq\pc(R_1)$
  \end{itemize}
  If, in addition, $\langle R_2\rangle$ is infinite, then
  $R_2\leq\pc(R_1)$ implies $\pc(R_1)=\pc(R_2)$ is higher rank affine. 
\end{lemma}

\begin{example}\label{caprace_counterexample}
 Consider the Coxeter system $(W,S)$ defined by the Coxeter diagram
 $a\stackrel{4}{-}b\stackrel{4}{-}c\stackrel{4}{-}d$, which is not affine.
 Consider the finite sets of reflections $R_1:=\{ bab,c,d\}$ and
 $R_2:=\{ a\}$.
 Then $aca=c$, $ada=d$ and
 $a(bab)a=(abab)a=(baba)a=bab$, so $R_1$ and $R_2$ commute.
 In fact, from Allcock's description of reflection centralizers
 \cite{All13}, one finds that $\langle
 R_1\rangle$ is the entire reflection part of $\centralizer_W(\langle R_2\rangle)$.
The group $\pc(R_2)\cong\mathbb{Z}/2\mathbb{Z}$ is irreducible. 
 The group $\langle R_1\rangle$ is infinite and irreducible; to see this, consider $P:=\langle a,b,c\rangle$, which is a Euclidean
 triangle parabolic subgroup of $W$.
 In its usual  $\Isom(\mathbb{E}^2)$ representation the elements
 $bab,c\in R_1\cap P$ are reflections through parallel lines, so
 $\langle R_1\rangle$
 contains an infinite dihedral subgroup.
 This also shows that $\langle R_1\rangle$ is irreducible, since $c$
 commutes with neither $bab$ nor $d$.
Furthermore, $P$ is  2--spherical and rank 3, so every proper parabolic subgroup is
finite, so $\pc(\langle bab,c\rangle)=P$.
But $d\in R_1\leq \pc(R_1)\geq \pc(\langle bab,c\rangle)$, so
$\pc(R_1)=W$.
Thus, $\pc(R_2)\lneq\pc(R_1)$ and $\pc(R_1)$ is not affine. 
This contradicts both conclusions in the $R_2\leq\pc(R_1)$ branch of \cite[Lemma~2.3]{Cap09}.
\end{example}

\begin{lemma}\label{lem:easy_grid}
    Let $(\mathcal{V},\mathcal{H})$ be a grid in $\Davis$.
  Suppose one of the following is true:
  \begin{itemize}
  \item $W(\mathcal{V})$ and $W(\mathcal{H})$ commute.
    \item $W(\mathcal{V})$ is dihedral and $\mathcal{V}$ contains at
      least 8 walls. 
  \end{itemize}
  Then one of the following
  is true:
  \begin{itemize}
  \item $\pc(\mathcal{H})\leq\pc(\mathcal{V})^\perp$
    \item $\pc(\mathcal{H})=\pc(\mathcal{V})$ is an irreducible higher rank affine parabolic.
  \end{itemize}
\end{lemma}
\begin{proof}
  By \fullref{lem:chains_irreducible}, $\pc(\mathcal{V})$ and
  $\pc(\mathcal{H})$ are irreducible and nonspherical.
  Furthermore, since it does not change $\pc(\mathcal{V})$ and
  $\pc(\mathcal{H})$, we may assume $\mathcal{V}$ and $\mathcal{H}$
  are finite chains, so that $W(\mathcal{V})$ and $W(\mathcal{H})$ are
  finitely generated reflection groups. 
  
   By \fullref{prop:affine} we can partition the set $R:=\{\refl{H}\mid
     H\in\mathcal{H}\}$ of reflections through
  walls of $\mathcal{H}$ into disjoint subsets:
  \begin{itemize}
  \item $R_0:=\{\refl{H}\mid H\in\mathcal{H},\,\refl{H}\notin
    \centralizer_W(W(\mathcal{V}))\}$
     \item $R_1:=\{\refl{H}\mid H\in\mathcal{H},\,\refl{H}\in
       \centralizer_W(W(\mathcal{V}))\cap\pc(\mathcal{V})\}$
        \item $R_2:=\{\refl{H}\mid H\in\mathcal{H},\,\refl{H}\in \pc(\mathcal{V})^\perp\}$
        \end{itemize}
        Furthermore, if $|R_1|\geq 2$ then $\pc(R_1)=\pc(\mathcal{V})$ is
        higher rank affine.
        
  Reflections in distinct walls of $\mathcal{H}$ do not
commute, so at least one of  $R\cap\pc(\mathcal{V})$ or $R_2$ is empty.

$W(\mathcal{V})$ and $W(\mathcal{H})$ commute if and only if $R_0=\emptyset$, in which case we are done, since then either
$R=R_1$ or $R=R_2$.
The former gives $|R_1|\geq 2$, so
$\pc(\mathcal{H})=\pc(R_1)=\pc(\mathcal{V})$ is higher rank
affine.
The latter gives $\pc(\mathcal{H})=\pc(R_2)\leq\pc(\mathcal{V})^\perp$.

If $R_0\neq\emptyset$ then, by hypothesis,  $W(\mathcal{V})$ is dihedral and $\mathcal{V}$
contains at least 8 walls.
For $\refl{H}\in R_0$, 
   \fullref{Caprace_11} says $W(H,\mathcal{V})$
is a Euclidean triangle group, so \fullref{Caprace_17} gives
$\refl{H}\in\pc(\mathcal{V})$, and \fullref{Caprace_16}  says
$\pc(\mathcal{V})$ is higher rank affine.
Since $\emptyset\neq R_0\subset R\cap\pc(\mathcal{V})$,
$R_2=\emptyset$, so  $R=R_0\cup R_1$, which gives
$\pc(\mathcal{H})\leq\pc(\mathcal{V})$.
Since irreducible affine parabolics do not contain proper nonspherical
parabolic subgroups and $\pc(\mathcal{H})$ is a nonspherical
parabolic, $\pc(\mathcal{H})=\pc(\mathcal{V})$.
\end{proof}

\begin{lemma}\label{lem:walls_through_grid}
  Let $(\mathcal{V},\mathcal{H})$ be a grid in $\Davis$, with
  $\mathcal{V}=\{V_i\}_{i\in I}$ and $\mathcal{H}=\{H_j\}_{j\in J}$.
  Suppose, for some $i_0\in I$ and $j_0\in J$ there is a wall $M$
  that contains $V_{i_0}\cap H_{j_0}$.
  Then $M$ is transverse to at least two of the following chains:
  \begin{align*}
    \mathcal{V}^-&:=\{V_i\}_{\{i\in I\mid i\leq i_0\}}&\qquad
  \mathcal{V}^+&:=\{V_i\}_{\{i\in I\mid i\geq i_0\}}\\
    \mathcal{H}^-&:=\{H_j\}_{\{j\in J\mid j\leq j_0\}}&\qquad
   \mathcal{H}^+&:=\{H_j\}_{\{j\in J\mid j\geq j_0\}}
  \end{align*}
\end{lemma}
\begin{figure}[h]
  \centering
  \includegraphics{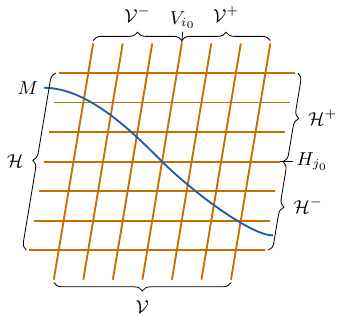}
  \hfill
    \includegraphics{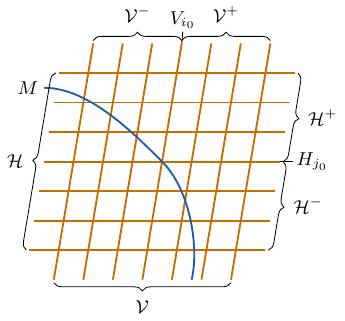}

  \caption{Two possibilities for a wall $M$ containing $V_{i_0}\cap H_{j_0}$, separating
    $V_{i_0}^+\cap H_{j_0}^+$ from $V_{i_0}^{-}\cap H_{j_0}^{-}$ and
    $M\not\transverse\mathcal{H}^+$. We have $M\transverse
    \mathcal{V}^{-}$ and at least one of
    $M\transverse\mathcal{V}^+$ or  $M\transverse\mathcal{H}^{-}$.}
  \label{fig:four_quarter_grids}
\end{figure}
\begin{proof}
  If $M$ is equal to $V_{i_0}$ then it is transverse to all of
  $\mathcal{H}$, so it is transverse to both of $\mathcal{H}^-$ and
  $\mathcal{H} ^+$.
  If $M$ is equal to $H_{j_0}$ then it is transverse to both
  $\mathcal{V}^-$ and $\mathcal{V}^+$.
Now assume $M$ is not one of $V_{i_0}$ or $H_{j_0}$.
Then $V_{i_0}\cup H_{j_0}$ divides  $\Davis$ into four regions, $V_{i_0}^-\cap H_{j_0}^-$,
  $V_{i_0}^+\cap H_{j_0}^-$, $V_{i_0}^+\cap H_{j_0}^+$, and $V_{i_0}^-\cap H_{j_0}^+$.
  The wall $M$ meets two of these, which are opposites, either
  $V_{i_0}^-\cap H_{j_0}^-$ and $V_{i_0}^+\cap H_{j_0}^+$ or $V_{i_0}^-\cap H_{j_0}^+$ and
  $V_{i_0}^+\cap H_{j_0}^-$, and separates the other two.
  Suppose, without loss of generality, that $M$ separates $V_{i_0}^-\cap H_{j_0}^-$ from
  $V_{i_0}^+\cap H_{j_0}^+$.

  Suppose $i_0$ and $j_0$ are not extremes.
  Suppose $M$ is not transverse to $\mathcal{H}^+$.
Then there is some first $j_1>j_0$ in $J$ such that $M$ is disjoint
from $H_{j_1}$. 
  For any $i<i_0$ there is a path from $V_{i_0}^-\cap H_{j_0}^-$ to
  $V_{i_0}^+\cap H_{j_0}^+$ by travelling through $V_i$ and $H_{j_1}$.
  Specifically, choose a point $a\in V_i\subset V_{i_0}^-$ in the interior of
  $H_{j_0}^-$, which exists, since $V_i\transverse H_{j_0}$, a point
  $b\in V_i\cap H_{j_1}$, which exists, since $V_i\transverse
  H_{j_1}$,  and a point $c\in H_{j_1}\subset H_{j_0}^+$ in the interior
  of $V_{i_0}^+$, which exists, since $V_{i_0}\transverse H_{j_1}$.
  By convexity of walls, $[a,b]\subset V_i\subset V_{i_0}^-$ and
  $[b,c]\subset H_{j_1}\subset H_{j_0}^+$.
  Now, since $M$ separates $V_{i_0}^-\cap H_{j_0}^-$ from
  $V_{i_0}^+\cap H_{j_0}^+$, it cuts either $[a,b]$ or $[b,c]$, but if
  $M$ is disjoint from $H_{j_1}\supset [b,c]$ then it must be that $M$
  cuts $[a,b]\subset V_i$, so
  $M\transverse V_i$.
  Since this was true for every $i<i_0$, we conclude $M\not\transverse
  \mathcal{H}^+$ implies $M\transverse\mathcal{V}^-$.
  
  An analogous argument yields that $M$ is transverse to at least one
  of $\mathcal{H}^-$ or $\mathcal{V}^+$.

  When one or both of $i_0$ and $j_0$ are extremes then the
  argument is similar, with the added observation that if, say, $i_0$
  is the least element of $I$, then $\mathcal{V}^-$ is the single wall
  $V_{i_0}$, which is transverse to $M$ by hypothesis. 
\end{proof}

\begin{lemma}\label{lem:three_walls}
  Let $M$ and $M'$ be distinct walls such that $\refl{M}$ and
  $\refl{M'}$ do not commute. Let $M'':=\refl{M}(M')$. Then $M''$ is
  distinct from $M$ and $M'$, and $M\cap M'=M\cap M''=M'\cap M''$.
\end{lemma}
\begin{proof}
If $M=M''=\refl{M}(M')$ then $M=\refl{M}(M)=\refl{M}^2(M')=M'$.
If $M'=M''=\refl{M}(M')$ then since
$\refl{M''}=\refl{M}\refl{M'}\refl{M}$, we get that $\refl{M}$ and
$\refl{M'}$ commute. Both are contradictions.

The first intersection equality is direct: $M\cap M'\subset M$, so it
is fixed pointwise by $\refl{M}$, but $\refl{M}(M\cap
M')=M\cap\refl{M}(M')=M\cap M''$.

For the equality $M\cap M'=M'\cap M''$, the 
main step is to show $\refl{M}\in\pc(M',M'')$.
In the canonical representation of $(W,S)$ on $\mathbb{R}^{|S|}$,
halfspaces of walls in $\Davis$ correspond to roots in
$\mathbb{R}^{|S|}$, so ($\pm$--pairs of) roots correspond to
reflections, see \cite[Section~5]{Hum90}.
It is well-known, cf
  \cite[Proposition~3.3]{BriHow99}, or \cite[Lemma~3.1]{Qi07} for a
  short proof, that  if a root $\alpha$ is in the linear span of roots
  $\beta_1,\dots,\beta_n$, then the corresponding reflections satisfy
  $\refl{\alpha}\in\pc(\refl{\beta_1},\dots,\refl{\beta_n})$.
  Since $M$, $M'$, and $M''$ are distinct, the corresponding roots
  $\alpha_M$, $\alpha_{M'}$, and $\alpha_{M''}$ are
  pairwise non-colinear, but since $M''=\refl{M}(M')$, the reflection
  formula for roots gives that $\alpha_{M''}$ is in the 2--dimensional
  subspace spanned by $\alpha_M$ and $\alpha_{M'}$.
  But then $\alpha_{M'}$ and $\alpha_{M''}$ also span that
  2--dimensional subspace, so $\alpha_M$ is in the span of $\alpha_{M'}$
  and $\alpha_{M''}$, so $\refl{M}\in\pc(M',M'')$.

  To finish the lemma, if $x\in M'\cap M''$ then $\stab(x)$ is a spherical parabolic
  subgroup that contains $\refl{M'}$ and $\refl{M''}$, so it contains
  $\pc(M',M'')$, so it contains $\refl{M}$. But $\refl{M}(x)=x$ if and
  only if $x\in M$, so $M'\cap M''\subset M$.
  Thus, $M'\cap M''=M'\cap M''\cap M=M'\cap M'\cap M=M'\cap M$.
\end{proof}

\begin{proposition}\label{prop:grid}
  Let $(\mathcal{V},\mathcal{H})$ be a grid in $\Davis$ such that 
  $\mathcal{V}$ and $\mathcal{H}$ are finite and each contain at least $(2\caprace-1)$--many
  walls.
  If
  $W(\interior{\mathcal{H}})\not\subset\centralizer_W(W(\mathcal{V}))$
  or  $W(\interior{\mathcal{V}})\not\subset\centralizer_W(W(\mathcal{H}))$
then 
  $\pc(\mathcal{V})=\pc(\mathcal{H})$ is irreducible higher rank affine.
  \end{proposition}
  \begin{proof}
    The bound on the chain size implies $\interior{\mathcal{V}}$ and
    $\interior{\mathcal{H}}$ are nonempty. 
    The cases are symmetric, so suppose
    $W(\interior{\mathcal{H}})\not\subset\centralizer_W(W(\mathcal{V}))$.
    For $\mathcal{V}=\{V_i\}_{i\in I}$ and
    $\mathcal{H}=\{H_j\}_{j\in J}$ there are $i_0\in I$ and $j_0\in J$ such that $\refl{V_{i_0}}$ does not
    commute with $\refl{H_{j_0}}$, and, furthermore,
    $H_{j_0}\in\interior{\mathcal{H}}$.
Since $V_{i_0}$ and $H_{j_0}$ are intersecting walls whose reflections
do not commute, \fullref{lem:three_walls} says that  $\refl{V_{i_0}}(H_{j_0})$ is distinct from $V_{i_0}$ and
$H_{j_0}$ and $\emptyset\neq V_{i_0}\cap H_{j_0}=V_{i_0}\cap
\refl{V_{i_0}}(H_{j_0})=H_{j_0}\cap\refl{V_{i_0}}(H_{j_0})$.
As described by \fullref{lem:walls_through_grid}, there are four
subchains 
$\mathcal{V}^+$, $\mathcal{V}^-$, $\mathcal{H}^+$, and $\mathcal{H}^-$
with respect to $i_0$ and $j_0$, and
$\refl{V_{i_0}}(H_{j_0})$ is transverse to at least two of them.
Since $j_0$ is an interior index, both of $\mathcal{H}^+$
and $\mathcal{H}^-$ contain at least $\caprace$--many walls, and at
least one of $\mathcal{V}^+$ or $\mathcal{V}^-$ does, since
$|\mathcal{V}^+|+|\mathcal{V}^{-}|=|\mathcal{V}|+1\geq 2\caprace$.
Thus, there is at least one of $\mathcal{V}^+$, $\mathcal{V}^-$,
$\mathcal{H}^+$, or $\mathcal{H}^-$ that contains at least
$\caprace$--many walls and is transverse to
$\refl{V_{i_0}}(H_{j_0})$.

Suppose that $\mathcal{V}^+$ is a
chain transverse to $\refl{V_{i_0}}(H_{j_0})$ with at least $\caprace$--many walls.
Then \fullref{Caprace_8} for $\mu:=\refl{V_{i_0}}(H_{j_0})$ and
$\mu':=H_{j_0}$ and the chain $\mathcal{V}^+$ gives  that
$W(H_{j_0},\refl{V_{i_0}}(H_{j_0}),\mathcal{V}^+)=W(H_{j_0},\mathcal{V}^+)$
is a Euclidean triangle group that is contained in an irreducible
affine parabolic $\mathcal{A}$, which has rank at least 3 since it
contains the triangle group. 
Then  \fullref{Caprace_17} says
$\refl{H_{j_0}}\in\pc(\mathcal{V}^+)\leq\mathcal{A}$.
Irreducible affine Coxeter groups have no proper nonspherical
parabolics, so since $\pc(\mathcal{V}^+)$ is nonspherical and
parabolic in $\mathcal{A}$, it equals $\mathcal{A}$.
Furthermore, a chain in an irreducible affine Coxeter group consists
of a family of parallel hyperplanes, so it generates an infinite
dihedral reflection subgroup. 
Since $\refl{H_{j_0}}\in\pc(\mathcal{V}^+)$, we do not have
$\pc(\mathcal{H})\leq\pc(\mathcal{V}^+)^\perp$, so 
\fullref{lem:easy_grid} implies
$\pc(\mathcal{H})=\pc(\mathcal{V}^+)=\mathcal{A}$, with
$W(\mathcal{H})\cong\dihedral_\infty\cong W(\mathcal{V}^+)$.

Now consider $V_i$ with $i<i_0$. It is transverse to $\mathcal{H}$,
and $|\mathcal{H}|\geq 2\caprace-1\geq 8$, so by \fullref{Caprace_11}
and \fullref{Caprace_17} 
either $\refl{V_i}\in\mathcal{A}$ or $\refl{V_i}$ centralizes
$W(\mathcal{H})$.
In the latter case, apply \fullref{prop:affine} with
$R_1=\{\refl{H}\mid H\in\mathcal{H}\}$, which is finite, since
$\mathcal{H}$ is finite,  and $R_2=\{\refl{V_i}\}$.
The first option of \fullref{prop:affine} implies  $\refl{V_i}$
centralizes $\pc(\mathcal{H})=\pc(\mathcal{V}^+)$, which is not true,
since $V_i$ is disjoint from the walls of $\mathcal{V}^+$.
Thus, the second option of \fullref{prop:affine} holds, so
$\refl{V_i}\in\pc(\mathcal{H})=\mathcal{A}$.
Thus, $\mathcal{A}=\pc(\mathcal{V})=\pc(\mathcal{H})$ is irreducible
higher rank affine.

The proof in the case that $\mathcal{V}^-$ is the long chain
transverse to $\refl{V_{i_0}}(H_{j_0})$ is exactly the same.
If instead we use one of the chains $\mathcal{H}^+$ or $\mathcal{H}^-$
the only change is to swap the roles of $\mathcal{V}$ and $\mathcal{H}$ and take $\mu':=V_{i_0}$ in the application of
\fullref{Caprace_8}.
\end{proof}

\begin{corollary}\label{cor:nondihedral_grids_commute}
  Let $(\mathcal{V},\mathcal{H})$ be a grid in $\Davis$ such that
  $2\caprace-1\leq|\mathcal{V}|<\infty$ and $2\caprace\leq|\mathcal{H}|<\infty$.
  If $W(\mathcal{V})\not\cong\dihedral_\infty$ then $\pc(\interior{\mathcal{H}})\leq\pc(\mathcal{V})^\perp$.
\end{corollary}
\begin{proof}
 If $\pc(\mathcal{V})$ is affine then $W(\mathcal{V})$ is dihedral,
 contrary to hypothesis, so \fullref{prop:grid} forces
 $W(\interior{\mathcal{H}})$ to centralize $W(\mathcal{V})$.
Apply \fullref{prop:affine} to $W(\mathcal{V})$ and
$W(\interior{\mathcal{H}})$, both of which are nonspherical since
$|\mathcal{V}|\geq 2\caprace-1>2$ and $|\interior{\mathcal{H}}|\geq
2$, and irreducible by \fullref{lem:chains_irreducible}.
Again, since $\pc(\mathcal{V})$ cannot be
affine, the remaining alternative is $\pc(\interior{\mathcal{H}})\leq\pc(\mathcal{V})^\perp$.
\end{proof}

  \begin{proof}[{Proof of \fullref{lem:grid_wide}}]
    By \fullref{lem:chains_irreducible}, if $\mathcal{V}$ is infinite
    we can pass to a finite subchain of length at least $2\caprace-1$
    and long enough that its parabolic closure is all of
    $\pc(\mathcal{V})$.
    Thus, we may assume $\mathcal{V}$ is finite.
    
    If $\mathcal{H}$ is infinite then every wall $H$ in its interior is in
    the interior of some finite subchain of length at least
    $2\caprace$ and long enough that its parabolic closure is all of
    $\pc(\mathcal{H})$.
Proving the proposition for such finite horizontal chains implies
either $\pc(\mathcal{H})=\pc(\mathcal{V})$ is higher rank irreducible
affine, and we are done, or not, in which case $\refl{H}\in
\pc(\mathcal{V})^\perp$. Since the latter is true for every
$H\in\interior{\mathcal{H}}$, we have $\pc(\interior{\mathcal{H}})\leq\pc(\mathcal{V})^\perp$.
Thus, it suffices to prove the proposition for finite chains $\mathcal{V}$ and
$\mathcal{H}$.

  The cardinality lower bounds give that $\pc(\mathcal{V})$ and
  $\pc(\interior{\mathcal{H}})$ are irreducible and nonspherical, by \fullref{lem:chains_irreducible}.
  If $W(\mathcal{V})\not\cong\dihedral_\infty$ then 
   \fullref{cor:nondihedral_grids_commute} says $\pc(\interior{\mathcal{H}})\leq\pc(\mathcal{V})^\perp$.
If $W(\mathcal{V})\cong\dihedral_\infty$   then
\fullref{lem:easy_grid} says either $\pc(\mathcal{H})$ is contained in
$\pc(\mathcal{V})^\perp$ or
 $\pc(\mathcal{V})=\pc(\mathcal{H})$ is irreducible higher rank
 affine.
 Thus, either $\pc(\mathcal{V})\times\pc(\mathcal{V})^\perp$ is a
 product of two infinite factors, or the first factor is irreducible
 higher rank affine, so in either case
 $\pc(\mathcal{V})\times\pc(\mathcal{V})^\perp$ is wide.
\end{proof}

%% file: flats.tex
Recall we defined constants $\radius$ and $\caprace$ and
functions $\density(\lambda,\epsilon)$ and $\length(\lambda,\epsilon)$
in \fullref{constants}.

\begin{definition}
  A path $\gamma\from I\to \Davis$ has \emph{connected wall preimages}
  if for every wall $M$ the set $\gamma^{-1}(\gamma(I)\cap M)$ is a
  connected subset of $I$ (possibly empty).
\end{definition}

\begin{theorem}\label{thm:wide_parallel}
  For all $\lambda$ and $\epsilon$, we have the implications \ref{wide_parallel:segment_in_wide_parabolic_nbhd}$\implies$\ref{wide_parallel:variable_radius}$\implies$\ref{wide_parallel:consecutive_walls_with_many_crossers}$\implies$\ref{wide_parallel:consecutive_walls_with_crossing_chain}$\implies$\ref{wide_parallel:long_transverse_grids}
  for all $(\lambda,\epsilon)$--quasigeodesics $\gamma\from
  I\to\Davis$.
If $\gamma$ is also continuous with connected wall preimages then
\ref{wide_parallel:long_transverse_grids}$\implies$\ref{wide_parallel:segment_in_wide_parabolic_nbhd}.  
Moreover, there are explicit bounds between the constraints of each
item that depend on $\lambda$ and $\epsilon$, but are independent of $\gamma$. 
\begin{enumerate}[label=(\roman*)]
\item For all $L$ there exists a wide parabolic subcomplex $\Upsilon$
  and compact subinterval $I'\subset I$ of length $L$ such that $\gamma(I')\subset\bar\nbhd_\radius(\Upsilon)$.\label{wide_parallel:segment_in_wide_parabolic_nbhd}
  \item There exists $R$ such that for all $L$ there exists a wide parabolic subcomplex
    $\Upsilon$ and compact subinterval $I'\subset I$ of length $L$ such that $\gamma(I')\subset\bar\nbhd_R(\Upsilon)$. \label{wide_parallel:variable_radius}
\item For all $0<\rho$ there exists $L_0$ such that for all
  $L_1\geq L_0$ there exists a compact subinterval $I'\subset I$ of length at
  least $L_1$ such that every chain $\mathcal{V}$ that separates the
  endpoints of  $\gamma(I')$ and has $|\mathcal{V}|/|I'|\geq\rho$ contains a pair of
consecutive walls mutually transverse to an infinite chain.\label{wide_parallel:consecutive_walls_with_many_crossers}
\item For all $0<\rho$ there exists $L_0$ such that for all
  $L_1\geq L_0$ there exists a compact subinterval $I'\subset I$ of 
  length at least $L_1$ such that every chain $\mathcal{V}$ separating
  the endpoints of $\gamma(I')$ with $|\mathcal{V}|/|I'|\geq\rho$ contains a pair of
consecutive walls mutually transverse to a chain of length
$2\caprace$.\label{wide_parallel:consecutive_walls_with_crossing_chain}
\item For all $L'$ there exists a grid $(\mathcal{V},\mathcal{H})$
  with chain lengths  $|\mathcal{V}|\geq L'$ and
  $|\mathcal{H}|=2\caprace$, and such that there is a compact
  subinterval $I'\subset I$ such that $\mathcal{V}$ separates the
  endpoints of $\gamma(I')$. 
 \label{wide_parallel:long_transverse_grids}
\end{enumerate}
\end{theorem}
The remainder of the
subsection proves the theorem in steps.
The implication
\ref{wide_parallel:segment_in_wide_parabolic_nbhd}$\implies$\ref{wide_parallel:variable_radius}
follows by choosing $R=\radius$.
\fullref{prop:near_flat_implies_long_grids} gives \ref{wide_parallel:variable_radius}$\implies$\ref{wide_parallel:consecutive_walls_with_many_crossers}.
The implication
\ref{wide_parallel:consecutive_walls_with_many_crossers}$\implies$\ref{wide_parallel:consecutive_walls_with_crossing_chain}
follows by passing to a subchain.
\fullref{lem:chaintogrid} gives \ref{wide_parallel:consecutive_walls_with_crossing_chain}$\implies$\ref{wide_parallel:long_transverse_grids}.
The proof for
\ref{wide_parallel:long_transverse_grids}$\implies$\ref{wide_parallel:segment_in_wide_parabolic_nbhd}
is broken into three pieces: \fullref{lem:induction_step} is the
inductive step for \fullref{prop:induction_argument}, which is then
applied in \fullref{prop:highly_woven_implies_thick}.

We will see in the proof of \fullref{thm:morse_characterization} that
\fullref{thm:wide_parallel} characterizes when $\gamma$ fails to
be Morse.
Notice that the items of \fullref{thm:wide_parallel} are impossible to
satisfy if $I$ is bounded, because each of them asserts the existence
of arbitrarily long subintervals with some property.
This is the expected behavior; bounded quasigeodesic segments
are always Morse.

\begin{lemma}\label{prop:near_flat_implies_long_grids}
  Suppose $\gamma\from I\to \Davis$ is a
  $(\lambda,\epsilon)$--quasigeodesic such that there
  exists $R$ such that for
  all $L$ there exists a wide parabolic subcomplex $\Upsilon$ and a
  compact subinterval $I'\subset I$ of length $L$ such that 
  $\gamma(I')\subset\bar\nbhd_R(\Upsilon)$.
  Then for all $0<\rho$ there exists $L_0$ such that for all
  $L_1\geq L_0$ there exists a compact subinterval $I'\subset I$ of
  length at least $L_1$ such that every chain $\mathcal{V}$ separating
  the endpoints of $\gamma(I')$ with $|\mathcal{V}|/|I'|\geq\rho$ contains a pair of
consecutive walls mutually transverse to a biinfinite chain.
\end{lemma}
\begin{proof}
Fix $R$ as in the hypothesis and  take any $0<\rho$.
By \fullref{all_metrics_are_equivalent} there exists $L_0\geq\length(\lambda,\epsilon)$ such that
for all $x,y\in\Davis$:
\[\dcat(x,y)\leq R\implies \dchain(x,y)\leq \rho L_0/4-2\]

By hypothesis, given $L_1\geq L_0$ there exists a wide parabolic subcomplex
$\Upsilon$ and compact $I'\subset I$ of length at least $L_1$ such that
$\gamma(I')\subset\bar\nbhd_R(\Upsilon)$.
Suppose that $\mathcal{V}$ is a chain separating the endpoints of $\gamma(I')$ with
$|\mathcal{V}|/|I'|\geq\rho$.

Let $I'=[z_0,z_1]$.
  Apply the 4--gon trick, \fullref{quad_trick}, to the four points
  $\gamma(z_0)$, $\pi_\Upsilon(\gamma(z_0))$,
  $\pi_\Upsilon(\gamma(z_1))$, and $\gamma(z_1)$ to find a convex
  subchain $\mathcal{V}'$ of $\mathcal{V}$ separating
  $\pi_\Upsilon(\gamma(z_0))$ and $\pi_\Upsilon(\gamma(z_1))$, with:
\[|\mathcal{V}|-|\mathcal{V}'|\leq
  \dchain(\gamma(z_0),\pi_\Upsilon(\gamma(z_0)))+\dchain(\gamma(z_1),\pi_\Upsilon(\gamma(z_1)))+2\leq
  2(\rho L_0/4-2)+2\]
But
  $|\mathcal{V}|\geq \rho L_1$, so $|\mathcal{V}'|\geq |\mathcal{V}|/2+2$.
In particular, there is at least one 
  consecutive pair $V$,$V'$ of 
 walls of $\mathcal{V}$
that separate $\pi_\Upsilon(\gamma(z_0))$ and
$\pi_\Upsilon(\gamma(z_1))$, so cut $\Upsilon$.

If $\Upsilon$ is a product of two infinite factors then up to
translation by $W$ we may assume $\Upsilon=\Davis_{T\sqcup T'}$ where
$T$ and $T'$ are subsets of $S$ with no edges connecting them in the
Coxeter diagram.
A wall $M$ cutting $\Upsilon$ has reflection $\refl{M}\in W_T\times
W_{T'}$.
Two walls $M$ and $M'$ cutting $\Upsilon$ give a reflection subgroup
$\langle \refl{M},\refl{M'}\rangle$ that is finite if and only if
$M\transverse M'$.
\fullref{lem:reflection_splitting} says $\langle
\refl{M},\refl{M'}\rangle$ splits as a direct product of reflection
subgroups of the two factors.
If $\refl{M}$ and $\refl{M'}$ live in opposite factors then $\langle
\refl{M},\refl{M'}\rangle\cong(\mathbb{Z}/2\mathbb{Z})^2$.
If they live in the same factor, say $W_T$, then $\langle
\refl{M},\refl{M'}\rangle$ is a subgroup of $W_T$, so $M$ and $M'$
cross in $\Davis$ if and only if they cross in $\Davis_T$.
Thus, walls from $\mathcal{V}$ that cut $\Upsilon$, being disjoint in $\Davis$, either all
have reflections in $W_T$ or all have reflections in $W_{T'}$.
Without loss of generality, assume the former.
By hypothesis $W_{T'}$ is infinite, so there exist walls $H$ and $H'$ of $\Davis$ such that $\langle
\refl{H},\refl{H'}\rangle\cong\dihedral_\infty\leq W_{T'}$.
Then the $\langle\refl{H},\refl{H'}\rangle$ orbits of $H$ and $H'$
give a chain in $\Davis$ consisting of walls with reflections in
$W_{T'}$.
These are transverse to $\{V,V'\}$, since every reflection in $W_T$
commutes with every reflection in $W_{T'}$. 

The other option is that up to $W$--translation
$\Upsilon=\Davis_{T\sqcup T'}$ where $T$ is irreducible higher
rank affine and $T'$ is spherical.
Walls of $\mathcal{V}$ are disjoint, so those that cut $\Upsilon$ have
reflections in $W_T$, since in a spherical parabolic subcomplex all
walls cross. 
Since $W_T$ is affine, $\Davis_T$ is a polyhedral tiling of
$\mathbb{E}^n$ for some $n>1$, and walls of $\Davis_T$ are affine
codimension--1 hyperplanes that come in finitely many parallel
families, with every wall from one family crossing every wall from
every other family.
Since the walls of $\mathcal{V}$ are disjoint in $\Davis$, their
intersections with $\Davis_T$ all belong to the same parallel
family. Any one of the other parallel families gives a chain
transverse to $\Davis_T\cap\mathcal{V}$, so taking the
corresponding walls of $\Davis$ gives an infinite chain transverse to $\{V,V'\}$.
\end{proof}

\begin{lemma}\label{lem:chaintogrid}
  Suppose $\gamma\from I\to \Davis$ is a $(\lambda,\epsilon)$--quasigeodesic such that for all $0<\rho$ there exists $L_0$ such that for all
  $L_1\geq L_0$ there exists a compact subinterval $I'\subset I$ of
  length at least $L_1$ such that every chain $\mathcal{V}$ separating
  the endpoints of $\gamma(I')$ with $|\mathcal{V}|/|I'|\geq\rho$ contains a pair of
consecutive walls mutually transverse to a chain of length
$2\caprace$.
Then for all $L'$ there exists a grid $(\mathcal{V},\mathcal{H})$
  with $\mathcal{V}$ separating the endpoints of a compact subsegment
  of $\gamma$, $|\mathcal{V}|\geq L'$, and $|\mathcal{H}|=2\caprace$. 
\end{lemma}
\begin{proof}
 Assume $L'\geq 2$, and set $\rho:=\density(\lambda,\epsilon)/(L'-1)$.
 Take $L_1:=\max\{\length(\lambda,\epsilon), L_0(\rho)\}$.
The definitions of $\length(\lambda,\epsilon)$ and
$\density(\lambda,\epsilon)$ imply that for any compact 
subinterval $I'\subset I$  of length at least $L_1$, there does exist a
 chain $\mathcal{V}=\{V_i\}_{i_0\leq i\leq i_1}$ separating the
 endpoints of $\gamma(I')$ with $|\mathcal{V}|/|I'|\geq\density(\lambda,\epsilon)$.
Consider the subchain $\{V_{j_k}\}_{0\leq k\leq (i_1-i_0)/(L'-1)}$ where
$j_k=i_0+(L'-1)k$, which contains at least $\rho|I'|$--many walls, so, by
hypothesis, contains a pair of consecutive walls $V$ and $V'$,
transverse to a chain $\mathcal{H}$ of length $2\caprace$. 
Let $\mathcal{V}'$ be the subchain of $\mathcal{V}$ consisting of all
walls from $V$ to $V'$.
By construction,  $|\mathcal{V}'|=L'$, and the two extremes are
transverse to $\mathcal{H}$, so $\mathcal{V}'\transverse\mathcal{H}$.
\end{proof}

In the following, when $\gamma\from I\to \Davis$ is a compact path with connected
wall preimages and $\mathcal{K}=\{K_i\}_{i_0\leq i\leq i_1}$ is a
  chain separating the endpoints of $\gamma(I)$, we assume that the
  parameterization of $\gamma$ and the chain order are compatible.
  The \emph{subsegment
    of $\gamma$ between the extremes of $\mathcal{K}$} means
  $\gamma(J)$ for $J:=[\max\gamma^{-1}(\gamma(I)\cap
    K_{i_0}),\min\gamma^{-1}(\gamma(I)\cap K_{i_1})]$.
Note that all but the extreme walls of $\mathcal{K}$ separate the
endpoints of $\gamma(J)$.

\begin{lemma}\label{lem:induction_step}
  Let $\gamma\from I\to\Davis$ be a compact path with connected wall preimages.
  Let $\mathcal{V}$ be a chain that separates the endpoints of
  $\gamma(I)$ and has  $|\mathcal{V}|\geq 2\caprace$.
  Let $\Upsilon$ be the unique parabolic subcomplex with stabilizer
  $\pc(\mathcal{V})\times\pc(\mathcal{V})^\perp$ (recall
  \fullref{normalizer_parabolics_have_well_defined_subcomplex}).
  Suppose there is a $t\in I$ such that $\dchain(\gamma(t),\Upsilon)\geq 4\caprace-1$.
  Define $\mathcal{V}^{-}$ to be the set of walls $V\in\mathcal{V}$ for which $\max\gamma^{-1}(\gamma(I)\cap
  V)<t$, and define $\mathcal{V}^+$ to be the set of walls $V\in\mathcal{V}$ for which
  $\min\gamma^{-1}(\gamma(I)\cap V)>t$.
  Then $\mathcal{V}^{-}$ and $\mathcal{V}^+$ are disjoint convex
  subchains of $\mathcal{V}$, at least one of which is
    transverse to a chain $\mathcal{H}$ with $|\mathcal{H}|\geq
    2\caprace$ and $\refl{H}\notin
    \pc(\mathcal{V})\times\pc(\mathcal{V})^\perp$ for all $H\in\mathcal{H}$.
  \end{lemma}
  \begin{figure}[h]
    \centering
    \includegraphics{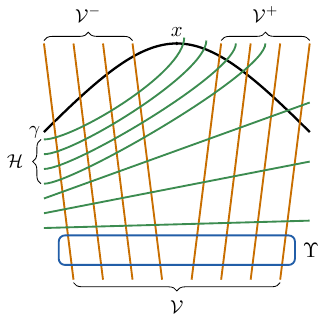}
    \caption{Inductive step: A point $x\in\gamma$ between the extremes
      of $\mathcal{V}$ and far from
      $\Upsilon$ gives a
        splitting of $\mathcal{V}$ into subchains $\mathcal{V}^+$ and
        $\mathcal{V}^{-}$ such that one of them is part of a grid with
        a chain $\mathcal{H}$ consisting of ``most''  of the walls separating $x$ from
        $\Upsilon$.
        Only at most $2\caprace-1$ of the walls separating $x$ from
        $\Upsilon$ can be
        transverse to all of $\mathcal{V}$ without cutting $\Upsilon$.}
      \label{fig:inductive_step}
  \end{figure}
  \begin{proof}
    By hypothesis, there is a chain $\mathcal{K}$ with
    $|\mathcal{K}|\geq 4\caprace-1$ separating $x:=\gamma(t)$ from $\Upsilon$.
    Every wall $K\in\mathcal{K}$ is disjoint from
    $\{x\}\cup\Upsilon$.
    Every reflection in $\pc(\mathcal{V})\times\pc(\mathcal{V})^\perp$
    cuts $\Upsilon$, so 
    $\refl{K}\notin\pc(\mathcal{V})\times\pc(\mathcal{V})^\perp$ for
    all $K\in\mathcal{K}$.

    If $I=[a,b]$, let $\gamma^{-}=\gamma([a,t])$, and let
    $\gamma^+=\gamma([t,b])$.

    Since $\gamma$ has connected wall preimages and $\mathcal{V}$ is a
    chain separating the endpoints of $\gamma$, $V$ comes before $V'$
    in $\mathcal{V}$ if and only if $\gamma^{-1}(\gamma(I)\cap V)$ is
    a nonempty closed interval that occurs in $I$ before the nonempty
    closed interval $\gamma^{-1}(\gamma(I)\cap V')$.
    It follows that $\mathcal{V}^{-}$ is an initial subchain of
    $\mathcal{V}$, $\mathcal{V}^+$ is a terminal subchain of
    $\mathcal{V}$, they are disjoint, and they account for all but at
    most one wall of $\mathcal{V}$. 

    For every $V\in \mathcal{V}^{-}$ there is a path $\delta_V$ from $\Upsilon$
    to $x$ that travels from a point in $V\cap\Upsilon$ through
    $V$ to $V\cap\gamma\subset\gamma^-$, and then along $\gamma^-$ to
    $x$.
    The same is true for $V\in\mathcal{V}^+$, with $\gamma^+$ replacing
    $\gamma^-$.
    Since $K\in \mathcal{K}$ separates $x$ from $\Upsilon$, it cuts
    every such path.
    If $K$ does not meet $\gamma^-$ then for all $V\in\mathcal{V}^{-}$
    it must be that $K$ cuts the part of $\delta_V$ contained in $V$,
    so $K\transverse\mathcal{V}^{-}$.
    Similarly, $K\cap \gamma^+=\emptyset$ implies $K\transverse\mathcal{V}^+$.
    At least one of these is true, since the fact that $\gamma^{-1}(K\cap\gamma(I))$ is connected and does not
    contain $t$ implies that $\gamma^{-1}(K\cap\gamma(I))$ does not
    contain points from both $[a,t]$ and $[t,b]$.
    Moreover, suppose some wall $K\in\mathcal{K}$ meets $\gamma^+$.
    If there is a wall $K'\in\mathcal{K}$ separating $K$ from $x$ then
    it must cut the subsegment of $\gamma^+$ between $x$ and
    $K\cap\gamma^+$, so $K'$ meets $\gamma^+$ and does not meet
    $\gamma^-$.
    Conversely, suppose there is a wall $K'\in\mathcal{K}$ that is separated from $x$
    by $K$.
    Since $K$ meets $\gamma^+$ it does not meet $\gamma^-$, so
    $K'$ also does not meet $\gamma^-$, else we would have a path from
    $K'$ to $x$ through $\gamma^-$ avoiding $K$.
    Thus, if $\mathcal{K}$ is ordered from 
    $\Upsilon$ toward $x$, then
    $\mathcal{K}$ splits into an initial subchain $\mathcal{K}'$
    consisting of walls disjoint from $\gamma$ and a terminal subchain
    $\mathcal{K}''$ consisting of walls that either all meet $\gamma^+$
    or all meet $\gamma^-$.
    
    Suppose $|\mathcal{K}'|\geq 2\caprace$.
    Since $\mathcal{K}'$ is
    disjoint from $\gamma$ it is transverse to $\mathcal{V}$.
    Apply \fullref{lem:grid_wide} to $(\mathcal{V},\mathcal{K}')$.
    Then for $K\in\interior{\mathcal{K}}'$ we have $\refl{K}\in
    \pc(\mathcal{V})\times\pc(\mathcal{V})^\perp$, which is a
    contradiction.
Thus, $|\mathcal{K}'|\leq 2\caprace-1$, which implies $|\mathcal{K}''|\geq
2\caprace$.
Take $\mathcal{H}:=\mathcal{K}''$.
If the walls of $\mathcal{K}''$ meet $\gamma^+$ then
    $\mathcal{H}\transverse\mathcal{V}^-$, and if the walls of
    $\mathcal{K}''$ meet $\gamma^-$ then
    $\mathcal{H}\transverse\mathcal{V}^+$.
  \end{proof}

\begin{proposition}\label{prop:induction_argument}
  Let $C:= 3^{|S|-3}$.
  For every compact path $\gamma\from I\to \Davis$ with connected wall
  preimages, for every grid $(\mathcal{V},\mathcal{H})$ with
  $\mathcal{V}$ separating the endpoints of $\gamma$ and chain lengths
  $|\mathcal{V}|\geq C(2\caprace-1)$ and $|\mathcal{H}|\geq
  2\caprace$, there exists a wide parabolic subcomplex $\Upsilon$ and
  a convex subchain $\mathcal{V}'$ of $\mathcal{V}$ such that
  $|\mathcal{V}'|\geq |\mathcal{V}|/C$ and the subsegment of $\gamma$
  between the extremes of $\mathcal{V}'$ is contained in $\bar\nbhd_\radius(\Upsilon)$.
\end{proposition}
 \begin{proof}
Suppose there is a grid $(\mathcal{V},\mathcal{H})$ satisfying the
hypotheses.
Note that if such a grid exists then by \fullref{lem:grid_wide}
either both $\pc(\mathcal{V})$ and $\pc(\mathcal{V})^\perp$ are
nonspherical, in which case $|S|\geq 4$, 
or $\pc(\mathcal{V})$ is irreducible higher rank affine, in which case
$|S|\geq 3$ with equality only if $W=\pc(\mathcal{V})$ is a Euclidean
triangle group.
If $W$ is a Euclidean triangle group we are finished by taking
$\Upsilon=\Davis$. Thus, we may suppose $|S|>3$.

The proof is by induction on rank: we guess a wide product parabolic
subcomplex and subsegment of $\gamma$, and if the lemma is not
satisfied we change the guesses in such a
way that the rank of the first factor of the parabolic decreases. 
 
   Let $\mathcal{V}_0:=\mathcal{V}=\{V_i\}_{i_0\leq i\leq i_1}$ and
   $\mathcal{H}_0:=\mathcal{H}$.
   Let $j_0:=i_0$, $j_1:=i_0+\lfloor\frac{i_1-i_0}{3}\rfloor$,
   $j_2:=j_1+\lfloor\frac{i_1-i_0}{3}\rfloor$, and $j_3:=i_1$.
Let $\mathcal{V}_{01}:=\{V_j\}_{j_0\leq j\leq j_1}$,
$\mathcal{V}_{02}:=\{V_j\}_{j_1\leq j\leq j_2}$, and
$\mathcal{V}_{03}:=\{V_j\}_{j_2\leq j\leq j_3}$; each of these is a
convex subchain of $\mathcal{V}_0$ containing at least
$\frac{|\mathcal{V}_0|}{3}$-many walls.
Define
$\Upsilon_0$ to be the unique parabolic subcomplex with stabilizer $\pc(\mathcal{V}_0)\times\pc(\mathcal{V}_0)^\perp$,
which is wide by \fullref{lem:grid_wide} applied to $(\mathcal{V},\mathcal{H})$.
Consider the subsegment $\gamma'$ of $\gamma$ between the extremes of
$\mathcal{V}_{02}$.
If $\gamma'\subset\bar\nbhd_\radius(\Upsilon_0)$ then we are done, since $|\mathcal{V}_{02}|\geq|\mathcal{V}|/3\geq|\mathcal{V}|/C$. 
Otherwise, let $x=\gamma'(t)$ be a point with $t$ in the interior of
the domain of $\gamma'$ and $\dcat(x,\Upsilon_0)>\radius$.
Define $\mathcal{V}_0^-$ to be the convex subchain of
$\mathcal{V}$ consisting of walls $V$ with
$\max\gamma^{-1}(\gamma(I)\cap V)<t$, and define $\mathcal{V}_0^+$ to be the convex subchain of
$\mathcal{V}$ consisting of walls $V$ with
$\min\gamma^{-1}(\gamma(I)\cap V)>t$.

Apply \fullref{lem:induction_step}  to produce a chain
$\mathcal{H}_1'$ with $|\mathcal{H}_1'|\geq 2\caprace$
that is transverse to either $\mathcal{V}_0^-$ or $\mathcal{V}_0^+$.
Let $\mathcal{V}_1$ be whichever of $\mathcal{V}_0^-$ or
$\mathcal{V}_0^+$  is transverse
to $\mathcal{H}_1'$.
\fullref{lem:induction_step} also says that reflections in walls of
$\mathcal{H}_1'$ are not in $\pc(\mathcal{V}_0)\times\pc(\mathcal{V}_0)^\perp$.
Observe that $\mathcal{V}_{01}\subset\mathcal{V}_0^-$ and
$\mathcal{V}_{03}\subset\mathcal{V}_0^+$, so in 
either case
$|\mathcal{V}_1|\geq|\mathcal{V}_0|/3\geq 3^{|S|-4}(2\caprace-1)$.

If $W(\mathcal{V}_0)\cong\dihedral_\infty$ we get a contradiction,
because $\mathcal{V}_1$ contains more than 8 walls of
$\mathcal{V}_0$, so \fullref{Caprace_11} upgrades
$\mathcal{H}_1'\transverse\mathcal{V}_1$ to $\mathcal{H}_1'\transverse\mathcal{V}_0$, but then
\fullref{lem:easy_grid} says $\pc(\mathcal{H}_1')$ is either contained
in $\pc(\mathcal{V}_0)$ or $\pc(\mathcal{V}_0)^\perp$.
Thus, if $W(\mathcal{V}_0)\cong\dihedral_\infty$ then we must have had
$\gamma'\subset\bar\nbhd_\radius(\Upsilon_0)$, and we are done. 

Otherwise, \fullref{lem:grid_wide} says
$\pc(\mathcal{V}_1)\times\pc(\mathcal{V}_1)^\perp$ is wide
and contains reflections through all walls of
$\mathcal{H}_1:=\interior{\mathcal{H}}_1'$.
Moreover, if $W(\mathcal{V}_1)\not\cong\dihedral_\infty$ then
$\pc(\mathcal{V}_1)$ is not higher rank irreducible affine, so
$\pc(\mathcal{H}_1)\leq\pc(\mathcal{V}_1)^\perp$. 
Now, $\mathcal{V}_1\subset\mathcal{V}_0$ implies
$\pc(\mathcal{V}_1)\leq\pc(\mathcal{V}_0)$, which implies
$\pc(\mathcal{V}_0)^\perp\leq\pc(\mathcal{V}_1)^\perp$.
In the $W(\mathcal{V}_1)\not\cong\dihedral_\infty$ case these cannot
be equalities, because in that case we know that for all
$H\in\mathcal{H}_1$ we have
$\refl{H}\in\pc(\mathcal{V}_1)^\perp\setminus\pc(\mathcal{V}_0)^\perp$.
Thus $W(\mathcal{V}_1)\not\cong\dihedral_\infty\implies\pc(\mathcal{V}_1)\lneq\pc(\mathcal{V}_0)$.

\begin{figure}[h]
  \centering
  \includegraphics{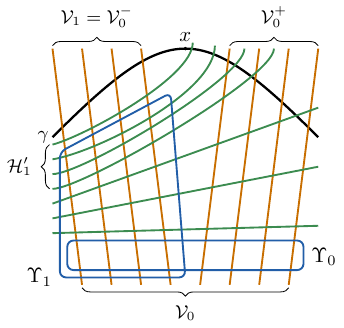}
  \caption{Using the induction step to produce a new candidate wide
    parabolic subcomplex $\Upsilon_1$ corresponding to
    $\mathcal{V}_1\subset\mathcal{V}_0$.}
  \label{fig:induction_argument}
\end{figure}

Define
$\Upsilon_1$ to be the unique parabolic subcomplex with stabilizer  $\pc(\mathcal{V}_1)\times\pc(\mathcal{V}_1)^\perp$,
and repeat the construction. See \fullref{fig:induction_argument}.
Either the induction will stop with the segment of $\gamma$
corresponding to the middle third of $\mathcal{V}_1$ contained in
$\bar\nbhd_\radius(\Upsilon_1)$, which is guaranteed to be the case if
$W(\mathcal{V}_1)\cong\dihedral_\infty$, or
$\pc(\mathcal{V}_1)\lneq\pc(\mathcal{V}_0)$, in which case we produce a further
$\mathcal{V}_2$, $\mathcal{H}_2$, $\Upsilon_2$, and repeat.

The ranks of parabolics in a strictly decreasing sequence are strictly
decreasing, so this iteration terminates in at most $|S|$ steps.
In fact, it terminates in at most $|S|-4$ steps, since either at some
point in the procedure we reach an affine $\pc(\mathcal{V}_i)$, or at
each step we have both $\pc(\mathcal{V}_i)$ and
$\pc(\mathcal{V}_i)^\perp$ nonspherical.
In the latter case, $\pc(\mathcal{V}_0)^\perp$ has rank at least 2, which means 
$\pc(\mathcal{V}_0)$ has rank at most $|S|-2$, so the rank of
$\pc(\mathcal{V}_i)$ is at least 2 and at most $|S|-2-i$, whence
$i\leq |S|-4$.
Thus, in $0\leq i\leq |S|-4$ steps we produce a wide parabolic subcomplex
$\Upsilon_i$ such that $\bar\nbhd_\radius(\Upsilon_i)$ contains the subsegment of
$\gamma$ between the extremes of $\mathcal{V}_{i2}$, where:
\[|\mathcal{V}_{i2}|\geq\frac{|\mathcal{V}_i|}{3}\geq \frac{|\mathcal{V}|}{3^{i+1}}\geq\frac{|\mathcal{V}|}{3^{|S|-3}}=\frac{|\mathcal{V}|}{C}\qedhere\]
\end{proof}

\begin{lemma}\label{prop:highly_woven_implies_thick}
  Suppose $\gamma\from I\to\Davis$ is a
  $(\lambda,\epsilon)$--quasigeodesic path with connected wall
  preimages, such that
  for all $L'$ there exists a grid $(\mathcal{V},\mathcal{H})$
  with $\mathcal{V}$ separating the endpoints of a compact subsegment
  of $\gamma$ and chain lengths $|\mathcal{V}|\geq L'$ and
  $|\mathcal{H}|=2\caprace$.
  Then for all $L$ there exists a wide parabolic subcomplex $\Upsilon$
  and subinterval $I'\subset I$ of length at least $L$ such that $\gamma(I')\subset\bar\nbhd_\radius(\Upsilon)$.
  \end{lemma}
 \begin{proof}
   Fix arbitrary $L$.
   Let $C$ be the constant of \fullref{prop:induction_argument}.
By \fullref{all_metrics_are_equivalent}, we
can choose $L'$ sufficiently large that for all $x,y\in\Davis$:
\[\dchain(x,y)\geq L'/C-2\implies \dcat(x,y)\geq \lambda L+\epsilon\]
Set $L'':=\max\{C(2\caprace-1),L'\}$.
Take a grid $(\mathcal{V},\mathcal{H})$ as hypothesized, with 
$|\mathcal{V}|\geq L''$ and pass to a compact subsegment of $\gamma$
whose endpoints are separated by $\mathcal{V}$.
Since $|\mathcal{V}|\geq C(2\caprace-1)$, we can apply \fullref{prop:induction_argument} to $(\mathcal{V},\mathcal{H})$ to
get a convex subchain $\mathcal{V}'$ of $\mathcal{V}$ and a wide
parabolic subcomplex $\Upsilon$ such that the subsegment $\gamma(I')$ of $\gamma$
between the extremes of $\mathcal{V}'$ is contained in
$\bar\nbhd_\radius(\Upsilon)$.
All but the extreme walls of $\mathcal{V}'$ separate the endpoints of
$\gamma(I')$. The number of such walls is at least 
$|\mathcal{V}|/C-2\geq L'/C-2$, so the $\dcat$--distance between the
endpoints of $\gamma(I')$ is at least $\lambda L+\epsilon$.
Since $\gamma$ is a $(\lambda,\epsilon)$--quasigeodesic, $|I'|\geq L$.
\end{proof}

%% file: proof_main_theorem.tex
We can now characterize certain Morse quasigeodesics, specifically,
Morse quasigeodesic paths with connected wall preimages.
This family includes both the $\dcomb$-- and $\dcat$--geodesics:
$\dcomb$--geodesics are uniformly $\dcat$--quasigeodesic, by
\fullref{all_metrics_are_equivalent}, and every wall preimage is empty
or a single point. For a $\dcat$--geodesic, connected wall preimages
follows from convexity of walls. 

\begin{definition}\label{def:fine_chain}
  A chain is \emph{fine} if for every pair of consecutive walls $V$
  and $V'$ in the chain, every chain $\mathcal{H}$ transverse to both
  $V$ and $V'$ has $|\mathcal{H}|< 2\caprace$, where $\caprace$ is
  the constant defined in \fullref{constants}.
\end{definition}

\begin{theorem}[Characterization of the Morse property]\label{thm:morse_characterization}
  Let $\Davis=\Davis(W,S)$ be the Davis complex of a finite rank
  Coxeter system. There are constants $\radius$ and $\caprace$ depending only on $(W,S)$, such that for all
  $\lambda$ and $\epsilon$, we have implications
\ref{main:long_transverse_grids}$\implies$\ref{main:fine_chain}$\implies$\ref{main:bounded_width_chain}$\implies$\ref{main:morse}$\implies$\ref{main:segment_in_wide_parabolic_nbhd}
for all $(\lambda,\epsilon)$--quasigeodesic paths
$\gamma\from I\to \Davis$.
If $\gamma$ additionally has connected wall preimages then
\ref{main:segment_in_wide_parabolic_nbhd}$\implies$\ref{main:long_transverse_grids}.
Moreover, there are effective bounds between the constraints of each
implication that depend on $\lambda$ and $\epsilon$ but are independent of $\gamma$.
  \begin{enumerate}[label=(\arabic*)]
    \item There is an upper bound on the lengths of chains
      $\mathcal{V}$ that separate the endpoints of a compact
      subsegment of $\gamma$ and for which there exists a
      grid $(\mathcal{V},\mathcal{H})$ with
      $|\mathcal{H}|=2\caprace$.\label{main:long_transverse_grids}
    \item There exists $0<\rho$ such that there
       exists $L$ such that for every compact subinterval $I'\subset I$ of
       length at least $L$ there exists a
       fine chain
       $\mathcal{V}$ separating the endpoints of $\gamma(I')$ with
       $|\mathcal{V}|/|I'|\geq \rho$. \label{main:fine_chain}
    \item There exist $C$ and $0<\rho$ such that there
       exists $L$ such that for every compact subinterval $I'\subset I$ of length at least $L$ there exists a chain
       $\mathcal{V}$ separating the endpoints of $\gamma(I')$ with
       $|\mathcal{V}|/|I'|\geq \rho$ such
      that for every pair $V$, $V'$ of consecutive walls of
      $\mathcal{V}$, every chain
      $\mathcal{H}\transverse\{V,V'\}$ has $|\mathcal{H}|\leq C$.\label{main:bounded_width_chain}
    \item There exists $\mu$ such that $\gamma$ is $\mu$--Morse.\label{main:morse}
    \item There is an upper bound on the lengths of subintervals
      $I'\subset I$ for which there exists a wide parabolic subcomplex
      $\Upsilon$ with $\gamma(I')\subset\bar\nbhd_\radius(\Upsilon)$. \label{main:segment_in_wide_parabolic_nbhd}
    \end{enumerate}
  \end{theorem}
  \begin{proof}
    Recall that \fullref{thm:wall_morse} characterized the Morse property in terms of
equivalent items
\ref{wall_morse:divergence}-\ref{wall_morse:bounded_gap_bounded_chains},
while 
\fullref{thm:wide_parallel} gave items
\ref{wide_parallel:segment_in_wide_parabolic_nbhd}-\ref{wide_parallel:long_transverse_grids}
describing the interaction of the quasigeodesic with wide parabolics.
Compared to
Items~\ref{main:long_transverse_grids}-\ref{main:segment_in_wide_parabolic_nbhd},
we have:
\ref{main:long_transverse_grids}$=\neg$\ref{wide_parallel:long_transverse_grids}, 
\ref{main:fine_chain}$=\neg$\ref{wide_parallel:consecutive_walls_with_crossing_chain}, 
\ref{main:bounded_width_chain}$=$\ref{wall_morse:bounded_gap_bounded_chains},
\ref{main:morse}$=$\ref{wall_morse:morse}, and 
 \ref{main:segment_in_wide_parabolic_nbhd}$=\neg$\ref{wide_parallel:segment_in_wide_parabolic_nbhd}.

Logically, the implication \ref{wide_parallel:consecutive_walls_with_crossing_chain} 
$\implies$\ref{wide_parallel:long_transverse_grids}
in \fullref{thm:wide_parallel} gives \ref{main:long_transverse_grids}
$\implies$\ref{main:fine_chain}, but we also need to demonstrate
effective quantitative bounds.
We argue as in the proof of \ref{wide_parallel:consecutive_walls_with_crossing_chain} 
$\implies$\ref{wide_parallel:long_transverse_grids}, which was \fullref{lem:chaintogrid}:
Suppose that \ref{main:long_transverse_grids} gives a positive integral upper bound
$B$, and define $\rho:=\density(\lambda,\epsilon)/B>0$.
For every parameter subinterval $I'$ of length at least
$\length(\lambda,\epsilon)$, there exists a chain separating the
endpoints of $\gamma(I')$ of density at least $\density(\lambda,\epsilon)$. Take such a chain $\mathcal{V}$, and pass to
the subchain $\mathcal{V}'$
consisting of every $B$-th wall, which has density at least $\rho$.
If $\mathcal{V}'$ is not fine then it has consecutive walls $V$
and $V'$ transverse to a chain $\mathcal{H}$ of length
$2\caprace$.
Then the convex subchain $\mathcal{V}''$ of $\mathcal{V}$
from $V$ to $V'$ makes a grid with $\mathcal{H}$.
This is a contradiction, since $|\mathcal{V}''|=B+1>B$, so
$\mathcal{V}'$ is fine. 
Thus, \ref{main:long_transverse_grids}  with bound $B$ implies
\ref{main:fine_chain} with
$\rho=\density(\lambda,\epsilon)/B$ and $L=\length(\lambda,\epsilon)$.

The implication
\ref{main:fine_chain}$\implies$\ref{main:bounded_width_chain} is
immediate from the definition of fine chain, by taking $C=2\caprace-1$.

\fullref{thm:wall_morse}
\ref{wall_morse:bounded_gap_bounded_chains}$\implies$\ref{wall_morse:morse}
gives 
\ref{main:bounded_width_chain}$\implies$\ref{main:morse}.

Suppose \ref{main:morse}. 
We have \ref{main:morse}$=$\ref{wall_morse:morse}$\implies$\ref{wall_morse:bounded_gap_bounded_chains}$=$\ref{main:bounded_width_chain}, 
by \fullref{thm:wall_morse}, so there are $C$, $\rho$, and $L$
depending on $\mu$, $\lambda$, and $\epsilon$ that satisfy \ref{main:bounded_width_chain}.
With this $\rho$, the proof of 
\fullref{prop:near_flat_implies_long_grids} constructs $L_0$ depending on
$\rho$, $R=\radius$, $\lambda$, and $\epsilon$ such that for any
subinterval $I'$ of length at least $L_0$ if $\gamma(I')$ is contained in the
closed $\radius$--neighborhood of a wide parabolic subcomplex
$\Upsilon$ then every chain of density at least $\rho$ and separating the endpoints of $\gamma(I')$
contains a consecutive pair of walls mutually transverse to an
infinite chain.
Now if $|I'|\geq \max\{L,L_0\}$ and there exists a wide parabolic
subcomplex $\Upsilon$ with $\gamma(I')\subset\bar\nbhd_\radius(\Upsilon)$, this contradicts
\ref{main:bounded_width_chain}.
Thus, $\max\{L,L_0\}$ is a bound satisfying \ref{main:segment_in_wide_parabolic_nbhd}.

Now suppose $\gamma$ also has connected wall preimages and suppose
\ref{main:segment_in_wide_parabolic_nbhd} holds with bound $A$.
\fullref{prop:highly_woven_implies_thick} takes a given $L$ and
produces $L'$ such that if $(\mathcal{V},\mathcal{H})$ is a grid with
$|\mathcal{V}|\geq L'$ and $|\mathcal{H}|=2\caprace$, with
$\mathcal{V}$ separating the endpoints of a compact subsegment of
$\gamma$, then there is a smaller compact subsegment $\gamma(I')$
contained in the closed $\radius$--neighborhood of a wide parabolic
subcomplex, with $|I'|\geq L$.
Take $L:=A+1$, and take the corresponding $L'$.
This $L$ exceeds the upper bound of
\ref{main:segment_in_wide_parabolic_nbhd}, so there is no grid $(\mathcal{V},\mathcal{H})$ with
$|\mathcal{V}|\geq L'$ and $|\mathcal{H}|=2\caprace$ that separates the
endpoints of a compact subsegment of $\gamma$.
Thus, \ref{main:long_transverse_grids} holds with bound $L'$. 
  \end{proof}

\fullref{thm:morse_characterization} applies to combinatorial
geodesics, but via the identification of $\Davis^{(1)}$ with the
Cayley graph of $W$ with respect to $S$, each edge of $\Davis^{(1)}$
is labelled with one of the generators. For combinatorial geodesics we
can give an additional characterization of the Morse property in terms
of labels on subsegments. Furthermore, we can tighten the geometric
condition of \fullref{thm:morse_characterization} of being contained
in a fixed neighborhood of a wide parabolic subcomplex to actually
being contained in a wide parabolic subcomplex:
  
 \begin{theorem}[Characterization of the Morse property for combinatorial geodesics]\label{main_for_combinatorial}
    Let $\Davis=\Davis(W,S)$ be the Davis complex of a finite rank
    Coxeter system.
    There are effective quantitative equivalences between the
    following items, uniformly over all combinatorial geodesics $\gamma$ in $\Davis$.
  \begin{enumerate}[label=(\arabic*)]
  \item There exists $\mu$ such that $\gamma$ is $\mu$--Morse.\label{comb:morse}
    \item There is an upper bound on the lengths of subsegments
      $\gamma'$ of $\gamma$ for which there exists a wide parabolic
      subcomplex $\Upsilon$ with
      $\gamma'\subset\Upsilon$.\label{comb:parabolic_wide}
      \item There is an upper bound on the lengths of subsegments
        $\gamma'$ of $\gamma$ for which there exists 
        $T\subset S$ such that $W_T$ is a wide special
        subgroup and all edges of $\gamma'$ have label
        in $T$.\label{comb:special_wide}
  \end{enumerate}
  \end{theorem}
  \begin{proof}
    The implications \ref{comb:morse}$\implies$\ref{comb:parabolic_wide} and \ref{comb:special_wide}$\implies$\ref{comb:parabolic_wide}
are easy, and \ref{comb:parabolic_wide}$\implies$\ref{comb:special_wide} follows from a similar argument to the second step of the proof of $\neg$\ref{comb:morse}$\implies\neg$\ref{comb:special_wide}.
We prove $\neg$\ref{comb:morse}$\implies\neg$\ref{comb:special_wide}.
There are two parts to the proof. The first is to go from a
combinatorial geodesic having a long subsegment contained in the
$\radius$--neighborhood of a wide parabolic subcomplex $\Upsilon=w\Davis_T$ to having a
long subsegment such that every edge $e_i$ is dual to a wall that cuts
$\Upsilon$. 
The second part is to argue that the edge labels of the $e_i$ are all
contained in a common wide subset of $S$. The complication in this
step is that in non-right-angled Coxeter groups the edges dual to a fixed wall
might not all have the same label; the labels might only be conjugate
generators.
This means the subset of $S$
containing the labels of the edges $e_i$ might not be $T$---it might be some
subset of $S$ related to $T$ by a conjugation.

For the first step, we start with a counting argument.
Suppose that $\mathcal{M}$ is the set of walls dual to edges of a
  combinatorial geodesic segment $\alpha$, and suppose $\mathcal{M}'$
  is some subset of those walls.  
Partition $\alpha$ into a concatenation of subpaths that are either
single edges dual to a wall in $\mathcal{M}\setminus\mathcal{M}'$ or a maximal
subpath such that every edge is dual to a wall in $\mathcal{M}'$.
Suppose that the lengths of the latter are bounded above by $N$.
There are at most $|\mathcal{M}|-|\mathcal{M}'|+1$ such subsegments,
so $|\mathcal{M}|=|\alpha|\leq N(|\mathcal{M}|-|\mathcal{M}'|+1)+|\mathcal{M}|-|\mathcal{M}'|$.
Thus:
\begin{equation}
  \label{eq:18}
\frac{|\mathcal{M}'|}{|\mathcal{M}|-|\mathcal{M}'|+1}\leq N\tag{$\dagger$}  
\end{equation}

Suppose $\gamma\from I\to \Davis$ is a combinatorial geodesic.
Then $\gamma$ is a $(\lambda,\epsilon)$--quasigeodesic
path with connected wall preimages.
Furthermore, $\gamma$ is injective and parameterized by arc length,
and the quasigeodesic constants are uniform over all combinatorial
geodesics. 

Suppose $\neg$\ref{comb:morse}, so $\gamma$ is not Morse.
It follows from \fullref{thm:morse_characterization} that for every $L$ there exists
a wide parabolic subcomplex $\Upsilon$ and subinterval $I'\subset I$ of length $L$ such
that $\gamma':=\gamma(I')$ is an
edge path contained in $\bar\nbhd_\radius(\Upsilon)$.
Let $x$ and $y$ be the end vertices of $\gamma'$, and consider the set of
all walls $\mathcal{M}$ separating $x$ and $y$, of which there are
exactly $L$--many, since they are the walls cutting the edges of $\gamma'$. 
Let $\mathcal{M}'$ be the subset of $\mathcal{M}$ consisting of
walls that cut $\Upsilon$.
A $\dwall$--version of the 4--gon trick, \fullref{quad_trick}, gives:
\[|\mathcal{M}|-|\mathcal{M}'|\cleq
  \dwall(x,\gate_\Upsilon(x))+\dwall(y,\gate_\Upsilon(y))\cleq \radius\]
Since the right-hand side is a fixed constant, 
this gives $|\mathcal{M}'|\cequiv |\mathcal{M}|= L$, so:
\begin{equation}
  \label{eq:21}
  \frac{|\mathcal{M}'|}{|\mathcal{M}|-|\mathcal{M}'|+1}\cgeq L\tag{$\ddagger$}
\end{equation}
The coarse comparison constants of \eqref{eq:21} are independent of the choices of $L$ and
$\gamma$, so, given $N$ we can choose $L$ large enough that
\eqref{eq:21} implies $\neg$\eqref{eq:18}, which
guarantees that $\gamma'$ contains a subsegment $\gamma''$ consisting
of at least $N$--many consecutive edges dual to walls of
$\mathcal{M}'$.
Thus, $\gamma$ contains arbitrarily long subsegments $\alpha:=\gamma''$ such that
there exists a wide parabolic subcomplex $\Upsilon$ such that
$\alpha\subset\bar\nbhd_\radius(\Upsilon)$ is an edge path, every edge
of which  is dual to a wall that cuts $\Upsilon$. 

For the second step, suppose $a$ and $b$ are the first and last vertices of $\alpha$,
respectively.
Replace $\Upsilon$, $\alpha$, $a$, and $b$ by their
$\gate_\Upsilon(a)^{-1}$--translates, so that $\Upsilon^{(0)}=W_T$ for
some wide $T\subset S$ and $\gate_\Upsilon(a)=\one$.
Suppose the vertices of $\alpha$ are $a=x_0,x_1,\dots,x_n=b$ and $e_i$
is the edge from $x_{i-1}$ to $x_i$.
Suppose the label of $e_i$ is $s_i\in S$.
Let $U:=\{s_1,\dots,s_n\}$, so that the vertices of $\alpha$ are
contained in the parabolic subcomplex $a\Davis_U$.

Since the wall dual to $e_i$ cuts $\Upsilon$, its reflection
$r_i:=x_is_ix_i^{-1}$ is in $W_T$.

Now, $x_1=r_1x_0\in W_Tx_0$.
By induction, $x_j=r_jx_{j-1}=r_j\cdots r_1x_0\in W_Tx_0$.
Plugging this expression for $x_j $ into the definition of $r_j$ gives
$r_j=r_j\cdots r_1x_0s_jx_0^{-1}r_1\cdots r_j$.
Rearranging gives $s_j=x_0^{-1}r_1\cdots r_{j-1}r_jr_{j-1}\cdots
r_1x_0\in x_0^{-1}W_Tx_0=a^{-1}W_Ta$.

Since $\gate_\Upsilon(a)=\one$, the minimal length element of $W_Ta$
is $a$,
so for all $u\in W_T$ we have $|ua|=|u|+|a|$ \cite[Lemma~4.3.3]{davisbook}.
Since $W_T\ni w_j:=as_ja^{-1}$, we have
$w_ja=as_j$.
The left-hand side is an element of $W_Ta$, so the minimal length
words in $S$ representing this element have length $|w_j|+|a|$.
The right-hand side is the same group element, expressed as a word of
length $|a|+1$. Thus $|w_j|=1$, which means $w_j\in T$.
We have shown $U\subset S\cap a^{-1}Ta$.

Once $|\alpha|>\maxsphere$ then $U$ cannot be spherical.
Our goal is to show that $U$ is contained in a wide subset of $S$,
establishing $\neg$\ref{comb:special_wide}.
This is true, in particular, if $U$ itself is wide, so suppose not.
Thus, the canonical product decomposition of $W_U$ is $W_{U_0}\times
W_{U_1}$, where $U_0$ is irreducible nonspherical and not higher rank
affine and $U_1$ is spherical. 
This implies
$U_1\subset U_0^\perp$, so $W_U\leq W_{U_0}\times W_{U_0^\perp}$.
Thus, it suffices to show $U_0^\perp$ is nonspherical. 

By \fullref{lem:nonspherical_normalizer}~\eqref{item:irreducible_rigidity},
since $U_0\subset S$ is nonspherical irreducible,
$aU_0a^{-1}\subset T\subset S$ implies $U_0=aU_0a^{-1}$.
Let $T_0$ be the irreducible component of $T$ containing $U_0$.
If $T_0$ is higher rank affine then $T_0=U_0$, since irreducible
affine parabolics have no proper nonspherical subparabolics.
This is impossible, since we already said that $U_0$ is not higher rank
affine.
Since $T$ is wide, the alternative is that $T_1:=T\setminus T_0$ is nonspherical.
But then $U_0\subset T_0$ implies $T_1\subset T_0^\perp\subset U_0^\perp$, so
$U_0^\perp$ is nonspherical. 
  \end{proof}
  The next example illustrates that knowing that an edge is dual to a
    wall that cuts a wide parabolic subcomplex does not imply that
    that edge belongs to a wide parabolic subcomplex. This explains
    the case split $U$ spherical/nonspherical at the end of the proof
    of \fullref{main_for_combinatorial}. This phenomenon cannot
    happen in the right-angled case, since in RACGs distinct
    generators are never conjugate. 
  \begin{example}\label{ex:bad_wall}
    Consider the Coxeter diagram shown in
    \fullref{fig:coxeter-diagram}.
    Let $T:=\{q,r,s,t\}$ and $U:=\{u\}$.
    The subset $T$ is wide;
    $W_T\cong\dihedral_\infty\times\dihedral_\infty$.
    The generator $u$ is not contained in any wide subset, but it is
    conjugate to a generator $t$ that is: $tut=utu\implies ututu=t\in T$.
    The Davis
    complex, a portion of which is shown in
    \fullref{fig:davis-example}, is a ``tree of flats'' where three
    flats ``meet'' by containing nonconsecutive edges of a hexagon, a
    translate of $\Davis_{\{t,u\}}$.
    As shown in     \fullref{fig:davis-example}, a wall dual to a
    $u$--edge cuts the hexagon containing it, as well as the opposite
    $t$--edge, then branches into the flat containing that $t$--edge,
    so it cuts a wide parabolic subcomplex, but it also cuts an edge
    not contained in any wide parabolic subcomplex.

\begin{figure}[h!]
  \centering
  \begin{subfigure}[t]{0.35\textwidth}
    \centering
    \includegraphics[scale=2]{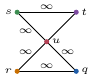}
    \caption{The Coxeter diagram.}
    \label{fig:coxeter-diagram}
  \end{subfigure}
  \hfill
  \begin{subfigure}[t]{0.6\textwidth}
    \centering
    \includegraphics[scale=2]{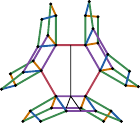}
    \caption{A portion of the Davis complex with part of one wall shown.}
    \label{fig:davis-example}
  \end{subfigure}

  \caption{The Coxeter diagram of \fullref{ex:bad_wall} and its Davis complex.}
  \label{fig:coxeter-davis}
\end{figure}
    
  \end{example}

%% file: hyperbolic_intro.tex
In \fullref{sec:fine_chain_metric} we use fine chains to define a metric $\dfine$ on $\NR^{(0)}$, which
we show is hyperbolic. Since it is defined in terms of walls, it also
gives a hyperbolic metric on $\Davis^{(0)}$.

A hyperbolic metric on $\Davis^{(0)}$ must somehow eliminate the known sources of non-hyperbolicity, the wide
parabolics.
In fact, in \fullref{lem:wide_parabolics_have_diameter_one} we observe that it collapses wide
parabolics to diameter 1 subsets.
This has the side effect, noted in
\fullref{prop:Davis_vertices_coarsely_dense_in_NR_dfine}, that it
collapses the Niblo-Reeves cube complex to within bounded distance of
the canonical copy of $\Davis^{(0)}$ in $\NR^{(0)}$.

Using \fullref{thm:morse_characterization} to relate Morse geodesics with fine chains
along them, we deduce that $(\Davis^{(0)},\dfine)$ is a \emph{Morse recognizing
space}, in \fullref{sec:morse_recognition}.

A common strategy for producing hyperbolic spaces is to cone off known
non-hyperbolic regions to make them have bounded diameter.
In \fullref{sec:coneoff} we construct such a  \emph{coned-off space}
$\coneoff$  by
coning off the wide parabolic subcomplexes of $\Davis$ and show that
it is $W$--equivariantly quasiisometric to $(\Davis^{(0)},\dfine)$.
This gives a more concrete construction of a hyperbolic Morse recognizing
space for $W$.

%% file: hyperbolic.tex
Recall, from \fullref{def:fine_chain}, that a chain
$\mathcal{V}:=\{V_i\}_{i\in I}$ is \emph{fine} if for all
$i\neq j$, or, equivalently, for all consecutive $i,j$, every
chain transverse to $\{V_i,V_j\}$ has length less than $2\caprace$.

\begin{definition}
  Define $\dfine\from \Davis^{(0)}\times\Davis^{(0)}\to\mathbb{R}$ by
  $\dfine(x,x):=0$ and for $x\neq y$:
  \[\dfine(x,y):=\max\{|\mathcal{V}|\mid \mathcal{V} \text{ is a fine
      chain separating $x$ and $y$}\}\]
  Make the same definition for $\dfine\from
  \NR^{(0)}\times\NR^{(0)}\to\mathbb{R}$.
\end{definition}

In this section we will show that $\dfine$ defines a hyperbolic metric
on $\NR^{(0)}$.
This is a variation of a construction of Genevois.
We further show that combinatorial geodesics are unparameterized
quasirulers that can be reparameterized to be rough geodesic
quasirulers, and that $\dfine$ restricts to a hyperbolic metric on
$\Davis^{(0)}$.

  \begin{lemma}\label{lem:fine_isometric_embedding}
 The fine chain metric $\dfine$ is a metric on $\Davis^{(0)}$ and on
 $\NR^{(0)}$.
 The canonical inclusion $\Davis^{(0)}\into\NR^{(0)}$ induces an
 isometric embedding $(\Davis^{(0)},\dfine)\into(\NR^{(0)},\dfine)$
 that takes the vertex sequence of a combinatorial geodesic in
 $\Davis$ to the vertex sequence of a combinatorial geodesic in $\NR$.
  \end{lemma}
  \begin{proof}
    A single wall is a fine chain, so $x\neq y\in\NR^{(0)}$ implies
    $\dfine(x,y)\geq 1$.
    
    Vertices of $\NR$ do not lie on walls, so for all $x$, $y$, $z$
    in $\NR^{(0)}$, if $\mathcal{V}$ is a
    chain realizing $\dfine(x,z)$ then
    $\mathcal{V}=\mathcal{V}'\sqcup\mathcal{V}''$ where $\mathcal{V}'$
    is an initial subchain separating $x$ from $y$ and $\mathcal{V}''$
    is a terminal subchain separating $y$ from $z$.
    Both $\mathcal{V}'$ and $\mathcal{V}''$ are fine, so
    $\dfine(x,z)=|\mathcal{V}|=|\mathcal{V}'|+|\mathcal{V}''|\leq\dfine(x,y)+\dfine(y,z)$.

    This shows that $\dfine$ is a metric on $\NR^{(0)}$. The proof for
    $\Davis^{(0)}$ is identical.

    Recall that a vertex of $\Davis^{(0)}$ defines a principal ultrafilter
on the halfspace system that for each wall picks the complementary
halfspace containing that vertex.
This gives a canonical inclusion $\iota\from\Davis^{(0)}\into\NR^{(0)}$ with the
additional property that a wall separates two vertices of
$\Davis^{(0)}$ if and only if the corresponding wall in $\NR$
separates the $\NR^{(0)}$ images of those vertices.
Thus, for $x,y\in\Davis^{(0)}$ the walls separating them correspond to
the walls separating their images in $\NR^{(0)}$, and fine chains are
preserved by this identification, so $\dfine(x,y)=\dfine(\iota(x),\iota(y))$.

The vertex sequence $v_0$, \dots,$v_n$ of a combinatorial geodesic, in
either $\Davis$ or $\NR$, is characterized by the properties that
there is a unique wall separating $v_i$ and $v_{i+1}$ for each $0\leq
i<n$ and for every wall $M$, if for some $j$ we have $v_j\in M^-$ and
$v_{j+1}\in M^+$ then for all $i\leq j<k$ we have $v_i\in M^{-}$ and
$v_k\in M^+$.  
  \end{proof}

  \begin{lemma}[{cf \cite[Lemma~6.54]{Gen19}}]\label{lem:low_slack_on_NR_geodesics}
Combinatorial geodesics in $\NR$ are unparameterized
$\tau$--quasirulers in $(\NR^{(0)},\dfine)$, for constant $\tau$
depending only on $\caprace$.
  \end{lemma}
  \begin{proof}
    Let $x$, $y$, and $z$ be vertices on a combinatorial geodesic of
    $\NR$ such that $y$ occurs between $x$ and $z$.
     Let $\mathcal{V}:=\{V_1,\dots,V_m\}$ be a longest fine chain
    separating $x$ and $y$ and let
    $\mathcal{H}:=\{H_1,\dots,H_n\}$ be a longest fine chain
    separating $y$ and $z$.
    Assume that indices and halfspaces are chosen so that for all $i$
    and $j$,  $y\in
    V_i^-\subset V_{i+1}^-$ and $y\in H_j^-\subset H_{j+1}^-$.
    
    Suppose for some $i_0$ and $j_0$ that $V_{i_0}\cap H_{j_0}=\emptyset$.
    Since $y\in V_{i_0}^-\cap H_{j_0}^-$, we have $V_{i_0}^+\cap
    H_{j_0}^+ =\emptyset$.
    Thus, for all $i\geq i_0$ and
    $j\geq j_0$ we have $V_i$ and $H_j$ are disjoint, since $V_i\subset
    V_{i_0}^+$ and $H_j\subset H_{j_0}^+$.
    Conversely, if $V_{i_0}\transverse H_{j_0}$ then $V_i\transverse
    H_j$ for all $i\leq i_0$ and all $j\leq j_0$.

    If $m\geq 3$ and $n\geq 2\caprace$ then set
    $\mathcal{K}:=\{V_m,\dots,V_3,H_{2\caprace},\dots,H_n\}$.
    We claim $\mathcal{K}$ is a fine chain.
    This follows because if $V_2\transverse
    H_{2\caprace}$ then
    $\{V_1,V_2\}\transverse\{H_1,\dots,H_{2\caprace}\}$,
    contradicting fineness of $\mathcal{V}$, so $V_2$ separates $V_3$
    from $H_{2\caprace}$. Moreover, the only consecutive pair of
    walls at which $\mathcal{K}$ could fail to be fine is
    $\{V_3,H_{2\caprace}\}$, but walls transverse to
    $\{V_3,H_{2\caprace}\}$ are also transverse to $\{V_2,V_3\}$, so
    fineness of $\mathcal{V}$ forces fineness of $\mathcal{K}$.
   
     The hypothesis that $y$ is between $x$ and $z$ in $\NR$ implies
    that any wall separating $y$ from one of $x$ or $z$ in $\NR$ also separates
    $x$ and $z$, so  $\mathcal{K}$ is a fine chain separating $x$ and
    $z$, which gives a bound:
    \[\dfine(x,z)\geq
      |\mathcal{K}|=|\mathcal{V}|-2+|\mathcal{H}|-(2\caprace-1)=\dfine(x,y)+\dfine(y,z)-(2\caprace+1)\]
    If $\min\{m,n\}<2\caprace$ then take $\mathcal{K}$ to be the
    longer of $\mathcal{V}$ and $\mathcal{H}$.  Then
    $\dfine(x,z)\geq|\mathcal{K}|\geq \dfine(x,y)+\dfine(y,z)-(2\caprace-1)$.
    Thus, it suffices to take $\tau:=2\caprace+1$.

    The bounded jump condition is trivial, since combinatorial
    geodesics are continuous. 
  \end{proof}

  The next result says the vertex sequence of a combinatorial geodesic
  can be parameterized as a
  rough geodesic in $(\NR^{(0)},\dfine)$.
  The point is that along a combinatorial geodesic the
  $\dfine$--distance from the initial vertex is nondecreasing, but it can be constant
  over long subintervals. However, since it is discrete valued, we can
  reparameterize so that those constant valued intervals are each traversed in time 1. 
  \begin{lemma}\label{lem:combinatorial_geodesics_reparameterize_to_rough_geodesics}
    Let $\alpha\from [0,n]\to \NR$ be a combinatorial geodesic, and let
    $T:=\dfine(\alpha(0),\alpha(n))$.
    There exist numbers $0=b_0<b_1<\cdots<b_n=T$ such that the map
    $\beta\from [0,T]\to\NR^{(0)}$ defined by sending $T$ to
    $\alpha(n)$ and $t<T$ to
    $\alpha(i)$ when $b_i\leq t <b_{i+1}$ is a $(\tau+2)$--rough
    geodesic in $(\NR^{(0)},\dfine)$ that traverses the vertices of the
    image of $\alpha$ in the same order as $\alpha$, where $\tau$ is the quasiruler constant of \fullref{lem:low_slack_on_NR_geodesics}.
  \end{lemma}
  \begin{proof}
    For each $i\in [0,n]\cap\mathbb{Z}$, define:
    \[\operatorname{FineProgress}(i)=\FP(i):=\dfine(\alpha(0),\alpha(i))\]
Since $\alpha$ is a combinatorial geodesic, for all $i\leq j$ in
$[0,n]\cap\mathbb{Z}$, every wall separating $\alpha(0)$ and
$\alpha(i)$ also separates $\alpha(0)$ and $\alpha(j)$, and the only
wall separating $\alpha(0)$ and $\alpha(i+1)$ that does not also
separate $\alpha(0)$ and $\alpha(i)$ is the wall transverse to the
edge between $\alpha(i)$ and $\alpha(i+1)$. 
Thus, $\FP(i+1)-\FP(i)\in\{0,1\}$.

\fullref{lem:low_slack_on_NR_geodesics} says $\alpha$ is an
unparameterized $\tau$--quasiruler. Rearranging the almost degeneracy
condition gives that for integers $i\leq j$:
\begin{equation}
  \label{eq:26}
  \FP(j)-\FP(i)\leq \dfine(\alpha(i),\alpha(j))\leq \FP(j)-\FP(i)+\tau
\end{equation}

Define $b_0:=0$.
For each integer $D$ in $[1,T]$ consider the indices $\{i\mid \FP(i)=D\}$, which is
an integer interval $[j,j+k]$ for some $j$ and $k$, by 
monotonicity of $\FP$. 
For $i\in [j,j+k]$ define $b_i:=D+\frac{i-j-k}{k+1}$, so that
$b_j=D-1+\frac{1}{k+1}$ and $b_{j+k}=D$, with the other values $b_i$
linearly interpolating between these two values.
By construction, $i\mapsto b_i$ is strictly increasing and for all $i$
we have:
\begin{equation}
  \label{eq:27}
0\leq \FP(i)-b_i<1  
\end{equation}
For integers $i\leq j$ in $[0,n]$, relations \eqref{eq:26} and \eqref{eq:27} give:
\[b_j-b_i-1\leq \dfine(\beta(b_i),\beta(b_j))\leq b_j-b_i+\tau+1\]
Since $0<b_{i+1}-b_i\leq 1$, if $x\leq y$ with $b_i\leq x< b_{i+1}$ and
$b_j\leq y<b_{j+1}$ then
$\dfine(\beta(x),\beta(y))=\dfine(\beta(b_i),\beta(b_j))$, by
definition of $\beta$, and $y-x\in b_j-b_i\pm 1$, so:
\[y-x-2\leq b_j-b_i-1\leq \dfine(\beta(x),\beta(y))\leq
  b_j-b_i+\tau+1\leq y-x+\tau+2\qedhere\]
  \end{proof}

  \begin{corollary}\label{cor:combinatorial_geodesics_rough_and_quasiruler}
For every combinatorial geodesic segment $\alpha$ in $\NR$ there is a parameterized
$(1,\tau+2,\tau)$--quasiruler in $(\NR^{(0)},\dfine)$ that has the
same vertex sequence as $\alpha$.
The same is true for combinatorial geodesic segments in $\Davis$.
\end{corollary}
\begin{proof}
 In $\NR$, this is a combination of \fullref{lem:low_slack_on_NR_geodesics} and
 \fullref{lem:combinatorial_geodesics_reparameterize_to_rough_geodesics},
 since the almost degeneracy of the triangle inequality does not
 depend on the parameterization. 

The same is true in $\Davis$ because
\fullref{lem:fine_isometric_embedding} says the inclusion of
$\Davis^{(0)}$ into $\NR^{(0)}$ is a $\dfine$--isometric embedding
that takes combinatorial geodesics to combinatorial geodesics.
\end{proof}

The following is essentially Genevois's proof that his $\delta_L$ metric on a
CAT(0) cube complex defines a non-geodesic hyperbolic metric,
\cite[Proposition~6.53]{Gen19}.
The difference is that his metric is defined in terms of two disjoint
walls being $L$--well-separated if every
\emph{set of walls} transverse to both of them and containing no facing
triple has cardinality at most $L$, whereas we are saying two disjoint
walls make a fine chain if every \emph{chain} transverse to both of them contains
at most $2\caprace -1$ walls.
A pair of $(2\caprace -1)$--well-separated walls makes a fine
chain, and a two-wall fine chain gives a
$(2\caprace-1)\cdot\dim\NR$--well-separated pair, so these two notions
are related, but it turns out to be fairly easy to just rerun the
hyperbolicity  
argument using $\dfine$, as opposed to trying to compare $\dfine$ to
$\delta_{2\caprace-1}$ and $\delta_{(2\caprace-1)\cdot\dim\NR}$.
  \begin{proposition}[{cf \cite[Lemma~2.9 and Proposition~6.53]{Gen19}}]\label{lem:four_point_hyperbolicity}
    $(\NR^{(0)},\dfine)$ is $(10\caprace+3)$--hyperbolic. 
  \end{proposition}
  \begin{proof}
        Let $x_0$, $x_1$, $x_2$, and $x_3$ be vertices of $\NR$. 
        Since $\NR$ is a CAT(0) cube complex, its 1--skeleton is a
        median graph. 
    For each $i$, let $y_i$ be the median of the triple
    $(x_{i+1},x_{i+2},x_{i+3})$, and let $z_i$ be the median of the
    triple 
    $(y_{i+1},y_{i+2},y_{i+3})$, with subscripts mod 4.
    \fullref{fig:four_point_median} shows the idealized version when
    all of these points are distinct. 
    \begin{figure}[h]
      \centering
      \includegraphics{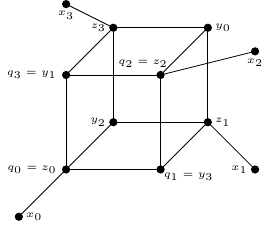}
      \caption{Setup for \fullref{lem:four_point_hyperbolicity}
        pictured as the
      free median algebra on $\{x_0, x_1, x_2, x_3\}$.}
      \label{fig:four_point_median}
    \end{figure}
    A key fact from median algebras is that sides of the
    median cube in \fullref{fig:four_point_median} have a parallelism
    relation such that a wall that separates the endpoints of one edge
    in the cube also separates the endpoints of the 3 edges parallel
    to it.

    The intuition for the hyperbolicity proof is to show that in any
    configuration as in \fullref{fig:four_point_median} at least two
    of the dimensions of the cube are small, so the figure is
    uniformly tree-like.

   Formally, there is a four-point $\delta$--hyperbolicity condition \cite[Definition~III.H.1.20]{BriHae99}:
       \begin{equation}
         \label{eq:four_point_hyperbolicity}
         \begin{split}
           \dfine(x_0,x_2)+&\dfine(x_1,x_3)\leq\\
           &\max\{\dfine(x_0,x_1)+\dfine(x_2,x_3), \dfine(x_0,x_3)+\dfine(x_1,x_2)\}+2\delta
         \end{split}\tag{$Q(\delta)$}
       \end{equation}
We will establish condition $Q(10\caprace+3)$.

It will be enough in the argument to focus on one face of the cube.
We arbitrarily choose the front face in the figure, and define its
corners to be $q_0:=z_0$, $q_1:=y_3$, $q_2:=z_2$, and $q_3:=y_1$.
The parallelism relation implies $a:=\dfine(q_0,q_1)=\dfine(q_2,q_3)$ and $b:=\dfine(q_0,q_3)=\dfine(q_1,q_2)$.
        Furthermore, every wall separating $q_0$ and $q_1$ crosses
        every wall separating $q_0$ and $q_3$.
        In particular, if $\mathcal{V}$ is a fine chain realizing
        $a=\dfine(q_0,q_1)$ and $\mathcal{H}$ is a fine chain
        realizing $b=\dfine(q_0,q_3)$ then
        $\mathcal{V}\transverse\mathcal{H}$, so fineness demands:
        \begin{equation}
          \label{eq:12}
         \min\{a,b\}=\min\{|\mathcal{V}|,|\mathcal{H}|\}\leq
         1\quad\text{or}\quad\max\{a,b\}=\max\{|\mathcal{V}|,|\mathcal{H}|\}\leq
         2\caprace-1
       \end{equation}
    
Let $[p_0,p_1]$ denote the \emph{interval} between $p_0$ and $p_1$, which is
all vertices that lie on combinatorial geodesics between $p_0$ and $p_1$.
From the median point definitions we have, for each $i$:
    \[[x_i,q_i]\cup[q_i,q_{i+1}]\cup[q_{i+1},x_{i+1}]\subset
      [x_i,x_{i+1}]\]
    Applying \fullref{lem:low_slack_on_NR_geodesics} twice:
    \begin{align*}
      \dfine(x_3,x_0)&\geq \dfine(x_3,q_3)+\dfine(q_3,x_0)-\tau\\
      &\geq
        \dfine(x_3,q_3)+\dfine(q_3,q_0)+\dfine(q_0,x_0)-2\tau\\
      &=\dfine(x_0,q_0)+\dfine(x_3,q_3)+b-2\tau
    \end{align*}
    A similar estimate holds for $\dfine(x_1,x_2)$. Combining them gives:
    \begin{equation}
      \label{eq:1}
      \dfine(x_3,x_0)+\dfine(x_1,x_2)\geq
      2b-4\tau +\sum_{i=0}^3\dfine(x_i,q_i)
    \end{equation}
    A similar argument gives:
    \begin{equation}
      \label{eq:9}
       \dfine(x_0,x_1)+\dfine(x_2,x_3)\geq
      2a-4\tau+\sum_{i=0}^3\dfine(x_i,q_i)
    \end{equation}
    On the other hand, the triangle inequality gives:
    \begin{equation}
      \label{eq:4}
        \dfine(x_0,x_1)+\dfine(x_2,x_3)\leq 2a+\sum_{i=0}^3\dfine(x_i,q_i)
    \end{equation}
    Suppose, without loss of generality, that:
    \begin{equation}
      \label{eq:15}
      \dfine(x_0,x_3)+\dfine(x_1,x_2)\leq\dfine(x_0,x_1)+\dfine(x_2,x_3)
    \end{equation}
        Then
        \eqref{eq:1}, \eqref{eq:4}, and \eqref{eq:15} combine to give:
        \begin{equation}
          \label{eq:5}
          b-a\leq   2\tau
        \end{equation}

        Now, from the triangle inequality:
        \[\dfine(x_0,x_2)+\dfine(x_1,x_3)\leq 2a+2b+\sum_{i=0}^3\dfine(x_i,q_i)\]
        Combined with \eqref{eq:1}:
        \begin{equation}
          \label{eq:10}
          \dfine(x_0,x_2)+\dfine(x_1,x_3)\leq \dfine(x_0,x_3)+\dfine(x_1,x_2)+4\tau+2a
        \end{equation}
        Similarly, using \eqref{eq:9}:
        \begin{equation}
          \label{eq:11}
           \dfine(x_0,x_2)+\dfine(x_1,x_3)\leq \dfine(x_0,x_1)+\dfine(x_2,x_3)+4\tau+2b
        \end{equation}

Recall from \eqref{eq:12} that if  $b\geq 2$ then $a\leq 2\caprace-1$, so
\eqref{eq:5} says $b\leq 2\tau+2\caprace-1$, whereas if $b<2$ then 
$b\leq 2\tau+2\caprace-1$ is still true, so \eqref{eq:11} gives:
\begin{equation}
  \label{eq:13}
  \dfine(x_0,x_2)+\dfine(x_1,x_3)\leq \dfine(x_0,x_1)+\dfine(x_2,x_3)+8\tau+4\caprace-2
\end{equation}

The assumption \eqref{eq:15} says that \eqref{eq:13} realizes the
maximum appearing in \ref{eq:four_point_hyperbolicity}, so \eqref{eq:13}  establishes
\ref{eq:four_point_hyperbolicity} with $\delta:=4\tau+2\caprace-1$.
Using the value $\tau:=2\caprace+1$ in the proof of \fullref{lem:low_slack_on_NR_geodesics}, this gives $\delta=10\caprace+3$.
\end{proof}

\begin{theorem}\label{thm:dfine_hyp_qruled_on_Davis}
  $(\Davis^{(0)},\dfine)$ is a quasiruled hyperbolic space. 
\end{theorem}
\begin{proof}
  By \fullref{lem:fine_isometric_embedding}, 
  the canonical inclusion $\Davis^{(0)}\into\NR^{(0)}$ is a $\dfine$--isometric
  embedding.
  Thus, the 4--point hyperbolicity condition in the proof of
  \fullref{lem:four_point_hyperbolicity} also applies to
  $(\Davis^{(0)},\dfine)$.
  Quasiruledness is established by
  \fullref{cor:combinatorial_geodesics_rough_and_quasiruler}.
\end{proof}
 
\begin{lemma}\label{lem:wide_parabolics_have_diameter_one}
    If $\Upsilon$ is a wide parabolic subcomplex of $\Davis$ then
    $\Upsilon^{(0)}$ has diameter 1 in $(\Davis^{(0)},\dfine)$.
  \end{lemma}
  \begin{proof}
    Recall from the analysis in the proof of
    \fullref{prop:near_flat_implies_long_grids} that fine chains
     in a wide Coxeter group are single walls, since every pair of
     disjoint walls is mutually transverse to a biinfinte chain.
    It follows that fine chains that cut a wide parabolic subcomplex
    of $\Davis$ consist of single walls. 
  \end{proof}

  Recall that the $W$--action on the Niblo-Reeves cube complex $\NR$
  is not cocompact if and only if $W$ contains irreducible higher rank
  affine parabolic subgroups, \cite{CapMuh05}.
  It turns out that since the $\dfine$ metric collapses such
  parabolics, it collapses the difference between $\Davis^{(0)}$ and
  $\NR^{(0)}$:
  \begin{proposition}\label{prop:Davis_vertices_coarsely_dense_in_NR_dfine}
   The vertices of $\Davis^{(0)}$ are coarsely dense in  $(\NR^{(0)},\dfine)$.
\end{proposition}
\begin{proof}
  We identify $\Davis^{(0)}$ with its canonical image in $\NR^{(0)}$.
For $x\in\NR^{(0)}$, choose $v_0\in \Davis^{(0)}$ minimizing the
combinatorial distance in $\NR^{(1)}$ from $x$ to $\Davis^{(0)}$.

We first recall that $x$ lies in the cubical chamber containing $v_0$,
as in \cite[Section~7.5]{Cap06}:  let $\Psi(v_0)$ be the set of
halfspaces of $\NR$ that contain $v_0$ but not one of its neighbors in $\Davis^{(0)}$.
If $x\notin\psi\in\Psi(v_0)$, then there exists a combinatorial geodesic in $\NR$ from $x$ to $v_0$ that crosses the wall $\partial\psi$ in its last edge.
Let $v_1$ be the neighbor of $v_0$ in $\Davis^{(0)}$ not contained in $\psi$.
Then $v_1$ sits in the interior of a combinatorial geodesic from $x$
to $v_0$, so it is a vertex of $\Davis^{(0)}$ that is closer to $v$
than $v_0$ is, contradicting the choice of $v_0$.
Thus, $x$ belongs to the cubical chamber of $v_0$.

Let $\mathcal{F}$ be a longest fine chain separating $v_0$ from $x$.
The walls of $\mathcal{F}$ are pairwise disjoint and contained in the set of walls $\walls(v_0,x)$ separating $v_0$ from $x$.

By \cite[Proposition~24]{Cap06} there is a constant $K=K(W,S)$ such
that for any set of pairwise disjoint walls $\mathcal{M}$ contained in
$\walls(v_0,x)$ with $|\mathcal{M}|>K$, we have
$W(\mathcal{M})\cong\dihedral_\infty$ and $\pc(\mathcal{M})$ is
irreducible higher rank affine.

Thus, $|\mathcal{F}|>K$ implies $\pc(\mathcal{F})$ is irreducible
higher rank affine, which implies $|\mathcal{F}|=1$, which is a
contradiction, since $K$ is clearly at least 1.

Conclude that $\dfine(x,v_0)=|\mathcal{F}|\leq K$. 
\end{proof}

%% file: morse_recognition.tex

In \fullref{sec:fine_chain_morse_recognizing} we show that
quasigeodesics in $(\Davis,\dcat)$ are Morse if and only if they are
(parameterized) quasigeodesics in $(\Davis^{(0)},\dfine)$.
In \fullref{sec:recognition_equivalence} we discuss related notions in
the literature.
In \fullref{sec:morse_parabolics} we characterize Morse and stable parabolic subgroups. 

\subsubsection{The fine chain metric is Morse
  recognizing}\label{sec:fine_chain_morse_recognizing}
For the next theorem, choose a map $\Pi\from \Davis\to\Davis^{(0)}$
sending each point in $\Davis$ to a vertex contained in a common
closed chamber.
Every chamber is a $W$--translate of the fundamental chamber, which is
compact, and every chamber contains one vertex, so $\Pi$ is the
identity map on $\Davis^{(0)}$, and is coarsely equivalent to the
identity map on $\Davis$.
For all $x\in\Davis$, no wall
separates $x$ and $\Pi(x)$. 
\begin{theorem}\label{thm:Morse_recognizing_general}
    A quasigeodesic $\gamma$ is Morse in $(\Davis,\dcat)$ if and
    only if $\Pi\circ\gamma$ is quasigeodesic in $(\Davis^{(0)},\dfine)$.
    More specifically, since $\Davis$ is CAT(0), Morse is
    equivalent to strong contraction, by
    \fullref{morse_divergent_contracting_in_cat_zero}, and we show:
    \begin{enumerate}
    \item For all $\lambda$, $\epsilon$, $\chi$ there exist $\lambda'$
      and $\epsilon'$ such that for every $\chi$--strongly contracting
      $(\lambda,\epsilon)$--quasigeodesic $\gamma\from
      I\to (\Davis,\dcat)$ the map $\Pi\circ\gamma\from I\to
      (\Davis^{(0)},\dfine)$ is a $(\lambda',\epsilon')$--quasigeodesic.\label{item:morse_implies_fine_qgeod}
      \item  For all $\lambda$, $\epsilon$, $\lambda'$, $\epsilon'$
        there exists $\chi$ such that for every 
        $(\lambda,\epsilon)$--quasigeodesic 
        $\gamma\from I\to (\Davis,\dcat)$ such that $\Pi\circ\gamma\from I\to
        (\Davis^{(0)},\dfine)$ is $(\lambda',\epsilon')$--quasigeodesic,
        $\gamma$ is $\chi$--strongly contracting in $(\Davis,\dcat)$.\label{item:fine_qgeod_imlies_morse}
    \end{enumerate}
  \end{theorem}
  \begin{proof}
    It suffices to assume the interval $I\subset\mathbb{R}$ is
    compact, because if $\gamma$ restricted to compact subintervals is
  uniformly contracting or uniformly
    quasigeodesic then so is $\gamma$ on all of $I$. 

    First we prove Item~\eqref{item:morse_implies_fine_qgeod}.
    By \fullref{lem:robust} there are $\chi'$ and $D$ depending on $\chi$,
    $\lambda$, and $\epsilon$ such that the geodesic $\alpha$ between
    the endpoints of $\gamma(I)$ is $\chi'$--strongly contracting and
    has Hausdorff distance at most $D$ from $\gamma(I)$. 
   
Suppose we know  Item~\eqref{item:morse_implies_fine_qgeod} for
$\dcat$--geodesics, so can show that $\Pi\circ\alpha$ is
$(\lambda'',\epsilon'')$--quasigeodesic in $(\Davis^{(0)},\dfine)$,
for some $\lambda''$ and $\epsilon''$ depending on $\chi'$, hence on
$\chi$, $\lambda$, and $\epsilon$.
Let $\phi$ be a map from the domain of $\gamma$ to the domain of
$\alpha$ such that $\dcat(\gamma(t),\alpha(\phi(t)))\leq D$ for all
$t\in I$.
Since there is no wall separating $x$ from $\Pi(x)$, any fine chain separating $\Pi(x)$ from $\Pi(y)$ also separates $x$ from $y$, so $\dfine(\Pi(x),\Pi(y)) \leq\dwall(x,y)\ceq\dcat(x,y)$.
Thus, $\dfine(\Pi\circ\gamma(t),\Pi\circ\alpha(\phi(t)))\leq D'$, for some uniform $D'$.

Since $\gamma$ is a $\dcat$--quasigeodesic, for $s<t$ in $I$:
\[\dfine(\Pi\circ\gamma(s),\Pi\circ\gamma(t))\cleq\dcat(\gamma(s),\gamma(t))\cleq t-s\]
Conversely, we need the $\dfine$--quasigeodesic lower bound for
$\Pi\circ\gamma$.
We have the $\dcat$--quasigeodesic lower bound:
\[
  \frac{1}{\lambda}(t-s)-\epsilon\leq\dcat(\gamma(s),\gamma(t))\leq
  \dcat(\alpha(\phi(s)),\alpha(\phi(t)))+2D=|\phi(t)-\phi(s)|+2D
\]
Thus:
\begin{equation}
  \label{eq:29}
  |\phi(t)-\phi(s)|\geq \frac{1}{\lambda}(t-s)-\epsilon-2D
\end{equation}
Since we supposed that $\Pi\circ\alpha$ is a $\dfine$--quasigeodesic:
\begin{equation}
  \label{eq:2}
  \dfine(\Pi\circ\alpha(\phi(s)),\Pi\circ\alpha(\phi(t)))\geq \frac{1}{\lambda''}|\phi(t)-\phi(s)|-\epsilon''
\end{equation}
Combining \eqref{eq:29} and \eqref{eq:2} gives:
\begin{align*}
  \dfine(\Pi\circ\gamma(s),\Pi\circ\gamma(t))&\geq\dfine(\Pi\circ\alpha(\phi(s)),\Pi\circ\alpha(\phi(t)))-2D'\\
                             &\geq \frac{1}{\lambda''}|\phi(t)-\phi(s)|-\epsilon''-2D'\\
  &\geq \frac{1}{\lambda\lambda''}(t-s)-\frac{2D+\epsilon}{\lambda''}-(2D'+\epsilon'')  
\end{align*}
This shows that $\Pi\circ\gamma$ is a $\dfine$--quasigeodesic with constants
depending only on $\lambda$, $\epsilon$, and $\chi$.
Thus, it suffices to prove Item~\eqref{item:morse_implies_fine_qgeod}
in the case that $\gamma$ is a $\dcat$--geodesic.
One direction is easy, and is independent of $\chi$:
    \[\dfine(\Pi\circ\gamma(t_0),\Pi\circ\gamma(t_1))\cleq\dcat(\Pi\circ\gamma(t_0),\Pi\circ\gamma(t_1))\cleq\dcat(\gamma(t_0),\gamma(t_1))=|t_1-t_0|\] 
In the other direction, since $\dcat$--geodesics have connected wall preimages,
\fullref{morse_divergent_contracting_in_cat_zero} and the equivalence 
\fullref{thm:morse_characterization}~\ref{main:morse}$\iff$\ref{main:fine_chain},
imply that there exist $L$ and $\rho>0$ depending on $\chi$ such that for every
    subinterval $I'=[t_0,t_1]\subset I$
     of length at least $L$ there is a fine
     chain $\mathcal{V}$ separating the endpoints of $\gamma(I')$ with
     $|\mathcal{V}|/|I'|\geq\rho$.
      Since no wall
     separates $\gamma(t_i)$ from $\Pi(\gamma(t_i))$, the fine chain
     $\mathcal{V}$ separates $\Pi(\gamma(t_0))$ and
     $\Pi(\gamma(t_1))$, so
     $\dfine(\Pi\circ\gamma(t_0),\Pi\circ\gamma(t_1))\geq
     |\mathcal{V}|\geq\rho |t_1-t_0|$.
     If $|t_1-t_0|<L$ then $\rho|t_1-t_0|-\rho L<0$, so for all
     $t_0,t_1\in I$ we have:
     \[\dfine(\Pi\circ\gamma(t_0),\Pi\circ\gamma(t_1))\geq\rho
       |t_1-t_0|-\rho L\]

     Now we prove Item~\eqref{item:fine_qgeod_imlies_morse}.
     In view of \fullref{lem:robust} and connect-the-dots trick, it suffices to assume that $\gamma$ is continuous. 

Let $J:=[a,b]\subset I$.
Let $\mathcal{V}=\{V_i\}_{0\leq i\leq n+1}$ be a longest fine chain
separating $\Pi\circ\gamma(a)$ and $\Pi\circ\gamma(b)$, oriented from $\Pi\circ\gamma(a)$ to
$\Pi\circ\gamma(b)$. 
No wall separates $\gamma(a)$ from $\Pi(\gamma(a))$ or separates
$\gamma(b)$ from $\Pi(\gamma(b))$, but conceivably $\gamma(a)$ could
be a point on $V_0$ and/or $\gamma(b)$ could be a point on $V_{n+1}$.
Even so, $\mathcal{V}':=\{V_i\}_{1\leq i\leq n}$ is a fine chain separating
  $\gamma(a)$ from $\gamma(b)$. 
Combined with the hypothesis that $\Pi\circ\gamma$ is $\dfine$--quasigeodesic:
\[\frac{|\mathcal{V}'|}{|J|}=\frac{|\mathcal{V}|-2}{|J|}=\frac{\dfine(\Pi\circ\gamma(a),\Pi\circ\gamma(b))-2}{|J|}\geq\frac{|J|/\lambda'-\epsilon'-2}{|J|}=\frac{1}{2\lambda'}+\left(\frac{1}{2\lambda'}-\frac{\epsilon'+2}{|J|}\right)\]
Thus, there exists a fine
chain separating $\gamma(a)$ and $\gamma(b)$ of density at least
$\rho:=\frac{1}{2\lambda'}$, provided:
\begin{equation}
  \label{eq:16}
 |J|\geq 2\lambda'(\epsilon'+2)
\end{equation}

\fullref{separated_chain_implies_divergent} says $\gamma$ has at least
quadratic divergence, with gauge depending on $\lambda$, $\epsilon$,
$\caprace$, the density $1/2\lambda'$, and the lower bound of
\eqref{eq:16}, hence on $\lambda$, $\epsilon$, $\lambda'$, and $\epsilon'$.
\fullref{morse_divergent_contracting_in_cat_zero} says this divergence
gauge determines an upper bound on a strong contraction constant, so
$\gamma$ is $\chi$--strongly contracting for some $\chi$ bounded above
in terms of $\lambda$, $\epsilon$, $\lambda'$, and $\epsilon'$.
\end{proof}

\subsubsection{Morse recognizing is equivalent to stability
  recognizing}\label{sec:recognition_equivalence}
The purpose of this subsection is to give definitions of ``Morse
recognizing'' and ``stability recognizing'', compare them to
related notions that appear in the literature, and show that these two
recognition properties are equivalent.

\begin{definition}\label{def:morse_recognizing}
  $\phi\from X\to Y$ is \emph{Morse recognizing} if:
  \begin{itemize}
  \item For all $\lambda$, $\epsilon$, $\mu$ there exist $\lambda'$
    and $\epsilon'$ such that if $\gamma\from I\to X$ is a
    $\mu$--Morse $(\lambda,\epsilon)$--quasigeodesic then
    $\phi\circ\gamma\from I\to Y$ is a
    $(\lambda',\epsilon')$--quasigeodesic.
    \item For all $\lambda$, $\epsilon$, $\lambda'$, $\epsilon'$ there
      exists $\mu$ such that if $\gamma\from I\to X$ is a
      $(\lambda,\epsilon)$--quasigeodesic and $\phi\circ\gamma\from
      I\to Y$ is a $(\lambda',\epsilon')$--quasigeodesic then
      $\gamma\from I\to X$ is $\mu$--Morse. 
  \end{itemize}
\end{definition}

\begin{definition}
  A subset $Z\subset X$ is \emph{$(\lambda,\epsilon)$--quasigeodesically connected} if for
  every pair of points $z_0$ and $z_1$ in $Z$ there exists a
  $(\lambda,\epsilon)$--quasigeodesic $\gamma\from I\to X$ whose image
  is in $Z$ and whose
  endpoints are $z_0$ and $z_1$.
\end{definition}

\begin{definition}[{cf \cite{DurTay15,CorHum17,AbbBehDur21}\footnote{We are treating
      stability as a property of a subspace $Z$ of $X$ with the
      subspace metric. It is not hard to show this is equivalent to
      the standard definition in which $Z$ is considered as a separate
    metric space and stability is a property of a quasiisometric
    embedding of $Z$ into $X$.}}]\label{def:stable_subspace}
  A subset $Z\subset X$ is \emph{$(\lambda,\epsilon,\mu)$--stable} if for
  every pair of points $z_0$ and $z_1$ in $Z$ there exists a
  $\mu$--Morse $(\lambda,\epsilon)$--quasigeodesic $\gamma\from I\to
  X$ whose image is in $Z$ and has
  endpoints $z_0$ and $z_1$.
\end{definition}

\begin{definition}\label{def:stable_subspace_recognizing}
  $\phi\from X\to Y$ is \emph{stability recognizing} if:
  \begin{itemize}
  \item For all $\lambda$, $\epsilon$, $\mu$ there exist $\lambda'$
    and $\epsilon'$ such that for every
    $(\lambda,\epsilon,\mu)$--stable $Z\subset X$ the map $\phi|_Z$ is
    a $(\lambda',\epsilon')$--quasiisometric embedding of
    $(Z,d_X|_Z)$ into $Y$.
    \item For all $\lambda$, $\epsilon$, $\lambda'$, $\epsilon'$ there
      exists $\mu$ such that for every
      $(\lambda,\epsilon)$--quasigeodesically connected set $Z\subset
      X$ if $\phi|_Z$ is a $(\lambda',\epsilon')$--quasiisometric
      embedding of $(Z,d_X|_Z)$ into $Y$ then $Z$ is
      $(\lambda,\epsilon,\mu)$--stable in $X$.
  \end{itemize}
\end{definition}

\begin{proposition}\label{prop:recognizing_is_recognizing}
$\phi\from X\to Y$ is Morse recognizing if and only if it is stability recognizing.  
\end{proposition}

Usually one also requires that the space $Y$ is hyperbolic.\footnote{A
  notable early example of a recognition-type theorem where the target
  space is not hyperbolic is the characterization of convex cocompact subgroups of
  the mapping class group in terms of having quasiconvex orbits in
  Teichm\"uller space \cite{FarMos02}.}
In that case the Morse property, stability, and quasiconvexity are
equivalent, so the recognition properties can be restated.
When $Y$ is hyperbolic, ``$\phi\from X\to Y$ is Morse recognizing'' is
equivalent to ``a quasigeodesic $\gamma$ in $X$ is Morse if and only if
$\phi\circ\gamma$ is a Morse quasigeodesic in $Y$, with quantitative
control on the quasigeodesic constants and Morse gauges in both
directions.''
Similarly, ``$\phi\from X\to Y$ is stability recognizing'' is equivalent
to ``a quasigeodesically
connected subspace $Z$ is stable in $X$ if and
only if $\phi|_Z$ is a quasiisometric embedding with
quasiconvex image, with quantitative control between the stability
parameters and the quasiisometric embedding/quasiconvexity
parameters, in both directions.'' 
We do not need $Y$ to be hyperbolic for \fullref{prop:recognizing_is_recognizing}.

Russell, Spriano, and Tran \cite[Definition~4.17]{RusSprTra22} defined
a \emph{Morse detectable space} $X$ to be one such that there exists a
Morse recognizing map $\phi\from X \to Y$ as above, with the
additional conditions that $Y$ is hyperbolic and $\phi$ is coarse
Lipschitz.
This gave a name to the property characterizing contracting
quasigeodesics in hierarchically hyperbolic spaces described by
Abbott, Behrstock, and Durham \cite{AbbBehDur21}.
This terminology has been picked up in subsequent work related to the
Morse-local-global property \cite{AbbZbi25,Per26}.

Durham and Taylor \cite{DurTay15} defined a notion of \emph{stable
  subgroup} abstracting the properties of what were called convex
cocompact subgroups of mapping class groups of hyperbolic surfaces \cite{FarMos02}.
A subgroup $H$ of a finitely generated
group $G$ is stable\footnote{Quasigeodesic connectivity of $H$ already
forces it to be finitely generated and undistorted in $G$.} if it is stable as a subspace of a Cayley graph of
$G$. 
Kent and Leininger \cite{KenLei08} (see also \cite{Ham05}) proved what we would now call a
stable subgroup recognition theorem,  showing that convex cocompact
subgroups of the mapping class group are those for which the orbit map
into the curve graph is a quasiisometric embedding of the subgroup.
Analogous stable subgroup recognition theorems exist for groups
hyperbolic relative to peripherals with linear divergence
\cite{AouDurTay17}, right-angled Artin groups \cite{KobManTay17},
hierarchically hyperbolic groups \cite{AbbBehDur21}, etc
\cite{PetSprZal24,PetZal24,Zbi24,BalCheKer25}.

\fullref{def:stable_subspace_recognizing} for stable \emph{subspace} recognition is
formally stronger than stable \emph{subgroup} recognition, since if
$\phi\from G\to Y$ recognizes stable subgroups of $G$ it is
conceivable that there are stable subspaces of $G$ that are not close
to any finitely generated subgroup and for which $\phi$ does not give
a quasiisometric embedding. We do not know any examples.

\begin{proof}[Proof of \fullref{prop:recognizing_is_recognizing}]
  Suppose $\phi\from X\to Y$ is Morse recognizing.

  Suppose $Z\subset X$ is $(\lambda,\epsilon,\mu)$--stable.
  For all $z_0,z_1\in Z$ there exists a $\mu$--Morse
  $(\lambda,\epsilon)$ quasigeodesic $\gamma\from [t_0,t_1]\to X$
  with image in $Z$ and $\gamma(t_i)=z_i$ for $i\in \{0,1\}$. 
  Since $\phi\from X\to Y$ is Morse recognizing, $\phi\circ \gamma\from
  [t_0,t_1]\to Y$ is $(\lambda',\epsilon')$--quasigeodesic, so:
\[
  d_Y(\phi(z_0),\phi(z_1))=d_Y(\phi(\gamma(t_0)),\phi(\gamma(t_1)))\stackrel{\lambda'\!,\epsilon'}{\ceq}|t_1-t_0|\stackrel{\lambda,\epsilon}{\ceq}
  d_X(z_0,z_1)\]
Thus, $\phi|_Z$ is a quasiisometric embedding with constants
determined by $\lambda$, $\epsilon$, $\lambda'$, $\epsilon'$, and
$\lambda'$ and $\epsilon'$ were determined by $\lambda$, $\epsilon$,
and $\mu$.

Conversely, suppose $Z\subset X$ is
$(\lambda,\epsilon)$--quasigeodesically connected and $\phi|_Z$ is a
$(\lambda',\epsilon')$--quasiisometric embedding of $Z$ into $Y$.
For all $z_0,z_1\in Z$ there exists a 
$(\lambda,\epsilon)$--quasigeodesic $\gamma\from [t_0,t_1]\to X$ with
image in $Z$ and 
$\gamma(t_i)=z_i$ for $i\in \{0,1\}$.
Since $\phi|_Z$ is a $(\lambda',\epsilon')$--quasiisometric embedding,
$\phi\circ \gamma$ is a quasigeodesic with constants determined by
$\lambda$, $\epsilon$, $\lambda'$, and $\epsilon'$.
Since $\phi\from X\to Y$ is Morse recognizing, there exists $\mu$ such
that $\gamma$ is $\mu$--Morse, for $\mu$ depending on $\lambda$,
$\epsilon$ and the quasigeodesic constants of $\phi\circ\gamma$, hence
on $\lambda$, $\epsilon$, $\lambda'$, and $\epsilon'$.
Thus, the transitive family of uniform quasigeodesics verifying
quasigeodesic connectivity of $Z$ is also uniformly Morse, so $Z$ is
stable.

We have shown that $\phi\from X\to Y$ is Morse recognizing implies that
it is stability recognizing.
Now suppose $\phi\from X\to Y$ is stability recognizing.

Suppose $\gamma\from I\to X$ is a $\mu$--Morse
$(\lambda,\epsilon)$--quasigeodesic.
Then $Z:=\gamma(I)$ is a $(\lambda,\epsilon,\mu)$--stable subset of
$X$. Since $\phi\from X\to Y$ is stability recognizing, $\phi|_Z$ is a
$(\lambda',\epsilon')$--quasiisometric embedding.
Thus, $\phi\circ\gamma$ is a composition of quasiisometric embeddings,
so $\phi\circ\gamma\from I\to Y$ is a quasigeodesic with constants
determined by $\lambda$, $\epsilon$, $\lambda'$, and $\epsilon'$,
hence by $\lambda$, $\epsilon$, and $\mu$.

Conversely, suppose $\gamma\from I\to X$ is
$(\lambda,\epsilon)$--quasigeodesic and $\phi\circ\gamma\from I\to Y$
is $(\lambda',\epsilon')$--quasigeodesic.
Then $Z:=\gamma(I)$ is $(\lambda,\epsilon)$--quasigeodesically
connected in $X$.
Furthermore, for any $z_0,z_1\in Z$ there exist $t_0,t_1\in I$ such
that $\gamma(t_i)=z_i$, and:
\[d_X(z_0,z_1)=d_X(\gamma(t_0),\gamma(t_1))\stackrel{\lambda,\epsilon}{\ceq}|t_1-t_0|\stackrel{\lambda'\!,\epsilon'}{\ceq}d_Y(\phi(\gamma(t_0)),\phi(\gamma(t_1)))=d_Y(\phi(z_0),\phi(z_1))\]
This says $\phi|_Z$ is a quasiisometric embedding of $(Z,d_X|_Z)$ into $(Y,d_Y)$,
with constants depending on $\lambda$, $\epsilon$, $\lambda'$, and
$\epsilon'$.
Since $\phi\from X\to Y$ is stability recognizing, there exists $\mu$
depending on $\lambda$, $\epsilon$, $\lambda'$, and
$\epsilon'$ such that $Z$ is $(\lambda,\epsilon,\mu)$--stable.
Consider an interval $[t_0,t_1]\subset I$.
Since $Z$ is stable, there exists a $\mu$--Morse
$(\lambda,\epsilon)$--quasigeodesic $\delta\from J\to X$ with image in
$Z$ and endpoints at $\gamma(t_0)$ and $\gamma(t_1)$.
So $\gamma([t_0,t_1])$ and $\delta$ are
$(\lambda,\epsilon)$--quasigeodesics with the same endpoints such that
$\delta$ is $\mu$--Morse.
It follows that the Hausdorff distance between $\delta$ and
$\gamma([t_0,t_1])$ can be bounded in terms of $\lambda$, $\epsilon$,
and $\mu$.
Then it follows that there is $\mu'$ depending on $\lambda$, $\epsilon$,
and $\mu$ such that $\gamma([t_0,t_1])$ is $\mu'$--Morse.
Since this is true for every subsegment of $\gamma$, conclude that
$\gamma$ is $\mu'$--Morse.
\end{proof}

\subsubsection{Recognizing Morse and stable parabolics}\label{sec:morse_parabolics}
The combination of \fullref{thm:Morse_recognizing_general} and
\fullref{prop:recognizing_is_recognizing} gives:
\begin{proposition}\label{prop:stable_subgroup_qi_embeds}
  A subgroup $H\leq W$ is stable if and only if it is finitely
  generated and the orbit map is a quasiisometric embedding of $H$
  with respect to a word metric into $(\Davis^{(0)},\dfine)$.
\end{proposition}

It is not immediately clear from
\fullref{thm:Morse_recognizing_general} how to recognize Morse
subgroups of $W$.
However, for parabolic subgroups, equivalently, for special subgroups, there are nice criteria for being
stable and Morse. Given a finite rank Coxeter system $(W,S)$ and
$T\subset S$, let $\Gamma_T$ denote the subdiagram of the Coxeter
diagram of $(W,S)$ spanned by the vertices $T$. Call the subdiagram
``spherical'' if it defines a spherical Coxeter group. 
\begin{theorem}\label{thm:characterize_morse_stable_for_parabolics}
Let $(W,S)$ be a finite rank Coxeter system, and let $T\subset S$.

  $W_T$ is stable if and only if $\Gamma_{T\cap U}$ is spherical for every
  wide $U\subset S$.
  
  $W_T$ is Morse if and only if for every wide $U\subset S$, one of the following is true:
  \begin{enumerate}
  \item The intersection of every connected component of $\Gamma_U$ with
    $\Gamma_T$ is spherical.\label{item:condition_spherical_intersection}
    \item Every nonspherical connected component of $\Gamma_U$ is
      contained in $\Gamma_T$.\label{item:condition_engulfing}
  \end{enumerate}
\end{theorem}
\begin{proof}
  The stable case follows immediately from
  \fullref{main_for_combinatorial}: if $U$ is wide and $T\cap U$ is
  nonspherical then there is a biinfinite $\dcomb$--geodesic
  $\gamma$ contained in
  $\Davis_{T\cap U}=\Davis_T\cap\Davis_U$, so the entire $\gamma$ has
  edges with labels in the wide subset $U$, so $\gamma$ is not Morse.
  On the other hand, if $T\cap U$ is spherical for every wide $U$,
  then, since there are only finitely many spherical subsets of $S$,
  there is a uniform diameter bound on $\Davis_{T\cap U}$, so for
  every $\dcomb$--geodesic $\gamma$ in $\Davis_T$ there is a uniform
  bound on the lengths of its subsegments having label in some wide
  subset of $S$. 
\fullref{main_for_combinatorial} then says that geodesics in
  $\Davis_T$ are uniformly Morse. 

  Now consider Morseness of $\Davis_T$.
  Condition \eqref{item:condition_spherical_intersection} is
  equivalent to: $W_{T\cap U}=W_T\cap W_U$ is spherical. 
  Since nonspherical irreducible Coxeter groups do not have finite index proper
  parabolic subgroups, 
  condition \eqref{item:condition_engulfing} is equivalent to saying
  that the index of $W_{T\cap U}=W_T\cap W_U$ in $W_U$ is finite.
  The two conditions together say that every wide special subgroup is
  either virtually contained in $W_T$ or virtually has trivial
  intersection with $W_T$.

Suppose $U\subset S$ is wide and neither condition
\eqref{item:condition_spherical_intersection} nor
\eqref{item:condition_engulfing} is true.
Let $U=U_1\sqcup\cdots\sqcup U_k\sqcup K$ be the decomposition of $U$ into
irreducible nonspherical components $U_i$ and spherical $K$.
Since $U$ is wide, if $k=1$ then $U_1$ is irreducible higher rank
affine, so it has no proper nonspherical subparabolics.
Thus $T\cap U_1$ is either spherical or all of $U_1$.
This contradicts the hypothesis, so we must have $k\geq 2$, and there
is some $i$ such that $W_{T\cap U_i}$
has infinite index in $W_{U_i}$.
If $T\cap U_i$ is spherical then there is some $j\neq i$ such that
$T\cap U_j$ is nonspherical.
If $T\cap U_i$ is nonspherical then consider any $j\neq i$.
If $T\cap U_j$ is spherical then $W_{T\cap U_j}$ has infinite index in
$W_{U_j}$; in this case swap $i$ and $j$.
Thus, we may assume we have indices $i\neq j$ such that $W_{T\cap
  U_i}$ has infinite index in $W_{U_i}$ and $T\cap U_j$ is
nonspherical.
For $n\in\mathbb{N}$, choose $b_n\in W_{U_i}$ such that $\dcomb(b_n,\Davis_{T\cap
  U_i})\xrightarrow{n\to\infty}\infty$, which is possible since
$W_{T\cap U_i}$ has infinite index in $W_{U_i}$.
Then \fullref{lem:gate_projection_image_is_residue} implies that
$\dcomb(b_n,\Davis_T)\to\infty$.
Pick $a_n\in W_{T\cap U_j}$ with $|a_n|_S\geq |b_n|_S$. 
Since $W_{U_i}$ and $W_{U_j}$ commute, the geodesics
$[\one,a_n]\subset\Davis_T$ have endpoints that are connected by $(3,0)$--quasigeodesics $[\one,b_n]+[b_n,a_nb_n]+[a_nb_n,a_n]$ that contain
points $b_n$ arbitrarily far from $\Davis_T$.
Thus, $\Davis_T$ is not Morse. 

Now suppose that for every wide $U\subset S$ either \eqref{item:condition_spherical_intersection} or
\eqref{item:condition_engulfing} is true.
We derive a contradiction by supposing that $\Davis_T$ is not Morse.
Let $D$ be the diameter of a chamber of $\Davis$.
Since $(\Davis,\dcat)$ is CAT(0),
\fullref{morse_divergent_contracting_in_cat_zero} and
\fullref{equivalence_of_strong_contraction_conditions} combine to say
that if $\Davis_T$ is not Morse then it does not have the Bounded
Geodesic Image Property.
Thus, for every $n\in\mathbb{N}$ there exists a $\dcat$--geodesic $[a_n,b_n]$
that does not come $n$--close to $\Davis_T$, but
$\diam(\pi_{\Davis_T }([a_n,b_n]))>n$.
By passing to a subsegment, we may assume 
$\dcat(\pi_{\Davis_T}(a_n),\pi_{\Davis_T}(b_n))>n$. 
Let $x_n$ be a vertex closest to $a_n$, and let $y_n$ be a vertex
closest to $b_n$.
Then $\dcat(a_n,x_n)\leq D$, $\dcat(b_n,y_n)\leq D$, and the
Hausdorff distance between $[a_n,b_n]$ and $[x_n,y_n]$ is at most $D$,
so $\dcat(\pi_{\Davis_T}(x_n),\pi_{\Davis_T}(y_n))\geq\dcat(\pi_{\Davis_T}(a_n),\pi_{\Davis_T}(b_n))-2D\geq n-2D$ and
$\dcat(\Davis_T,[x_n,y_n])\geq n-D$.
Recalling the definition of the gate map from \fullref{sec:gate},
set $x_n':=\gate_{\Davis_T}(x_n)$ and $y_n':=\gate_{\Davis_T}(y_n)$. 
\fullref{cpp_and_gate_agree} implies
$\dcat(\pi_{\Davis_T}(x_n),x_n')\leq D$ and
$\dcat(\pi_{\Davis_T}(y_n),y_n')\leq D$, so $\dcat(x_n',y_n')\geq n-4D$.

Fix $n$ large enough that $\dcat$--distance at least $n-4D$ implies
$\dchain$ distance at least $4\caprace$ in
\fullref{all_metrics_are_equivalent} and \fullref{lem:distance_to_standard}.
The former implies $\dchain(x_n',y_n')\geq
4\caprace$, so there is a chain $\mathcal{V}^{-}\sqcup\mathcal{V}^+$
oriented from $x_n'$ to $y_n'$ such that the walls of
$\mathcal{V}^{-}$ come first, the walls of $\mathcal{V}^+$ come
second, and each of the two subchains contains at least $2\caprace$--many walls. 

Every wall of $\mathcal{V}^{-}\sqcup\mathcal{V}^+$ cuts $\Davis_T$, so $\pc(\mathcal{V}^{-})$
and $\pc(\mathcal{V}^+)$ are both irreducible nonspherical subgroups
of $W_T$.
By \fullref{cor:wall_gate}, no wall of $\mathcal{V}^{-}\sqcup\mathcal{V}^+$
separates $x_n$ from $x_n'$ or separates $y_n$ from $y_n'$, so they
all cut $\gamma:=[x_n,y_n]$.
Let $p$ be a point on $\gamma$ that occurs between the last wall of
$\mathcal{V}^{-}$ and the first wall of $\mathcal{V}^{+}$ and is not
contained in any wall.
Let $\gamma^{-}$ be the subsegment of $\gamma$ from $x_n$ to $p$, and
let $\gamma^{+}$ be the subsegment of $\gamma$ from $p$ to $y_n$.
Let $p'$ be a vertex closest to $p$.
Then $\dcat(p,p')\leq D$, no wall separates $p$ and $p'$, and
$\dcat(p',\Davis_T)\geq n-2D$.
Then the choice of $n$ with respect to
\fullref{lem:distance_to_standard} gives a chain $\mathcal{H}_0$ of
length at least $4\caprace$ separating $p'$ and $\Davis_T$.
Since no walls separate $p$ and $p'$, the chain $\mathcal{H}_0$ also
separates $p$ and $\Davis_T$.
The walls of $\mathcal{H}_0$ do not cut
$\Davis_T$, but all of them must cut the paths between
$\Davis_T$ and $p$ given by $[x_n',x_n]+\gamma^{-}$ and
$\gamma^++[y_n,y_n']$.
Since any wall cuts $\gamma$ at most once, each wall of
$\mathcal{H}_0$ cuts at
least one of $[x_n',x_n]$ and $[y_n,y_n']$. 
Let $\mathcal{H}^{-}$ be the subchain of $\mathcal{H}_0$ consisting of
walls that cut $[x_n,x_n']$, and let $\mathcal{H}^+:=\mathcal{H}_0\setminus\mathcal{H}^{-}$.
Without loss of generality, assume $|\mathcal{H}^{-}|\geq 2\caprace$,
and set $\mathcal{V}:=\mathcal{V}^{-}$ and
$\mathcal{H}:=\mathcal{H}^{-}$.
This is possible because otherwise  $|\mathcal{H}^{+}|\geq 2\caprace$,
so then we would take $\mathcal{V}:=\mathcal{V}^{+}$ and
$\mathcal{H}:=\mathcal{H}^{+}$ and swap $x_n$ and $x_n'$ with $y_n$ and $y_n'$ in what follows. 

Each wall $H\in\mathcal{H}$ cuts $[x_n,x_n']$ and separates $p$ and
$\Davis_T$.
Since $H$ separates $p$ from $\Davis_T$ and separates $x_n$ from
$x_n'\in\Davis_T$, we have that $x_n$ and $p$ are on the same side of
$H$.
Since the open halfspaces of $H$ are convex, $\gamma^{-}$ lies in the same open halfspace of $H$ as $x_n$ and $p$. 
For each $V\in\mathcal{V}$ there is a path from $p$ to $\Davis_T$
that follows $\gamma^{-}$ from $p$ to $V\cap\gamma$ and then follows $V$
to $\Davis_T$.
Since $H$ cuts this path but does not cut $\gamma^{-}$, we must have $H\transverse V$.

Since this is true for all $H\in\mathcal{H}$ and $V\in\mathcal{V}$, we
have a grid $(\mathcal{V},\mathcal{H})$ with $|\mathcal{V}|\geq
2\caprace$ and $|\mathcal{H}|\geq 2\caprace$.
\fullref{cor:P_or_refP} and \fullref{lem:grid_wide} say that
$\normalizer_W(\pc(\mathcal{V}))=\pc(\mathcal{V})\times\pc(\mathcal{V})^\perp$
is a wide parabolic, and either $\pc(\interior{\mathcal{H}})\leq
\pc(\mathcal{V})^\perp$ or $\pc(\mathcal{V})=\pc(\mathcal{H})$
is irreducible higher rank affine.

Since the walls of $\mathcal{H}$ do not cut $\Davis_T$, their
reflections are not in $W_T$.

The affine case is impossible because it yields $W_T\geq\pc(\mathcal{V})=\pc(\mathcal{H})\not\leq W_T$.

The commuting case is also impossible.
Up to the $W_T$--action, we may assume $\pc(\mathcal{V})=W_U$ for some
$U\subset T$, and it follows that
$\normalizer_W(\pc(\mathcal{V}))=W_U\times W_{U^\perp}$.
Now, $U$ is a nonspherical irreducible component of the wide set
$U\sqcup U^\perp$ that is contained in $T$, so
condition~\eqref{item:condition_spherical_intersection} is not
satisfied for $U\sqcup U^\perp$.
By hypothesis, 
condition~\eqref{item:condition_engulfing} must be satisfied,
meaning that $T$ contains every nonspherical irreducible component of
$U\sqcup U^\perp$, hence of $U^\perp$.
However, this is not true, because $\pc(\interior{\mathcal{H}})$ is a
nonspherical, irreducible reflection subgroup of $W_{U^\perp}$, so by
\fullref{lem:reflection_splitting} it is contained in
$W_{\mathcal{C}}$ for some connected component of $\Gamma_{U^\perp}$.
The component $\mathcal{C}$ is necessarily nonspherical, since
$\pc(\interior{\mathcal{H}})$ is.
We get the contradiction $W_T\not\geq\pc(\interior{\mathcal{H}})\leq
W_{\mathcal{C}}\leq W_T$.

Since both cases lead to contradictions, $\Davis_T$ is Morse.
\end{proof}

%% file: coneoff.tex
\begin{definition}
  The \emph{coned-off space} $(\coneoff,\dcone)$ is the metric graph $\coneoff$ and geodesic metric $\dcone$ obtained
  from $\Davis^{(1)}$ by adding a new vertex $v_\Upsilon$ for each
  wide parabolic subcomplex $\Upsilon$ of $\Davis$, and connecting
  $v_\Upsilon$ to each vertex of $\Upsilon$ by an edge of length
  $1/2$.
\end{definition}

In the case that $(W,S)$ is a right-angled Coxeter group, $\coneoff$
is $W$--equivariantly quasiisometric to the universal acylindrical
hyperbolic space for $W$ constructed by Abbott, Behrstock, and Durham
\cite{AbbBehDur21}. In the 2-dimensional, right-angled, strongly CFS case this also
coincides, again up to $W$--equivariant quasiisometry, with the
intersection complex of Oh \cite{Oh22}, see \cite[Section~6.2]{CasDanEdl}.
We also note that Genevois \cite{Gen21} has results producing
hyperbolic spaces by coning off certain convex subcomplexes of CAT(0)
cube complexes, but $\Davis$ is not a cube complex outside of the
right-angled case.

The goal of this subsection is to prove:
\begin{theorem}\label{thm:coneoff_qi_to_dfine}
  The identity map on $\Davis^{(0)}$ induces a quasiisometry
  $(\Davis^{(0)},\dfine)\to(\coneoff,\dcone)$.
\end{theorem}
\begin{corollary}\label{cor:coneoff_is_morse_recognizing}
  $(\coneoff,\dcone)$ is a geodesic hyperbolic space and
  $(\Davis^{(1)},\dcomb)\xrightarrow{\mathrm{Id_{\Davis^{(1)}}}}(\coneoff,\dcone)$
    is Morse recognizing.
\end{corollary}
\begin{proof}
   $(\coneoff,\dcone)$ is a geodesic space by construction.
 Hyperbolicity is preserved by quasiisometries between quasiruled
 spaces \cite[Section~2.1]{BlaHaiMat11}.
We have that $(\Davis^{(0)},\dfine)$ is a quasiruled hyperbolic space,
\fullref{thm:dfine_hyp_qruled_on_Davis}, and quasiisometric to
  the geodesic space $(\coneoff,\dcone)$, \fullref{thm:coneoff_qi_to_dfine}, so $(\coneoff,\dcone)$ is hyperbolic.
  The quasiisometry transfers the Morse recognizing property of
  $(\Davis^{(0)},\dfine)$, \fullref{thm:Morse_recognizing_general}, to $(\coneoff,\dcone)$.
\end{proof}

Since $\Davis^{(0)}$ is coarsely dense in 
$\coneoff$, the identity map on $\Davis^{(0)}$ is a coarse surjection
$\Davis^{(0)}\to\coneoff$, so to prove
\fullref{thm:coneoff_qi_to_dfine} we show $\dfine\ceq\dcone$ on $\Davis^{(0)}$.
One direction is easy: \fullref{lem:dfine_leq_dcone} says
$\dfine\leq\dcone$, because $\dfine$ gives wide parabolic subcomplexes
diameter 1, so it collapses as much as $\dcone$.
In the other direction we need to argue that $\dfine$ does not
collapse too much more than $\dcone$. That is
\fullref{dcone_cleq_dfine}.
To get the idea of the proof, recall that to show that a combinatorial
geodesic $\gamma$ in $\Davis$ is an unparameterized $\dfine$ quasiruler, we
defined a $\operatorname{FineProgress}$ function measuring the
distance travelled by $\gamma$ in terms of $\dfine$ against its
parameterization. We compressed parameter intervals on which
$\operatorname{FineProgress}$ was constant, so that
it had roughly constant growth rate relative to the reparameterization.
Now we look at $\gamma$ in $\coneoff$ and replace subsegments of
$\gamma$ by cone path shortcuts, when possible, to get a new path
$\realgreedy{\gamma}$, essentially making a geometric compression of
$\gamma$, instead of just changing its speed.
This compression is uniformly bounded in terms of $\dfine$, by
\fullref{small_dfine}, whose proof produces wide parabolic subcomplexes from grids to show that
if $\dfine$ collapses the length of a combinatorial geodesic segment
enough then the geodesic has a long subsegment in a uniform
neighborhood of a wide parabolic subcomplex, so that subsegment has
bounded  $\dcone$--diameter.
The proof of \fullref{dcone_cleq_dfine} cuts $\gamma$ along walls
of a longest fine chain separating its endpoints and uses the
quasiruler property to bound the $\dfine$ distance between endpoints
of the resulting subsegments, and then applies \fullref{small_dfine}
to control the $\dcone$--length of the subsegments.
It will turn out after the fact, \fullref{cor:greedy_are_qgeod}, that
the paths $\realgreedy{\gamma}$ are uniform quasigeodesics in $\coneoff$.

\bigskip

There are two kinds of edges in $\coneoff$, those coming from
$\Davis^{(1)}$ and those that connect a vertex $a\in\Upsilon^{(0)}$,
for some wide parabolic subcomplex $\Upsilon$, to the cone vertex
$v_\Upsilon$.
Furthermore, $v_\Upsilon$ is only connected by an edge to vertices
from $\Upsilon^{(0)}$, so a $\dcone$--geodesic $\gamma$ between $x$
and $y$ in $\Davis^{(0)}$ can be decomposed as a concatenation
$\beta_0+c_1+\beta_1+\cdots+c_m+\beta_m$ where the $\beta_i$ are disjoint
(possibly singleton)
$\dcomb$--geodesics and the $c_i$ are cone paths of the form
$a_i-v_{\Upsilon_i}-b_i$ for some $a_i\neq b_i\in\Upsilon_i^{(0)}$.
Moreover, two vertices of any $\beta_i$ are contained in a common wide
parabolic subcomplex only if they are adjacent in $\beta_i$.
Compare with the ``electric paths'' and ``electric path length'' in the
relatively hyperbolic case \cite{Far98}.

\begin{lemma}\label{lem:dfine_leq_dcone}
  On $\Davis^{(0)}\!$, $\dfine\leq \dcone$.
\end{lemma}
\begin{proof}
  Let $x,y\in\Davis^{(0)}\!$, and let $\gamma$ be a $\dcone$--geodesic
  connecting them, decomposed as above.
  Since $\gamma$ is $\dcone$--geodesic, so are its subpaths.
Thus, if $\beta_i$ is a $\dcomb$--geodesic subsegment from
$\beta_i^{-}$ to $\beta_i^+$ then
$|\beta_i|=\dcone(\beta_i^{-},\beta_i^+)$.
Since it is a combinatorial geodesic, we also have
$|\beta_i|=\dwall(\beta_i^{-},\beta_i^+)\geq\dfine(\beta_i^{-},\beta_i^+)$.
For cone paths $c_i=a_i-v_{\Upsilon_i}-b_i$
  \fullref{lem:wide_parabolics_have_diameter_one} says $\dfine(a_i,b_i)=1$.
  Thus: \[\dfine(x,y)\leq \sum_{i=0}^m\dfine(\beta_i^{-},\beta_i^+) +\sum_{i=1}^m\dfine(a_i,b_i) \leq \sum_{i=0}^m\dcone(\beta_i^{-},\beta_i^+) +\sum_{i=1}^m\dcone(a_i,b_i)=\dcone(x,y)\qedhere\]
\end{proof}

\begin{definition}[Shortcut graph, greedy path, realization]\label{def:greedy}
  Let $\gamma$ be a $\dcomb$--geodesic segment from $x$ to $y$ in
  $\Davis^{(1)}$ with vertex sequence
  $x=z_0,z_1,\dots,z_n=y$.
  Define its \emph{shortcut graph} $\shortcut{\gamma}$ to be the graph with vertices
  $z_i$, with an ``ordinary edge'' between $z_i$ and $z_j$ if they
  are adjacent in $\Davis^{(1)}$, and a ``shortcut edge'' between $z_i$
  and $z_j$ if they are not adjacent in $\Davis^{(1)}$, but are
  contained in a common wide parabolic subcomplex. 

  Define the \emph{greedy path} $\greedy{\gamma}$
to be the path in $\shortcut{\gamma}$ with vertex sequence  $z_{\phi(i)}$ where $\phi$ is defined
inductively by $\phi(0):=0$ and:
\[\phi(i+1):=\max\{m\mid \text{$z_{\phi(i)}$ and $z_m$ are adjacent in
    $\shortcut{\gamma}$}\}\]

Define a \emph{realization} $\realgreedy{\gamma}$ of $\greedy{\gamma}$
  to be a path in $\coneoff$ with vertices and ordinary edges
  corresponding to those of $\greedy{\gamma}$, and additional vertices
  and edges defined as follows.
  When $\greedy{\gamma}$ has a shortcut edge from $z_i$ to $z_j$ there
  exist wide parabolic subcomplexes containing $z_i$ and $z_j$. Pick
  one such parabolic subcomplex $\Upsilon$ and add a cone shortcut
  $z_i-v_\Upsilon-z_j$ to $\realgreedy{\gamma}$. 
\end{definition}
It follows from the definitions that
$\dcone(x,y)\leq d_{\shortcut{\gamma}}(x,y)\leq |\greedy{\gamma}|=|\realgreedy{\gamma}|$.

Observe that we have identified the vertices of $\gamma$ with those of
$\shortcut{\gamma}$, and we have identified the vertices of
$\greedy{\gamma}$ with the $\Davis^{(0)}$ vertices of
$\realgreedy{\gamma}$, which are a subset of the vertices of
$\gamma$. 
The distinction is that $\greedy{\gamma}$ is a canonical path in
$\shortcut{\gamma}$, whereas $\realgreedy{\gamma}$ is a path in
$\coneoff$ whose $\Davis^{(0)}$ vertices are canonical, but whose cone
vertices may not be. 

\begin{lemma}\label{lem:greedy}
  The greedy path is a shortest path from $z_0$ to $z_n$ in $\shortcut{\gamma}$.
\end{lemma}
\begin{proof}
Let $\phi$ be the indexing function of the greedy path, and
inductively define $\phi'$ by $\phi'(0):=0$ and:
\[\phi'(i+1):=\max\{m\mid \exists k\leq\phi'(i),\,\text{$z_k$ and
    $z_m$ are adjacent in $\shortcut{\gamma}$}\}\]
We claim that $\phi(i)=\phi'(i)$ for all $i$.
By definition, $\phi(0)=\phi'(0)=0$.
Suppose inductively that $\phi=\phi'$ on $[0,i]$, and consider $i+1$.
By definition, the maximum defining
$\phi'(i+1)$ is taken over a superset of the maximum defining $\phi(i+1)$, so
$\phi'(i+1)\geq\phi(i+1)$.
To derive the opposite bound, 
suppose  $j=\phi'(i+1)$.
Then there is some $k\leq\phi'(i)$ such that $z_k$ and $z_j$ are
adjacent in $\shortcut{\gamma}$.
If they are adjacent via a $\Davis^{(1)}$ edge then $j=k+1$, so
$\phi'(i+1)-1=k\leq\phi'(i)=\phi(i)<\phi(i+1)$ implies
$\phi'(i+1)\leq\phi(i+1)$.
If they are adjacent via a shortcut through a wide parabolic
subcomplex $\Upsilon$ then since parabolic subcomplexes are
$\dcomb$--convex and
$k\leq\phi'(i)=\phi(i)<\phi(i+1)\leq\phi'(i+1)=j$, the vertex
$z_{\phi(i)}$ is also in $\Upsilon$.
Since there is an edge from $z_{\phi(i)}$ to $z_j$ in $\shortcut{\gamma}$, it
follows that $\phi(i+1)\geq j=\phi'(i+1)$. Thus $\phi\equiv\phi'$.

Now take any injective edge path from $z_0$ to $z_n$ in $\shortcut{\gamma}$,
and let $q_0,q_1,\dots,q_m$ be its vertex sequence, and define $\psi$ by
$q_i=z_{\psi(i)}$.
It follows immediately that $\psi(i)\leq\phi'(i)$ for all $i$, so by
our previous argument $\psi(i)\leq\phi(i)$.
Thus, at every step the greedy path is at least as far along $\gamma$
as the $q$ path. 
In particular, since $\psi(m)=n$, the greedy path takes at most $m$
steps, so it is at least as short as every edge path from $z_0$ to $z_n$ in $\shortcut{\gamma}$.
\end{proof}

\begin{lemma}\label{small_dfine}
  There is a function $\collapse\from \mathbb{N}\to\mathbb{N}$ defined by:
  \[\collapse(N):=\sup\{d_{\shortcut{\gamma}}(x,y)\mid \dfine(x,y)\leq N\text{ and
    $\gamma$ is a $\dcomb$--geodesic from $x$ to $y$}\}\]
\end{lemma}
\begin{proof}
  We must show $\collapse(N)<\infty$.
  For $n\in\mathbb{N}$ with $n\leq \collapse(N)$ there exists a
$\dcomb$--geodesic $\gamma_n$ between some vertices $x_n$ and $y_n$ such that $\dfine(x_n,y_n)\leq
N$ and $d_{\shortcut{\gamma}_n}(x_n,y_n)\geq n$.
 For such $n$, let $p_0,p_1,\ldots,p_{q}$ be the vertex sequence
 of the greedy path $\greedy{\gamma_n}$ in $\shortcut{\gamma}_n$.
 Call these the ``greedy
 vertices'' of $\gamma_n$.
 By \fullref{lem:greedy}, $\greedy{\gamma_n}$  is a $\shortcut{\gamma}_n$-geodesic, so
$q=d_{\shortcut{\gamma}_n}(x_n,y_n)\geq n$.

Since $\gamma_n$ is a $\dcomb$--geodesic, the walls dual to the
edges of $\gamma_n$ are
distinct and separate $x_n$ and $y_n$.
For each $0\leq i<q$, call the wall dual to the first
edge of $\gamma_n$ after $p_i$ a ``greedy wall''. 
Dilworth's Theorem implies there is a subset of these
walls forming a chain $\mathcal{V}=\{V_i\}_{0\leq i\leq I}$ such that
$I+1=|\mathcal{V}|\geq n/\dim(\NR)$, where the indexing
and orientation are as
usual:  $x_n\in V_i^{-}\subset
V_{i+1}^{-}$ for all $i<I$.

Suppose $n\geq(N+1)\dim(\NR)$ so that $E:=\lfloor\frac{|\mathcal{V}|-1}{N}\rfloor\geq 1$.
Consider the subchain of $\mathcal{V}$ given by $\mathcal{V}':=\{V_0,V_E,V_{2E},\dots,
V_{N E}\}$.
Then $\dfine(x_n,y_n)\leq N<N+1=|\mathcal{V}'|$, so
$\mathcal{V}'$ is not fine, so there exists $0\leq j<N$ such that 
$\{V_{jE},V_{(j+1)E}\}$ is transverse to a chain $\mathcal{H}$ of
length $2\caprace$.
Let $\mathcal{V}'':=\{V_i\}_{jE\leq i\leq (j+1)E}$.
Then $(\mathcal{V}'',\mathcal{H})$ is a grid with
chains of lengths $E+1\geq |\mathcal{V}|/(N+1)$ and $2\caprace$.

By \fullref{prop:induction_argument} there exists a constant $C$ such
that $|\mathcal{V}''|\geq C(2\caprace-1)$ implies there exists a wide
parabolic subcomplex $\Upsilon$ and a convex subchain $\mathcal{V}'''$
of $\mathcal{V}''$ such that $|\mathcal{V}'''|\geq |\mathcal{V}''|/C$
and the subsegment $\beta$ of $\gamma_n$ between the extreme walls of
$\mathcal{V}'''$ is contained in $\bar\nbhd_\radius(\Upsilon)$.
We have $|\mathcal{V}''|\geq n/(N+1)\dim(\NR)$, so when $n\geq (2\caprace-1) (N+1) C\dim(\NR)$  this result
applies and gives $|\mathcal{V}'''|\geq n/(N+1)C\dim(\NR)$.

Now proceed as in the proof of \fullref{main_for_combinatorial}: Since $\beta\subset\bar\nbhd_{\radius}(\Upsilon)$ and $\Upsilon$ is
convex, the $\dwall$--version of the 4-gon trick (\fullref{quad_trick})
implies that, for $\beta^{-}$ and $\beta^{+}$ the endpoints of
$\beta$, if $\mathcal{M}$ is the set of all walls transverse to
$\beta$ and $\mathcal{M}'$ is the subset of those that cut $\Upsilon$,
we have:
\[D:=|\mathcal{M}|-|\mathcal{M}'|\stackrel{+}{\cleq}
  \dwall(\beta^{-},\Upsilon)+\dwall(\beta^{+},\Upsilon)\ceq
  \dcat(\beta^{-},\Upsilon)+\dcat(\beta^{+},\Upsilon)\leq 2\radius\]
Thus, there exists $D_0$,
depending only on $(W,S)$, such that $D\leq D_0$ for all $n$.

The count $D=|\mathcal{M}|-|\mathcal{M}'|$ gives us a decomposition of
$\beta$ as $D$--many single edges whose dual walls do
not cut $\Upsilon$ separating  $(D+1)$--many (possibly
trivial) subpaths consisting of edges all of whose dual walls cut
$\Upsilon$.
Consider one such subsegment $\beta'$.
By the same argument as in
the proof of \fullref{main_for_combinatorial}, its edge labels belong
to a set $U=S\cap a^{-1}Ta$ for some $a$, and if $U$ is
nonspherical then it is contained in a wide subset of $S$, which then
implies that $\beta'$ is contained in a wide parabolic subcomplex, so
its diameter in $\shortcut{\gamma}_n$ is 1.
It is also possible that $U$ is spherical and not contained in a wide
subset, recall \fullref{ex:bad_wall}, but in this case $|\beta'|\leq\maxsphere$, the maximal
$\dcomb$--diameter of a spherical special subgroup, so its diameter in
$\shortcut{\gamma}_n$ is also bounded above by $\maxsphere\geq 1$.
Thus, $\diam_{\shortcut{\gamma}_n}(\beta)\leq D+(D+1)\maxsphere\leq
D_0+(D_0+1)\maxsphere$.
Let $D_1:=D_0+(D_0+1)\maxsphere$.

For all but the first wall of $\mathcal{V}'''$, the edge path $\beta$
contains the vertex of $\gamma_n$ immediately prior to that wall.
Since the walls of $\mathcal{V}'''$ are greedy walls, this means that
$\beta$ contains $(|\mathcal{V}'''|-1)$--many greedy vertices.
Since the greedy path is a $\shortcut{\gamma}_n$ geodesic, this implies that:
\[\frac{n}{(N+1)C\dim(\NR)}-2\leq|\mathcal{V}'''|-2\leq\diam_{\shortcut{\gamma}_n}(\beta)\leq
  D_1\]
Thus, $n\leq (D_1+2)(N+1)C\dim(\NR)$ is uniformly bounded. 
We conclude:
\[\collapse(N)\leq (N+1)C\dim(\NR)\cdot\max\{2\caprace-1, D_1+2\}\qedhere\]
\end{proof}

\begin{proposition}\label{dcone_cleq_dfine}
  There exists $C$ such that for every $\dcomb$--geodesic
  $\gamma$ in $\Davis$ with endpoints
  $x,y\in\Davis^{(0)}$:
  \[\dcone(x,y) \leq d_{\shortcut{\gamma}}(x,y)\leq C\dfine(x,y)+C\]
\end{proposition}
\begin{proof}
Let  $\tau$ be the quasiruler constant of \fullref{lem:low_slack_on_NR_geodesics}.
  By \fullref{small_dfine}, $C:=\collapse(2\tau)+1<\infty$.
  
Let $\gamma\from [0,T]\to\Davis^{(1)}$ be the arc length
parameterization with $x:=\gamma(0)$ and $y:=\gamma(T)$.
  Let $\mathcal{V}:=\{V_i\}_{1\leq i\leq n}$ be a longest fine chain
  separating $x$ and $y$, with $x\in V_i^{-}\subset V_{i+1}^{-}$ for
  all $i$.
  Let $t_i$ be the half-integral time such that $\gamma(t_i)\in V_i$,
  and set $t_0:=-1/2$ and $t_{n+1}:=T+1/2$.
  For each $0\leq i\leq n$, let $\gamma_i$ be the maximal edge-subpath
  of $\gamma\setminus\mathcal{V}$ with endpoints
  $a_i:=\gamma(t_i+1/2)$ and $b_i:=\gamma(t_{i+1}-1/2)$.

For each $i$, $\{V_j\}_{1\leq j\leq i}$ is a fine chain separating $x$ and $a_i$,
so $\dfine(x,a_i)\geq i$,
and $\{V_j\}_{i+1\leq j\leq n}$ is a fine chain separating $b_i$ and
$y$, so $\dfine(b_i,y)\geq n-i$.
Applying \fullref{lem:low_slack_on_NR_geodesics} and
\fullref{cor:combinatorial_geodesics_rough_and_quasiruler} twice:
\[n=\dfine(x,y)\geq\dfine(x,a_i)+\dfine(a_i,y)-\tau\geq\dfine(x,a_i)+\dfine(a_i,b_i)+\dfine(b_i,y)-2\tau\geq i+\dfine(a_i,b_i)+n-i-2\tau\]
Thus, $\dfine(a_i,b_i)\leq 2\tau$, which implies  
$d_{\shortcut{\gamma}_i}(a_i,b_i)\leq
  \collapse(2\tau)$.

  Since $\shortcut{\gamma}_i$ is a subgraph of $\shortcut{\gamma}$,
  this implies $d_{\shortcut{\gamma}}(a_i,b_i)\leq \collapse(2\tau)$ for every $i$.

  Also, for each $1\leq i\leq n$ the vertices $b_{i-1}$ and $a_i$ are
  the endpoints of the edge of $\gamma$ crossing $V_i$, so they are adjacent in
  $\Davis^{(1)}$, so they are adjacent in $\shortcut{\gamma}$.

  Concatenating these edges with shortest paths in each
  $\shortcut{\gamma}_i$, we obtain a path in $\shortcut{\gamma}$ from $x$ to $y$ of
  length at most $(n+1)\collapse(2\tau)+n$.
  Thus:
  \[\dcone(x,y)\leq d_{\shortcut{\gamma}}(x,y)\leq (n+1)\collapse(2\tau)+n \leq Cn+C=C\dfine(x,y)+C\qedhere\]
\end{proof}

\begin{corollary}[Realizations of greedy paths are uniform $\dcone$--quasigeodesics]\label{cor:greedy_are_qgeod}
  There exists $C$ such that for every $\dcomb$--geodesic segment
  $\gamma$ in $\Davis^{(1)}$, any realization $\realgreedy{\gamma}$ in $\coneoff$ of the
  greedy path $\greedy{\gamma}$ in $\shortcut{\gamma}$ is a $(C,C)$--quasigeodesic.
\end{corollary}
\begin{proof}
  Let $C_0$ be the constant of \fullref{dcone_cleq_dfine}, and let
  $p_0$, \dots, $p_m$ be the vertex sequence of $\greedy{\gamma}$,
  which is the same as the $\Davis^{(0)}$--vertex sequence of $\realgreedy{\gamma}$.
  For $i<j$ let $\gamma_{ij}$ be the subsegment of $\gamma$ from $p_i$
  to $p_j$.
  By \fullref{lem:greedy}, the subpath of $\greedy{\gamma}$ from $p_i$
  to $p_j$ is a geodesic in $\shortcut{\gamma}$.
  Moreover, the vertices of the greedy path are vertices of $\gamma$
  occurring in the same order as they do in $\gamma$, so they all
  belong to $\gamma_{ij}$. 
  Since $\shortcut{\gamma}_{ij}$ is naturally identified with a
  subgraph of $\shortcut{\gamma}$, it follows that the subpath of $\greedy{\gamma}$ from $p_i$
  to $p_j$ is identified with $\greedy{\gamma}_{ij}$.
  Thus, $j-i=d_{\shortcut{\gamma}_{ij}}(p_i,p_j)$.
  By \fullref{dcone_cleq_dfine},
  $d_{\shortcut{\gamma}_{ij}}(p_i,p_j)\leq C_0\dfine(p_i,p_j)+C_0$.
  Since $\dfine\leq\dcone$, by \fullref{lem:dfine_leq_dcone}, conclude $j-i\leq C_0\dcone(p_i,p_j)+C_0$.
  On the other hand, the subpath of $\realgreedy{\gamma}$ from $p_i$
  to $p_j$ has length $j-i$, so $j-i\geq \dcone(p_i,p_j)$.

  We have shown that $i\mapsto p_i$ is a discrete
  $(C_0,C_0)$--quasigeodesic in $\coneoff$. Since every point on
  $\realgreedy{\gamma}$ is within distance $1/2$ of some $p_i$, we can
  increase $C_0$ a uniformly bounded amount such that $\realgreedy{\gamma}$
  is a continuous $(C,C)$--quasigeodesic.
\end{proof}

%% file: projection.tex
\subsubsection{Projections}\label{sec:projections}
 Recall from \fullref{sec:gate} that given a parabolic subcomplex
$\Upsilon$ there is a gate projection
$\gate_\Upsilon\from\Davis^{(0)}\to\Upsilon^{(0)}$, and it agrees with
$\dcomb$--closest point projection.
\begin{definition}\label{def:projection_pseudometric}
 Given a parabolic subcomplex $\Upsilon$, define its projection
 pseudometric on $\Davis^{(0)}$:
 \[d_\Upsilon(x,y) := \dcomb(\gate_\Upsilon(x),\gate_\Upsilon(y))\]
 \end{definition}
Now we want to make similar definitions of gate map and projection
pseudometric directly for a nonspherical irreducible parabolic subgroup $P$.
Recall from \fullref{lem:nonspherical_normalizer}
and \fullref{cor:P_or_refP} that there exists a parabolic subgroup
$P^\perp$ such that $\centralizer_W(P)=P^\perp$ and
$\normalizer_W(P)=P\times P^\perp$.
\begin{definition}
  For a parabolic subgroup $P$, define:
  \[\fibers(P):=\{\text{parabolic subcomplexes
    }\Upsilon\mid\stab(\Upsilon)=P\}\]
  When $P$ is irreducible and nonspherical, 
  let $\Upsilon_{N(P)}$ denote the unique element of
  $\fibers(\normalizer_W(P))$, recall \fullref{normalizer_parabolics_have_well_defined_subcomplex}.
\end{definition}
\begin{lemma}\label{lem:fibers}
  Let $P$ be a nonspherical irreducible parabolic subgroup of $(W,S)$.
  \begin{enumerate}
  \item $P^\perp$ acts freely and transitively by left multiplication
    on  $\fibers(P)$, and the action induces isomorphisms between the
    corresponding subcomplexes. \label{item:fibers_action}
    \item Distinct elements of $\fibers(P)$ are disjoint, and
      $\coprod_{\Upsilon\in\fibers(P)}\Upsilon^{(0)}=\Upsilon_{N(P)}^{(0)}$.\label{item:fibers_partition}
      \item For every $\Upsilon\in\fibers(P)$, the composition of the
        inclusion $\Upsilon^{(0)}\into \Upsilon_{N(P)}^{(0)}$ with the
        quotient $\Upsilon_{N(P)}^{(0)}\onto
        P^\perp\backslash\Upsilon_{N(P)}^{(0)}$ is a bijection.\label{item:fibers_bijection}
  \end{enumerate}
\end{lemma}
\begin{proof}
  By \fullref{irreducible_parabolics_have_type}, there is a unique
  $T\subset S$ such that $P$ is conjugate to $W_T$.
  Fix $w_0\in W$ such that $P=w_0W_Tw_0^{-1}$.
  Then $w_0\Davis_T\in\fibers(P)$.
   By \fullref{cor:P_or_refP}, $\normalizer_W(P)=P\times
   P^\perp$, with $P^\perp=w_0W_{T^\perp}w_0^{-1}$, so
   $\Upsilon_{N(P)}=w_0\Davis_{T\sqcup T^\perp}$.

   For \eqref{item:fibers_action}, the action is well-defined, because
   if $z\in P^\perp$ and $\Upsilon\in\fibers(P)$ then
   $\stab(z\Upsilon)=z\stab(\Upsilon)z^{-1}=zPz^{-1}=P$, since
   $z\in P^\perp=\centralizer_W(P)$.
   Transitivity and the fact that the multiplication induces an
   isomorphism of subcomplexes is
   \fullref{lem:same_stabilizer_implies_parallel}.
   Freeness follows because $P\cap P^\perp$ is trivial.

   By \eqref{item:fibers_action}, $\fibers(P)=\{zw_0\Davis_T\mid z\in
   P^\perp\}$, so:
   \[\bigcup_{\Upsilon\in\fibers(P)}\Upsilon^{(0)}=P^\perp
   w_0W_T=w_0W_{T^\perp}W_T=w_0W_{T\sqcup
     T^\perp}=\Upsilon_{N(P)}^{(0)}\]
 Suppose $z,z'\in P^\perp$ and $x,y\in W_T$ satisfy $zw_0x=z'w_0y$.
 Then $P^\perp\ni (z')^{-1}z=w_0yx^{-1}w_0^{-1}\in w_0W_Tw_0^{-1}=P$,
 so $z=z'$ and $x=y$. The former gives that two members of $\fibers(P)$ with a common vertex agree,
 so the union written above is a disjoint union. Moreover, 
 since members of $\fibers(P)$ are full subcomplexes, this implies
 that distinct members of $\fibers(P)$ are disjoint.
 The latter gives \eqref{item:fibers_bijection}, since the
 $P^\perp$--orbit of a vertex $zw_0x$ of $\Upsilon_{N(P)}$ is
 $\{z'w_0x\mid z'\in P^\perp\}$, which meets $z'w_0\Davis_T$ of
 $\fibers(P)$ in the vertex $z'w_0x$.
\end{proof}

We define $\gate_P$ by fixing any $\Upsilon\in\fibers(P)$ and setting:
\[\gate_P\from \Davis^{(0)}\to P^\perp\backslash \Upsilon_{N(P)}^{(0)} :
  x\mapsto P^\perp\gate_\Upsilon(x)\]
By \fullref{cor:W_action_on_projection} and \fullref{lem:perp_dont_change_gate}, if $g\in P^\perp$ and
$\Upsilon\in\fibers(P)$ then $\gate_{g\Upsilon}=g\gate_\Upsilon$.
Thus, for all $x\in \Davis^{(0)}$ and $\Xi\in\fibers(P)$, we have
$\gate_P(x)=P^\perp\gate_\Upsilon(x)=P^\perp\gate_{\Xi}(x)$, so
$\gate_P(x)$ is a well defined $P^\perp$--orbit of vertices in
$\Upsilon_{N(P)}^{(0)}$, independent of the choice of
$\Upsilon\in\fibers(P)$.

\begin{definition}\label{def:projection_pseudometric_subgroup}
Given a nonspherical irreducible parabolic subgroup $P$, define its
projection pseudometric on $\Davis^{(0)}$:
\[d_P(x,y):=\dcomb(\gate_P(x),\gate_P(y))\]
\end{definition}
The definition of $d_P$ is in terms
of the minimal $\dcomb$--distance between the $P^\perp$--orbits $\gate_P(x)$ and
$\gate_P(y)$. Since $\Upsilon_{N(P)}$ is a product and the $P^\perp$--action induces 
isomorphisms between $P$--fibers, an equivalent formulation is to
choose $\Upsilon\in\fibers(P)$, set $d_P=d_\Upsilon$, and observe that the
result is independent of the choice of $\Upsilon$.
This shows that $d_P$ actually is a pseudometric.

\begin{definition}
  For any parabolic subcomplex $\Upsilon$ and $x,y\in\Davis$:
  \[\walls_\Upsilon(x,y):=\{M\in\walls(x,y)\mid M\text{ cuts
    }\Upsilon\}=\{\text{walls }M\mid \text{$M$ 
      separates $x$ and $y$ and cuts $\Upsilon$}\}\]
\end{definition}
\begin{lemma} \label{lem:gate-wall}
  For any parabolic subcomplex $\Upsilon$ and $x,y\in\Davis^{(0)}$:
  \[\walls_\Upsilon(x,y)=\walls(\gate_\Upsilon(x),\gate_\Upsilon(y))\]
  In particular, this implies $d_\Upsilon(x,y)=|\walls_\Upsilon(x,y)|$.
\end{lemma}
\begin{proof}
  \fullref{cor:wall_gate} says that a wall $M$ separates $x$ from
  $\gate_\Upsilon(x)$ if and only if it separates $x$ from $\Upsilon$,
  so for any wall $M$ that cuts $\Upsilon$, vertices $x$ and
  $\gate_\Upsilon(x)$ are on the same side of $M$.
  
  For  $M\in\walls_\Upsilon(x,y)$, suppose $x\in M^-$. Then $y\in M^+$ and $M$ cuts $\Upsilon$, so  
  $\gate_\Upsilon(x)\in M^-$ and  
  $\gate_\Upsilon(y)\in M^+$ so, $M\in\walls(\gate_\Upsilon(x),\gate_\Upsilon(y))$.
Conversely, for $M\in\walls(\gate_\Upsilon(x),\gate_\Upsilon(y))$,
suppose $\gate_\Upsilon(x)\in M^-$.  Then $\gate_\Upsilon(y)\in M^+$,
so $M$ cuts $\Upsilon$, so $x\in M^-$ and $y\in M^+$, so $M\in\walls_\Upsilon(x,y)$.
\end{proof}

\subsubsection{Coned-off graphs of parabolic subgroups}
The following definition is very close to that of the \emph{Artin
  complex} of \cite{CumMarVas23}. It is also going to play a role
similar to that of the \emph{factored contact graph} \cite{BehHagSis17,AbbBehDur21}.
\begin{definition}[Parabolic coned-off graph]  \label{def:CP}
 Given a parabolic subcomplex $\Upsilon$, let $C(\Upsilon)$ be the graph obtained
 from $\Upsilon^{(1)}$ by coning off every proper parabolic subcomplex $\Xi
 \subsetneq \Upsilon$ by adding a vertex $v_\Xi$ and an edge of length
 $1/2$ from $v_\Xi$ to every vertex of $\Xi^{(0)}$.

 Given a nonspherical irreducible parabolic subgroup $P$, define:
 \[\widetilde{C(P)}:=\coprod_{\Upsilon\in\fibers(P)}C(\Upsilon)\]
Then $P^\perp$ acts on $\widetilde{C(P)}$ by left-multiplication, with $z\in P^\perp$ taking $v\in C(\Upsilon)$
to $zv\in zC(\Upsilon)=C(z\Upsilon)$.
Define $C(P)$ to be the quotient of $\widetilde{C(P)}$ by the left $P^\perp$--action.
\end{definition}

\begin{remark}
  In contrast to the curve complex/HHS setting, the graphs
  $C(\Upsilon)$ and $C(P)$ \textbf{need not be hyperbolic}.
  Indeed, suppose
  that $\Upsilon=\Davis_T$ for $T\subset S$ irreducible higher rank
  affine. Then $W_T$ has no proper nonspherical parabolics. Since
  special spherical subgroups have uniformly bounded diameter,
  $\Upsilon^{(1)}\into C(\Upsilon)$ is a quasiisometry. Thus,
  $C(\Upsilon)$ is quasiisometric to $\mathbb{Z}^{|T|-1}$.
\end{remark}

\begin{lemma}\label{lem:CP_well_defined}
  If $P$ is a nonspherical irreducible parabolic subgroup and
  $\Upsilon\in\fibers(P)$, then
  the composition $C(\Upsilon)\into \widetilde{C(P)}\to C(P)$ is a
  graph isomorphism.
  Moreover, for any two such parabolic subcomplexes $\Upsilon$ and
  $\Upsilon'$, the resulting isomorphism $C(\Upsilon)\to C(P)\to
  C(\Upsilon')$ is left-multiplication by the unique $z\in P^\perp$
  such that $z\Upsilon=\Upsilon'$.
  
There is a canonical assignment of a proper parabolic subgroup of $P$ to each
cone vertex of $C(P)$ by taking the common stabilizer of the corresponding
$P^\perp$--orbit of cone vertices of the $C(\Upsilon)$.
\end{lemma}
\begin{proof}
This is essentially the same argument as in the construction of
$\gate_P$ and $d_P$ in the previous subsection. 
 The ``ordinary'' vertices of $C(P)$ can be identified
  with $P^\perp$--orbits of pairs $(\Upsilon,v)$, where
  $\Upsilon\in\fibers(P)$ and $v\in\Upsilon^{(0)}$.
  The ``cone'' vertices of $C(P)$, which are distinguished by the fact that all of their
  incident edges have length $1/2$, can be identified with $P^\perp$--orbits of pairs $(\Upsilon,\Xi)$, where $\Upsilon\in\fibers(P)$ and
  $\Xi$ is a parabolic subcomplex contained in $\Upsilon$.

Note that if $z\in P^\perp$ and $(\Upsilon,\Xi)$ is a parabolic
subcomplex pair with $\Upsilon\in\fibers(P)$ then
$\stab(z\Xi)=z\stab(\Xi)z^{-1}=\stab(\Xi)$, so we can associate the subgroup
$\stab(\Xi)$ to the cone vertex of $C(P)$ corresponding to the
$P^\perp$--orbit of $(\Upsilon,\Xi)$, independent of the choice of
representative pair.
\end{proof}

Just as we defined projection pseudometrics $d_\Upsilon$ and $d_P$ in
\fullref{def:projection_pseudometric} and \fullref{def:projection_pseudometric_subgroup}, we now define their
$C(\Upsilon)$ and $C(P)$ analogues:
\begin{definition}\label{def:delta_pseudometric}
  If $\Upsilon$ is a parabolic subcomplex, let $d_{C(\Upsilon)}$ be
  the distance in the graph $C(\Upsilon)$.
  Define a projection pseudometric on $\Davis^{(0)}$:
  \[\delta_\Upsilon(x,y):=d_{C(\Upsilon)}(\gate_\Upsilon(x),\gate_\Upsilon(y))\]
Here we implicitly compose the gate map $\gate_\Upsilon\from
\Davis^{(0)}\to\Upsilon^{(0)}$ with the inclusion $\Upsilon^{(0)}\to
C(\Upsilon)^{(0)}$.

Similarly, if $P$ is a nonspherical irreducible parabolic subgroup, define:
  \[\delta_P(x,y):=d_{C(P)}(\gate_P(x),\gate_P(y))\]
\end{definition}
An equivalent formulation of $\delta_P$ is to choose
$\Upsilon\in\fibers(P)$, take $\delta_P:=\delta_\Upsilon$, and
observe that the resulting pseudometric is independent of the choice
of $\Upsilon$.

\begin{lemma} \label{lem:projection}
  If $P$ is a nonspherical irreducible parabolic subgroup and 
  $\Upsilon$ is a parabolic subcomplex:
  \[\diam_{\delta_P}(\Upsilon)>1\implies P \le \stab(\Upsilon)\]
\end{lemma}
\begin{proof}
  Choose any parabolic subcomplex $\Xi\in\fibers(P)$.
  By \fullref{lem:gate_projection_image_is_residue},
  $\gate_\Xi(\Upsilon)\subset \Xi^{(0)}$ spans a parabolic subcomplex
  of $\Xi$, so if $\gate_\Xi(\Upsilon)\subsetneq \Xi^{(0)}$ then it
  spans a proper parabolic subcomplex of $\Xi$, which is coned off in  
$C(\Xi)$, so $\gate_\Xi(\Upsilon)$ has
$\delta_P$--diameter at most 1.
Thus, $\gate_\Xi(\Upsilon)\subsetneq \Xi^{(0)}$  implies
$\diam_{\delta_P}(\Upsilon)\leq 1$, or, contrapositively, $\diam_{\delta_P}(\Upsilon)>1$
implies $\gate_\Xi(\Upsilon)=\Xi^{(0)}$.
 This gives,  by
  \fullref{cor:cross_projections_have_same_stabilizer} and
  \fullref{lem:nested_residues_have_nested_stabilizers}, that:
  \[ P=\stab(\Xi) = \stab(\gate_\Xi(\Upsilon)) =\stab(\gate_\Upsilon(\Xi))
    \leq\stab(\Upsilon)\qedhere\]
\end{proof}

\begin{definition}\label{def:prewide}
    A nonspherical irreducible parabolic subgroup is
    \emph{pre-wide} if its normalizer is wide.    
\end{definition}

\begin{lemma} \label{lem:pre-wide}
  Every nonspherical irreducible parabolic subgroup of the stabilizer
  of a wide parabolic subcomplex is pre-wide.
\end{lemma}
\begin{proof}
Let $\Upsilon$ be a wide parabolic subcomplex, and let $Q_1\times\cdots \times Q_n\times F$ be the canonical decomposition of
  $\stab(\Upsilon)$ into irreducible nonspherical factors $Q_i$ and a
  spherical factor $F$. 
Any nonspherical irreducible parabolic subgroup $P$ of $\stab(\Upsilon)$ is a
  reflection subgroup, so \fullref{lem:reflection_splitting} says it
  splits as a direct product of its intersections with the factors of
  $\stab(\Upsilon)$.
  Since it is nonspherical and irreducible, this implies there exists
  $i$ such that $P\leq Q_i$.
Assume, by renumbering, that $i=1$.
  Since $\stab(\Upsilon)$ is wide, either $n>1$ or $n=1$ and $Q_1$ is higher
  rank affine.
  In the former case, $Q_2\leq Q_1^\perp\leq P^\perp$, so $P^\perp$ is
  nonspherical, so $\normalizer_W(P)=P\times P^\perp$ is wide. 
  In the latter case, since proper parabolic subgroups of an
  irreducible affine Coxeter group are spherical, $P=Q_1$, so $P$ is
  irreducible higher rank affine, which also implies $\normalizer_W(P)$ is wide. 
\end{proof}

\subsubsection{Large links and bounded geodesic image for wide parabolics} 
Recall $\maxsphere:=\max_{\text{spherical }T\subset S}\diam_{\dcomb}W_T$.
\begin{lemma}\label{lem:far_wide_parabolics_have_finite_intersection}
  For any wide special subgroups $P$ and $Q$ and element $w \in W$, if
  $|P\cap wQw^{-1}|=\infty$, then $\dcone(\one,w) \le \maxsphere+2$.
\end{lemma}
\begin{proof}
Suppose $P\cap wQw^{-1}$ is infinite.
Intersections of parabolic subgroups are parabolic, and an infinite
parabolic has at least one nonspherical factor in its canonical
decomposition as a product of irreducible factors. 
  Thus there are nonspherical irreducible subsets $T,U\subset S$ and elements  $p\in P$ and $q\in Q$ such
  that $W_T\leq P$, $W_U\leq Q$, and $pW_Tp^{-1}=wqW_Uq^{-1}w^{-1}$.
By \fullref{irreducible_parabolics_have_type}, $T=U$, so $h:=p^{-1}wq\in\normalizer_W(W_T)$.
By \fullref{lem:nonspherical_normalizer}~\eqref{item:normalizer_splits},
$\normalizer_W(W_T)=W_T\times W_{T^\perp}$, so $h=ab$ for some $a\in
W_T$ and $b\in W_{T^\perp}$, and then $w=pabq^{-1}$.

  We now bound the $\dcone$--length of these factors uniformly. Since
  $W_T\subseteq P$ and $P$ is wide, every vertex of $\Davis_{T}$ has
  $\dcone$--distance at most $1$ from $\one$, so $\dcone(\one,pa)\leq
  1$.
  Similarly, $\dcone(\one,q)\leq 1$.
  If $W_{T^\perp}$ is finite then $\dcone(\one,b)\leq \dcomb(\one,b)\leq\maxsphere$, so
  $\dcone(\one,pabq^{-1})\leq \maxsphere+2$.
  If $W_{T^\perp}$ is not finite then $W_T\times W_{T^\perp}$ is wide,
  so $\dcone(\one,ab)\leq 1$ and $\dcone(\one,pabq^{-1})\leq 3\leq \maxsphere+2$.
\end{proof}

\begin{corollary} \label{cor:close-cone}
  If $\Upsilon$ and $\Xi$ are wide parabolic subcomplexes of $\Davis$ with
  corresponding cone vertices $v_\Upsilon$ and $v_\Xi$ in $\coneoff$,
  and $|\stab(\Upsilon) \cap \stab(\Xi)| = \infty$, then $\dcone(v_\Upsilon,v_\Xi) \le \maxsphere+3$.
\end{corollary}
\begin{proof}
  Write $\Upsilon= a \Davis_T$ and $\Xi = b\Davis_U$, where $T$ and
  $U$ are wide subsets of $S$.
  We have $\stab(\Upsilon) = aW_Ta^{-1}$ and $\stab(\Xi) =
  bW_Ub^{-1}$.
  Thus: \[ |\stab(\Upsilon) \cap \stab(\Xi)|=\infty \Longrightarrow
    |W_T \cap a^{-1} b W_U b^{-1} a|  = \infty\]
  By \fullref{lem:far_wide_parabolics_have_finite_intersection}
  $\dcone(a,b) = \dcone(\one,a^{-1}b) \le \maxsphere+2$.
  Then, by the triangle inequality:
  \[ \dcone(v_\Upsilon,v_\Xi) \le \dcone(v_\Upsilon,a) + \dcone(a,b) + \dcone(b,v_\Xi) \le \maxsphere+3 \qedhere\]  
\end{proof}

\begin{lemma} \label{lem:gridcompletion}
  Let $\Upsilon$ be a wide parabolic subcomplex, and let $\mathcal{V}$ be a
  chain of walls cutting $\Upsilon$ with $|\mathcal{V}|\geq 2$. There exists
  a biinfinite chain $\mathcal{H}$  cutting $\Upsilon$ such that $(\mathcal{V},\mathcal{H})$ is a grid. 
\end{lemma}
\begin{proof}
See the proof of \fullref{prop:near_flat_implies_long_grids}.
\end{proof}

\begin{lemma} \label{lem:LargeLink}
Let $x,y\in \Davis^{(0)}$, and let $\gamma$ be a geodesic in
$\coneoff$ from $x$ to $y$. Denote by
$\mathcal P(\gamma)$
the collection of wide parabolic subcomplexes whose cone vertices occur
on $\gamma$.
For any parabolic subcomplex $\Upsilon$ satisfying
$d_\Upsilon(x,y)>\dim(\NR)\cdot\dcone(x,y)$, 
there exists $\Xi\in\mathcal P(\gamma)$ such that
there is a pair of disjoint walls $M$ and $M'$ that cut both
 $\Upsilon$ and $\Xi$. This implies, in particular, that 
$|\stab(\Upsilon)\cap \stab(\Xi)|=\infty$.
\end{lemma}
\begin{proof}
  Let $x=z_0,z_1,\ldots,z_m=y$ be the sequence of vertices of
  $\Davis^{(0)}$ occurring in $\gamma$, so that at each 
step $z_{i-1}-z_{i}$ either $z_{i-1}$ and $z_{i}$ are adjacent in
$\gamma$ and in $\Davis^{(1)}$, or $\gamma$ contains a cone
shortcut $z_{i-1}-v_\Xi-z_i$, where $\Xi$ is a wide parabolic subcomplex
containing $z_{i-1}$ and $z_{i}$.

Since ordinary edges have length $1$ and cone edges have length $1/2$, we have:
\[|\mathcal P(\gamma)|\leq \dcone(x,y) = m\]

By Lemma \ref{lem:gate-wall}, $|\walls_\Upsilon(x,y)|=d_\Upsilon(x,y)$.
Dilworth's theorem implies $\walls_\Upsilon(x,y)$ contains a chain
$\mathcal{V}$ of length at least $d_\Upsilon(x,y)/\dim(\NR)$.
Thus, supposing $d_\Upsilon(x,y)>\dim(\NR)\cdot \dcone(x,y)$, we have $|\mathcal V|>m.$

Every wall $M\in\mathcal V$ separates $x$ from $y$, so some step of $\gamma$ has endpoints in opposite halfspaces of $M$. Assign $M$ to one such step.
An ordinary Davis edge crosses exactly one wall, so the ``excess''
$(|\mathcal{V}|-m)$--many walls come from cone shortcut steps. In
particular, there is some cone shortcut step $z_{i-1}-v_\Xi-z_i$ such
that $z_{i-1}$ and $z_i$ are  separated by distinct walls $M,M'\in\mathcal V$.
Because $z_{i-1},z_i\in \Xi$, both walls cut $\Xi$.
By construction, both walls also cut $\Upsilon$.
Furthermore, they are disjoint, since they came from the chain
$\mathcal{V}$.
\end{proof}

\begin{proposition}[Intersection large links for wide parabolics] \label{prop:LargeLink}
  There are constants $\lambda$, $D_0$, and $K_0$ such that for every
  pair of vertices $x,y\in \Davis^{(0)}$ there exists a collection of wide parabolic subcomplexes $\mathcal{P}(x,y)$ satisfying the following conditions:
  \begin{itemize}
    \item $|\mathcal{P}(x,y)| \le \lambda\dcone(x,y)+\lambda^2$.
    \item For all $\Xi \in \mathcal{P}(x,y)$ and any geodesic $\alpha$ in $\coneoff$ from $x$ to $y$, $\dcone(v_\Xi,\alpha) \le D_0$.
    \item If $\Upsilon$ is a wide parabolic subcomplex and
      $d_\Upsilon(x,y) > K_0$, then there exists $\Xi \in
      \mathcal{P}(x,y)$ such that a pair of disjoint walls cut both
      $\Upsilon$ and $\Xi$. In particular, $|\stab(\Upsilon) \cap \stab(\Xi)| =
      \infty$. 
  \end{itemize}
\end{proposition}
The name of \fullref{prop:LargeLink} is meant to invoke
\cite[Proposition~9.4 ``large link lemma'']{BehHagSis17}, where the
analogous conclusion would be that $\Upsilon$ is parallel into
$\Xi$. Our proposition is intentionally coarser: wide parabolic
subcomplexes are not the analogues of HHS domains.
Indeed, even in RACGs there are examples, see \fullref{ex:intersection_of_wide_not_wide}, containing maximal wide parabolic
subcomplexes whose intersection is nonspherical but not wide, so their
unbounded overlap is not captured by a lower
``domain''.
\begin{example}\label{ex:intersection_of_wide_not_wide}
   Take the Coxeter diagram to be a hexagon with vertices
  1,2,3,4,5,6. Add one more vertex 0 adjacent to 1.
  Label all edges $\infty$.
  Consider the complement of $\{2,5\}$ and the complement of
  $\{3,6\}$.
  These are maximal wide, and their intersection $\{0,1,4\}$ is
  nonspherical and nonwide, since $W_{\{0,1,4\}}\cong \dihedral_\infty\times\mathbb{Z}/2\mathbb{Z}$.
\end{example}
 
\begin{proof}[Proof of \fullref{prop:LargeLink}]
  Let $\gamma$ be a $\dcomb$--geodesic in $\Davis^{(1)}$ from $x$ to $y$,
  and let $\greedy{\gamma}$ be the greedy path in the shortcut graph
  $\shortcut{\gamma}$ of $\gamma$, and let $\realgreedy{\gamma}$ be a
  choice of realization of $\greedy{\gamma}$ in $\coneoff$.
  Recall that this means that for each shortcut edge
  $\greedy{\gamma}(i)-\greedy{\gamma}(i+1)$, we choose some wide
  parabolic subcomplex $\Xi$ containing those two vertices, and
  $\realgreedy{\gamma}$ contains a cone shortcut
  $\greedy{\gamma}(i)-v_\Xi-\greedy{\gamma}(i+1)$.
   Let $\mathcal{P}(x,y) = \mathcal{P}(\realgreedy{\gamma})$ be the
   collection of wide parabolic subcomplexes $\Xi$ so chosen.
   
It follows from the construction in the proof of  \fullref{lem:greedy}
that the order of vertices in $\greedy{\gamma}$ is the same as their
order in $\gamma$, hence the order of $\Davis^{(0)}$ vertices of
$\realgreedy{\gamma}$ is the same as it is in $\gamma$.

 By \fullref{cor:greedy_are_qgeod}, there is a $\lambda$ such that $\realgreedy{\gamma}$ is a uniform
 $(\lambda,\lambda)$--quasigeodesic from $x$ to $y$ in $\coneoff$, so:
  \[ |\mathcal{P}(x,y)| \le |\realgreedy{\gamma}| \le \lambda \dcone(x,y)+\lambda^2\]
 
  Since $(\coneoff, \dcone)$ is $\delta$--hyperbolic, the
  quasigeodesic $\realgreedy{\gamma}$ is contained in a fixed radius $D_0$ neighborhood of
  any $\dcone$--geodesic $\alpha$ from $x$ to $y$ in $\coneoff$, where
  the radius $D_0$ depends only on $\lambda$ and $\delta$.

   Let $C$ and $\radius$ be the constants of
   \fullref{prop:induction_argument}. Let $K=K(\radius)$ be such
   that whenever $\dcat(u,v) \le \radius+1/2+\diam(\text{chamber})$,
   then $\dcomb(u,v)\le K$. Since $\dcone\leq\dcomb$, this also gives $\dcone(u,v)\le K$.
   We set $D:=\lambda(2K+3+\lambda)$. By enlarging $\lambda$, if
   necessary, we may assume $D\geq (2\caprace-1)/\dim(\NR)$.
   Finally, let: \[ K_0 := CD\dim(\NR)^2\]
  
  Now let $\Upsilon$ be a wide parabolic subcomplex such that
  $d_\Upsilon(x,y) > K_0$. By \fullref{lem:gate-wall}, 
  $ d_\Upsilon(x,y) = |\walls_\Upsilon(x,y)|$.
By Dilworth's Theorem, $\walls_\Upsilon(x,y)$ contains a chain $\mathcal V$ with:
\[
|\mathcal V|
\ge \frac{d_\Upsilon(x,y)}{\dim(\NR)} > \frac{K_0}{\dim(\NR)} =
CD\dim(\NR)\geq C(2\caprace-1)\]
By \fullref{lem:gridcompletion} there is a chain $\mathcal{H}$ of
$2\caprace$--many walls cutting $\Upsilon$ such that $(\mathcal{V},
\mathcal{H})$ is a grid.
\fullref{prop:induction_argument} applied to that grid produces a wide parabolic
subcomplex $\Upsilon'$ and a convex subchain $\mathcal{V}' \subset
\mathcal{V}$ with $|\mathcal{V}'| \ge \frac{|\mathcal{V}|}{C}$, and a
subpath $\gamma'$ of $\gamma$ between the extremal walls of
$\mathcal{V}'$ lying within $\bar\nbhd_\radius(\Upsilon')$. Let $a$ be
the last vertex of $\gamma$ before $\gamma'$, and let $b$ be the first
vertex of $\gamma$ after $\gamma'$, so that $\dcat(a,\Upsilon')\leq
\radius+1/2$, $\dcat(b,\Upsilon')\leq \radius+1/2$, and every wall of
$\mathcal{V}'$ separates $a$ and $b$.
Let $a'$ and $b'$ be vertices of
$\Upsilon'$ contained in a common chamber with $\pi_{\Upsilon'}(a)$
and $\pi_{\Upsilon'}(b)$, respectively. Then:
\[ \dcone(a,b) \le \dcone(a,a') + \dcone(a',v_{\Upsilon'}) + \dcone(v_{\Upsilon'},b') + \dcone(b',b) \le 2K+1\]

Let $\greedy{\gamma}(l)$ be the vertex of $\greedy{\gamma}$ closest to
$a$ in $\shortcut{\gamma}$ lying on the $x$--side of $a$ in
$\gamma$. Similarly, let $ \greedy{\gamma}(r)$ be the vertex of
$\greedy{\gamma}$ closest to $b$ in $\shortcut{\gamma}$ lying on the $y$--side of $b$ along
$\gamma$.
We have $\dcone(a,\realgreedy{\gamma}(l)) \le 1$ and $\dcone(b,\realgreedy{\gamma}(r)) \le 1$, so
$\dcone(\realgreedy{\gamma}(l),\realgreedy{\gamma}(r)) \le 2K+3$.
In particular, the quasigeodesic condition gives:
\[ r-l \le \lambda(2K+3) + \lambda^2 = D\]
Thus, the segment of $\realgreedy{\gamma}$ from $\realgreedy{\gamma}(l)$ to $\realgreedy{\gamma}(r)$ has $\dcone$--length at most $D$.

Since every wall of $\mathcal{V}'$ separates $a$ and $b$, our choices
of $\realgreedy{\gamma}(l)$ and $\realgreedy{\gamma}(r)$ imply that every wall of $\mathcal{V}'$ separates
$\realgreedy{\gamma}(l)$ and $\realgreedy{\gamma}(r)$ in $\Davis$.
Since $|\mathcal{V}'|> D\dim(\NR)$ and
$d_{\shortcut{\gamma}}(\greedy{\gamma}(l),\greedy{\gamma}(r))=r-l\leq D$, there is some $s \in \{ l,\ldots
r-1\}$ such that there are  more than $\dim(\NR)$--many walls of
$\mathcal{V}'$ separating $\realgreedy{\gamma}(s)$ and $\realgreedy{\gamma}(s+1)$.
Since $\mathcal{V}'\subset\mathcal{V}$, those walls cut $\Upsilon$, so
\fullref{lem:gate-wall} yields $d_\Upsilon(\realgreedy{\gamma}(s),\realgreedy{\gamma}(s+1))>\dim(\NR)$.

Note that since  $\realgreedy{\gamma}(s)$ and $\realgreedy{\gamma}(s+1)$ are separated by multiple walls,
they are not adjacent in $\Davis^{(1)}$, so their adjacency in
$\shortcut{\gamma}$ is due to a shortcut edge corresponding to one of our chosen
wide parabolic subcomplexes $\Xi \in \mathcal{P}(x,y)$.
Thus:
\[ d_\Upsilon(\realgreedy{\gamma}(s),\realgreedy{\gamma}(s+1)) > \dim(\NR) = \dim(\NR)\cdot 1=\dim(\NR)\cdot\dcone(\realgreedy{\gamma}(s),\realgreedy{\gamma}(s+1))\]
Since $v_\Xi$ is the only cone vertex of the $\dcone$--geodesic
segment $\realgreedy{\gamma}(s)-v_\Xi-\realgreedy{\gamma}(s+1)$,
\fullref{lem:LargeLink} implies there is a pair of disjoint walls
cutting both $\Upsilon$ and $\Xi$, so $\stab(\Upsilon) \cap \stab(\Xi)$ is infinite. 
\end{proof}

\begin{theorem}[Bounded Geodesic Image] \label{thm:BGI}
  Let $D_0$ and $K_0$ be the constants of \fullref{prop:LargeLink}.
  Define $B_0:=D_0+\maxsphere+3$.
  Let $\gamma$ be a $\dcone$--geodesic between $x,y\in\Davis^{(0)}$,
  and let $\Upsilon$ be a wide parabolic subcomplex.
  \[\dcone(v_\Upsilon,\gamma) > B_0\implies d_\Upsilon(x,y) \le K_0\]
\end{theorem}
\begin{proof}
  By \fullref{prop:LargeLink}, if $d_\Upsilon(x,y) > K_0$,
  then there exist $\Xi \in \mathcal{P}(x,y)$ with $|\stab(\Upsilon)
  \cap \stab(\Xi)| = \infty$.
  \fullref{cor:close-cone} together with the same proposition yields:
  \[ \dcone(v_\Upsilon,\gamma) \le \dcone(v_\Upsilon, v_\Xi) + \dcone(v_\Xi,\gamma) \le \maxsphere+3+D_0=B_0\qedhere\]
\end{proof}

\begin{lemma}[One-sided distance formula]\label{lem:no_large_projection}
  There is a constant $C$ such that for all $x,y\in\Davis^{(0)}$:
  \[\dcomb(x,y)\leq C(1+\dcone(x,y))\cdot\max\{1,\sup_{\Upsilon}
    d_\Upsilon(x,y)\}\]
  where the supremum is taken over wide parabolic subcomplexes $\Upsilon$.
\end{lemma}
\begin{proof}
  It suffices to take $C$ to be the constant of
  \fullref{dcone_cleq_dfine}.
  
    Let $\gamma$ be a $\dcomb$--geodesic in $\Davis^{(1)}$ from $x$ to
    $y$, and let $x=z_0,\dots,z_{|\greedy{\gamma}|}=y$ be the vertex sequence of
  the greedy path $\greedy{\gamma}$. 
  By \fullref{dcone_cleq_dfine} and \fullref{lem:dfine_leq_dcone},
  $|\realgreedy{\gamma}|\leq C\dcone(x,y)+C$.

    If the step from $z_i$ to $z_{i+1}$ in $\realgreedy{\gamma}$ is an edge of
    $\Davis^{(1)}$ then $\dcomb(z_i,z_{i+1})=1$.
    Otherwise it is a  cone path shortcut, so $z_i$ and $z_{i+1}$ are
    contained in a common wide parabolic subcomplex $\Upsilon$.
    By $\dcomb$--convexity, the subsegment $\gamma_i$ of $\gamma$ between $z_i$
    and $z_{i+1}$ is a $\dcomb$--geodesic in $\Upsilon$.
   By definition, $d_\Upsilon(x,y)$ is the $\dcomb$--distance between the
   gate projections $\gate_\Upsilon(x)$ and $\gate_\Upsilon(y)$ of $x$ and $y$ to
   $\Upsilon$.
   Now:
   \begin{align*}
     \dcomb(x,z_i)+&|\gamma_i|+\dcomb(z_{i+1},y)=|\gamma|=\dcomb(x,y)\\
     &\leq \dcomb(x,\gate_\Upsilon(x))+\dcomb(\gate_\Upsilon(x),\gate_\Upsilon(y))+\dcomb(\gate_\Upsilon(y),y)
   \end{align*}
     Since the gate map agrees with $\dcomb$--closest point
     projection, this gives:
   \[0\leq\dcomb(x,z_i)-\dcomb(x,\gate_\Upsilon(x))+\dcomb(z_{i+1},y)-\dcomb(\gate_\Upsilon(y),y)\leq \dcomb(\gate_\Upsilon(x),\gate_\Upsilon(y))-|\gamma_i|\]
   Thus, $|\gamma_i|\leq \dcomb(\gate_\Upsilon(x),\gate_\Upsilon(y)) =d_\Upsilon(x,y)$.
The claim follows. 
\end{proof}

\subsubsection{Pre-wide witnesses for large projections}
\begin{lemma} \label{lem:factorization}
  Let $P=W_T$ be nonspherical and irreducible, and take
  $\Upsilon\in\fibers(P)$.
  Given $u,v \in \Upsilon^{(0)}$, let $m:=\delta_P(u,v)$.
  The element $w = u^{-1}v$ admits a factorization $w=h_1\cdots h_m$,
  where $h_i \in W_{J_i}$ for $J_i \subsetneq T$. Moreover, this
  factorization can be chosen so that each $h_i$ is a reduced word in
  the corresponding $W_{J_i}$ and $h_1\cdots h_m$ is a reduced word in
  $W$.
\end{lemma}
\begin{proof}
  Let $u=z_0, z_1, \ldots, z_m =v$ be the $\Upsilon^{(0)}$--vertices
  along a geodesic in $C(\Upsilon)$.
  Each step $z_{i-1}$ to $z_i$ joins vertices lying in a common proper parabolic subcomplex of $\Upsilon$, say
  $z_{i-1}, z_i \in  a_i W_{J_i}$ for some  $J_i \subsetneq T$.
    Thus $h_i = z_{i-1}^{-1} z_i\in W_{J_i}$. 
This is true even if $z_{i-1}$ and $z_i$ are adjacent in
$\Davis^{(1)}$, since then their difference is in $S$, so $J_i$
can be taken to be a singleton, which is a proper subset
of $T$, since $P$ nonspherical implies $|T|\geq 2$.
  
  Among $C(\Upsilon)$--geodesics from $u$ to $v$, choose one that
  minimizes $\sum_{i=1}^m |h_i|$.
  Choose a reduced word for each $h_i$.
  If their concatenation was not reduced, then the Coxeter deletion
  condition gives two letters whose deletion leaves a word
  representing the same element $w$. Recall that a reduced word
  representing an element of $W_J$ only uses letters in $J$;
  therefore, the remaining letters in the $i$-th block still belong to
  $J_i$, and hence represent an element of $W_{J_i}$.
  Thus, after further reducing each remaining $W_{J_i}$--word, either
  some block becomes trivial, contradicting the definition of $m$, or
  all $m$ factors remain nontrivial, but the total word length
  decreases, contradicting minimality of the sums of the word lengths.
\end{proof}

\begin{lemma}\label{witness_claim}
  If $P$ is a nonspherical irreducible parabolic subgroup 
  and $x,y\in\Davis^{(0)}$ with $d_P(x,y)>\maxsphere\delta_P(x,y)$ then there is a nonspherical irreducible parabolic subgroup $Q\lneq P$ such that:
  \[d_Q(x,y)\geq\frac{\frac{d_P(x,y)}{\delta_P(x,y)}-\maxsphere}{\lfloor(\rank(P)-1)/2\rfloor }\]
\end{lemma}
Note that if $\rank(P)=2$ then $d_P=\delta_P$, so the hypothesis is
impossible to satisfy, so there is no division-by-zero problem. 
\begin{proof}
  Up to the $W$--action, we may assume $P=W_T$ and $\Upsilon=\Davis_T$ for
    some $T\subset S$ with $|T|=r:=\rank(P)$.
Let  $m:=\delta_P(x,y)$. 
Let $u:=\gate_\Upsilon(x)$ and $v:=\gate_\Upsilon(y)$.
    By  \fullref{lem:factorization}, the element $w=u^{-1}v$ admits a
  factorization $w=h_1 \cdots h_m$ satisfying $h_i \in W_{J_i}$ for
  $J_i \subsetneq T$ and: 
  \[ d_P(x,y) = d_P(u,v)=\sum_{i=1}^m |h_i|\]
  For at least one $i$ the word $h_i$  has at least average length, $|h_i|\geq d_P(x,y)/m>\maxsphere$.
Take the canonical decomposition of $W_{J_i}$ into a product of
nonspherical irreducible factors $W_{J_{i,j}}$ and spherical factor $F_i$:
 \[ W_{J_i} = W_{J_{i,1}} \times W_{J_{i,2}} \times \cdots \times
   W_{J_{i,k}} \times F_i\]
 There must be at least one nonspherical factor, since
 $|h_i|>\maxsphere$ implies $h_i\notin F_i$.
 Since  $2\leq |J_i| \lneq
 |T|=r$, and since the irreducible nonspherical components of $J_i$
 are disjoint sets containing at least two vertices each, the number
 of non-spherical components is $0<k \le \lfloor\frac{
   r-1}{2}\rfloor$. 
 
 Since different factors commute, we can further factorize:
 \[ h_i = h_{i,1} h_{i,2} \cdots h_{i,k} \cdot f_i, 
\quad\text{ where } h_{i,j} \in W_{J_{i,j}},\, f_i \in
   F_i\text{, and } |h_i| = \sum_{j=1}^k |h_{i,j}| +
   |f_i|\]
Thus, there is at least one $j$ such that:
 \[|h_{i,j}|\geq \frac{|h_i|-|f_i|}{k}\geq \frac{d_P(x,y)/\delta_P(x,y)-\maxsphere}{\lfloor(r-1)/2\rfloor}\]

  Because the full concatenated expression for $w$ is reduced, we may
  commute the component words within $h_i$ and obtain a reduced word
  for $w$ in which a reduced expression for $h_{i,j}$ occurs as a
  contiguous subword. The corresponding subpath of the combinatorial
  geodesic lies in a parabolic subcomplex $\Xi$ whose stabilizer
  $Q$ is conjugate to $W_{J_{i,j}}$.
  Moreover, the prefix in $w$ preceding this subword is a word in $T$,
  so $\Xi\subsetneq\Upsilon$ and $Q\lneq P$.
  Every wall crossed by this subpath separates the gates of $x$ and $y$ in $\Upsilon$,
  hence separates $x$ and $y$ and cuts $\Xi$.
  Thus:
  \[ d_Q(x,y) \ge |h_{i,j}| \geq \frac{d_P(x,y)/\delta_P(x,y)-\maxsphere}{\lfloor(r-1)/2\rfloor}\qedhere\]
\end{proof}

\begin{proposition}[Pre-wide witness]\label{prop:witness}
  There is a function $M\from \mathbb{N}^{|S|-1}\to\mathbb{N}$ 
  such that for all $x,y \in \Davis^{(0)}$ and
  all wide parabolic subcomplexes $\Upsilon$, if $d_\Upsilon(x,y) > M(c_2,\ldots,c_{|S|})$, then there exists a pre-wide $P \le \stab(\Upsilon)$ such that $\delta_P(x,y) > c_{\rank(P)}$.
\end{proposition}
\begin{proof}
We will choose $M'_2 < M'_3 < \ldots < M'_{|S|}$ recursively.
Set $M'_1:=0$. Then, for $2\leq r\leq |S|$, choose:
  \begin{equation} \label{eqn:Mr}
 L_r:=\maxsphere + \big\lfloor\frac{r-1}{2}\big\rfloor M'_{r-1} \qquad  M'_r :=c_r L_r
  \end{equation}

  We first prove a statement for irreducible parabolics by induction on rank. 
  
  \begin{claim}
    Suppose that $P$ is a nonspherical irreducible parabolic subgroup
    with $d_P(x,y) > M'_{\rank(P)}$. Then either $\delta_P(x,y) >
    c_{\rank(P)}$ or there exists a nonspherical irreducible
    parabolic $Q \lneq P$  such that $d_Q(x,y) > M'_{\rank(Q)}$.
  \end{claim}
  \begin{claimproof}
    Let $r:=\rank(P)$, and suppose $\delta_P(x,y)\leq c_r$.
    Then $d_P(x,y)>c_rL_r\geq \delta_P(x,y)\maxsphere$, so \fullref{witness_claim} gives
    a nonspherical irreducible parabolic $Q\lneq P$ with:
    \[d_Q(x,y)\geq\frac{\frac{d_P(x,y)}{\delta_P(x,y)}-\maxsphere}{\lfloor(r-1)/2\rfloor
      }>\frac{\frac{M'_r}{c_r}-\maxsphere}{\lfloor
        (r-1)/2\rfloor}=M_{r-1}'\geq M'_{\rank(Q)}\]
This proves the claim.
  \end{claimproof}
 
Take the canonical decomposition $\stab(\Upsilon)=P_1 \times \cdots \times P_k \times F $ as a product of
  nonspherical irreducible factors and a spherical factor.
  Since gate maps to nested parabolic subcomplexes compose, \fullref{cor:composition_of_nested_gate_maps}, gate distance splits as the sum of the component distances,
  while the spherical contribution is at most $\maxsphere$. Therefore,
  setting $M(c_2,\dots,c_{|S|}):=\maxsphere+\lfloor\frac{|S|}{2}\rfloor M'_{|S|}$ guarantees that whenever $d_\Upsilon(x,y) >M(c_2,\dots,c_{|S|})$,
  then at least one of the $k\leq\lfloor\frac{|S|}{2}\rfloor$ many $P_i$ will have $d_{P_i}(x,y) >
  M'_{\rank(P_i)}$.
  Applying the claim by inducting on $\rank(P_i)$, we see that there
  is an irreducible nonspherical $P \le P_i$ such that $\delta_P(x,y)
  > c_{\rank(P)}$.
  Finally, $P$ is pre-wide by \fullref{lem:pre-wide}.
\end{proof}

\subsubsection{Assigning cone vertices to pre-wide parabolics}
\begin{lemma}\label{lem:rho_map_from_prewide_to_cone_vertices}
  There is a map $\rho$ from pre-wide parabolics to cone vertices of
  $\coneoff$ defined by sending a pre-wide parabolic $P$ to the cone
  vertex corresponding to the unique parabolic subcomplex
  $\Upsilon_{N(P)}\in\fibers(\normalizer_W(P))$.
  The map has the following properties:
  \begin{enumerate}
      \item $g\rho(P)=\rho(gPg^{-1})$\label{item:rho_map_group_action}
    \item If $P\leq Q$ are pre-wide parabolic subgroups, then
      $\dcone(\rho(P),\rho(Q)) \le 1$.\label{item:rho_map_nested}
      \item If $P\leq Q$ are pre-wide parabolic subgroups, then for all
        $x,y\in \Davis^{(0)}$ and $\Upsilon_Q\in\fibers(Q)$, we have
        $\delta_P(x,y)=\delta_P(\gate_{\Upsilon_Q}(x),\gate_{\Upsilon_Q}(y))$.\label{item:rho_map_extra}
           \item For all $x,y\in\Davis^{(0)}$, $\delta_P(x,y)\leq d_{\Upsilon_{N(P)}}(x,y)$.\label{item:rho_map_delta}
      \end{enumerate}
\end{lemma}
\begin{proof}
  Item~\eqref{item:rho_map_group_action} is clear.

For Item~\eqref{item:rho_map_nested}, for any $\Upsilon_Q\in\fibers(Q)$, choose $\Upsilon_P\in
\fibers(P)$ with $\Upsilon_P\subset\Upsilon_Q$.
  Then $\Upsilon_P\subset \Upsilon_{N(P)}\cap\Upsilon_{N(Q)}$.
  In particular, every vertex of $\Upsilon_P$ is a $\Davis^{(0)}$
  vertex at distance $1/2$
  from both the cone vertex for $\Upsilon_{N(P)}$, which is $\rho(P)$, and
  the cone vertex for $\Upsilon_{N(Q)}$, which is $\rho(Q)$.
  Thus, $\dcone(\rho(P),\rho(Q))\leq 1$.

Continuing with the same choices, for Item~\eqref{item:rho_map_extra}, apply
\fullref{cor:composition_of_nested_gate_maps}:
\begin{align*}
  \delta_P(x,y)&=d_{C(\Upsilon_P)}(\gate_{\Upsilon_P}(x),\gate_{\Upsilon_P}(y))\\
               &=d_{C(\Upsilon_P)}(\gate_{\Upsilon_P}\circ\gate_{\Upsilon_Q}(x),\gate_{\Upsilon_P}\circ\gate_{\Upsilon_Q}(y))\\
               &=\delta_P(\gate_{\Upsilon_Q}(x),\gate_{\Upsilon_Q}(y))
\end{align*}

Finally, for any $\Upsilon_P\in\fibers(P)$, we have, from 
  \fullref{cor:composition_of_nested_gate_maps} and
  the Lipschitz property of gate maps:
  \begin{align*}
    \delta_P(x,y)&=d_{C(\Upsilon_P)}(\gate_{\Upsilon_P}(x),\gate_{\Upsilon_P}(y))\\
                 &\leq \dcomb(\gate_{\Upsilon_P}(x),\gate_{\Upsilon_P}(y))\\
                 &=\dcomb(\gate_{\Upsilon_P}\circ\gate_{\Upsilon_{N(P)}}(x),\gate_{\Upsilon_P}\circ\gate_{\Upsilon_{N(P)}}(y))\\
    &\leq
      \dcomb(\gate_{\Upsilon_{N(P)}}(x),\gate_{\Upsilon_{N(P)}}(y))\\
    &=d_{\Upsilon_{N(P)}}(x,y)\qedhere
  \end{align*}
\end{proof}

\begin{corollary}\label{cor:BGIII}
  If $P$ is a pre-wide parabolic and $\gamma$ is a $\dcone$--geodesic
  segment in $\coneoff$ between $x,y\in\Davis^{(0)}$ then
  $\dcone(\gamma,\rho(P))>B_0$ implies $\delta_P(x,y)\leq K_0$,
  where $B_0$ and $K_0$ are as in \fullref{thm:BGI}.
\end{corollary}
\begin{proof}
  Since $\rho(P)$ is the cone vertex corresponding to the unique
  parabolic subcomplex $\Upsilon_{N(P)}$ with stabilizer
  $\normalizer_W(P)$, 
  we have $\delta_P(x,y)\leq d_{\Upsilon_{N(P)}}(x,y)\leq K_0$, where
  the first inequality is
  \fullref{lem:rho_map_from_prewide_to_cone_vertices}~\eqref{item:rho_map_delta}
  and the second is from \fullref{thm:BGI}.
\end{proof}

%% file: acylindricity.tex
Bowditch \cite{Bow08} defines a group action $G\act X$ on a metric
space to be \emph{acylindrical} if for all $\epsilon\geq 0$ there
exist $R_\epsilon$ and $N_\epsilon$ such that for all points $x$ and
$y$ with $d(x,y)\geq R_\epsilon$, the set $\{g\in G\mid d(x,gx)\leq
\epsilon\text{ and }d(y,gy)\leq \epsilon\}$ has cardinality at most
$N_\epsilon$.

\begin{theorem}[Acylindricity] \label{thm:acylindricity}
    The action $W\act (\coneoff,\dcone)$ is acylindrical. That is,
    for all $\epsilon\geq 0$ there are constants $R_\epsilon$ and
    $N_\epsilon$ such that $\dcone(x,y) \geq R_\epsilon$ implies
    $|A_\epsilon(x,y)|\leq N_\epsilon$, where:
    \[A_\epsilon(x,y) := \{ w\in W \mid
    \dcone(x,wx)\leq \epsilon \text{ and } \dcone(y,wy)\leq \epsilon
    \}\]
\end{theorem}
Most of this section is dedicated to proving
\fullref{thm:acylindricity}, but before diving in we draw some
further conclusions, assuming that \fullref{thm:acylindricity} holds.  

Osin \cite{Osi16} formalized the notion of an \emph{acylindrically
  hyperbolic group} as a group $G$ that admits an acylindrical,
non-elementary action on a hyperbolic space. 
In that work, he also defined a \emph{generalized loxodromic} (genlox)
as an
element $g\in G$ such that there exists some acylindrical action of $G$ on a hyperbolic space $X_g$ for which the
element $g$ acts by a loxodromic isometry.
A \emph{universal acylindrical action} $G\act X$ is an acylindrical action of $G$ on a hyperbolic space $X$ such that 
every generalized loxodromic element $g\in G$ acts loxodromically on
$X$. See also \cite{AbbHum20}.

Generalized loxodromic elements $g\in G$ are Morse  \cite{Sis16}; that
is, $\mathbb{Z}\to G : z\mapsto g^z$ is a Morse quasigeodesic with
respect to any word metric on $G$.
The converse is not true in general; Morse elements need not be
generalized loxodromic \cite{AbbHum20}. 
However, in Coxeter groups we do get equivalence:
\begin{proposition}\label{prop:morse_elements_are_genlox}
  The following are equivalent, for each $w\in W$:
  \begin{itemize}
  \item $w$ is a Morse element of $W$.
    \item $w$ is loxodromic in the action $W\act(\coneoff,\dcone)$.
      \item $w$ is a generalized loxodromic element of $W$.
  \end{itemize}
\end{proposition}
\begin{proof}
  An element $w\in W$ is Morse if and only if $n\mapsto w^n$ is a
  Morse quasigeodesic in $(\Davis^{(1)},\dcomb)$.
  This excludes finite order elements, which also cannot act
  loxodromically, so it suffices to consider $w$ of infinite order.
  Since $(\Davis,\dcat)$ is CAT(0), cyclic subgroups of $W$ are
  undistorted, so $n\mapsto w^n$ is quasigeodesic in $(\Davis,\dcat)$
  for all infinite order $w$.
  Since quasiisometries preserve quasigeodesics, the
  same is true in $(\Davis^{(1)},\dcomb)$, so it remains only to
  decide when $n\mapsto w^n$ is Morse. 
   By \fullref{cor:coneoff_is_morse_recognizing}, this is true if and
   only if $n\mapsto w^n$ is a quasigeodesic in
   $(\coneoff,\dcone)$.
   This is equivalent to $w$ being loxodromic for
   $W\act (\coneoff,\dcone)$.
   Since this particular action is an acylindrical action on a
   hyperbolic space, every element that acts loxodromically here is a
   generalized loxodromic element of $W$.
   All generalized loxodromic elements of $W$ are Morse elements of
   $W$ \cite{Sis16}. 
 \end{proof}

 There is a straightforward characterization of Coxeter groups for
 which $W\act (\coneoff,\dcone)$ is non-elementary:
\begin{proposition}\label{prop:nonelementary_action}
  The following are equivalent:
  \begin{itemize}
  \item $W\act (\coneoff,\dcone)$ is non-elementary.
  \item $(W,S)$ is nonspherical, not wide, and not virtually
      infinite cyclic.
  \item The canonical decomposition $P_1\times\cdots \times P_n\times F$ of
    $(W,S)$ into a product of  irreducible nonspherical factors $P_i$
    and a spherical factor $F$ has $n=1$ and $P_1$ is not affine.
  \end{itemize}
\end{proposition}
\begin{proof}
  The action is acylindrical, by \fullref{thm:acylindricity}, so by \cite[Theorem~1.1]{Osi16}, there are three possibilities:
 \begin{itemize}
 \item The action is non-elementary.
 \item The action has bounded orbits.
   \item $W$ is virtually infinite cyclic and contains a loxodromic element.
   \end{itemize}
   If $(W,S)$ is spherical or wide then $\coneoff$ is bounded.
   Otherwise, $W$ contains a Morse element
   \cite{CapFuj10}, and such an element acts loxodromically on
   $\coneoff$, by \fullref{prop:morse_elements_are_genlox}.
   In particular, $\coneoff$ is not bounded.
   Therefore, outside the spherical and wide cases, the action
   contains a loxodromic element and the action does not have bounded
   orbits, so the only possibilities are $W$ is virtually infinite
   cyclic or the action is non-elementary. 

    To see the equivalence with the product decomposition description,
   observe that $(W,S)$ is spherical if and only if $n=0$. It is wide
   if and only if either $n>1$ or
   $n=1$ and $P_1$ is higher rank affine. The remaining case to rule
   out is that $n=1$ and $P_1$ is infinite dihedral, which is
   precisely the case
   that $W$ is virtually infinite cyclic. 
\end{proof}
\begin{corollary}\label{cor:universal}
    The action $W\act (\coneoff,\dcone)$ is a universal
  acylindrical action.
  Moreover, if $(W,S)$ is nonspherical, not wide, and not virtually
  infinite cyclic, then the action is non-elementary.
\end{corollary}
\begin{proof}
  It is an acylindrical action, by \fullref{thm:acylindricity}.
  The action is universal, by
  \fullref{prop:morse_elements_are_genlox}.
  With the additional conditions,  the action is non-elementary, by
  \fullref{prop:nonelementary_action}.
\end{proof}

Abbott, Balasubramanya, and Osin \cite{AbbBalOsi19} define a poset of
hyperbolic structures on a group: if $X$ and $Y$ are (possibly
infinite) generating sets of a group $G$ then $Y$ \emph{dominates} $X$
if the identity map on $G$ induces a Lipschitz map from the Cayley
graph $\Gamma(G,Y)$ of $G$ with respect to $Y$ to $\Gamma(G,X)$.
Two generating sets are equivalent if they dominate
each other. A \emph{hyperbolic structure} is an equivalence class of
generating sets of $G$ for which the corresponding Cayley graphs are
hyperbolic. An \emph{acylindrically hyperbolic structure} is a
hyperbolic structure with the additional property that the actions of
$G$ on the corresponding Cayley graphs are acylindrical.
The \emph{largest} acylindrically hyperbolic structure, if it exists,
is the one that dominates all the others.

\begin{theorem}\label{thm:largest}
  $\coneoff$ defines the largest acylindrically hyperbolic structure
  for $W$. 
\end{theorem}
It is noted in \cite{AbbBalOsi19} that if a largest acylindrically
hyperbolic structure exists then it is universal, so
\fullref{thm:largest} also implies \fullref{cor:universal}.
\begin{proof}
  The 1--skeleton of the Davis complex is the Cayley graph
  $\Gamma(W,S)$ of $W$ with
  respect to $S$. The coned-off space $\coneoff$ contains the graph
  $\Davis^{(1)}$ and adds additional length $1/2$ edges and cone
  vertices whose effect is to make any two
  distinct  vertices contained in a wide parabolic subcomplex
  have distance exactly 1. Restricted to $\Davis^{(0)}$, which is coarsely dense
  in $\coneoff$,  this is exactly the same metric as the word metric
  defined by the generating set $S':=S\cup\bigcup_{\text{wide }T\subset
    S}W_T$.
  Therefore, $\Gamma(W,S')$ is equivariantly quasiisometric to $\coneoff$, so it is hyperbolic and
  the $W$--action is acylindrical, by
  \fullref{cor:coneoff_is_morse_recognizing} and \fullref{thm:acylindricity}.
  
  The remaining claim is that $S'$ dominates every other
  acylindrically hyperbolic structure on $W$.
  This part of the argument is the same as in \cite[Theorem~7.9]{AbbBalOsi19}.
Suppose that $X$ is some other acylindrically hyperbolic structure
on $W$.
The restriction of $W\act\Gamma(W,X)$ to $W_T$ is also acylindrical,
for any $T\subset S$.

We claim that when $T$ is wide every infinite order
element  $w\in W_T$ has a power with non-virtually cyclic centralizer.
If $W_T$ is virtually irreducible higher rank affine then it has a finite index
translation subgroup isomorphic to $\mathbb{Z}^n$ for $n\geq 2$, so
some power of $w$ is contained in that subgroup, hence has centralizer
containing $\mathbb{Z}^2$.
If $W_T=A\times B$ with $A$ and $B$ infinite then write $w=(a,b)$. If
both $a$ and $b$ have infinite order then $\langle
(a,1),(1,b)\rangle\cong\mathbb{Z}^2\leq\centralizer_{W_T}(w)$. If one of
them, say $b$, has finite order $n$ then $w^n=(a^n,1)$ has
$\langle w^n\rangle\times B\leq\centralizer_{W_T}(w^n)$. 

Powers of loxodromic elements are loxodromic, and centralizers of loxodromic elements in acylindrical actions are
virtually cyclic \cite[Corollary~6.9]{Osi16}, so this shows that $W_T\act\Gamma(W,X)$ contains no
loxodromic elements, and then Osin's trichotomy implies it
has bounded orbits.
Since there are only finitely many wide subsets of $S$, we can take
such a bound uniformly.
Since $S$ is also finite, $\sup_{s\in S'}|s|_X=C<\infty$, which implies that the identity
map on $W$ from $\Gamma(W,S')$ to $\Gamma(W,X)$ is $C$--Lipschitz.
Hence, $S'$ dominates $X$.
\end{proof}

\medskip

The remainder of this subsection builds towards the proof of acylindricity.
The proof is strongly inspired by that of acylindricity
for the action of a group on a hierarchically hyperbolic space
\cite[Theorem~14.3]{BehHagSis17}.
In particular, they define strong and weak sets $L_1(m)$ and $L_2(m)$
associated to the midpoint $m$ of a geodesic.
We make analogous definitions in \fullref{sec:strongweak}.
However, we do not know how to complete the direct analogue of their
proof for only a midpoint $m$.
Instead we take two points $m^-$ and $m^+$ trisecting the geodesic,
define strong and weak sets for both of them, and use the separation
between $m^-$ and $m^+$ and a Coxeter specific argument to conclude.

\medskip

Note that acylindricity is trivial if $\coneoff$ is bounded---just
take all $R_\epsilon$ bigger than the diameter of $\coneoff$---so we
may assume throughout that $(W,S)$ is not wide and nonspherical. 

The setup for \fullref{thm:acylindricity} is written assuming $x$ and
$y$ are vertices of the Davis complex. Since every point of $\coneoff$
is distance at most $1/2$ from $\Davis^{(0)}$, the extension to
arbitrary $x,y\in\coneoff$ is an easy step at the very end of the
proof.

\subsubsection{Some parallel geodesic estimates and constant roundup}\label{sec:roundup}
\begin{lemma} \label{lem:fellow-travel}
  Let $X$ be a $\delta$--hyperbolic metric space. 
Given $\epsilon$
there exists $H_\epsilon$ such that for any geodesic segments 
$\alpha\from [0,L]\to X$ and $\beta\from [0,L]\to X$, if $d(\alpha(0),\beta(0))\leq \epsilon$ and
$d(\alpha(L),\beta(L))\leq \epsilon$, then 
$d(\alpha(t),\beta(t))\leq H_\epsilon$ for every $t\in [0,L]$. 
\end{lemma}
\begin{proof}
  This is an elementary $\delta$--hyperbolicity calculation. We claim
  it suffices to take $H_\epsilon:=4\delta+3\epsilon$.
\end{proof}

\begin{lemma}\label{lem:cone_geod_close_to_comb_geod}
  There exists $J$ such that if $\alpha$ is a $\dcone$--geodesic between
  $x,y\in\Davis^{(0)}\subset\coneoff$ and $m\in\Davis^{(0)}$ is a vertex of $\alpha$ then for
  every $\dcomb$--geodesic $\gamma$ from $x$ to $y$ there exists a
  vertex $p\in\gamma$ such that $\dcone(m,p)\leq J$.
\end{lemma}
\begin{proof}
  Let $\greedy{\gamma}$ be the greedy path in the shortcut graph of
  $\gamma$, and let $\realgreedy{\gamma}$ be a realization in $\coneoff$.
 By \fullref{cor:greedy_are_qgeod}, realizations of greedy paths are
 uniformly quasigeodesic in $\coneoff$.
 
 There is a $\delta$ such that  $\coneoff$ is $\delta$--hyperbolic, by
 \fullref{cor:coneoff_is_morse_recognizing}.

 Since $\coneoff$ is hyperbolic, quasigeodesics are Morse, with Morse
 gauge depending only on the quasigeodesic constants and the
 hyperbolicity constant, so $\realgreedy{\gamma}$ is $\mu$--Morse for
 a Morse gauge $\mu$ independent of $\gamma$. 
 Since $\alpha$ is a geodesic with the same endpoints as $\realgreedy{\gamma}$, every point of
 $\alpha$ is within $\dcone$--distance $\mu(1,0)$ of some point of $\realgreedy{\gamma}$.
 Since every point of $\realgreedy{\gamma}$ is within $\dcone$--distance $1/2$ of $\gamma$, it
 suffices to take $J:=\mu(1,0)+1/2$.
\end{proof}

We now collect the constants that have appeared so far and define some more.
\begin{itemize}
  \item $\maxsphere:=\max_{\text{spherical }T\subset S}\diam_{\dcomb}W_T$.
\item $\delta$ is the hyperbolicity constant of $\coneoff$.
\item $\epsilon$ is the acylindricity displacement parameter.
  \item $H_\epsilon$ is the fellow-traveler constant from
    \fullref{lem:fellow-travel}.
\item $J$ is the $\dcone$--quasiconvexity constant of \fullref{lem:cone_geod_close_to_comb_geod}.
\item $D_0$ and $K_0$ are the constants of \fullref{prop:LargeLink}.
\item $B_0=D_0+\maxsphere+3$, as in \fullref{thm:BGI}.
\end{itemize}
Note that $\maxsphere$ is an integer and $D_0$ can be taken to be an
integer, which also makes $B_0$ an integer.
For our later convenience we replace $H_\epsilon$, $J$, and $K_0$ by their
integer ceilings. 
Then we make the following proleptic choices of additional constants,
which are also integers.
\begin{itemize}
  \item $C_0 := H_\epsilon+B_0+3J$
  \item $C_1 := C_0 +2$
  \item $C_2 := C_1+H_\epsilon$
  \item $C_3:=C_2+B_0+1$
\end{itemize}

Also for each $2 \le r \le |S|$, we choose positive integers $a_r<c_r<b_r$ such that:
  \begin{itemize}
  \item $a_r>2K_0+2C_3$
  \item $c_r > a_r+2K_0$
  \item $b_r > c_r+2K_0$
  \item $a_s > b_r+4K_0$ when $s<r$
  \end{itemize}
  Such choices exist, for example by choosing $a_{|S|}>2K_0+2C_3$
  first, then choosing $c_{|S|}$, $b_{|S|}$, and $a_{|S|-1}$
  satisfying the required conditions, then continuing inductively.

  \begin{itemize}
  \item $M_\epsilon:=M(c_2,\ldots,c_{|S|})$, where $M$ is the witness
    threshold function of \fullref{prop:witness}, and
    $c_2,\ldots,c_{|S|}$ are the constants $c_r$ chosen above.
  \end{itemize}

\subsubsection{Parabolic carriers}
Let $u,v \in \Davis^{(0)}$.
For $i=1,\ldots,|S|-1$, we inductively define three collections
$\mathcal{P}_i$, $\mathcal{D}_i$, and $\mathcal{F}_i$ of parabolic
subgroups.
Elements of $\mathcal{P}_i$ are called the \emph{parent carriers} at level $i$; those of $\mathcal{D}_i$ are the \emph{descendant
  carriers} at level $i$; and those of $\mathcal{F}_i$ are the
\emph{factor carriers} at level $i$.

Set $\mathcal{P}_1 := \{W\}$.

Choose a $\dcone$--geodesic $\gamma$ from $u$ to $v$ in $\coneoff$.

Let $\mathcal{D}_1$ be the set of stabilizers of wide parabolic
subcomplexes associated to the cone vertices that occur on $\gamma$. 

Define $\mathcal{F}_1$ to be the set of irreducible nonspherical
factors occurring in the canonical product decomposition of some $Q\in \mathcal{D}_1$.

Now assume for $i\ge 1$, the sets $\mathcal{P}_i$, $\mathcal{D}_i$, and
$\mathcal{F}_i$ have been defined.
Let: \[ \mathcal{P}_{i+1} := \{ P \in \mathcal{F}_i \mid \delta_P(u,v)\le
  b_{\rank(P)}+2K_0\}\]
For each $P \in \mathcal{P}_{i+1}$, choose a geodesic $\hat{\gamma}_P$
in $C(P)$ from $\gate_P(u)$ to $\gate_P(v)$, where $C(P)$ and
$\gate_P$ are as defined in
\fullref{sec:projections} and \fullref{def:CP}.

Let $\mathcal{D}_{i+1}$
be the set of parabolic subgroups associated to the cone vertices on
$\hat{\gamma}_P$ for each $P \in \mathcal{P}_{i+1}$.

Let $\mathcal{F}_{i+1}$ be the set of irreducible nonspherical factors
occurring in the canonical product decomposition of some $Q \in \mathcal{D}_{i+1}$.

We set $\rank(\mathcal{P}_i) := \sup \{\rank(P)\mid P \in
\mathcal{P}_i\}$, which is a maximum unless $\mathcal{P}_i=\emptyset$,
and define $\rank(\mathcal{D}_i)$ and
$\rank(\mathcal{F}_i)$ similarly.
By construction, if $\mathcal{P}_i\neq\emptyset$ then: \[ \rank(\mathcal{P}_i) > \rank(\mathcal{D}_i) \ge \rank(\mathcal{F}_i) \ge \rank(\mathcal{P}_{i+1})\]
In particular, $\rank(\mathcal{P}_{i})\leq |S|+1-i$, so $\rank(\mathcal{P}_{|S|-1})\leq 2$, which implies
$\mathcal{P}_{|S|}\subset \mathcal{F}_{|S|-1}=\emptyset$.

\begin{lemma}\label{lem:elements_of_Fi_are_pre-wide}
  For all $i$, each $F\in \mathcal{F}_i$ is pre-wide.
\end{lemma}
\begin{proof}
  Each $F\in \mathcal{F}_1$ is a nonspherical irreducible parabolic
  subgroup of a wide parabolic subgroup, by definition of
  $\mathcal{F}_1$ and $\mathcal{D}_1$, so $F$ is pre-wide, by 
\fullref{lem:pre-wide}.
  
  Suppose the lemma is true for all $F\in \mathcal{F}_i$ for all $i\leq i_0$.
  Consider $F\in \mathcal{F}_{i_0+1}$.
  By construction, $F$ is a nonspherical irreducible factor in the
  canonical decomposition of some $Q\in \mathcal{D}_{i_0+1}$.
  In turn, $Q$ is the parabolic subgroup associated to some cone
  vertex of $C(P)$ for some $P\in
  \mathcal{P}_{i_0+1}\subset\mathcal{F}_{i_0}$.
  By the induction hypothesis, $P$ is pre-wide, so $\normalizer_W(P)$
  is a wide parabolic subgroup.
  Thus, $F\leq Q\leq P\leq \normalizer_W(P)$.
  Since $F$ is a nonspherical irreducible parabolic contained in
  the wide parabolic $\normalizer_W(P)$, it is pre-wide, again by \fullref{lem:pre-wide}.
\end{proof}

\begin{lemma} \label{lem:count_domains}
  There is a function $\factorsum\from \mathbb{N}\to\mathbb{N}$ defined by:
  \[\factorsum(L):=\sup_{\dcone(u,v)=2L}\sup \sum_{i=1}^{|S|-2}|\mathcal{F}_i|\]
  The first supremum is taken over pairs $u,v\in\Davis^{(0)}$ such
  that $\dcone(u,v)=2L$. The second supremum is taken over all possible
  choices of carrier systems for the given $u$ and $v$.
\end{lemma}
\begin{proof}
  Take any $u,v\in\Davis^{(0)}$ such that $\dcone(u,v)=2L$.
  Construct parent, descendant, and factor carriers for each
  level, including the choice of $\dcone$--geodesic $\gamma$ from $u$
  to $v$ in $\coneoff$.
  We must show $\sum_{i=1}^{|S|-2}|\mathcal{F}_i|$ is bounded above,
  depending only on $L$.

  By construction, elements of $\mathcal{F}_i$ are irreducible nonspherical factors
  of the canonical product decomposition of elements of
  $\mathcal{D}_i$.
  If $\mathcal{D}_i=\emptyset$ then $|\mathcal{F}_i|=0$.
  Otherwise, since such a factor has rank at least 2, the number
  of such factors of $Q\in \mathcal{D}_i$ is at most $\left\lfloor\frac{\rank(Q)}{2}\right\rfloor$.
  Therefore:
  \[ |\mathcal{F}_i| \le \left\lfloor\frac{\rank({\mathcal{D}_i})}{2}\right\rfloor \cdot |\mathcal{D}_i| <\frac{ |S| \cdot |\mathcal{D}_i|}{2}\]
  Since $\dcone(u,v)=2L$, there can be at most $2L$ cone vertices on $\gamma$, so
  $|\mathcal{D}_1|\leq 2L$ and $|\mathcal{F}_1| < L|S|$.
  For $i \ge 1$, recall the collection of parent carriers:
  \[  \mathcal{P}_{i+1} := \{ P \in \mathcal{F}_i \mid
    \delta_P(u,v)\le b_{\rank(P)}+2K_0\}\]
  Each element $P \in \mathcal{P}_{i+1}$ has a corresponding
  path $\hat\gamma_P$ in $C(P)$ of length at most $b_{\rank(P)}+2K_0$, therefore
  contributing at most that number of cone vertex parabolics to the
  descendant carriers $\mathcal{D}_{i+1}$.
  Since $b_2 > b_3 > \cdots $, we have for $i \ge 2$:
  \[ |\mathcal{D}_i| \le (b_2+2K_0)\, |\mathcal{P}_{i}| \le
    (b_2+2K_0)\, |\mathcal{F}_{i-1}|\]
  Therefore, $|\mathcal{F}_{i}| \le \frac{|S|}{2} (b_2+2K_0)\,
  |\mathcal{F}_{i-1}|$.
  Recursively, this yields:
  \[|\mathcal{F}_i| \le L|S|\left(\frac{|S| (b_2+2K_0)}{2}\right)^{i-1}\]
  This gives a uniform upper bound on $\sum_{i=1}^{|S|-2}
  |\mathcal{F}_i|$, depending only on $L$ and constants:
\[\sum_{i=1}^{|S|-2} |\mathcal{F}_i|\leq  L|S|\sum_{j=0}^{|S|-3}\left(\frac{|S| (b_2+2K_0)}{2}\right)^{j}\qedhere\]
\end{proof}

\subsubsection{The strong and weak sets for trisection vertices}\label{sec:strongweak}
Let $x,y\in\Davis^{(0)}$, and let  $\gamma$ be a $\dcone$--geodesic
from $x$ to $y$ in $\coneoff$.
Choose two vertices $m_-$ and $m_+$ in $\Davis^{(0)}$ on $\gamma$ that
are nearest to the points $1/3$ and $2/3$ of the way along ${\gamma}$.
For all $g \in A_\epsilon(x,y)$ and $m\in \{m_\pm\}$, we have $\dcone(m,g m)\leq
H_\epsilon$, by definition of $H_\epsilon$.

\begin{definition}[The strong set $L_1$ and weak set $L_2$]\label{def:strong_weak}
For fixed $x$, $y$, and $\gamma$ as above, for each  $m\in \{m_\pm\}$, let $L_1(m)$ be the collection of pre-wide parabolics $P$ such that: 
\begin{itemize}
    \item $\delta_P(x,y)>c_{\rank(P)}$
    \item $\dcone(m, \rho(P))\leq C_1$ 
    \item There is no pre-wide parabolic $Q \gneq P$ with $\delta_Q(x,y)>c_{\rank (Q)}$.
\end{itemize}
    
Similarly, let $L_2(m)$ be the collection of pre-wide parabolics $P$ such that:
\begin{itemize}
    \item $\delta_P(x,y)>a_{\rank(P)}$ 
    \item $\dcone(m,\rho(P))\leq C_2$ 
    \item There is no pre-wide parabolic $Q \gneq P$ with $\delta_Q(x,y)>b_{\rank(Q)}$.
\end{itemize}
\end{definition}

In this section we will show that $|L_2(m)|$ is uniformly finite
(\fullref{prop:L2finite}), that $A_\epsilon(x,y)$ conjugates $L_1(m)$ into
$L_2(m)$ (\fullref{prop:stability}), and that if $L_1(m)=\emptyset$
then there is a vertex of $\Davis^{(0)}$ uniformly $\dcone$--close to $m$ that
has bounded $\dcomb$--displacement under the $A_\epsilon(x,y)$--action (\fullref{prop:no_large_links}).

\begin{proposition}[Uniform finiteness of $L_2(m)$]
\label{prop:L2finite}
    Let $C_3$ be as in
    \fullref{sec:roundup}.
    Let $x,y\in\Davis^{(0)}$ with $\dcone(x,y) \ge R_\epsilon>3(C_3+1)$.
    Let $\gamma$ be any $\dcone$--geodesic from $x$ to
    $y$, and let $m_\pm$ be trisection vertices of $\gamma$.
    For $m\in \{m_\pm\}$ we
    have $|L_2(m)| \le \factorsum(C_3)$, where $\factorsum$ is the function of \fullref{lem:count_domains}.
\end{proposition}
\begin{proof}
  The lower bound on $R_\epsilon$ guarantees that the subsegment $\gamma'$ of $\gamma$ centered about $m$ of
  radius $C_3$ exists and does not contain $x$ or $y$.
 Let $u$ and $v$ be the end vertices of $\gamma'$.
 Let $\mathcal{P}_i$, $\mathcal{D}_i$, and $\mathcal{F}_i$ be parent,
 descendant, and factor carriers for $u$ and $v$ at
 levels $i=1,\ldots,|S|-1$, with the assumption that $\gamma'$ is the
 geodesic used to define $\mathcal{D}_1$.
By \fullref{lem:elements_of_Fi_are_pre-wide}, all of the factor
carriers are pre-wide, hence every parent carrier is pre-wide except $W\in\mathcal{P}_1$.
 Our goal is to show that  $L_2(m) \subset
 \bigcup \mathcal{F}_i$.
 By \fullref{lem:count_domains}, this would imply the desired bound $|L_2(m)| \le \factorsum(C_3)$. 
  
 Let $P \in L_2(m)$. Since $\dcone(m,\rho(P))\le C_2$ and
 $\dcone(m,u)=\dcone(m,v)=C_3=C_2+B_0+1$, 
 both segments $[x,u]$ and $[v,y]$ are more than $B_0$ away from
 $\rho(P)$.
 Therefore, by \fullref{cor:BGIII}, $\delta_P(x,u),\, \delta_P(v,y) \le K_0$.
 Hence:
  \begin{equation} \label{eqn:biggerthan2L}
      \delta_P(u,v) \geq \delta_P(x,y) - 2K_0 > a_{\rank(P)}-2K_0 
  \end{equation}
  We now make the following claim.

  \begin{claim}
     Let $P \in L_2(m)$. If $P \lneq Q$ for some $Q \in \mathcal{P}_i$, then either $P \in \mathcal{F}_i$ or 
     $P \lneq Q'$ for some $Q' \in \mathcal{P}_{i+1}$.
  \end{claim}
  \begin{claimproof}
    We deal with the case of $i=1$ and $i>1$ separately. They are
    similar, but the definition of $\mathcal{D}_1$ is slightly
    different than $\mathcal{D}_i$ for $i>1$. 

    For $i=1$, recall that $\mathcal{D}_1$ is the collection of
  wide parabolic subgroups  $\stab(\Upsilon)$, where $v_\Upsilon$ is a cone vertex on $\gamma'$.
    We  show that $P \le \stab(\Upsilon)$ for some $\stab(\Upsilon) \in
    \mathcal{D}_1$.
    If not, then by \fullref{lem:projection}, the $\delta_P$--diameter of
    $\Upsilon$ is at most 1, for every such $\Upsilon$.
    Thus, every cone shortcut step in $\gamma'$ has $\delta_P$--diameter at most 1.
    Every $\Davis^{(1)}$ step in $\gamma'$ also has $\delta_P$--diameter
    at most 1, so $\delta_P(u,v)\le 2C_3$.
    Since $a_{\rank(P)}-2K_0  > 2C_3$, this contradicts
    \eqref{eqn:biggerthan2L}.
    Therefore, there exists a cone vertex $v_\Upsilon$ on $\gamma'$ such that
    $P\leq\stab(\Upsilon)$.
Since $P$ is irreducible and nonspherical, $P\leq Q$ for one of the
canonical nonspherical irreducible factors $Q$ of $\stab(\Upsilon)$.
By construction, such a $Q$ belongs to $\mathcal{F}_1$, so if $P=Q$ we
are done.
Otherwise, $P\lneq Q$, so the definition of $L_2(m)$ requires that
$\delta_Q(x,y) \le b_{\rank(Q)}$.
Since $P$ and $Q$ are
pre-wide we can apply
\fullref{lem:rho_map_from_prewide_to_cone_vertices}~\eqref{item:rho_map_nested}
to get:
\[\dcone(m, \rho(Q)) \leq 1+\dcone(m,\rho(P))\le 1+C_2\]
    Since $\dcone(m,\rho(P))$ is half-integral and $C_2$ is integral,
    this is actually a strict inequality, so the subsegments of $\gamma$ from $x$ to $u$ and from $v$ to
    $y$ are both more than $\dcone$--distance $B_0$ away
    from $\rho(Q)$, so \fullref{cor:BGIII} gives
    $\delta_Q(x,u),\delta_Q(y,v)\leq K_0$.
    Then the  triangle inequality yields:
    \[\delta_Q(u,v) \le \delta_Q(x,y) + 2K_0 \le b_{\rank(Q)}+2K_0\]
    This shows $Q \in \mathcal{P}_2$.

    For $i>1$, suppose $P\lneq Q \in \mathcal{P}_i$.
    By definition of $\mathcal{P}_i$, this implies $\delta_Q(u,v)\leq b_{\rank(Q)}+2K_0$.
    Let $\hat{\gamma}_Q$ be the geodesic in $C(Q)$ connecting
    $\gate_Q(u)$ to $\gate_Q(v)$ used to define $\mathcal{D}_{i}$.
    It has length at most $b_{\rank(Q)}+2K_0$.
     For any $\Upsilon_Q\in\fibers(Q)$,
    \fullref{lem:CP_well_defined}, says $C(Q)$ is isomorphic to
    $C(\Upsilon_Q)$, and 
    \fullref{lem:rho_map_from_prewide_to_cone_vertices}~\eqref{item:rho_map_extra}
    says $\delta_P(u,v)=\delta_P(\gate_{\Upsilon_Q}(u),\gate_{\Upsilon_Q}(v))$.
    The same argument as in the $i=1$ case gives that  if
    $P\not\leq\stab(\Upsilon)$ for every cone vertex $v_\Upsilon$ of
    $\hat{\gamma}_Q$, then $\delta_P(\gate_{\Upsilon_Q}(u),\gate_{\Upsilon_Q}(v))\leq
    b_{\rank(Q)}+2K_0$.
    However, since $\rank(P) <\rank(Q)$ and $a_{\rank(P)}>
    b_{\rank(Q)}+4K_0$, this contradicts \eqref{eqn:biggerthan2L}.
    Thus, there is some cone vertex $v_\Upsilon$ on $\hat{\gamma}_Q$ with $P
    \le \stab(\Upsilon)$, and, by construction, $\stab(\Upsilon)\in \mathcal{D}_i$.
    Let $Q'$ be the irreducible nonspherical factor of
    $\stab(\Upsilon)$ containing $P$.
    Then $Q' \in \mathcal{F}_i$.
    A similar argument to the $i=1$ case gives that either $P=Q'\in
    \mathcal{F}_i$, or $Q'\in\mathcal{P}_{i+1}$.
  \end{claimproof}

  Now we will apply the claim to see that eventually
  $P\in\mathcal{F}_i$.
  At $i=1$ we have 
   $\mathcal{P}_1=\{W\}$, so $P\leq Q_1:=W\in\mathcal{P}_1$, but $P$ is certainly not
   all of $W$ since that would mean $W$ is wide and $\coneoff$ is bounded.
   Thus, $P\lneq Q_1\in\mathcal{P}_1$, and the assumption of the claim holds at
   level $1$, so the claim says that either $P\in\mathcal{F}_1$, in which case we are done, or
   $P\lneq Q_2$ for some $Q_2\in\mathcal{P}_2$. 
Continue inductively. Since $\rank(\mathcal{P}_{j})\leq |S|+1-j$,
there is some first $j_0\leq |S|$ such that 
$\mathcal{P}_{j_0}=\emptyset$.
Thus, $P\in\mathcal{F}_i$ for some $i<j_0\leq |S|$.
\end{proof}

\begin{proposition} \label{prop:stability}
  Let $x,y\in\Davis^{(0)}$ with $\dcone(x,y)\geq
  R_\epsilon>3(C_3+\epsilon+1)$.
  For any $\dcone$--geodesic from $x$ to $y$, let $m_\pm$ be its
  trisection vertices.
  For all $m\in \{m_\pm\}$, $g\in A_\epsilon(x,y)$, and $P\in L_1(m)$, we have $gPg^{-1}\in L_2(m)$.
\end{proposition}
\begin{proof}
Let  $P':=gPg^{-1}$.
  The definition of $L_1(m)$ gives $\dcone(m,\rho(P))\leq C_1$.
  Since $g\rho(P)=\rho(P')$ and $\dcone(m,gm)\le H_\epsilon$, we have
  $\dcone(m,\rho(P')) \le H_\epsilon+C_1 =C_2$. This shows the second
  condition of $P'\in L_2(m)$.
  
  By our assumption on $R_\epsilon$, the $\dcone$--distance from
  $\rho(P')$ to any geodesic from $x$ to $gx$ is at least:
  \begin{equation}
    \label{eq:3}
    \dcone(m,x)-\dcone(m,\rho(P'))-\dcone(x,gx)\geq R_\epsilon/3-1-C_2-\epsilon>C_3-C_2=B_0+1
  \end{equation}
Therefore, by
  \fullref{cor:BGIII}, 
  $\delta_{P'}(x,gx) \le K_0$.
  Similarly, $\delta_{P'}(y,gy) \le K_0$. Hence:
  \begin{align*}
    \delta_{P'}(x,y) 
    &\ge \delta_{P'}(gx,gy)-\delta_{P'}(gx,x) - \delta_{P'}(y, gy) \\
    & \ge \delta_P(x,y) - 2K_0 \\
    & > c_{\rank(P)}-2K_0 \\
    &>a_{\rank(P)}\\
    &=a_{\rank(P')} 
  \end{align*}
  This shows the first condition of $P'\in L_2(m)$.
   
  Lastly, suppose that $P' \lneq Q$, for some pre-wide $Q$.
  Suppose that $\delta_Q(x,y)>b_{\rank(Q)}$.
  Then, by
  \fullref{lem:rho_map_from_prewide_to_cone_vertices}~\eqref{item:rho_map_nested},
  $\dcone(\rho(P'),\rho(Q))\leq 1$.
  The extra $+1$ slack in \eqref{eq:3} accommodates an extra $\dcone(\rho(P'),\rho(Q))$ term in a similar computation, so again  \fullref{cor:BGIII} implies
  $\delta_Q(x,gx), \delta_Q(y,gy) \le K_0$ and then:
  \begin{align*}
    \delta_{g^{-1}Qg}(x,y) 
    &= \delta_Q(gx,gy)\\
    &\ge \delta_Q(x,y) - \delta_Q(x,gx) - \delta_Q(y,gy) \\ 
    &> b_{\rank(Q)}-2K_0 > c_{\rank(Q)}
  \end{align*}
  This contradicts $P \in L_1(m)$, so we must have had
  $\delta_Q(x,y)\leq b_{\rank(Q)}$, which gives the final condition
  for $P'\in L_2(m)$.
\end{proof}

\begin{lemma} \label{prop:small_projection}
  Fix $x,y\in\Davis^{(0)}$ and a $\dcone$--geodesic $\gamma$ from $x$
  to $y$ in $\coneoff$. Let $m_\pm$ be the trisection vertices of
  $\gamma$. 
   If $L_1(m)=\emptyset$ for some
    $m \in \{m_\pm\}$ and $\Upsilon$ is a wide parabolic subcomplex
    with $\dcone(m,v_\Upsilon)\le C_0$, then $d_\Upsilon(x,y) \le M_\epsilon$.
\end{lemma}
\begin{proof}
  Suppose $d_\Upsilon(x,y) > M_\epsilon$.
  Then \fullref{prop:witness}  says there exists some pre-wide witness
  $P \le
  \stab(\Upsilon)$ with $\delta_P(x,y) > c_{\rank(P)}$.
If $\Upsilon=w\Davis_T$ then $P\leq\stab(\Upsilon)$ implies
$P=wvW_Uv^{-1}w^{-1}$ for some $v\in W_T$ and $U\subset T$.
Then $\Xi:=wv\Davis_U$ has $\stab(\Xi)=P$ and $\Xi\subset\Upsilon$.
On the other hand, by \fullref{lem:fibers}, $\Xi$ is a subcomplex of
the unique parabolic subcomplex with stabilizer $\normalizer_W(P)$, which is the subcomplex corresponding to the cone vertex $\rho(P)$.
Hence, the cone vertices $\rho(P)$ and $v_\Upsilon$ have every vertex of $\Xi^{(0)}$ as mutual neighbors, so  $\dcone(v_\Upsilon,\rho(P))\le 1$.

Take a maximal $Q$ among the set of pre-wide parabolics that contain
$P$ and satisfy $\delta_Q(x,y) > c_{\rank(Q)}$. The set is not empty,
since it contains $P$, and chains are uniformly bounded in length by $|S|-2$,
since $\rank(P)\geq 2$ and strict containment of parabolic subgroups
forces strict increase in rank. 
Thus, a maximal element exists. 

Since $P$ and $Q$ are nested,
\fullref{lem:rho_map_from_prewide_to_cone_vertices}~\eqref{item:rho_map_nested}
says $\dcone(\rho(P),\rho(Q)) \leq 1$.
Therefore:
\[\dcone(m,\rho(Q))\leq\dcone(m,v_\Upsilon)+\dcone(v_\Upsilon,\rho(P))+\dcone(\rho(P),\rho(Q))\leq C_0+2=C_1\]
We have confirmed the three membership conditions for $Q \in L_1(m)$,
but $L_1(m)=\emptyset$, so this is a contradiction and we must have
had $d_\Upsilon(x,y)\leq  M_\epsilon$.
\end{proof}

\begin{proposition} \label{prop:no_large_links}
  There exists $D_\epsilon$ such that if $\dcone(x,y) \ge
  R_\epsilon>3(B_0+C_0+\epsilon+1)$, for any $\dcone$--geodesic
  $\gamma$ from $x$ to $y$ with trisection vertices $m_\pm$,  if $L_1(m)=\emptyset$ for some
    $m\in \{m_\pm\}$, then for any choice of a $\dcomb$--geodesic
    from $x$ to $y$ and a vertex $p$ on that geodesic
    $\dcone$--closest to $m$, we have $\dcomb(p,gp) \le D_\epsilon$ for all $g \in A_\epsilon(x,y)$.
\end{proposition}
\begin{proof}
We show it suffices to take $D_\epsilon:=C(1+
H_\epsilon+2J)(2M_\epsilon +K_0)$, where $C$ is
the constant of \fullref{lem:no_large_projection}, 
  $M_\epsilon$ is as in
  \fullref{prop:small_projection}, and $H_\epsilon$, $J$, and $K_0$ are as in
  \fullref{sec:roundup}. 

  Suppose $L_1(m)=\emptyset$.
  For any choice of $\dcomb$--geodesic from $x$ to $y$, take a vertex
  $p$ that is $\dcone$--closest to
  $m$.
  Then $\dcone(m,p)\leq J$, by \fullref{lem:cone_geod_close_to_comb_geod}.

  For all $g \in A_\epsilon(x,y)$, we have $\dcone(m,gm) \le
  H_\epsilon$, by \fullref{lem:fellow-travel}, so $\dcone(p,gp)\leq H_\epsilon+2J$.

  We claim  $d_\Upsilon(p,gp) \le 2M_\epsilon+K_0$, for all wide parabolic
  subcomplexes $\Upsilon$.
  If $d_\Upsilon(p,gp) \leq K_0$ we are happy, so suppose not.
  Then \fullref{thm:BGI}  implies that the $\dcone$--distance from
  $v_\Upsilon$ to any $\dcone$--geodesic from $p$ to $gp$ is at most $B_0$.
  Since the length of any such geodesic is at most $H_\epsilon+2J$,
  it follows that:
  \[\dcone(v_\Upsilon,m)\leq\dcone(v_\Upsilon,p)+\dcone(p,m)\leq  B_0+H_\epsilon+3J=C_0\]
  The same estimate holds using $gp$ and $gm$, so:
  \[\dcone(m,v_{g^{-1}\Upsilon})=\dcone(gm,v_\Upsilon)\leq
    \dcone(v_\Upsilon,gp)+\dcone(gp,gm)\leq C_0\]
  \fullref{prop:small_projection} gives $d_\Upsilon(x,y)\leq M_\epsilon$ and $d_{g^{-1}\Upsilon}(x,y)\leq M_\epsilon$.

  By \fullref{lem:gate-wall}, $d_\Upsilon(x,y)=|\walls_\Upsilon(x,y)|$.
  Since $p$ is a vertex on a $\dcomb$--geodesic between $x$ and $y$:
\[d_\Upsilon(x,p)\leq d_\Upsilon(x,p)+d_\Upsilon(p,y)=d_\Upsilon(x,y)\leq M_\epsilon\]
Similarly:
\[d_\Upsilon(gx,gp)=d_{g^{-1}\Upsilon}(x,p)\leq d_{g^{-1}\Upsilon}(x,y)\leq M_\epsilon\]

Any $\dcone$--geodesic
from $x$ to $gx$ has $\dcone$--distance from $v_\Upsilon$ at least:
\[\dcone(m,x)-\dcone(m,v_\Upsilon)-\dcone(x,gx)\geq R_\epsilon/3-1 -C_0-\epsilon>B_0\]
Then \fullref{thm:BGI} implies $d_\Upsilon(x,gx) \le K_0$.

Putting the last three bounds together:
\begin{align*}
  d_\Upsilon(p,gp)&\leq d_\Upsilon(p,x)+d_\Upsilon(x,gx)+d_\Upsilon(gx,gp)\\
  &\leq M_\epsilon+K_0+M_\epsilon
\end{align*}
This proves the claim.

Now apply \fullref{lem:no_large_projection}:
\begin{align*}
  \dcomb(p,gp)&\leq C(1+\dcone(p,gp))\cdot\max\{1,\sup_{\text{wide }\Upsilon} d_\Upsilon(p,gp)\}\\
  &\leq C(1+ H_\epsilon+2J)(2M_\epsilon+K_0)\qedhere
\end{align*} 
\end{proof}

\subsubsection{Proof of acylindricity}

Choose $R_\epsilon$ large enough so that:
\begin{equation}
  \label{eq:R_conditions}
  R_\epsilon/3-1>\max\{C_3+\epsilon,\, \maxsphere+2C_1+4\}
\end{equation}
The first condition guarantees $R_\epsilon$ is
sufficiently large to apply \fullref{prop:L2finite}, 
\fullref{prop:stability}, and \fullref{prop:no_large_links}.

\begin{proof}[Proof of \fullref{thm:acylindricity}]
 Let $D_\epsilon$ be the constant of \fullref{prop:no_large_links}.
 Let $R_\epsilon$ satisfy \eqref{eq:R_conditions}.
 Let $x,y\in\Davis^{(0)}$, with $\dcone(x,y)\geq R_\epsilon$,  fix a
 $\dcone$--geodesic from $x$ to $y$, and let $m_\pm$ be its trisection vertices. 
  We have two cases:

  \paragraph{\bf Case A} $L_1(m)=\emptyset$ for some $m \in \{m_\pm\}$.
  
  Then, by Proposition \ref{prop:no_large_links},  for any choice of a
  $\dcomb$--geodesic from $x$ to $y$ and a vertex $p$ on that geodesic
    $\dcone$--closest to $m$, we have $\dcomb(p,gp) \le D_\epsilon$ for
    all $g \in A_\epsilon(x,y)$.
    Since $W$ acts freely on $\Davis^{(0)}$, which is locally finite,
    this implies:
    \[|A_\epsilon(x,y)|\leq |\{w\in W\mid \dcomb(p,wp)\leq
      D_\epsilon\}|=|\bar\nbhd_{D_\epsilon}(1)|<\infty\]

\paragraph{\bf Case B} $L_1(m) \ne \emptyset$ for both $m\in \{m_\pm\}$.
  
  Fix a pair $P_{-}\in L_1(m_{-})$ and $P_+\in L_1(m_+)$. By
  \fullref{prop:stability}, for any $g\in A_\epsilon(x,y)$:
  \[ gP_{-} g^{-1}\in L_2(m_{-})\quad\text{and}\quad gP_{+} g^{-1}\in L_2(m_{+})\]
  
  Define the map:
  \[ \Phi\from A_\epsilon(x,y) \rightarrow L_2(m_-) \times L_2(m_+) : g \mapsto (gP_-g^{-1},gP_+ g^{-1})
\]

 By \fullref{prop:L2finite}, $|L_1(m)| \le |L_2(m)| \le \factorsum(C_3)$,
 so the cardinality of the image of $\Phi$ is bounded by
 $\factorsum(C_3)^2$. We now bound the cardinality of each fiber.
 Suppose $g, h \in A_\epsilon(x,y)$ are in the same fiber, so:
\[gP_{-} g^{-1}=hP_{-} h^{-1}\quad\text{and}\quad gP_+ g^{-1}=hP_+ h^{-1}\]

 Equivalently, $h^{-1}g \in \normalizer_W(P_-)
 \cap \normalizer_W(P_+)$.
 By \fullref{def:strong_weak}, $P_{-}\in L_1(m_{-})$ implies
 $\dcone(m_{-},\rho(P_{-})) \le C_1$ and, similarly, $\dcone(m_{+},\rho(P_{+})) \le C_1$.
 Now:
 \[\dcone(\rho(P_{-}),\rho(P_+))\geq \dcone (m_-,m_+)-2C_1 \geq \dcone(x,y)/3-2-2C_1\geq R_\epsilon/3-2-2C_1 \stackrel{\eqref{eq:R_conditions}}{>} \maxsphere+3\]
 Then the contrapositive of \fullref{cor:close-cone} implies
 $\normalizer_W(P_-) \cap \normalizer_W(P_+)$ is a spherical parabolic.  
It follows that $|A_\epsilon(x,y)|$ is bounded by $\factorsum(C_3)^2$ times the
maximum cardinality of a spherical parabolic subgroup of $(W,S)$. 
  
 Acylindricity for $x,y\in\Davis^{(0)}$ follows by taking:
\[N_\epsilon:=\max\{|\bar\nbhd_{D_\epsilon}(1)|,
 \factorsum(C_3)^2\cdot\max\{|W_T|\mid \text{spherical } T\subset
 S\}\}\]

For the general case, use the above $R_\epsilon$ and $N_\epsilon$ to define $R'_\epsilon:=R_{\epsilon+1}+1$ and $N'_\epsilon:=N_{\epsilon+1}$.
Suppose $x',y'\in\coneoff$ with $\dcone(x',y')\geq R'_\epsilon$.
Choose closest $\Davis^{(0)}$ vertices $x$ to $x'$ and $y$ to $y'$. 
Then $\dcone(x,x'),\dcone(y,y')\leq 1/2$, so $\dcone(x,y)\geq
R_{\epsilon+1}$ and $A_\epsilon(x',y')\subset A_{\epsilon+1}(x,y)$.
Thus: \[\dcone(x',y')\geq R'_\epsilon\implies|A_\epsilon(x',y')|\leq
|A_{\epsilon+1}(x,y)|\leq N_{\epsilon+1}=N'_\epsilon\qedhere\]
\end{proof}

%% file: stable.tex

We already have a characterization of subgroup stability in terms of a geometric condition that is uniform over the whole subgroup:
\begin{corollary}\label{cor:qie_stable}
  A finitely generated subgroup $G\leq W$ is stable if and only if for some/every finite generating set $S_G$ of $G$ the orbit map of $G$ into $\coneoff$, obtained by sending $g\in G$ to $g.\one\in\Davis^{(0)}$ and then including into $\coneoff$, is a quasiisometric embedding $(G,d_{S_G})\to (\coneoff,\dcone)$, where $d_{S_G}$ is the word metric on $G$ with respect to the generating set $S_G$. 
\end{corollary}
This follows by combining \fullref{prop:stable_subgroup_qi_embeds} and \fullref{thm:coneoff_qi_to_dfine}.

In this section we give an element-wise characterization.

We first recall the mapping class group situation.
\fullref{cor:qie_stable} is analogous to the result that a subgroup of MCG is stable if and only if it is finitely generated and the orbit map is a quasiisometric embedding into the curve graph \cite{KenLei08,Ham05,DurTay15}.
Bestvina, Bromberg, Kent, and Leininger \cite{BBKL20} give a different characterization\footnote{Here we elide a difference between stability and convex cocompactness in exceptional low complexity cases.}: a finitely generated subgroup of the mapping class group is stable if and only if it is undistorted and purely pseudo-Anosov, where \emph{purely pseudo-Anosov} means that every infinite order element of the subgroup is a pseudo-Anosov surface homeomorphism class.
The pseudo-Anosovs are precisely the elements that are loxodromic for the action on the curve graph, so ``purely pseudo-Anosov'' could be equivalently described as being ``purely curve-graph-loxodromic'', or just ``purely loxodromic'', with the action on the curve graph implicit.
Leaving the action implicit is natural here since the action on the curve graph gives the largest acylindrical hyperbolic structure on MCG.
This is also the terminology used in the RAAG and graph product cases \cite{KobManTay17,BalCheKer25}.
We make the corresponding definition for Coxeter groups:
\begin{definition}
  Let $\Davis=\Davis(W,S)$ be the Davis complex of a finite rank Coxeter system. 
  A subgroup $G$ of $W$ is \emph{purely loxodromic} if every infinite order element of $G$ acts loxodromically on the coned-off space $\coneoff$.
  Equivalently, for every infinite order element $g\in G$, the orbit map gives a quasiisometric embedding $\mathbb{Z}\to\coneoff : n\mapsto g^n\!.\one$. 
\end{definition}
\begin{definition}
    A finitely generated subgroup of a finitely generated group is \emph{undistorted} if the inclusion map is a quasiisometric embedding, with respect to word metrics for any choices of finite generating sets of the two groups.
\end{definition}

The goal of this section is to prove the following analogue of \cite[Main Theorem]{BBKL20}:

\begin{theorem} \label{thm:chara_stable}
  A finitely generated subgroup of $W$ is stable if and only if it  is undistorted and purely loxodromic.
\end{theorem}

\begin{remark}
  It is not known if the analogue of \cite[Main Theorem]{BBKL20} is true for every HHG with a universal action on the maximal hyperbolic space of its HHG structure, \cite[Question~C]{AbbBehDur21}.
  \fullref{thm:chara_stable} does not address that question in that generality, but, interestingly, provides some examples where the corresponding conclusion is true even though the Coxeter group is known \emph{not} to be an HHG, recall \fullref{fn:azuhag}.
  These are not the first examples of this phenomenon; stability in many graph products can be characterized in terms of being purely loxodromic \cite{BalCheKer25}. It is easy to make non-HHG graph products by taking one of the vertex groups to have a distorted finitely generated Abelian subgroup (Plummer, unpublished, cf \cite[Proposition~2.17]{HagRusSis25}).
  Our examples are somewhat more subtle; since Coxeter groups are CAT(0) groups, this kind of easy obstruction to being an HHG is not available.
\end{remark}
\begin{remark}
  In the MCG setting, it is an open question, asked by Farb and Mosher \cite{FarMos02} in the free case, whether the characterization of stable subgroups requires ``undistorted'' as a separate hypothesis. Stable subgroups are undistorted, so this is really a question about the diversity of purely pseudo-Anosov subgroups.
  For arbitrary hyperbolic groups ``undistorted'' really is necessary as a hypothesis: consider a hyperbolic group constructed as the mapping torus of pseudo-Anosov homeomorphism of a closed hyperbolic surface.
  The surface group fiber is a hyperbolic subgroup of a hyperbolic group, so it is purely loxodromic with respect to its action on the ambient hyperbolic group, but it is distorted, hence not stable.
  Every Coxeter group acting as a cocompact geometric reflection group on $\mathbb{H}^3$ is virtually of this form \cite[Theorem~5.2]{Ago08}, 
 so undistorted is necessary as a separate hypothesis in our theorem.
 Notice further that if $W$ is hyperbolic then it has no wide parabolic subgroups to cone off, so $\coneoff$ is just $\Davis^{(1)}$, which is the Cayley graph of $W$ with respect to $S$, so loxodromic for the action on the ambient group and loxodromic for the action on $\coneoff$ coincide in the hyperbolic case.
\end{remark}
\begin{remark}
There are several classes of finitely generated, purely pseudo-Anosov subgroups of MCG that are known to be stable without an additional undistortedness hypothesis, see, for example \cite{LeiRus23} and the references in its Section~1.1.
It would be interesting to know if there are natural classes of examples of finitely generated, purely loxodromic subgroups of Coxeter groups in which undistortedness is automatic. 
\end{remark}

\medskip

This subsection has the same architecture as \cite{BBKL20}.
\fullref{prop:criterion} corresponds to \cite[Proposition~3.1]{BBKL20}, \fullref{lem:H_subword} and \fullref{prop:H_subword} correspond to \cite[Proposition~4.1]{BBKL20}, and the proof of \fullref{thm:chara_stable} imitates the proof of \cite[Main Theorem]{BBKL20}.
The main difference is that \cite[Proposition~4.1]{BBKL20} uses extensive MCG/subsurface machinery to decompose a word $g$ into subwords with linearly summing projections.
We replace that machinery with greedy paths and wall combinatorics, resulting in an expression of $g$ as a combination of possibly overlapping subpaths, but with enough control on the overlaps to reach a similar conclusion. 

\medskip

Throughout this section we have a finitely generated subgroup $G\leq W$, for a finite rank Coxeter system $(W,S)$. We will choose a finite generating set $S_G$ of $G$, and let $|\cdot|_G$ denote word length in $G$ with respect to $S_G$, and define $\lambda_G:=\max_{s\in S_G}|s|_S$.
\begin{lemma} \label{lem:finite_big_projection}
  Let $(W,S)$ be a finite rank Coxeter system. For all $g\in W$, we have $|\operatorname{BigProj}(g)|\leq 2^{|g|_S}$, where:
  \[\operatorname{BigProj} (g) := \{ \text{nonspherical irreducible parabolic subgroup }Q\mid \delta_Q(1,g)>1 \} \]
\end{lemma}
\begin{proof}
  Suppose $\Upsilon$ is a parabolic subcomplex with stabilizer $Q$ such that  $\delta_Q(1,g)> 1$.
  Let $u:=\gate_\Upsilon(1)$ and $v:=\gate_\Upsilon(g)$.
  Then $ d_{C(\Upsilon)}(u,v)=\delta_\Upsilon(1,g)=\delta_Q(1,g)>1$, so there is no proper parabolic subcomplex of $\Upsilon$ containing both $u$ and $v$.
  Since $u$ and $v$ are vertices of $\Upsilon$, every wall separating them cuts $\Upsilon$, hence has reflection in $Q=\stab(\Upsilon)$. Thus, $\pc(\walls(u,v))$ is a parabolic subgroup of $Q$.

  Choose a combinatorial geodesic from $1$ to $u^{-1}v$, whose length is $n:=\dcomb(u,v)$.
  Let $s_i\in S$ be the label of the $i$th edge, and let $r_i$ be reflection in that edge. Then $r_i=s_1\cdots s_{i-1}s_is_{i-1}\cdots s_1$.
  Conversely, $s_i=r_1\cdots r_{i-1}r_ir_{i-1}\cdots r_1$, so if
    $T:=\{s_1,\dots,s_n\}\subset S$ then $W_T=W(r_1,\dots,r_n)$.
    Translating by $u$, the successive reflections along a $\dcomb$--geodesic from $u$ to $v$ are $ur_iu^{-1}$, so $v=(ur_nu^{-1})\cdots(ur_1u^{-1})u$, so $u^{-1}v=r_n\cdots r_1=s_1\cdots s_n$.
    Thus, $v\in uW_T=uW(r_1,\dots,r_n)=W(ur_1u^{-1},\dots,ur_nu^{-1})u\subset\pc(\walls(u,v))u\subset Qu$.
    Now, $uW_T\subset Qu\subset\Upsilon^{(0)}$, so $u\Davis_T$ is a parabolic subcomplex of $\Upsilon$.
It contains both $u$ and $v$, so it is not a proper subcomplex of $\Upsilon$, so $\pc(\walls(u,v))=uW_Tu^{-1}=\stab(u\Davis_T)=\stab(\Upsilon)=Q$.

By \fullref{lem:gate-wall}, $\walls(u,v) \subset \walls(1,g)$.
Thus, a parabolic subgroup $Q\in \operatorname{BigProj} (g)$ can be specified as the parabolic closure of a set of reflections corresponding to some subset of the walls separating $1$ and $g$.
There are only $2^{|g|_S}=2^{|\walls(1,g)|}$--many distinct subsets of $\walls(1,g)$, so $|\operatorname{BigProj} (g)|\leq 2^{|g|_S}$.
\end{proof}

\begin{proposition}[Criterion for reducibility/normalizing] \label{prop:criterion}
 Let $(W,S)$ be a finite rank Coxeter system, and let $G \le W$ be a subgroup generated by a finite set $S_G$.
For all $c>0$ there exists $R>0$ such that if there exists $g \in G$ and a pre-wide parabolic subgroup $P$  with $|g|_G \ge R$ and $d_P(1,g) \ge c|g|_G$, then $G$ contains a nontrivial element normalizing some pre-wide parabolic subgroup.
\end{proposition}
\begin{proof}
  Suppose not, so there is some $c$ for which the conditions are satisfied for arbitrarily large $R$, but $G$ contains no nontrivial element normalizing a pre-wide parabolic subgroup. 
  Let $r$ be the smallest rank for which there exists $a>0$ such that there are $g\in G$ of unbounded length for which there exists a pre-wide parabolic $P$ of rank $r$ with $d_P(1,g) \ge a|g|_G$.
  Since the ranks of pre-wide parabolics are between $2$ and $|S|$, to establish that such a minimal $r$ exists it is enough to argue that we are not minimizing over the empty set.
  To see this, take $a$ to be the offending $c$.
  Then there are arbitrarily long $g\in G$ such that there exists a pre-wide $P$ satisfying $d_P(1,g) \ge c|g|_G$.
  Since these elements $g$ have unbounded lengths, there are infinitely many of them, but only finitely many possible ranks of the corresponding $P$, so some rank reoccurs infinitely often. 

Notice that if $0<a'\leq a$ then $d_P(1,g)\geq a|g|_G\geq a'|g|_G$, so we may assume $0<a<1$.

  For every $q<r$ and every $b>0$, define:
  \[R(q,b):=\sup\{R\mid \exists g\in G, \text{ pre-wide }P,\,\rank(P)=q\text{ and }d_P(1,g)/b\geq |g|_G>R\}\]
  Minimality of $r$ implies $R(q,b)<\infty$ for all $q<r$ and $b>0$.
  Recalling that $\maxsphere:=\max_{\text{spherical }T\subset S}\diam_{\dcomb} W_T$ and $\sup\emptyset=-\infty$, choose an integer:
  \[L>\max\left\{\frac{4\maxsphere}{a},\sup_{2\leq q<r}R(q,\frac{a}{4\lfloor(r-1)/2\rfloor})\right\}\] 

  By construction, there are arbitrarily long elements $g\in G$ such that there exists a pre-wide parabolic $P$ of rank $r$ with $d_P(1,g)\geq a|g|_G$.
  Take an $S_G$--geodesic word representing such a $g$, and cut it into $n:=\lfloor |g|_G/L\rfloor$--many consecutive subwords of length $L$, plus a possibly empty remainder word:
  \[g=g_1\cdots g_ng_{n+1}\]
  Let $p_0:=1$, and for  $i\leq n$, let $p_i:=p_{i-1}g_i$ be the $i$th prefix of $g$, which has length $Li$. For $1\leq i\leq n+1$ let $P_i:=p_{i-1}^{-1}Pp_{i-1}$.

  Suppose for some $1\leq i\leq n$ we have $d_{P_i}(1,g_i)\geq La/2$ and $\delta_{P_i}(1,g_i)= 1$. Then $d_{P_i}(1,g_i)=d_{P_i}(1,g_i)/\delta_{P_i}(1,g_i)>\maxsphere$, so \fullref{witness_claim} produces a nonspherical irreducible parabolic $Q\lneq P_i$ with:
  \[d_Q(1,g_i)\geq \frac{d_{P_i}(1,g_i)-\maxsphere}{\lfloor (r-1)/2\rfloor}>\frac{La/4}{\lfloor (r-1)/2\rfloor}\]
  Note in this case that $2\leq\rank(Q)<\rank(P)=r$, so the denominators are positive.
  Rearranging gives:
  \begin{equation}
    \label{eq:14}
    \frac{d_Q(1,g_i)}{\frac{a}{4\lfloor (r-1)/2\rfloor}}>L
  \end{equation}
  Since $P_i$ is pre-wide, \fullref{lem:pre-wide} says $Q$ is pre-wide, so \eqref{eq:14} shows that $R(\rank(Q),\frac{a}{4\lfloor (r-1)/2\rfloor})\geq L$, which contradicts the definition of $L$. 
Thus:
\begin{equation}
  \label{eq:19}
 \forall 1\leq i\leq n,\, d_{P_i}(1,g_i)\geq La/2\implies\delta_{P_i}(1,g_i)>1
\end{equation}
Now consider:
\[d_P(1,g)\leq\sum_{i=1}^n d_P(p_{i-1},p_i)+d_P(p_n,g)=\sum_{i=1}^{n+1} d_{P_i}(1,g_i)\]
For all $i$ we have $d_{P_i}(1,g_i)\leq |g_i|_S\leq \lambda_G|g_i|_G\leq \lambda_G L$.

Suppose there are $(n-k)$--many indices between $1$ and $n$ with $d_{P_i}(1,g_i)< La/2$.
Then:
\[anL\leq a|g|_G\leq d_P(1,g)\leq (n-k)La/2+(k+1)\lambda_G L\]
Since $\lambda_G\geq 1$ and $0<a<1$, we have $\lambda_G-a/2>0$, so this expression can be rearranged to bound $k$:
\[k\geq\frac{an/2-\lambda_G}{\lambda_G-a/2}=\frac{(a/2)\lfloor\frac{|g|_G}{L}\rfloor-\lambda_G}{\lambda_G-a/2}\]
Conclude that the number of indices $i$ for which $d_{P_i}(1,g_i)\geq La/2$ grows linearly with $|g|_G$.
Applying \eqref{eq:19}, the number of indices $i$ for which $\delta_{P_i}(1,g_i)>1$ grows linearly with $|g|_G$.

Take the union of the sets $\operatorname{BigProj}(h)$ of \fullref{lem:finite_big_projection} as $h$ ranges over the $L$--sphere in $G$ with respect to $|\cdot|_G$.
The sphere contains less than $(2|S_G|)^L$--many elements $h$, and \fullref{lem:finite_big_projection} implies $|\operatorname{BigProj}(h)|\leq 2^{\lambda_G L}$, so this is a finite collection of nonspherical irreducible parabolic subgroups whose cardinality depends on $L$, not on $|g|_G$. 
Thus, whenever $|g|_G$ is long enough and $P$ is its corresponding parabolic, the decomposition $g=g_1\cdots g_{n+1}$ has more indices $1\leq i\leq n$ for which $\delta_{P_i}(1,g_i)>1$ than there are potential distinct parabolics $P_i$, so there are indices $i<j$ such that $P_i=P_j$.
Equivalently, $p_{i-1}^{-1}Pp_{i-1}=p_{j-1}^{-1}Pp_{j-1}$, so:
\[p_{j-1}p_{i-1}^{-1}\in \normalizer_W(P)\cap G\]
This element is nontrivial, since $p_{i-1}$ and $p_{j-1}$ are distinct prefixes of a geodesic word.
This contradicts the supposition that $G$ has no nontrivial element normalizing a pre-wide parabolic.
\end{proof}

For any $g \in W$, let $\alpha_g$ be a combinatorial geodesic from $1$ to $g$ in $\Davis$, and let $\greedy{\alpha_g}$ be the greedy path in the shortcut graph of $\alpha_g$.
Choose a realization $\realgreedy{\alpha}_g$ of $\greedy{\alpha_g}$ in $\coneoff$.
Let $x$ and $y$ be a pair of adjacent vertices in $\greedy{\alpha_g}$.
Either $x$ and $y$ are adjacent in $\realgreedy{\alpha}_g$ and in  $\alpha_g$, or there is a cone path shortcut $x-v_\Upsilon-y$ in $\realgreedy{\alpha}_g$ for some wide parabolic subcomplex $\Upsilon$ containing $x$ and $y$.
Since parabolic subcomplexes are $\dcomb$--convex, the subsegment of $\alpha_g$ between $x$ and $y$ is contained in $\Upsilon$. 
Call such a segment of $\alpha_g$ a \emph{wide segment}.

The walls $\walls(1,g)$ separating $1$ and $g$ form a finite poset under inclusion of the halfspaces containing 1.
Consider a Dilworth partition of $\walls(1,g)$ into $D \le \dim(\NR)$ many chains:
\[ \walls(1,g) = \mathcal{K}_1 \sqcup \cdots \sqcup \mathcal{K}_D \] 
For each wide segment $\alpha$ of $\alpha_g$, let $\walls(\alpha)$ be the set of walls transverse to its edges.
Pick an index $j$ such that $\walls(\alpha)\cap\mathcal{K}_j$ has maximal cardinality, and set $\mathcal{K}_\alpha:=\mathcal{K}_j$.
Maximality ensures $\mathcal{K}_\alpha$ contains at least the average number of $\walls(\alpha)$ walls per chain:
\begin{equation} \label{eqn:max}
  \frac{|\walls(\alpha)|}{D} \le |\walls(\alpha) \cap \mathcal{K}_\alpha|
\end{equation}

Partition $\alpha_g$ into subsegments of the following types.
\begin{itemize}
  \item $E_0$ consists of single ordinary $\Davis^{(1)}$ edges corresponding to an edge in $\greedy{\alpha_g}$.
  \item $E_1$ consists of wide segments $\alpha$ of $\alpha_g$ such that $|\walls(\alpha) \cap \mathcal{K}_\alpha| = 1$.
  \item $E_2$ are the remaining wide subsegments $\alpha$ of $\alpha_g$, which have  $|\walls(\alpha) \cap \mathcal{K}_\alpha| \geq 2$.
\end{itemize}
Order the segments in $E_2$ according to their appearance along $\alpha_g$, say $E_2 = \{\alpha_1,\ldots,\alpha_m\}$.
For each $\alpha_i$, let $\mathcal{V}_i = \walls(\alpha_i) \cap \mathcal{K}_{\alpha_i}$, and define:
\[ P_i = \pc(\mathcal{V}_i)\]
Because $\mathcal{V}_i$ is a chain with at least two walls, $P_i$ is a nonspherical irreducible parabolic subgroup.
Further, all walls of $\mathcal{V}_i$ cut any wide parabolic subcomplex $\Upsilon_i$ containing $\alpha_i$, so their reflections lie in its stabilizer:
\[ P_ i \le \stab(\Upsilon_i) \]
By \fullref{lem:pre-wide}, each $P_i$ is pre-wide.

Using the above notation, we establish the following statement.

\begin{lemma} \label{lem:H_subword}
  Let $(W,S)$ be a finite rank Coxeter system, let $G\leq W$ be a subgroup generated by a finite set $S_G$. 
  Given $g\in G$, identify it with an $S_G$--shortest representative word, which is the label of a geodesic $\beta_g$ from $1$ to $g$ in  the Cayley graph of $G$ with respect to $S_G$.
  Let $\alpha_g$ be a $\dcomb$--geodesic from $1$ to $g$ in $\Davis^{(1)}$.
  With notation as above, for each $\alpha_i \in E_2$, there exists a subword $g_i$ of $g$ and a conjugate $Q_i$ of $P_i$ such that:
  \begin{enumerate}
  \item $\forall i,\,Dd_{Q_i}(1,g_i)\geq |\alpha_i|$\label{item:H_subword_large_projection}
    \item $\sum_i |g_i|_G\leq (1+\lambda_G/2) D|g|_G$\label{item:H_subword_controlled_sum}
  \end{enumerate}
\end{lemma}
\begin{proof}
  Let $1=x_0,\ldots,x_n=g$ be the consecutive vertices of the geodesic $\beta_g$ in the Cayley graph of $G$. 
  For every wall $M \in \walls(1,g)$, orient its halfspaces by choosing $1 \in M^-$ and $g \in M^+$.
  Define the first-entry index of $M$ by: \[ s(M) := \min \{ j \mid x_j \in M^+\}\]
  Order any chain extracted from $\walls(1,g)$ according to our choices for $M^-$ and $M^+$, so that chain order is consistent with first-entry indices. 

For each $\alpha_i \in E_2$, let $M_i$ and $N_i$ be the first and last walls of $\mathcal{V}_i$, and set:
    \[ a_i := s(M_i)-1\quad\text{and} \quad b_i=s(N_i)\]
    Let $g_i := x_{a_i}^{-1} x_{b_i} \in G$ and $Q_i:=x_{a_i}^{-1}P_ix_{a_i}$. 
    Note that every wall of $\mathcal{V}_i$ separates $x_{a_i}$ and $x_{b_i}$ and has corresponding reflection in $P_i$, so \fullref{lem:gate-wall} gives:
\[    d_{Q_i}(1,g_i) = d_{P_i}(x_{a_i},x_{b_i}) \ge |\mathcal{V}_i| \ge \frac{|\walls(\alpha_i)|}{D} = \frac{|\alpha_i|}{D}\]
    This proves Item~\eqref{item:H_subword_large_projection}.
    
    Let $\beta_i$ be the subsegment of  $\beta_g$ from $x_{a_i}$ to $x_{b_i}$, which is nontrivial, since $a_i < b_i$.
    We can organize the segments $\beta_i$ according to the $D$--many chains in $\walls(1,g)$, as follows.
     Suppose $\beta_i$ and $\beta_j$, with $i<j$,  are defined using the same chain, i.e.\ $\mathcal{K}_{\alpha_i} = \mathcal{K}_{\alpha_j}$, with $\mathcal{V}_i=\walls(\alpha_i)\cap\mathcal{K}_{\alpha_i}$ and $\mathcal{V}_j=\walls(\alpha_j)\cap\mathcal{K}_{\alpha_j}$.
    Since $\alpha_g$ is a $\dcomb$--geodesic, the last wall $N_i$ of $\mathcal{V}_i$ occurs in $\mathcal{K}_{\alpha_i}$ before the first wall $M_j$ of $\mathcal{V}_j$, which implies $s(N_i)\leq s(M_j)$, so $b_i\leq a_j+1$.
        A strict inequality $b_i<a_j+1$ would mean that $\beta_i$ and $\beta_j$ have disjoint interiors, and equality would mean the last edge of $\beta_i$ is the first edge of $\beta_j$.
        This could happen; a single $S_G$--edge has vertices whose $\dcomb$--distance in $\Davis$ is at most $\lambda_G$, so such vertices could be separated by as many as $\lambda_G$--many distinct walls---in particular by both $N_i$ and $M_j$.

        Pushing this reasoning further, if $i<j_1<\cdots<j_{m}<k$ and $\beta_i$, $\beta_{j_h}$, and $\beta_k$ are all defined using the same chain of $\walls(1,g)$, and if $\beta_i$ and $\beta_k$ share an edge $e$, then a $\dcomb$--geodesic in $\Davis$ between the endpoints of $e$ crosses the last wall of $\mathcal{V}_i$ and the first wall of $\mathcal{V}_k$, so also crosses every wall of every $\mathcal{V}_{j_h}$. Since $|\mathcal{V}_{j_h}|\geq 2$, such a $\dcomb$--geodesic, which has length at most $\lambda_G$, crosses at least $(2+2m)$--many walls, so $m\leq \lambda_G/2-1$.
        Thus, the edge $e$ belongs to at most $(\lambda_G/2+1)$--many of the  $\beta$--paths from a single chain.
        Summing the lengths of the $\beta$ coming from a single chain, we have at most $|g|_G$--many edges each counted with multiplicity at most $\lambda_G/2+1$. There are $D$--many distinct chains to sum over, so we obtain the estimate of Item~\eqref{item:H_subword_controlled_sum}:
    \[ \sum_i |g_i|_G = \sum_i |\beta_i| \le (1+\lambda_G/2)D|g|_G \qedhere \] 
\end{proof}

\begin{proposition} \label{prop:H_subword}
  Let $(W,S)$ be a finite rank Coxeter system, let $G\leq W$ be a subgroup generated by a finite set $S_G$.
  Assume $G$ is undistorted.
    There exists a constant $K$ such that the following statements hold.
  Given nontrivial $g \in G$, let $\walls(1,g)=\mathcal{K}_1\sqcup\dots\sqcup\mathcal{K}_D$ be a partition of $\walls(1,g)$ into $D\leq \dim(\NR)$--many disjoint chains, and let $E_0$, $E_1$, and $E_2$ be the partition of a combinatorial geodesic $\alpha_g$ of $g$ as discussed above.
  For each $\alpha_i \in E_2$, let $P_i$ be the associated pre-wide parabolic subgroup, let $g_i \in G$ be the associated subword of a fixed $S_G$--minimal word representing $g$, and let $Q_i$ be the conjugate of $P_i$, as given by \fullref{lem:H_subword}.
  \begin{enumerate}
  \item $|E_0|+|E_1|+|E_2|\leq K\dcone(1,g)$\label{item:H_subword_1}
    \item $ \frac{|g|_G}{KD} \le |E_0| + |E_1| + \sum_{i=1}^{|E_2|} d_{Q_i}(1,g_i)$\label{item:H_subword_2}
  \end{enumerate}
\end{proposition}
\begin{proof}
Recall, from \fullref{cor:greedy_are_qgeod}, that realizations of greedy paths are uniform quasigeodesics in $\coneoff$.
  Since $G$ is undistorted, by hypothesis, we can choose a constant $K$ such that (we can dispense the additive error by enlarging $K$):
  \[ \frac{1}{K} |g|_S \le |g|_G \le K |g|_S\quad \text{ and } \quad |\realgreedy{\alpha}_g| \le K\dcone(1,g). \]
  Since $|E_0|+|E_1|+|E_2|=|\realgreedy{\alpha}_g|$, the latter gives Item~\eqref{item:H_subword_1}.

  For Item~\eqref{item:H_subword_2}, we first observe:
  \[ |g|_S = |\alpha_g| = |E_0| + \sum_{\alpha \in E_1} |\alpha| + \sum_{\alpha \in E_2}  |\alpha| \]
  Recall that $\alpha\in E_1$ means $\alpha$ is a wide subsegment of $\alpha_g$ for which $|\walls(\alpha)\cap\mathcal{K}_\alpha|=1$, which, by
  \eqref{eqn:max}, implies $|\alpha|=|\walls(\alpha)|\leq D$.
For $\alpha_i \in E_2$, \fullref{lem:H_subword}~\eqref{item:H_subword_large_projection} says $|\alpha_i| = |\walls(\alpha_i)| \le D d_{Q_i}(1,g_i)$.
This yields the second statement:
\[ \frac{|g|_G}{K} \le |g|_S \le |E_0| + D |E_1| + D \sum_{i=1}^{|E_2|} d_{Q_i}(1,g_i) \le D \left( |E_0| + |E_1| + \sum_{i=1}^{|E_2|} d_{Q_i}(1,g_i) \right) \qedhere\] 
  
\end{proof}

\begin{proof}[Proof of Theorem \ref{thm:chara_stable}]
  Every stable subgroup $G \le W$ is necessarily undistorted and purely loxodromic.
  We prove the opposite direction: suppose $G$ is undistorted and purely loxodromic.
  Since finitely generated Coxeter groups are virtually torsion-free \cite[Corollary~6.12.12]{davisbook}, every subgroup of $W$ has a finite-index torsion-free subgroup.
  Since stability is a virtual property, up to replacing $G$ by a finite-index subgroup, we may assume every nontrivial element of $G$ has infinite order, hence is loxodromic.
  Supposing that $G$ is not stable, we derive a contradiction.

  Fix a finite generating set $S_G$ of $G$.
  By \fullref{cor:qie_stable}, if $G$ is not stable then the orbit map of $G$ into $\coneoff$ is not a quasiisometric embedding. It is Lipschitz, so the failure is due to elements of $G$ being too short in $\coneoff$; specifically,  there exists a sequence $h_n \in G$ such that:
  \[ r_n := \frac{|h_n|_G}{\dcone(1,h_n)} \to \infty\] 

  The goal is to show that once $r_n$ is large enough we can apply  \fullref{prop:criterion} to $h_n$.
  
  Let $K \ge 1$ be the constant of \fullref{prop:H_subword}.
  Choose  $0<c < \frac{1}{3(\lambda_G+2)K \dim(\NR)^2}$, and let $R:=R(c) \ge 1$ be the constant of \fullref{prop:criterion}.
  
  For any $n$ large enough that $r_n > 4K^3 R \dim(\NR)$, take $g:=h_n$.
  We use the notation of \fullref{prop:H_subword} with respect to this $g$.
  Recall that $\walls(1,g)$ was partitioned into $D$--many disjoint chains, for some $D$ depending on $g$ with $D\leq \dim(\NR)$.
  We claim:
  \begin{equation} \label{eqn:big_projection}
      \frac{|g|_G}{2KD} < \sum_{i=1}^{|E_2|} d_{Q_i}(1,g_i)
    \end{equation}
    This follows from \fullref{prop:H_subword}, because by Item~\eqref{item:H_subword_2}, the alternative is that $|E_0|+|E_1|$ is at least half of $|g|_G/KD$, but then using Item~\eqref{item:H_subword_1} we could derive the following contradiction:
  \[ 4K^3R \dim(\NR)<r_n = \frac{|g|_G}{\dcone(1,g)} \le \frac{K|g|_G}{|E_0|+|E_1|+|E_2|}\leq \frac{K|g|_G}{|g|_G/2KD+|E_2|} \le 2K^2 D \leq 2K^2 \dim(\NR)\] 
  Consider the collection of subwords $g_i$ of $g$ provided by \fullref{lem:H_subword}.
  Partition their indices:
  \[ J^+ := \{ i\mid  d_{Q_i}(1,g_i) \ge c |g_i|_G \} \qquad  J^- := \{ i \mid  d_{Q_i}(1,g_i) < c |g_i|_G \} \]
  In view of (\ref{eqn:big_projection}), we claim:
  \begin{equation}
    \label{eq:20}
    \frac{|g|_G}{4KD} \le \sum_{i\in J^+} d_{Q_i}(1,g_i)
  \end{equation}
  If not, then by \fullref{lem:H_subword}~\eqref{item:H_subword_controlled_sum}:
  \[ \frac{|g|_G}{4KD} \le \sum_{i\in J^-} d_{Q_i}(1,g_i) <  \sum_{i\in J^-} c|g_i|_G \le c(1+\lambda_G/2)D |g|_G \]
  This would give the contradiction: 
  \[3(\lambda_G+2)KD^2\leq 3(\lambda_G+2)K \dim(\NR)^2< 1/c <  2(\lambda_G+2)KD^2\]

  Next, suppose $|g_i|_G < R$ for every $i \in J^+$.
  Then, since  $d_{Q_i}(1,g_i) \le \dcomb(1,g_i) \le K|g_i|_G < KR$, the bound \eqref{eq:20} gives the estimate:
  \[ \frac{|g|_G}{4KD} \le \sum_{i\in J^+} d_{Q_i}(1,g_i) \le K R |J^+| \le KR|E_2|\]
This estimate combined with \fullref{prop:H_subword}~\eqref{item:H_subword_1} gives the contradiction:
\[ 4K^3R \dim(\NR)<r_n = \frac{|g|_G}{\dcone(1,g)} \le  \frac{4K^2DR|E_2|}{\dcone(1,g)} \leq \frac{4K^3DR|E_2|}{|E_0|+|E_1|+|E_2|}\leq 4K^3 DR \leq 4K^3R \dim(\NR)\]
Conclude that there exists some $i \in J^+$ such that $|g_i|_G \ge R$.
By definition of $J^+$, we have $d_{Q_i}(1,g_i) \ge c|g_i|_G$.
Since $Q_i$ is a conjugate of $P_i$, and $P_i$ is pre-wide, $Q_i$ is also pre-wide.
Thus, \fullref{prop:criterion} says that $G$ contains a nontrivial element normalizing some pre-wide parabolic subgroup $P$.
Such an element fixes the cone vertex $\rho(P)$ of $\coneoff$ corresponding to the unique wide parabolic subcomplex whose stabilizer is $\normalizer_W(P)$, as described in \fullref{lem:rho_map_from_prewide_to_cone_vertices}, so
this contradicts the assumption that every non-trivial element of $G$ acts loxodromically on $\coneoff$.
\end{proof}